\documentclass[final]{report}

\usepackage[utf8]{inputenc}
\usepackage[T1]{fontenc}
\usepackage[a4paper]{geometry}

\usepackage{pdflscape}
\usepackage{amsrefs}

\RequirePackage[orthodox]{nag}
\usepackage{microtype}
\usepackage{booktabs}
\usepackage{tabularx}
\usepackage{ltablex}
\usepackage{multicol}
\usepackage{mathptmx} 
\usepackage{mathtools}
\usepackage{amssymb}
\usepackage{amsthm}
\usepackage{stmaryrd}
\usepackage{enumitem}

\makeatletter
\renewcommand\tableofcontents{\@starttoc{toc}}
\makeatother

\usepackage{bbding}
\usepackage{nicefrac}
\usepackage{xspace}
\usepackage[table,dvipsnames]{xcolor}
\definecolor{mylightblue}{rgb}{0.93,0.95,1.0}
\definecolor{mylightgray}{rgb}{0.95,0.95,0.95}

\definecolor{blueish}{RGB}{71,98,168}
\usepackage{tikz}
\usepackage{pgfplots}
\pgfplotsset{compat=1.14}
\usetikzlibrary{external}
\usetikzlibrary{calc}
\usetikzlibrary{positioning}
\usepackage[framemethod=TikZ]{mdframed} 

\usepackage{ifthen}
\usepackage[position=top,labelformat=empty]{subfig}

\usepackage{algpseudocode}
\usepackage[frozencache]{minted} 

\usepackage[colorlinks=true,linkcolor=blueish,citecolor=TealBlue]{hyperref}
\hypersetup{breaklinks=true}
\usepackage[obeyFinal]{todonotes}

\usepackage{pifont}
\newcommand{\cmark}{\ding{51}}%
\newcommand{\xmark}{\ding{55}}%

\usepackage{makeidx}
\makeindex

\theoremstyle{plain}
\newtheorem{Thm}{Theorem}
\newtheorem{Rmk}[Thm]{Remark}
\newtheorem*{Rmk*}{Remark}

\newtheorem{Cor}[Thm]{Corollary}
\newtheorem*{Cor*}{Corollary}
\newtheorem{Lem}[Thm]{Lemma}
\newtheorem*{Lem*}{Lemma}
\newtheorem{Pro}[Thm]{Proposition}
\newtheorem*{Pro*}{Proposition}
\newtheorem{Heu}[Thm]{Heuristic}
\numberwithin{Thm}{section}

\usepackage{collect}
\newtheorem{Que}{Question}

\makeatletter
 \definecollection{qu}
 \newenvironment{Qu}{%
   \@nameuse{collect*}{qu}{%
     \begin{Que}
   }{%
     \end{Que}
   }{}{}%
}{%
  \@nameuse{endcollect*}%
}
\makeatother

\theoremstyle{definition}
\newtheorem{Def}[Thm]{Definition}

\newcommand{\CRM}{CRM\xspace} 

\newcommand{\BISH}{\textsf{BISH}\xspace}
\newcommand{\INT}{\textsf{INT}\xspace}
\newcommand{\RUSS}{\textsf{RUSS}\xspace}
\newcommand{\CLASS}{\textsf{CLASS}\xspace}

\newcommand{\NN}{\mathbb{N}}
\newcommand{\RR}{\mathbb{R}}
\newcommand{\QQ}{\mathbb{Q}}

\DeclareMathOperator{\id}{\mathrm{id}}
\newcommand{\AWWKL}{\textrm{Anti-WWKL}\xspace}
\newcommand{\AFAN}{\textrm{Anti-FAN}\xspace}
\newcommand{\POS}{\hyperref[PR:POS]{\textrm{POS}}\xspace}
\newcommand{\POSp}[1]{\hyperref[PR:POSp]{\textrm{POS}_{#1}^{\textrm{pw}}}\xspace}

\newcommand{\FAN}{\hyperref[PR:FAN]{\textrm{FAN}}}
\newcommand{\FAND}{\hyperref[PR:FAN]{\textrm{FAN}\texorpdfstring{\ensuremath{_{\Delta}}}{Î}}\xspace}
\newcommand{\FANP}{\hyperref[PR:FAN]{\textrm{FAN}\texorpdfstring{\ensuremath{_{\smash{\Pi^{0}_{1}}}}}{Î -0-1}}\xspace}
\newcommand{\FANc}{\hyperref[PR:FAN]{\textrm{FAN}\texorpdfstring{\ensuremath{_{c}}}{c}}\xspace}
\newcommand{\FANst}{\hyperref[PR:FANst]{\texorpdfstring{\textrm{FAN}\ensuremath{_{\textrm{stable}}}}{FAN-stable}}\xspace}
\newcommand{\FANf}{\hyperref[PR:FAN]{\texorpdfstring{\textrm{FAN}\ensuremath{_{\textrm{full}}}}{FAN-full}}\xspace}
\newcommand{\BDN}{\hyperref[PR:BDN]{\textrm{BD-N}}\xspace}
\newcommand{\BD}{\hyperref[PR:BD]{\textrm{BD}}\xspace}
\newcommand{\wBDN}{\hyperref[PR:wBDN]{\textrm{wBD-N}}\xspace}
\newcommand{\LEM}{\hyperref[PR:LEM]{\textrm{LEM}}\xspace}
\newcommand{\WLEM}{\hyperref[PR:WLEM]{\textrm{WLEM}}\xspace}

\newcommand{\LLPO}{\hyperref[PR:LLPO]{\textrm{LLPO}}\xspace}
\newcommand{\LLPOn}[1]{\hyperref[PR:LLPOn]{\texorpdfstring{\ensuremath{\textrm{LLPO}_{#1}}\xspace}{LLPOn}}}
\newcommand{\LLPOnp}[1]{\hyperref[PR:LLPOn]{\texorpdfstring{\ensuremath{\textrm{LLPO}_{#1}^\prime}\xspace}{LLPOn}}}

\newcommand{\WLPO}{\hyperref[PR:WLPO]{\textrm{WLPO}}\xspace}
\newcommand{\LPO}{\hyperref[PR:LPO]{\textrm{LPO}}\xspace}
\newcommand{\nWLPO}{\ensuremath{\neg\textrm{WLPO}}\xspace}
\newcommand{\nLPO}{\texorpdfstring{\ensuremath{\neg\textrm{LPO}}\xspace}{¬LPO}}
\newcommand{\WKL}{\hyperref[PR:WKL]{\textrm{WKL}}\xspace}
\newcommand{\WKLppp}{\textrm{WKL!!!}\xspace}
\newcommand{\WKLpp}{\textrm{WKL!!}\xspace}
\newcommand{\WKLp}[1]{\texorpdfstring{\hyperref[PR:WKLp]{\ensuremath{\textrm{WKL}^{\prime}(#1)}}\xspace}{'}}
\newcommand{\WWKL}{\hyperref[PR:WWKL]{\textrm{WWKL}}\xspace}
\newcommand{\UCT}{\hyperref[PR:UCT]{\textrm{UCT}}\xspace}
\newcommand{\KT}{\hyperref[PR:KT]{\textrm{KT}}\xspace}
\newcommand{\SinC}{\hyperref[PR:SinC]{\textrm{SC}}\xspace}
\renewcommand{\SS}{\hyperref[PR:SS]{\textrm{SS}}\xspace}
\newcommand{\iSS}{\hyperref[PR:SS]{\textrm{iSS}}\xspace}
\newcommand{\KTs}[1]{\ensuremath{\textrm{KTs}(#1)}\xspace}
\newcommand{\AS}[2]{\hyperref[PR:AS]{\ensuremath{\textrm{AS}_{#1}^{\, #2}}}}
\newcommand{\iAS}{\hyperref[PR:iAS]{\ensuremath{\textrm{iAS}}}\xspace}

\newcommand{\ASL}[1]{\hyperref[PR:ASL]{\ensuremath{\textrm{AS}_{#1}^{L}}}}
\newcommand{\MPv}{\texorpdfstring{\ensuremath{\hyperref[PR:MPv]{\textrm{MP}^{\lor}}}\xspace}{MPv}}
\newcommand{\IIIa}{\texorpdfstring{\ensuremath{\hyperref[PR:IIIa]{\textrm{III}_{a}}}\xspace}{IIIa}}
\newcommand{\MP}{\hyperref[PR:MP]{\textrm{MP}}\xspace}
\newcommand{\WMP}{\hyperref[PR:WMP]{\textrm{WMP}}\xspace}
\newcommand{\WPFP}{\hyperref[PR:WPFP]{\textrm{WPFP}}\xspace}
\newcommand{\PFP}{\hyperref[PR:PFP]{\textrm{PFP}}\xspace}
\newcommand{\coWPFP}{\hyperref[PR:coWPFP]{\textrm{co-WPFP}}\xspace}
\newcommand{\coPFP}{\hyperref[PR:coPFP]{\textrm{co-PFP}}\xspace}

\newcommand{\KS}{\hyperref[PR:KS]{\textrm{KS}}\xspace}
\newcommand{\CC}{\hyperref[PR:CC]{\textrm{CC}}\xspace}
\newcommand{\WCN}{\hyperref[PR:WCN]{\textrm{WCN}}\xspace}
\newcommand{\MC}{\hyperref[PR:MC]{\textrm{MC}}\xspace}

\newcommand{\OI}[1]{\texorpdfstring{\hyperref[PR:OI]{\textrm{OI}{\ensuremath{_{#1}}}}\xspace}{OI}}

\newcommand{\DGP}{\hyperref[Pr:DGP]{\textrm{DGP}}\xspace}

\renewcommand{\mid}{\, \middle| \,}
\newcommand{\menge}[1]{\ensuremath{\left\{ #1 \right\}}}
\newcommand{\set}[2]{\mbox{\ensuremath{\left\{ \,#1 \mid #2 \,\right\}}}}

\newcommand{\ext}[1]{\llbracket #1  \rrbracket}
\newcommand{\fa}[2]{\forall {#1} : {#2}}
\newcommand{\ex}[2]{\exists {#1} : {#2}}
\newcommand{\Interior}[1]{\mathrm{Int}\left(#1\right)}
\newcommand{\Closure}[1]{\overline{#1}}
\newcommand{\Complement}[1]{{#1}^\prime}
\DeclarePairedDelimiter{\abs}{\lvert}{\rvert}

\DeclarePairedDelimiter{\bracks}{\lbrack}{\rbrack}

\newcommand{\co}{\colon\!}
\newcommand{\CS}{2^{\NN}}
\newcommand{\cS}{2^{\ast}}
\newcommand{\BS}{\NN^{\NN}}
\newcommand{\Ninf}{\NN_{\infty}}

\newcommand{\one}{111\ldots}
\newcommand{\zero}{000\ldots}

\newcommand{\ACC}{\hyperref[Sec:Choice]{\textrm{ACC}}\xspace}
\newcommand{\ADC}{\hyperref[Sec:Choice]{\textrm{ADC}}\xspace}

\newcommand{\pylist}[3]{#1[#2 \co #3]}
\newcommand{\define}[1]{\textit{#1}\index{#1}}
\newenvironment{principle}[2][]{\index{\ifx\newenvironment#1\newenvironment\else#1@\fi#2}\begin{mdframed}[leftmargin=2em,rightmargin=3em,skipabove=1em,skipbelow=1em, innerleftmargin=-1em,innerrightmargin=0em, nobreak=true, linecolor=blueish, linewidth=5pt, bottomline=false, topline=false, rightline=false, backgroundcolor=blueish!10] \begin{quote} \textbf{(#2)} \quad }{\end{quote}\end{mdframed}}

\newenvironment{bareprinciple}{\begin{mdframed}[leftmargin=2em,rightmargin=3em,skipabove=1em,skipbelow=1em, innerleftmargin=-1em,innerrightmargin=0em, nobreak=true, linecolor=blueish, linewidth=5pt, bottomline=false, topline=false, rightline=false, backgroundcolor=blueish!10] \begin{quote}}{\end{quote}\end{mdframed}}

\begin{document}

\title{Constructive Reverse Mathematics \\ {\small version 1.3}}
\author{Hannes Diener}
\date{April, 2020}
\maketitle

\newpage

\tableofcontents
\setcounter{chapter}{-1}

\chapter{Introduction}

\section{Constructive Mathematics}
Almost all proofs in traditional mathematics invoke the law of excluded middle (\LEM) at some point. Sometimes this use is as obvious as starting a proof of $\varphi$ by ``let us assume $\varphi$ is false''  followed by a derivation of a contradiction, and sometimes this use is  subtly hidden in  details, as it is in the usual interval-halving proof of the intermediate value theorem.\footnote{In each halving step one needs to decide whether at the midpoint $c$ we have $f(c) \geqslant 0$ or not.} 
One should remember though that \emph{any} application of \LEM comes at a price. Sometimes the price to pay is that what is really going on is obfuscated. Sometimes the price is as low as leading to inefficient programs, since a proof might use unbounded search. Often enough, though, the price to pay is that a proof provides \emph{no} algorithmic information at all. We believe that in all these cases it is a price we should not pay.

It is one of the criminally underrated insights of 20\textsuperscript{th} century mathematics that we can capture the idea of constructiveness not by adding layers of notions onto traditional mathematics but  simply by removing layers---namely barring the use of \LEM. Doing so one might expect to end up stuck in a barren mathematical landscape devoid of any interesting results; or in the words of Hilbert:
\begin{quote}
Taking away this law of excluded middle from the mathematician is about the same as taking away the telescope from the astronomer or forbid the boxer the use of his fists.\footnote{In the original German: ``Dieses Tertium non datur dem Mathematiker zu nehmen, w\"are etwa, wie wenn man dem Astronomen das Fernrohr oder dem Boxer den Gebrauch der F\"auste untersagen wollte.''}
\end{quote}
 
It was Bishop who showed us that this concern is unfounded, and that  mathematical life can thrive in the absence of \LEM. It can be, arguably, even a much richer existence than traditional mathematics, extending our horizons to places previously hidden by \LEM. Continuing in Hilbert's own analogy: Bishop discovered unknown galaxies without \emph{any} astronomical equipment.

\section{Constructive Reverse Mathematics}
The focus of reverse mathematics---as opposed to normal, everyday mathe\-matics---is \emph{not} to find what theorems we can prove from certain axioms, but to ask which axioms are also \emph{necessary} to prove certain theorems. As such the idea is neither new nor revolutionary. However, applying this approach systematically to some fragment of mathematics is a fairly recent development. Maybe the best known and most developed of these reverse approaches is ``Simpson style'' reverse mathematics \cite{Simpson:1999lr} (initiated by H.~Friedmann \cite{hF74}), whose goal is to examine which set existence axioms  need to be added to classical second order arithmetic to prove theorems in mathematics (where objects are coded by natural numbers).
Approaching from a very different point of view, but resulting in a somewhat similar hierarchy, is the theory of Weihrauch reducibility \cite{Brattka:2010fk}. The question there is: ``which theorems can be computably transformed into others?''

The aim of Constructive Reverse Mathematics is, to borrow Ishihara's description from \cite{hI06a},  
\begin{quote}
 "to classify [over intuitionistic logic] various theorems in intuitionistic, constructive recursive and classical mathematics by logical principles, function existence axioms and their combinations."
\end{quote}
Interestingly enough, all three of these schools of reverse mathematics  share common themes. For example, in all of these approaches one can identify a hierarchical level corresponding to \LPO and one to \WKL, which have a remarkable number of theorems being categorised into the same levels.

We will describe some of the foundational aspects of Constructive Reverse Mathematics in the next section, but before doing so, we would like to approach the question of why one would be interested in pursuing it. Naturally, most of these remarks also apply to other the kinds of reverse mathematics mentioned above.
\begin{enumerate}
	\item Probably the foremost aim of reverse mathematics is to gain insight into various theorems and principles, and in particular see whether certain assumptions are really necessary. Mathematicians generally try to be as economical with their assumptions as possible and dispense with unused ones.\footnote{This is very reminiscent of the principle of Chekhov's gun, as explained by the well-known playwright Anton Chekhov himself: ``Remove everything that has no relevance to the story. If you say in the first chapter that there is a rifle hanging on the wall, in the second or third chapter it absolutely must go off. If it's not going to be fired, it shouldn't be hanging there.''} Reverse results show that a theorem is optimal in this sense.
	\item Naturally, as with any other form of mathematics, the intrinsic challenge can be reason enough for pursuing it. In our opinion, the proofs and techniques in CRM are more diverse and can be more intricate than in standard Bishop style constructive mathematics. 
One reason for this is that, since the latter is very minimal in its assumptions most proofs are somewhat ``synthetic''---that is information gets combined straightforwardly to get some result. The absence of \LEM means  there are very few general, non-trivial disjunctions,\footnote{See \cite{hD12b} for some notable exceptions.} which means that proofs tend to be of a very ``linear shape.'' However, in CRM there are general, non-trivial disjunctions coming from the principles themselves,  thus breaking this linear pattern.

	\item (Constructive) reverse mathematics allows one to  explore the foundations of mathematics by ana\-lysing what consequences which axioms have. More importantly it allows us to do so without having to actually accept these principles. 
	\item As described in the next section, results in constructive reverse mathematics often involve a Brouwerian Counterexample. As such they  provide details of why certain statements cannot be proved constructively. Often, this gives us the exact information needed to ``constructivise'' a theorem. For example if a theorem is shown to be equivalent to the uniform continuity theorem (see Section \ref{Sec:UCT}) it is most likely to be true constructively if we assume that all functions involved are uniformly continuous. Similarly if an existence statement is equivalent to weak K\"{o}nig's Lemma (see Section \ref{Sec:LLPO}), experience has shown that the addition of the assumption that there is at most one solution frequently leads to a fully constructive theorem \cite{pS06}. Of course, these are only heuristics, but they seem to provide a very natural and fruitful way of constructivising classical results.
\end{enumerate}

\section{A very Short History of Constructive Reverse Mathematics}
It is generally accepted that the following 1984 result by Julian and Richman is the first result in constructive reverse mathematics \cite{wJ84}*{Theorem 2.4}.\footnote{A new, alternative, proof can be found in \cite{jB08b}.}
\begin{Pro*}
The fan theorem for decidable bars (\FAND) is equivalent to the statement that every uniformly continuous, positively valued function  $f:[0,1] \to \RR$ has a positive infimum. 
\end{Pro*} 
Julian and Richman noticed that in \RUSS (\cites{bK84,oA80}, see also Section \ref{Sec:varieties}) one could actually construct a positively valued, uniformly continuous function $f:[0,1] \to \RR$ with $\inf f =0$ (that is a uniformly continuous function $f:[0,1] \to \RR$ which, in particular, does not attain its minimum) and they were wondering  what principles \emph{exactly} were responsible for assuring the existence or ruling out the existence of such a strange object. 
After classifying this statement in a logical form it was easy to settle this question. 

Many principles we consider in \CRM these days were used long before 1984.  Historically, Brouwer started the tradition of using sequences such as 
\[ \alpha_n = \begin{cases}
	0 & a^{n+2} + b^{n+2} = c^{n+2} \text{ has no integer solutions} \\
	1 & a^{n+2} + b^{n+2} = c^{n+2} \text{ has an integer solution,} \\
\end{cases} \]
or
\[ \beta_n = \begin{cases}
	0 & 2n \text{ is the sum of two primes} \\
	1 & \text{otherwise} , \\
\end{cases} \]
to show that if a theorem $T$ constructively implies, for example, 
\begin{equation} \label{Eqn:taboo}
	\fa{n \in \NN}{\alpha_n = 0} \  \lor  \ \ex{n \in \NN}{\alpha_n = 1} \ ,
\end{equation}   
then there cannot be a constructive proof of $T$, since otherwise we would have a constructive proof of \ref{Eqn:taboo}. Since, under the BHK interpretation \cite{dD04}*{Chapter 5} of the logical connectives this implies that we have either proved Fermat's last theorem or found a counterexample, Brouwer rejected such theorems $T$. 

Sequences like $\alpha_n$ and $\beta_n$  might become obsolete for this purpose, since the underlying problem might get solved; indeed, since Andrew Wiles' 1995 proof we know that 
\[ \fa{n \in \NN}{\alpha_n = 0} \ .  \]
Most mathematicians believe that there will always be important, unsolved problems which we can use in the place of Fermat's last theorem, but that, in general, there is no algorithmic way, given an \emph{arbitrary} binary sequence $\gamma_n$, to decide whether
\[ \fa{n \in \NN}{\gamma_n = 0} \ \lor \  \ex{n \in \NN}{\gamma_n = 1}  \ . \]
It was Bishop who grandiosely named this principle the ``limited principle of omniscience'' (\LPO).

And indeed, one can see that from the 1970s on Brouwerian ``unstable'' examples of the kind
\[ T \implies \left( \fa{n \in \NN}{\alpha_n = 0} \  \lor \ \ex{n \in \NN}{\alpha_n = 1}  \right) \ , \]
were more and more replaced\footnote{There seems to be  a certain tendency to stick with the sort of specific Brouwerian counterexamples in papers with a general audience (for example the first half of \cite{mM89}), since paradoxically a classical mathematician will be more suspicious of being able to decide the Goldbach conjecture than of making the arbitrarily phrased decision in \LPO.} by ``stable'' ones of the form
\[ T \implies \left( \fa{\gamma \in \CS}{\left(  \fa{n \in \NN}{\gamma_n = 0} \  \lor  \ \ex{n \in \NN}{\gamma_n = 1} \right)} \right) \ ; \]
that is, by
\[ T \implies P  \ , \]
where $P$ is a general, constructively dubious statement such as \LPO.
Because of this, some authors have cited Brouwerian counterexamples as the first instances of \CRM. Strictly speaking, this is not \CRM since the forward direction of the classification is missing. Indeed, as Iris Loeb has argued in \cite{iL12}, Brouwer would have seen no point in proving $P \implies T$, and that these results therefore can not be claimed as the beginning of \CRM.
While agreeing that this is formally correct, we would like to point out that often the implication $P \implies T$ is very easy or even trivial and that most of the work and the interesting constructions are happening in the proof of $T \implies P$. So a Brouwerian counterexample $T \implies P$ is, with regards to mathematical content, more than half of a full equivalence. We would also like to add that if the point of an Brouwerian counterexample were only to show that some statement $T$ implied an unacceptable statement and is therefore unacceptable itself, then there would be no point in distinguishing between, say, \LPO and \LLPO and simply work with the weakest one. The fact that researchers were using different taboo-statements hints that they, at least implicitly, were thinking reversely.

The main group of results in Constructive Reverse Mathematics were proved from around 1988 on, with the---in our personal opinion---deepest results and notions (such as \BDN) being due to Hajime Ishihara. He has also authored the only overview of the area \cite{hI06a}, which is slightly harder to get hold of, but also contains more results than its predecessor \cite{hI04b}.

Apart from Ishihara, and the author of this thesis himself, many people have contributed to the area, and we hope to have cited most of their relevant articles. These people are, in  alphabetical order, J.~Berger, D.S.~Bridges, M.~Hendtlass, I.~Loeb, M.~Mandelkern, J.~Moschovakis, T.~Nemoto, F.~Richman, P.~Schuster, and  W.~Veldman.\footnote{We would like to point out explicitly that W.~Veldman, in particular, has many (sometimes unpublished) results in constructive reverse mathematics that have, unfortunately, not made it into this thesis.}
 
\section{Foundational Aspects}
In the tradition of Bishop-style constructive mathematics \cites{eB67,dB85,dBlV06} we will be working informally, in the same sense that most mathematicians work informally. That is not to say that we will be vague or imprecise, but that we are happy to skip details for the sake of readability and clarity of ideas. However, we are as sure as every other mathematician that all our results can be (almost mechanically) formalised in an appropriate system.
 
 In this vein we will abstain from choosing one of the many possible set-theoretic or type-theoretic foundations such as \cite{pA01}, or \cite{pML98}, or one of the more restricted formal systems that could serve as the basis for our endeavours such as W.~Veldmann's Basic Intuitionistic Mathematics (BIM) \cite{wV11}, or Heyting arithmetic in all finite types and related systems \cite{uK08}.

We believe that it is not necessary to choose a formal system, but at the same time believe that anybody preferring to work in one of these systems will be able to easily read our results and translate them into their preferred style. A good analogy to justify this approach might be to compare it to the use of pseudo-code to present an algorithm over choosing a specific programming language.

Choosing a particular framework also seems to be against the spirit of constructive mathematics:
As constructivists, we are very happy to live with the fact that for a real number $x$ we cannot decide whether $x \in [0,1]$ or $x \notin [0,1]$. It seems just as dubious to be able, on the meta-level, to answer the question of what, for example, a function $\NN \to \NN$ is. Notice that this does not stop us from developing interesting mathematics. Just as there are numbers which are definitely members of $[0,1]$ and numbers which are definitely not, there are things that we definitely believe are functions $\NN \to \NN$ such as primitive recursive functions, and things that we definitely cannot define a priori, such as a discontinuous function. Notice that, however, we do not rule out the existence of the latter. We would like to call this approach \define{humble foundations}. That is, we demand a maximal burden of proof for us when it comes to showing the existence of an object outright, but make minimal assumptions on arbitrary objects.

Even though there is no reason not to take a humble approach with a formal system, we believe that there is a danger of getting distracted, and, continuing the above example, switch from  assuming that every primitive recursive function $\NN \to \NN$ exists to assuming also the reverse, namely that these are the only functions. 

Related to this, there is also a danger of getting distracted and studying properties of the chosen formal system rather than doing mathematics. It is one of the enduring strengths of Bishop's approach  to skip past foundational details and ``do mathematics''. It also gave his results a  robustness to  survive changing foundational fashions.

There is one aspect in this sometimes heated debate about how much formalism is appropriate, where we have to concede that more precision is needed.  Even though the full axiom of choice is a definite constructive taboo, since it implies \LEM \cites{rD75,nG78}, traditionally the use of the axioms of countable and dependent choice has been  more or less tacitly accepted in Bishop style constructive mathematics. These are the following principles.

\begin{description}
	\item[\ACC] If $S$ is a subset of $\NN \times B$, and for each $n \in \NN$ there exists $b \in B$ such that $(n, b) \in S$, then there is a function $f: \NN \to B$ such that $(n,f(n)) \in S$ for each $n \in \NN$.
	\item[\ADC] If $a \in A$ and $S \subset A \times A$, and for each $x$ in $A$ there exists $y$ in $A$ such that $(x, y) \in S$, then there exists a sequence of elements $a_1, a_2,\dots$ of $A$ such that $a_1 = a$ and $(a_n, a_{n+1}) \in S$ for each $n \in \NN$.  
\end{description}

We conjecture that the reason why the use of these choice principles is such a divisive topic might be that there are two fundamentally different ways in which they get used in constructive mathematics:

\begin{itemize}
  \item Often \ACC or \ADC is applied to get a representation of an object.

  For example, given a real number $x$, we often need a binary sequence $(a_n)_{n \geqslant 1}$ such that 
\begin{align*}
	a_n = 0 & \implies x < \frac{1}{2^n} \ , \\
	a_n = 1 & \implies x > \frac{1}{2^{n+1}}  \ .
\end{align*}

One might assume that real numbers are Cauchy sequences of rationals with a fixed modulus of Cauchyness, or fast converging Cauchy sequences.  
Let us simply call these reals ``represented'' and denote the set of them by $\RR_r$. For these represented reals we can easily find a binary sequence $(a_n)_{n \geqslant 1}$ as above with minimal (primitive recursive) effort. 

 However, if one makes the minimal ``humble'' axiomatic assumption that the real numbers $\RR$ are not of a specific form but merely satisfy
 \[ \fa{x \in \RR, \ n \in \NN}{\ex{q \in \QQ}{ \abs*{x-q} < \frac{1}{2^n} }} \]
 then we cannot guarantee the existence a sequence $(a_n)_{n \geqslant 1}$ as above (see Section \ref{Sec:topmodels}). Of course this difference is minute, and under the assumption of \ACC we have $\RR  = \RR_r$.
 
Now, if we have a result that relies on the existence of  sequences $(a_n)_{n \geqslant 1}$ as above we can use \ACC to show it. Formally
 \[ \ACC \vdash \fa{x \in \RR}{ \dots }  \ ; \]
But if \ACC was only used to ensure that $\RR  = \RR_r$,  we could also restate our result as
 \[ \vdash \fa{x \in \RR_r}{ \dots }  \ . \]
 
These kinds of use of \ACC can therefore be seen as simply a matter of style and simplicity. Large parts of traditional Bishop style constructive mathematics could be rendered choice-free by switching to represented reals and making similar definitions and arguments for other objects such as point-wise continuous functions and so on.

  \item However, there are also uses of \ACC and \ADC of a more structural kind. The proof that \LLPO implies \WKL needs \ACC not because we do not assume that we have a nice representations of binary trees, but because we need to use \LLPO countable many times.

\end{itemize}

To gain more insights into phenomena of the second kind, attempts have been made to simply move choice principles into the list of principle studied \cite{jB12}. And indeed, the original plan for this thesis was to work choice-sensitive and distinguish, for example, between the sequential version of \LPO and the real version. However this quickly turned out to be too ambitious a project. The \hyperlink{Ch:BigPicture}{big picture} (Section \ref{Ch:BigPicture}) is already very complicated and the number of principles would multiply in the absence of choice.  Any attempt to do \CRM without the use of \ACC or \ADC would need to find a way to present  results in a way that highlights the interesting issues of the second kind and somehow manages to not give too much prominence to issues of the first kind.

\section{Overview and Plan}

Contrary to Simpson style reverse mathematics, in which most theorems fall into one of the ``big five'' categories,\footnote{Although recent work has shown that there are more than the big five.} there is a plethora of principles that have been considered in constructive reverse mathematics, with a quick count totalling about $17$ major ones. We believe that the presentation we will give is a sensible way to group them. If we consider the big three varieties \CLASS, \INT, \RUSS (see Section \ref{Sec:varieties}) there are seven possible combinations of these varieties such that a principle is true in at least one of them and possibly fails to hold in others.  Five of these combinations form our first five chapters.

\def\firstcircle{(90:0.4cm) circle (0.5cm)}
\def\secondcircle{(210:0.4cm) circle (0.5cm)}
\def\thirdcircle{(330:0.4cm) circle (0.5cm)}
\begin{tabularx}{0.7\textwidth}{X c}
& \\
\textbf{Chapter \ref{Ch:OmniPr}}: Omniscience principles which are true classically, but not in \INT or \RUSS. 
& 
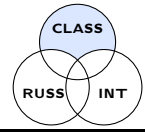
  \begin{tikzpicture}[font=\tiny,baseline=(current bounding box.north)]
    \begin{scope}
      \fill[mylightblue!200]\firstcircle;
      \fill[white] \secondcircle;
      \fill[white] \thirdcircle;
    \end{scope}
    \draw \firstcircle node[above] {\CLASS};
    \draw \secondcircle node [below left=-0.1cm and -0.3cm] {\RUSS};
    \draw \thirdcircle node [below right=-0.1cm and -0.2cm] {\INT};
  \end{tikzpicture}  
\\ \midrule
\textbf{Chapter \ref{Ch:MP}}: Markov's principle and its weakenings which are true in \CLASS and \RUSS. Actually, \WMP is true everywhere, but fits better into this chapter than into the chapter about \BDN.
& 
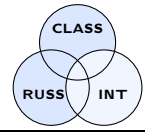
  \begin{tikzpicture}[font=\tiny,baseline=(current bounding box.north)]
    \begin{scope}
      \fill[mylightblue!200]\firstcircle;
      \fill[mylightblue!200] \secondcircle;
      \fill[mylightblue!80] \thirdcircle;
    \end{scope}
    \draw \firstcircle node[above] {\CLASS};
    \draw \secondcircle node [below left=-0.1cm and -0.3cm] {\RUSS};
    \draw \thirdcircle node [below right=-0.1cm and -0.2cm] {\INT};
  \end{tikzpicture}  \\ \midrule
\textbf{Chapter \ref{Ch:fan}}: The Fan theorems, which are true in \CLASS and \INT. 
& 
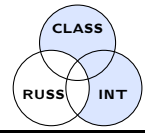
  \begin{tikzpicture}[font=\tiny,baseline=(current bounding box.north)]
    \begin{scope}
      \fill[mylightblue!200]\firstcircle;
      \fill[mylightblue!200] \thirdcircle;
      \fill[white] \secondcircle;
    \end{scope}
    \draw \firstcircle node[above] {\CLASS};
    \draw \secondcircle node [below left=-0.1cm and -0.3cm] {\RUSS};
    \draw \thirdcircle node [below right=-0.1cm and -0.2cm] {\INT};
  \end{tikzpicture}
\\ \midrule
\textbf{Chapter \ref{Ch:BDN}}: BD-N which is true in all varieties.
& 
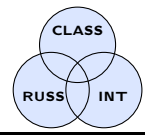
  \begin{tikzpicture}[font=\tiny,baseline=(current bounding box.north)]
    \begin{scope}
      \fill[mylightblue!200]\firstcircle;
      \fill[mylightblue!200] \thirdcircle;
      \fill[mylightblue!200] \secondcircle;
    \end{scope}
    \draw \firstcircle node[above] {\CLASS};
    \draw \secondcircle node [below left=-0.1cm and -0.3cm] {\RUSS};
    \draw \thirdcircle node [below right=-0.1cm and -0.2cm] {\INT};
  \end{tikzpicture}
\\ \midrule
\textbf{Chapter \ref{Ch:recside}}: Recursive principles, which are only true in \RUSS, and might be the strangest of all principles considered. 
& 
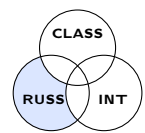
  \begin{tikzpicture}[font=\tiny,baseline=(current bounding box.north)]
    \begin{scope}
      \fill[mylightblue!200] \secondcircle;
      \fill[white]\firstcircle;
      \fill[white] \thirdcircle;
    \end{scope}
    \draw \firstcircle node[above] {\CLASS};
    \draw \secondcircle node [below left=-0.1cm and -0.3cm] {\RUSS};
    \draw \thirdcircle node [below right=-0.1cm and -0.2cm] {\INT};
  \end{tikzpicture}
\end{tabularx}
 
\bigskip

There are two combinations of varieties that are missing: Principles that are only true in \INT and principles that are only true in \INT and \RUSS. The reason is that even though there are principles that fall into these two categories---namely continuous choice (see Section \ref{Sec:Fan_collapse}) in the first and the negation of \WLPO (see Section \ref{Sec:nLPO}) in the second one---there are simply not enough equivalences known to warrant their having a chapter by themselves.

After the five principle-chapters, we talk about 

\begin{tabularx}{0.7\textwidth}{X c}

relationships between these principles in \textbf{Chapter \ref{Ch:relations}} and sketch \\ \midrule
 techniques to separate them in \textbf{Chapter \ref{Ch:separating}}. \\ \midrule
 
 We  finish with collecting some leftover but intellectually appetizing bits in the last \textbf{Chapter \ref{Ch:bitsnpieces}}.
 \end{tabularx}

\section{Version History}

The numbering of theorems and equations is kept consistent between minor version increases, for example from 1.0 to 1.1, but not between major version increases, for example from 1.9 to 2.0. Minor minor version changes, such as 1.1 to 1.1.2, are reserved for correction of typos and changes in presentation, but not mathematics. There is no guarantee that page numbers are consistent between any versions.

\begin{tabularx}{\textwidth}{l l X}
1.0 & 2018, Feb & original version, Habilitationsschrift, University of Siegen\\ \midrule
1.1 & 2018, Apr &  \begin{minipage}[t]{4in} \begin{itemize}[leftmargin=1.5 em,itemsep = 0pt]   \item Added section on \IIIa.  \item Question 6 answered positively (see Proposition \ref{Pro:FANc_equiv_Dini} and Proposition \ref{Pro:Qu6answered}). \item Added a short section on \BD. \item Added Proposition \ref{Pro:IVTunique} (IVT and \MPv). \end{itemize} \end{minipage} \\
\midrule
1.2 & 2018, May  & Rewrite of Subsection \ref{SubS:ReverseReverse}. \\
\midrule
1.3 & 2020, April  & \begin{minipage}[t]{4in} \begin{itemize}[leftmargin=1.5 em,itemsep = 0pt]   \item Clarification of the relationship of \nWLPO/\WLPO and the statements ``all functions are non-discontinuous''/``there exists discontinuous functions'': Propositions \ref{Pro:WLPO-Equiv-disc} and \ref{Pro:nWLPO_equiv_non-disc}.
\item Added Proposition 1.3.5.½
~on intersection of compact and located sets.
\item One more implication (every almost Cauchy sequence is Cauchy implies the Riemann Permutation Theorem) has been added in Section \ref{Sec:belowBDN}.
\end{itemize} \end{minipage}
\end{tabularx}


\chapter{Omniscience Principles}\label{Ch:OmniPr}
\section{\texorpdfstring{\LEM and \WLEM}{LEM and WLEM}} 
The possibly strongest of all omniscient principles is the \define{law of excluded middle}\footnote{Often, such as in \cite{Troelstra1988a}, it is called the \define{principle of excluded middle}---PEM.\index{PEM|see {\LEM}} Of course it also well known under its Latin name \define{tertium non datur}.} itself. 
\begin{principle}[LEM]{\LEM} \label{PR:LEM}
If $\varphi$ is any syntactically correct statement, then
\begin{equation*} 
 \varphi \lor \lnot \varphi \ . 
\end{equation*}
\end{principle}
Over intuitionistic logic \LEM is equivalent\footnote{To be a tiny bit pedantic, this is not an instance-wise equivalence. If $ \varphi \lor \lnot \varphi$, then for this $\varphi$ also $\lnot \neg \varphi \implies \varphi$. The converse \[ (\lnot \neg \varphi \implies \varphi) \implies  (\varphi \lor \lnot \varphi)\ , \]   however, does not hold in intuitionistic logic. Nevertheless, for an arbitrary formula $\varphi$ one can easily see that $\lnot \neg \left( \varphi \lor \lnot \varphi \right)$ holds. So if \ref{Eqn:stability} holds for \emph{any} formula, we can obtain $\varphi \lor \lnot \varphi$.} to \define{double negation elimination}, also known as \define{stability} or \define{proof by contradiction}, that is that for any  syntactically correct statement $\varphi$
\begin{equation} \label{Eqn:stability}
 \lnot \neg  \varphi \implies \varphi \ . 
 \end{equation}
\LEM's slightly weaker sibling is the \define{weak law of excluded middle}.
\begin{principle}[WLEM]{\WLEM} \label{PR:WLEM}
If $\varphi$ is any syntactically correct statement, then
\[ \lnot \varphi \lor \lnot \neg \varphi \ . \]
\end{principle}

We will start with the well-known and simple observation that constructively the notion of a set $S$ being inhabited, that is $\ex{x}{x\in S}$ is stronger than being non-empty, that is $\neg(\emptyset = S)$; in fact the equivalence of both notions is equivalent to the law of excluded middle. 
\begin{Pro} \label{Pro:LEM_equiv_inhabited}
\LEM is equivalent to the following statement.

For all  $S \subset \menge{0}$
\[ \lnot (S = \emptyset)  \implies S \neq \emptyset \ . \] 
\end{Pro}
\begin{proof}
\LEM implies that either $0 \in S$ or $0 \notin S$. The second alternative is ruled out, since it implies $S=\emptyset$. Hence the first alternative holds and we are done. 

Conversely let $\varphi$ be any syntactically correct closed statement and consider the set 
\[ S = \set{x}{ x = 0 \land (\varphi \lor \lnot \varphi)} \ . \]
Then the assumption that $S=\emptyset$ leads to a contradiction and thus, $S \neq \emptyset$. So there exists $x\in S$ such that $x=0$ and more importantly  $\varphi \lor \lnot \varphi$.
\end{proof}

The construction of the set $S$ above is an instance of a very common trick in the toolbox of a constructive (reverse) mathematician. It can be found in many different variations---see for example the next proposition. Interestingly enough, in our experience, mathematicians without a strong background in formal logic are very uncomfortable when they first encounter this construction and will either dispute its validity or at best judge it ``pathological''. We  assure the reader that there is, though, no problem from a set-theoretic viewpoint as long as one has some form of set comprehension at one's disposal, which is the case in the big (constructive) mainstream set theories (ZF, IZF, CZF). 

It is also worth highlighting that $S$ is actually a family of sets $S_\varphi$, and that the fact that non-emptyness does not imply inhabitedness is not due us actually having a specific counterexample, but to it failing in a uniform fashion. Nevertheless it is common and productive to think of this sort of family of sets as one concrete example of a set, albeit a fairly fuzzy one.

A variation of this construction can also be seen in the next Proposition.
\begin{Pro}
\LEM is equivalent to the statement that the supremum\footnote{Notice that we use Bishop's definition of the supremum \cite{dB85}*{Chapter 2, Defintion 4.2}, which constructively differs from the usual least upper bound one. In fact the set from this proposition's proof shows that both definitions being equivalent is equivalent to \LEM.} of every bounded, inhabited subset of reals exists.
\end{Pro}
\begin{proof}
One direction is a well-known result (or even an axiom) in basic, classical analysis. For the other direction consider the set 
\[ S = \menge{0} \cup \set{x}{ x = 1 \land (\varphi \lor \lnot \varphi)}  \ . \qedhere \]
\end{proof}
Notice that for the set $S$ of the previous proof we can actually show that the supremum cannot be distinct from $1$; in other words the supremum of $\lnot \neg S$ exists. However, as the part \ref{Equiv2WLEM} of the next proposition will show, this is not always possible. Part \ref{Equiv3WLEM} improves upon a result of Mandelkern \cite{mM82}, who showed that it implies \WLPO.
\begin{Pro} \label{Pro:equival_WLEM}
The following are equivalent to \WLEM
\begin{enumerate}
\item  \label{Equiv2WLEM} If $S$ is a bounded, inhabited subset of real number, then $\sup \lnot \neg S$ exists.
\item \label{Equiv3WLEM} Whenever two inhabited open subsets $U,V$ of a bounded interval are disjoint there exists a point $x \notin U \cup V$.
\end{enumerate}
\end{Pro}
\begin{proof}
We will first show that \WLEM implies \ref{Equiv2WLEM}. To this end let $S \subset \RR$ and $B\in \NN$ be such that \[ \fa{s \in S}{-B \leqslant s \leqslant B} \ . \]
By \cite{dB85}*{Proposition 4.3, Chapter 1} it suffices to show that $\lnot \neg S$ is order located; that is for all $a,b \in \RR$ with $a<b$
\[  \left( \lnot \fa{s \in S}{s \leqslant a} \right) \ \lor  \  \left( \fa{s \in S}{s < b} \right) \ ; \] but this follows easily from \WLEM applied to the formula $\fa{s \in S}{s \leqslant a}$.

It is also easy to see that \ref{Equiv2WLEM} implies \ref{Equiv3WLEM}: if $U$ and $V$ are as stated, then  $x=  \sup \lnot \neg U$ exists, by \ref{Equiv2WLEM}. Since $U$ and $V$ are open, both $x \in U$ as well as $x \in V$ lead to contradictions, and hence $x \notin U \cup V$. 

Finally let $\varphi$ be any syntactically correct statement. Let
\[ U = \left(0,\frac{1}{3}\right) \cup \set{x \in \left(0,\frac{2}{3}\right)}{\varphi} \]
and
\[ V =  \set{x \in \left(\frac{1}{3}, 1\right)}{\lnot \varphi} \cup \left(\frac{2}{3},1\right)  \  . \]
Then $0 \in U$ and $1 \in V$ so they are inhabited. They are also easily seen to be  open and disjoint.  
Now assume there exists $x \notin U \cup V$. If $x < \frac{2}{3}$, then $\lnot \varphi$ leads to the contradiction $x \in V$, so $\lnot \neg \varphi$.  Similarly $x > \frac{1}{3}$ implies $\lnot \varphi$. Hence \WLEM holds.
\end{proof}

The rules known as \define{De'Morgan's laws} are
\begin{equation} \label{DM1}
	\tag{DM1} \lnot (\varphi \land \psi) \iff \lnot \varphi \lor \lnot \psi \ ,
\end{equation}
and
\begin{equation} \label{DM2}
	\tag{DM2} \lnot (\varphi \lor \psi) \iff \lnot \varphi \land \lnot \psi \ .
\end{equation}
where $\varphi$ and $\psi$ are syntactically correct statements. It is easy to see that \ref{DM2} is provable in intuitionistic logic. The same is true for the direction from the right to the left of \ref{DM1}. However, we have
\begin{Pro} \label{Pro:WLEM_equiv_DM1}
\WLEM is equivalent to \ref{DM1}.
\end{Pro}
\begin{proof}
First assume \ref{DM1} and let $\varphi$ be arbitrary. Since $\neg(\lnot \varphi \land \varphi)$ is provable in intuitionistic logic we have $\lnot \neg \varphi \lor \lnot \varphi  $; that is \WLEM holds. 

Conversely assume that $\neg(\varphi \land \psi)$. By \WLEM either $\lnot \varphi$ or $\lnot \neg \varphi$. It is easy to see that in the second case the assumption that $\psi$ holds leads to a contradiction. Hence $\lnot \psi$ and we are done.	
\end{proof}

Sometimes De Morgan's laws are stated slightly differently---in its ``substitution form''---in which case we get an equivalence to \LEM:
\begin{Pro} \label{Pro:WLEM_equiv_DM12p}
\LEM is equivalent to the following versions of De Morgan's law
\begin{equation} \label{DM1p}
	\tag{DM1$^\prime$} \lnot (\lnot \varphi \land \lnot \psi) \iff  \varphi \lor  \psi
\end{equation}
as well as to
\begin{equation} \label{DM2p}
	\tag{DM2$^\prime$} \lnot (\lnot \varphi \lor \lnot \psi) \iff  \varphi \land  \psi
\end{equation}	
\end{Pro}
\begin{proof}
Of course, \LEM implies both versions. Conversely let $\varphi$ be a syntactically correct statement. In intuitionistic logic we have $\neg(\lnot \varphi \land \lnot \neg \varphi)$, so \ref{DM1p} implies $\varphi \lor \lnot \varphi$. We also have 
\[ \lnot \left ( \lnot ( \varphi \lor \lnot \varphi  ) \lor \lnot ( \varphi \lor \lnot \varphi  )   \right) \ , \]
so \ref{DM2p} also implies $\varphi \lor \lnot \varphi$. Thus in both cases \LEM holds.
\end{proof}

There are quite a few other basic logical principles equivalent to \LEM over intuitionistic logic. (See, however, Section \ref{Sec:PimplQveeQimplP} for $(\varphi \implies \psi) \lor (\psi \implies \varphi)$)

\begin{Pro} \label{Pro:paradoxes_of_material_impl}
The following are equivalent to \LEM
\begin{enumerate}
  \item \label{LEMequiv1} $((\varphi \implies \psi) \implies \varphi) \implies \varphi$  \quad (Peirce's law)
  \item \label{LEMequiv1b} $(\varphi \implies \vartheta) \implies (((\varphi \implies \psi) \implies \vartheta) \implies \vartheta)$  \quad (generalisation of Peirce's law)

  \item \label{LEMequiv2} $(\lnot \varphi  \implies \varphi) \implies \varphi $
  \item \label{LEMequiv3} $(\psi \implies \varphi) \lor (\varphi \implies \vartheta)$ (Linearity)
  \item \label{LEMequiv4} $(\varphi \implies \psi) \lor \lnot \psi$
  \item \label{LEMequiv5} $(\lnot (\varphi \implies \psi)) \implies (\varphi \land \lnot \psi)$ \quad (the counterexample principle)
\end{enumerate}
\end{Pro}
\begin{proof}
	Clearly all of these can be proved in classical logic, i.e.\ follow from \LEM. Furthermore \ref{LEMequiv2} is a special case of \ref{LEMequiv1} ($\psi \equiv \bot$), which in turn is a special case of \ref{LEMequiv1b}. To see that \ref{LEMequiv2} implies \LEM we will show that it implies \ref{Eqn:stability}. Clearly $\lnot \neg \varphi$ implies $(\lnot \varphi \implies \varphi)$, and hence $\varphi$. To see that \ref{LEMequiv3} implies \LEM, we simply take $\psi \equiv \top$ and $\vartheta \equiv \bot$. Similarly \ref{LEMequiv4} implies \LEM by choosing $\varphi \equiv \top$. Lastly, \ref{LEMequiv5} implies \ref{Eqn:stability} and therefore \LEM if one chooses $\psi \equiv \bot$ (and weakens the consequence).
\end{proof}

\section{\texorpdfstring{\LPO}{LPO}}
The ubiquitousness of the \define{limited principle of omniscience} (\LPO) in analysis might only be rivalled by \LLPO's. This is mainly due to the fact, that real numbers and sequences feature prominently in analysis, and \LPO tells us everything we want to know about both.
\begin{principle}[LPO]{\LPO}  \label{PR:LPO}
For every binary sequence $(a_n)_{n \geqslant 1}$ we can decide whether \[ \fa{n \in \NN}{a_n = 0}  \lor \ex{n \in \NN}{a_n = 1}  . \]
\end{principle}

\subsection{Basic equivalencies of \texorpdfstring{\LPO}{LPO}}
Mostly taken directly from \cite{hI06} are the following equivalences:
\begin{Pro} \label{Thm:LPO-equivs} The following are equivalent to \LPO
\begin{enumerate}
\item \label{LPOEquiv2} $\fa{x \in \RR}{x < 0 \, \lor \, x = 0 \, \lor \, 0 < x}$
\item \label{LPOEquiv1} For every binary sequence $(a_n)_{n \geqslant 1}$ we can decide whether \[ \ex{n \in \NN}{\fa{n \geqslant N}{a_n = 0}}  \lor \ex{k_n \in \BS}{\fa{n \in \NN}{a_{k_n} = 1}}  . \]
\item \label{LPOEquiv4} Every bounded monotone sequence of real numbers converges.
\item \label{LPOEquiv5} (Bolzano-Weierstra\ss theorem) Every sequence in a compact (totally bounded and complete) set has a convergent subsequence.
\item \label{LPOEquiv6} Every sequence of closed subsets of a compact metric space with the finite intersection property has nonempty intersection.
\item \label{LPOEquiv7} Ascoli's Lemma.
\end{enumerate}
\end{Pro}
\begin{proof} 
The equivalence of \LPO to \ref{LPOEquiv2} is standard.
\ref{LPOEquiv1} obviously implies \LPO. Conversely we can show \ref{LPOEquiv1} by applying \LPO  countably many times:
using \LPO (and unique choice) construct a binary sequence $b_n$ such that
\begin{align*}
b_k = 0 & \implies \ex{n \geqslant k}{a_n = 1} \ , \\
b_k = 1 & \implies \fa{n \geqslant k}{a_n = 0}  \ .
\end{align*}
Now, using \LPO again, either $\ex{N \in \NN}{b_N = 1}$ or $\fa{k}{b_k = 0}$. In the first case $\fa{n \geqslant N}{a_n = 0}$. In the second case we can use dependent choice to find $k_n \in \BS$  such that $\fa{n\in \NN}{a_{k_n} = 1}$.

The equivalence of \LPO to \ref{LPOEquiv4} can be found in \cite{mM88a}, the one to \ref{LPOEquiv5} in \cite{dB02b},  the one to \ref{LPOEquiv6} in \cite{hI04b}, and the one to \ref{LPOEquiv7} in \cite{hD08}.
\end{proof}

\begin{Pro} \label{Pro:LPO<->ratveeirrat} \LPO is equivalent to the following statements: 
\begin{itemize}
\item Every real number is either rational or irrational. 
\item Every real number is either algebraic or transcendental. 
\item For all sequences $r_n$ of reals and for all $x \in \RR$ either there exists $n $ such that $x=r_n$ or $x \neq r_n$ for all $n \in \NN$.
\end{itemize}
\end{Pro}
\begin{proof} 
First notice that using \LPO and countable choice, we can, for every $n \in \NN$ and $x \in \RR$,  decide whether $x = r_n$ or $x\neq r_n$. So using (unique) countable choice there exists a binary sequence $(a_n)_{n \geqslant 1}$ such that $a_n=0 \iff \lnot (x\neq r_n) $. By \LPO either $a_n=0$ for all $n \in \NN$ or there exists $n\in \NN$ such that $a_n=1$. That is we can decide whether $x \neq r_n$ for all $n\in \NN$ or whether there exists $n \in \NN$ such that $x=r_n$. 

Conversely let $(a_n)_{n \geqslant 1}$ be a decreasing binary sequence. Define \[ a = \sum_{n \geqslant 1} \frac{a_n}{n!} \ . \]
Notice that if $a_n = 1$ for all $n \in \NN$, then $a = e$, and thus transcendental (irrational). If $a$ is algebraic (rational) there exists $\varepsilon > 0$ such that $\abs{e-a}> \varepsilon$. Hence we can easily find $n$ such that $a_n =0$. 
\end{proof}
\begin{Pro}
\LPO is equivalent to the statement that the transitive closure of a decidable relation $R \subset \NN \times \NN$ is, again, decidable.
\end{Pro}
\begin{proof}
Let $(a_n)_{n \geqslant 1}$ be a binary sequence. We define a decidable relation $R$ on $\NN$ by $(n,m) \in R$ for all $n,m \geqslant 2$ and $(1,n) \in R$ if $a_n=1$. Now it is easy to see that $(1,2) \in R^{+}$ (the transitive closure) if and only if there exists  $n$ such that $a_n=1$. 

Conversely let $R$ be a decidable relation on $\NN$ and assume that \LPO holds. Furthermore let $k,\ell$ be arbitrary natural numbers. Now define a binary sequence $(a_n)_{n \geqslant 1}$ by
\[
a_n = \begin{cases}
1 & \ex{k_{1}, \dots, k_{j} \leqslant n}{(k,k_{1}) \in R \land (k_{1},k_{2}) \in R \land \dots \land (k_{j},\ell) \in R}\\
0 & \text{otherwise.}  
\end{cases}
\]
Then, if there exists $n$ such that $a_n=1$ we have $(k,\ell) \in R^{+}$ and if $a_n=0$ for all $n \in \NN$ we have $(k,\ell) \notin R^{+}$.
\end{proof}

\begin{Pro} \label{Pro:LPO<->IPP} 
\LPO is equivalent to the statement, that if $A \cup B$ is infinite (i.e.\ there is an injection $\NN \to A \cup B$), then either $A$ or $B$ is infinite.
\end{Pro}
\begin{proof}
Let $f:\NN \to A \cup B$ an injection. Using countable choice we can assume that there is a binary sequence such that
\[ a_n = \begin{cases} 0 & \implies f(n) \in A \\ 1 & \implies f(n) \in B 
\end{cases}
\]
Using \LPO and countable choice iteratively we can fix another binary sequence $(b_n)_{n \geqslant 1}$ such that
\[ b_n = \begin{cases} 0 & \ex{k > n}{ a_{k} = 0} 
\\ 1 & \fa{k > n}{ a_{k} = 1}
\end{cases}
\]
One final application of \LPO yields that either there is $N\in \NN$ such that $b_n =1$, in which case $n \mapsto f(N+n+1)$ is an injection of $\NN$ into $B$, or $b_n =0$ for all $n \in \NN$, in which case we can find an injection of $\NN$ into $A$.
Conversely assume that $(a_n)_{n \geqslant 1}$ is an, without loss of generality increasing, binary sequence. Define the sets  $A = \set{n}{a_n=0 } $ and $B = \set{n}{a_n=1}$. Then $A \cup B = \NN$ is infinite and $A$ and $B$ are disjoint. Now if $A$ is infinite there cannot be a $n$ with $a_n=1$, and thus $a_n=0$ for all $n\in \NN$. If $B$ is infinite it is, in particular, inhabited and so we can find $n \in \NN$ with $a_n=1$.
\end{proof}
The last proposition iterated finitely many times also shows that \LPO is equivalent to the \define{infinite pigeonhole principle}.

\begin{principle}{IPP}
For every $f:\NN \to \menge{1, \dots, k}$ there exists an infinite set $A \subset \NN$ and $1 \leqslant K \leqslant k$ such that $f(x)=K$ for all $x\in A$.
\end{principle}

\begin{Pro}
\LPO is equivalent to the statement that every countable subset of the natural numbers is decidable.
\end{Pro}
\begin{proof}
It is clear that \LPO is enough to prove the statement. Conversely, if $(a_n)_{n\geqslant 1}$ is a binary sequence, then the set $A=\set{a_n}{n \in \NN}$ countable. If this set is decidable, then either $1 \in A$ or $1  \notin A$. In the first case there exists $n\in \NN$ with $a_n=1$. In the second case there cannot be such an $n$. Hence, in that case, $\fa{n \in \NN}{a_n=0}$.
\end{proof}

\subsection{Metastability} 

In a program suggested by  Tao \cite{tT07}, it is proposed to recover the ``finite'' (constructive) content of theorems by replacing them with logically (using classical logic) equivalent ones that can be proved by ``finite methods.'' For example, since there often is no way to attain the Cauchy condition it is replaced with the following notion of meta-stability. A sequence $(x_n)_{n \geqslant 1}$ in a metric space $(X,d)$ is called \emph{metastable} if 
\[ \fa{ \epsilon >0, f}{ \ex{m}{ \fa{ i,j \in [m, f(m)] }{ \; d( x_{i}, x_{j}) < \varepsilon }}} \ , \]
where $[k,\ell]$ denotes the set $\menge{k,k+1, \dots, \ell}$. 
One can trivially show  that a Cauchy sequence is metastable. However, as we will see metastability contains almost no constructive content.

As noted in \cite{jA12} every  non-decreasing sequence of reals bounded by $B \in \RR$ is metastable since it is impossible that $d (x_{m},x_{f(m)}) > \frac{\varepsilon}{2}$ for all $1 \leqslant m \leqslant \frac{2B}{\varepsilon}$.
The converse fails constructively, since we can show:
\begin{Pro}
\LPO is equivalent to the statement that every metastable, increasing sequence of rationals is bounded.
\end{Pro}
\begin{proof}
Assume $(x_n)_{n \geqslant 1}$ is non-decreasing and metastable. Then, using \LPO, we can, for every $n$ decide whether $n$ is an upper bound or not. Using \LPO again, we can thus either find an upper bound or  a function $f: \NN \to \NN$ such that $x_{f(n)} > x_n + 1$ for all $n \in \NN$. Without loss of generality we may assume that $f$ is increasing. Furthermore \[  d(x_n,x_{f(n)}) > 1 \ ; \]  a contradiction to the metastability, and hence $x_n$ is bounded.

Conversely, let $(a_n)_{n \geqslant 1}$ be a binary sequence that has, without loss of generality, at most one $1$. Now consider \[ x_n = \sum_{i=1}^n i a_{i} \ . \] It is easy to see that $x_n$ is metastable: if $f: \NN \to \NN$ is increasing, then either $a_{i} = 0$ for all $i \in [1, f(1)]$ or  $a_{i} = 0$ for all $i \in [f(2), f(f(2))]$. In both cases $x_{i}$ is constant.

Now if $x_n$ is bounded, there is $N \in \NN$ with $x_n < N$. If there was $i > N$ with $a_{i} = 1$, then $x_{i} = i > N$ which is a contradiction. Hence $a_{i}=0$ for all $ i > N$, that is we only need to check finitely many entries to see if $(a_n)_{n \geqslant 1}$ consists of $0$s or whether there is a term equalling $1$.
\end{proof}

One might thus hope that there is maybe a chance that  every bounded, metastable sequence converges. However, also this statement is equivalent to \LPO.

\begin{Pro}
\LPO is equivalent to the statement that every bounded, metastable sequence of reals converges.
\end{Pro}
\begin{proof}
Assume that \LPO holds and that $(x_n)_{n \geqslant 1}$ is a bounded and metastable sequence of reals.  
Since \LPO implies the Bolzano Weierstra\ss\xspace theorem (see Theorem \ref{Thm:LPO-equivs}) there exists $x \in \RR$ and $k_n \in \BS$ such that $x_{k_n}$ converges to $x$. Now let $\varepsilon >0$ be arbitrary and using \LPO fix a binary sequence $(\lambda_n)_{n \geqslant 1}$ such that 
\begin{align*}
\lambda_n = 0 & \implies \abs*{x-x_n} < \varepsilon \ ,  \\
\lambda_n = 1 & \implies \abs*{x-x_n} \geqslant \varepsilon \ .
\end{align*}
By Theorem \ref{Thm:LPO-equivs} either there exists $N$ such that $\lambda_n=0$ for all $n \geqslant N$ or there exists a strictly increasing $\ell_n \in \BS$ such that $\lambda_{\ell_n} =1$ for all $n \in \NN$. We will show that the second alternative is ruled out by the metastability: fix $M$ such that $\abs*{x_{k_n} - x} < \frac{\varepsilon}{2}$ for $n \geqslant M$ and hence 
\begin{equation} \label{Eqn:metstable}
\abs*{x_{k_n} - x_{\ell_n}} \geqslant \frac{\varepsilon}{2} \text{ for } n \geqslant M \ .
\end{equation}
 Now define $f: \NN \to \NN$ by $f(n) = \max \menge{k_{n+M},\ell_{n+M}}$.
Then $f$ is increasing. Since $(x_n)_{n \geqslant 1}$ is metastable there exists $m$ such that for all $i,j \in [m,f(m)]$ we have $\abs*{x_{i} - x_{j}} < \frac{\varepsilon}{2}$. Since $k_{m+M},\ell_{m+M} \in [m+M,f(m)]$ we get the desired contradiction to \ref{Eqn:metstable}.

Conversely let $(a_n)_{n \geqslant 1}$ be a binary sequence with at most one $1$. We will show that $(a_n)_{n \geqslant 1}$ is metastable. So let $f : \NN \to \NN$ an increasing function. Now either there exists $i \in [1,f(1)]$ such that $a_{i}=1$ or for all $i \in [1,f(1)]$ we have $a_{i}=0$. In the first case, since $(a_n)_{n \geqslant 1}$ has at most one $1$, for all $i \in [f(1)+1,f(f(1)+1)]$ we have $a_{i}=0$. In both cases there exists $m$ such that, regardless of $\varepsilon >0$, we have
 \[ \fa{i,j \in [m,f(m)]}{\abs*{a_{i} - a_{j}} = 0 < \varepsilon} \ ;  \]
 that is $(a_n)_{n \geqslant 1}$ is metastable. Now if this sequence converges it must converge to $0$. So there exists $N \in \NN$ such that for all $n \geqslant N$ we have $a_n=0$. So we only need to check finitely many indices $n \in \NN$ for $a_n=1$, and hence \LPO holds.
\end{proof}

\subsection{Hillam's theorem}
Hillam's theorem \cite{bH76} states that if $f:[0,1] \to [0,1]$ is a continuous map and one defines a sequence by choosing an arbitrary $x_{0} \in [0,1]$ and then taking $x_{n+1} = f(x_n)$ for all $n > 0$, then 
\[ \abs*{x_n - x_{n+1}} \to 0  \iff (x_n)_{n \geqslant 1} \text{ is Cauchy} \ .\]
The interesting direction here is the one from the left to the right, since the converse holds trivially. Notice that if the sequence in the theorem converges, it converges to a fix-point. Hillam's theorem is therefore very interesting, since it is a fix-point theorem not equivalent to \LLPO (cf.\ Section \ref{SSec:Minima}).
\begin{Pro} \label{Pro:Hillam}
Hillam's theorem is equivalent to \LPO. (Even for uniformly continuous functions).
\end{Pro}
\begin{proof}
It is easy to see that the non-constructive steps in original proof \cite{bH76} hold under the assumption of \LPO; and that therefore the latter is enough to prove Hillam's theorem. To prove the converse let $a \in \RR$ be a non-negative real number. Without loss of generality assume that $a<1$.  Now define $f: [0,1] \to [0,1]$ by 
\[ f(x) =  (1- a)x +a \ ; \]
 that is $f$ is the linear function through $(0,a)$ and $(1,1)$. Now define a sequence of real numbers by $x_n = f^n(0)$. It is easily shown by induction that $ x_n = 1 - (1-a)^n$ and that $x_n \leqslant x_{n+1}$. We want to show that $\abs*{x_n - x_{n+1}} \to 0$ as $n \to \infty$.  For an arbitrary $\varepsilon >0$ either $a < \varepsilon$ or $a > \frac{\varepsilon}{2}$. In the first case, since for any $x \in [0,1] $   \[ \abs*{x-f(x)} \leqslant a \ , \] for any $n \in \NN$ also $\abs*{x_n - x_{n+1}} < \varepsilon$. In the second case choose $N$ such that $(1-a)^n < \varepsilon$. Now for any $n \geqslant  N$ we get 
 \[ \abs*{x_n - x_{n+1}} \leqslant  x_{n+1} - x_n \leqslant 1 - x_n = (1-a)^n \leqslant (1-a)^n < \varepsilon \ . \] 
 In both cases there exists $N \in \NN$, such that $\abs*{x_n - x_{n+1}} < \varepsilon$ for all $n \geqslant N$; therefore, by assumption, $(x_n)_{n \geqslant 1}$ is a Cauchy sequence and hence converges to a point $x_{\infty}$. Either $x_{\infty} < 1$, in which case $a = 0$, or $x_{\infty} > 0$. In the second case we can find $N$ such that $x_n > x_{\infty}/2$. Then, by Bernoulli's inequality 
 \[ x_{\infty}/2  < x_n = 1 - (1-a)^n \leqslant 1 - 1 -  N(-a) = Na  \ . \] Hence $a > x_{\infty}/(2N)$, after dividing by $N$, and therefore $a >0$.
\end{proof}
\subsection{Cardinality}
It is a well known exercise in first year mathematics to prove that there is a bijection between $[0,1)$, and $(0,1)$ and that therefore both sets have the same cardinality.

\begin{Pro} \label{Pro:LPO_equiv_extbij}
\LPO is equivalent to the existence of a strongly extensional bijection\footnote{To be precise: by bijection we obviously mean surjective and injective. Furthermore, a function $f:X \to Y$ is injective, if $f(x) \neq f(y)$ whenever $x \neq y$. This is a stronger requirement than  $x=y$ whenever $f(x) = f(y)$. The latter is called \emph{weakly injective} by Bishop \cite{dB85}*{p. 63}.} $f:[0,1) \to (0,1)$.
\end{Pro}
\begin{proof}
It is easy to see that with the help of \LPO
\[ f(x) =\begin{cases} \frac{1}{2} & \text{ if } x = 0 \\
\frac{1}{n+1} & \text{ if } x = \frac{1}n \\
x & \text{ else}
\end{cases}
\]
is a well-defined, strongly extensional bijection from $[0,1) \to (0,1)$. 

Conversely assume that such a function $f$ exist. Then $0<f(0) <1$. Since $f$ is surjective there exist $a,b$ such that $0<f(a)<f(0)<f(b)<1$. Since $f$ is strongly extensional either $0<a<b$ or $0<b <a$. Without loss of generality assume the first. Setting $a_{1}=a$ and $b_{1} = b$ and using the usual interval-halving technique (notice that since $f$ is injective we have $f(x)<f(0) \lor f(x)>f(0)$ for  $x \neq 0$ ), we  construct sequences $(a_n)_{n \geqslant 1}$ and $(b_n)_{n \geqslant 1}$ such that
\begin{itemize}
\item $ 0< a_n \leqslant a_{n+1}<b_{n+1} \leqslant b_n$
\item $f(a_n) < f(0) < f(b_n)$
\item $\abs*{a_n - b_n} < \left(\frac{1}{2}\right)^n$
\end{itemize}
Both sequences converge to the same limit $z > 0$. Now, by injectivity, $f(z) \neq f(0)$, so either $f(z) > f(0)$ or $f(z) <  f(0)$. Without loss of generality assume the first alternative holds and let $\varepsilon = f(z) - f(0) > 0$. Using Ishihara's tricks (see \cite{hD12b} or the original \cite{hI91}) either $\abs*{f(a_n) - f(z)} < \varepsilon$ infinitely often or \LPO holds. The first alternative is ruled out, since
\[ \abs*{f(a_n) - f(z)} =  f(0) -f(a_n) + \varepsilon  > \varepsilon \ , \] 
 and hence \LPO holds.
\end{proof}

\section{\texorpdfstring{\WLPO}{WLPO}}
Slightly weaker than \LPO is the \emph{weak limited principle of omniscience}.
\begin{principle}[WLPO]{\WLPO} \label{PR:WLPO}
For every binary sequence $(a_n)_{n \geqslant 1}$ we can decide whether \[ \fa{n \in \NN}{a_n = 0}  \lor \lnot \fa{n \in \NN}{a_n = 0}  . \]
\end{principle}
It was also called $\forall$-$\mathrm{PEM}$ in \cite{Troelstra1988a}. \index{$\forall$-$\mathrm{PEM}$|see {\WLPO}}
\WLPO has  fewer equivalences than \LPO or \LLPO. The following are taken straight from \cite{hI06a}.  
\begin{Pro} \label{Pro:WLPO-Equiv}
The following are equivalent to \WLPO:
\begin{enumerate}
\item For all $x,y \in \RR$ we have $x \geqslant y$ or $\lnot (x \geqslant y) $.
\item For all $x,y \in \RR$ we have $x = y$ or $\lnot (x = y) $.
\end{enumerate}
\end{Pro}

\WLPO is also intrinsically linked with the existence of discontinuous \index{discontinuous} functions.\footnote{We call a function \(f:X \to Y \) between metric spaces \define{discontinuous}, if there exists \(x_n \to x\) and \(\varepsilon >0\) such that \(\rho(f(x_n),f(x)) > \varepsilon \) for all \( n \).} This is interesting from a purely logical form perspective, since it is one of the few places where a universal statement is equivalent to an existential one.
\begin{Pro*}[\textbf{\theThm{}½}] \label{Pro:WLPO-Equiv-disc}
The following are equivalent to \WLPO:
\begin{enumerate}
\item \label{WLPO_equiv_disc_CS} The existence of a discontinuous function from $\CS \to \menge{0,1}$.
\item \label{WLPO_equiv_disc_BS} The existence of a discontinuous function from $\BS \to \NN$.
\item \label{WLPO_equiv_disc_01} The existence of a discontinuous function $[0,1] \to \RR$. 
\end{enumerate}
\end{Pro*}
\begin{proof}
The proofs are well known or straightforward, for example the proof of the equivalence of \WLPO to \ref{WLPO_equiv_disc_BS} can be found in  \cite{mA02}*{Section 3.1} and the  proof of the equivalence of \WLPO to \ref{WLPO_equiv_disc_01} is spelled out in \cite{hD17c}. Interestingly enough the specific space hardly matters at all, as the following lemma, which also subsumes all of the cases above, shows.
\end{proof}
\begin{Lem*} $ $
\begin{enumerate}
  \item  If there exists a discontinuous function $f:X \to Y$, where $X$ is complete, then \WLPO holds.
  \item If $X$ has an accumulation point and $Y$ has two distinct points, then \WLPO implies that there is a discontinuous function $f:X \to Y$.
\end{enumerate}
\end{Lem*}
\begin{proof}
\begin{enumerate}
  \item Let $X$ be complete and \(f:X \to Y \) discontinuous. So there exists \(x_n \to x\) and \(\varepsilon >0\) such that \(\rho(f(x_n),f(x)) > \varepsilon \) for all \( n \). Now let $(a_n)_{n \geqslant 1}$ be a binary sequence that is, without loss of generality, increasing. Consider the sequence $(z_n)_{n \geqslant 1}$ defined by
\begin{equation*}
	z_n = \begin{cases}
		x & \text{if } a_n = 0 \\
		x_m & \text{if }  a_{n} = 1 \text{ and } a_{m}=1- a_{m+1} \ .
	\end{cases}
\end{equation*}
In the notation of \cite{hD17} this is just the sequence $(a \circledast x)$, which is therefore Cauchy. (This can also, straightforwardly, be established by a direct argument.) It therefore converges to a point $z \in X$.\footnote{This construction shows that this proposition could be generalised to the class of spaces that are \emph{complete enough}, as introduced in \cite{hD17}.} Notice that
\begin{align}
\label{Eqn:WLPOdisc1} 	\fa{n \in \NN}{a_n =0} & \implies z = x \\
\label{Eqn:WLPOdisc2}	\ex{n \in \NN}{a_n =1} & \implies z = a_m \ ,
\end{align}
where $m$ is chosen minimal, as above. Now either $\abs{f(z) - f(x)} > 0$ or $\abs{f(z) - f(x)} < \varepsilon$. In the first case we have $\lnot(z = x)$, which implies that $\lnot \fa{n \in \NN}{a_n =0}$ by \eqref{Eqn:WLPOdisc1}. In the second case we have $\lnot \ex{n \in \NN}{a_n =1}$ from \eqref{Eqn:WLPOdisc2}, so $\fa{n \in \NN}{a_n =0}$. Altogether \WLPO holds.
  \item Assume that $x\in X$ and $(x_n)_{n \geqslant 1}$  is a sequence of points  in $X$ all distinct from $x$ that are such that $\abs*{x_n-x} < \nicefrac{1}{n}$. Also assume that $y_1,y_2 \in Y$ are such that $y_1 \neq y_2$. Since we assume \WLPO the function $f:X \to Y$, given by 
\begin{equation*}
	f(t) = \begin{cases}
		y_1 & \text{if } t = x \\
		y_2 & \text{if }  \lnot(t - x)  \ ,
	\end{cases}
\end{equation*}
is well-defined. Since $\lnot(x_n = x)$ for all $n\in \NN$, we have $\abs*{f(x_n) -f(x)} =  \abs*{y_1 -y_2} > 0$ for all $n \in \NN$. That means $f$ is discontinuous.
\end{enumerate}
\end{proof}

The following is the analogue of Proposition \ref{Pro:LPO<->ratveeirrat}.
\begin{Pro} \WLPO is equivalent to the statement that every real number is either irrational or not. More general: if $r_n$ is a sequence of real numbers, then for all $x \in \RR$ either 
\[ \fa{n \in \NN}{\lnot (x =r_n)}  \lor \lnot \fa{n \in \NN}{\lnot (x =r_n)}  \] 
\end{Pro}
\begin{proof}
The proof of Proposition \ref{Pro:LPO<->ratveeirrat} easily adapts to this case.
\end{proof}
The proof of the next proposition requires a lemma.
\begin{Lem}
A set $A \subset \RR$ is located if and only if it is inhabited and for all $a<b$ and $\varepsilon > 0$ we can decide
\[ \fa{x \in [a,b]}{x \notin A} \ \lor \ \ex{x \in [a-\varepsilon,b+\varepsilon]}{x \in A}  \ . \]
\end{Lem}
\begin{proof}
Straightforward.
\end{proof}

\begin{Pro} \label{Pro:WLPO-equivs} The following are equivalent to \WLPO.
\begin{enumerate}

  \item  \label{Equ:WLPO2} The zero-set 
  \[Z_{f} =  \set{x \in [0,1]}{f(x) = 0}  \]
  of a point-wise continuous function $f:[0,1] \to \RR$ is located (and therefore compact since it is closed).
\item \label{Equ:WLPO3}The ``weak support'' 
  \[ S^w_f = \set{x \in [0,1]}{\lnot \left(f(x) = 0 \right)}  \]
of a point-wise continuous function $f:[0,1] \to \RR$ is located whenever it is inhabited.
\item \label{Equ:WLPO4} (Strong Intermediate Value Theorem) For any point-wise continuous function $f: [0,1] \to \RR$ such that
  $f(0) \leqslant 0$ and $f(1) \geqslant 0$  \[ \inf f^{-1}(\menge{0}) \ , \] that is the left-most root, exists.
  \end{enumerate}

\end{Pro}
\begin{proof}
We will show the implications in the following order: 
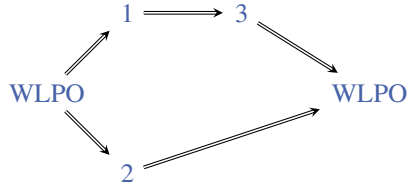
\begin{center}
\begin{tikzpicture}[>=stealth,shorten >=1pt, node distance=1.5cm]
\node (A)  {\WLPO};
\node[above right of=A] (B) {\ref{Equ:WLPO2}};
\node[below right of=A] (C) {\ref{Equ:WLPO3}};
\node[right of=B] (D) {\ref{Equ:WLPO4}};
\node[right = 3cm of A] (F) {\WLPO};
\draw[double,->] (A) -- (B);
\draw[double,->] (A) -- (C);
\draw[double,->] (C) -- (F);
\draw[double,->] (B) -- (D);
\draw[double,->] (D) -- (F);
\end{tikzpicture}
\end{center} 
First assume \WLPO. Note that since \WLPO implies \LLPO which in turn implies that $f$ is uniformly continuous (see Proposition \ref{Pro:LLPO_impl_UCT}). \LLPO also implies that a uniformly continuous $\abs*{f}$ attains its minimum $z$ and maximum $z^\prime$ (see section \ref{SSec:Minima}).
Now either $\abs*{f(z)} =0$ or $\lnot (\abs*{f(z)} = 0 )$. In the second case, since $z$ is minimal, there cannot be an $x \in [a,b]$ such that $\abs*{f(z)} = 0$. In particular, we can therefore decide  
\[\fa{x \in [a,b]}{ x \notin Z_f} \ \lor \ \ex{x \in [a,b]}{ x \in Z_f}  \ .\]
By the previous lemma $Z_f$ is located. 

Similarly either $\abs*{f(z^\prime)} =0$ or $\lnot (\abs*{f(z^\prime)} = 0 )$ and therefore, again by the previous Lemma $S^w_f$ is located.
Altogether we have shown that \WLPO implies \ref{Equ:WLPO2} as well as \ref{Equ:WLPO3}.

Since the infimum of a totally bounded set of real numbers exists \cite{dBlV06}*{Proposition 2.2.5} we have $\ref{Equ:WLPO2} \implies \ref{Equ:WLPO4}$.

To see that $\ref{Equ:WLPO4} \implies \WLPO$ holds, consider an arbitrary real number $a$. Construct the piecewise linear function $f$ by 
  \[f(x) = \begin{cases} 
     (3-3\abs{a}) x - 1& x \in \bracks*{0,\frac{1}{3}} \ , \\
    -\abs{a} & x \in \bracks*{\frac{1}{3},\frac{2}{3}} \ , \\
     3(1+\abs{a}) x -2 - 3\abs{a} & x \in \bracks*{\frac{2}{3},1} \ .
   \end{cases} \]
 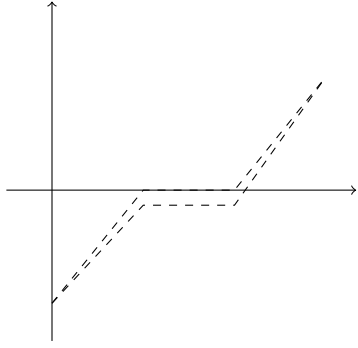
\begin{figure}[ht]
\centering
\begin{tikzpicture}[y=-1cm]
\draw[->] (3,5.5) -- (7.6,5.5);
\draw[->] (3.6,7.5) -- (3.6,3);
\draw[dashed] (3.6,7) -- (4.8,5.5) -- (6,5.5) -- (7.2,4);
\draw[dashed] (3.6,7) -- (4.8,5.7) -- (6,5.7) -- (7.2,4);
\end{tikzpicture}
\caption{We think of $a \geqslant 0$ as being very small. The function $f$ is a variation on the standard example that shows that the intermediate value theorem implies \LLPO.}
\end{figure}   
   
  By our assumption there exists $x$ such that $f(x) =0$ and for all $y $ with $y<x$   we have $\neg(f(y)=0)$. Now either $x < \frac{2}{3}$ or $x > \frac{1}{3}$. In the first case $a$ must be $0$ as the assumption $\abs{a} > 0$ implies that $f(x) < -\abs{a} < 0$, which is a contradiction to $f(x) = 0$. In the second case assume that $a = 0$, but then $f(\frac{1}{3}) = 0$ which contradicts $x > \frac{1}{3}$. Thus we have decided whether $a =0$ or $\neg(a =0)$.
  
Finally, to see that $\ref{Equ:WLPO3} \implies \WLPO$ we can use the same construction as before. Since $S^w_f$ is located the distance $r=d(1,S^w_f)$ exists. If $r < \frac{1}{3}$ we cannot have $a=0$, which would imply that $r=\frac{2}{3}$. If $r > 0$, then $a=0$, since $\abs{a} > 0$ implies $r=0$.
 \end{proof}
Similarly, we can get the following equivalence.
\begin{Pro}
\LPO is equivalent to the statement that if $f:[0,1] \to \RR$ is a point-wise continuous function, then its support
\[ S_{f} = \set{x \in [0,1]}{f(x) \neq 0} \]
 is located, provided it is inhabited.
\end{Pro}
\begin{proof}
One direction is clear by Proposition \ref{Pro:WLPO-equivs}, the fact that $\LPO \iff \WLPO \land \MP$ (see Section \ref{Sec:MP} for details on \MP), and the fact that under the assumption of \MP
\[ S_{f}  = \set{x \in [0,1]}{\neg(f(x) = 0)} \ .\]
For the other direction it is enough to show that the statement implies \MP . To this end let $a \in \RR$ be a real number such that $\neg(a = 0)$. Now consider $f:[0,1] \to \RR$ defined by 
\[
f(x) = \begin{cases}
a & \text{for } x \in \bracks*{0,\frac{1}{2}} \\
(1-a)2x -1 +2a & \text{for } x \in \bracks*{\frac{1}{2},1} 
\end{cases}
\]
Then $f(1)=1$, so $1 \in S_{f}$.  If the support is located that means that $\delta=d(0,S_{f})$ exists. Now either $\delta > 0$ or $\delta < \frac{1}{2}$. In the first case we must have $f(0) = 0$, since $f(0) \neq 0$ implies $\delta = 0$; but this cannot happen since $f(0)=a$. So we must have $\delta < \frac{1}{2}$. That means, by the definition of the infimum, that there exists $x \in [0,\frac{1}{2}]$ with $x \in S_{f}$; and hence $a = f(x) \neq 0$.
\end{proof} 

\begin{Pro*}[\textbf{\theThm{}½}] \label{Pro:WLPO_equiv_intersection_comapct_and_located}
The following are equivalent to \WLPO.
\begin{enumerate}
  \item \label{WLPO_comp_loc1} The (inhabited) intersection of a compact set $C \subset [0,1]$ with an interval is compact.
  \item \label{WLPO_comp_loc2} The intersection of a compact and a closed, located set is compact.\footnote{Notice that the \emph{union} of two located, inhabited sets are easily seen to be located: \[ \rho(x,A \cup B) = \inf_{y \in A \cup B}\menge{\rho(x,y)} = \min \menge{\rho(x,A), \rho(x,B)} \ . \]}
\end{enumerate}
\end{Pro*}
\begin{proof}
The Brouwerian counterexample part (\ref{WLPO_comp_loc1} implies \WLPO) is easy. Let $a \geqslant 0$. We may assume that $a \leqslant \frac{1}{2}$. Both $A = [-1,0]$ and $B = [a,1]$ are compact and therefore also closed and located. If $A \cap B$ is empty $\lnot (a =0)$, and if there is $x \in A \cap B$ then $a = 0$. We can, of course, also make sure that the intersection is always inhabited, by taking $A^\prime = [-1,0] \cup [1,2]$ and $B^\prime = [a,2]$.

Obviously \ref{WLPO_comp_loc2} implies \ref{WLPO_comp_loc1}. To prove the missing implication (\WLPO implies \ref{WLPO_comp_loc2}) let $A$ be a compact and $B$ be a closed and located subset of a metric space $(X,\rho)$. If we can show that $A \cap B$ is totally bounded, we are done. To this end let $\varepsilon > 0$ be arbitrary, and $x_1, \dots, x_n$ be a $\nicefrac{\varepsilon}{2}$-approximation to $A$. For every $i \in \menge{1, \dots,n}$ consider the function $g_i: X \to \RR^+_0$, defined by
\[ g_i(x) = \max \menge{\rho(x,x_i) - \nicefrac{\varepsilon}{2}, 0   } ,\] 
so that $g_i(x) = 0$ if and only if $\rho(x,x_i) \leqslant \nicefrac{\varepsilon}{2}$. Now consider the function $f: A \to \RR$ defined by $f(x) = \rho(x,B)$. This is well-defined, since $B$ is located.
Finally let $h_i = f + g_i$. Notice that, since $f$ and $g_i$ are both non-negative $h_i(x) = 0$ if and only if both $f(x) = 0$ and $g_i(x) = 0$. Since $h_i$ is continuous, by \WLPO (via \LLPO/\WKL) it follows from Proposition \ref{Pro:WKL_Min}.\ref{Equiv:WKL0b} that it attains its minimum at some $y_i \in X$. Now, by \WLPO, either $h_i(y_i) = 0$ or $\lnot h_i(z) = 0$. In the first case $y_i \in B$ and $\rho(x_i,y_i) \leqslant \nicefrac{\varepsilon}{2}$. In the second case there cannot be $z$ such that $z \in A \cap B$, such that $\rho(x_i,z) \leqslant \nicefrac{\varepsilon}{2}$.

We claim that $F = \set{y_i}{h_i(y_i) = 0}$ is an $\varepsilon$-approximation to $A \cap B$. For let $z \in A \cap B$, and let $i$ be such that $\rho(z,x_i) < \nicefrac{\varepsilon}{2}$. For that $i$ we have $g_i(z) = 0$. Since $z \in B$ we also have $f(z) = 0$ and  hence $h_i(z) = 0$. That means that $h_i(y_i) = 0$, by virtue of $y_i$ being a minimum. But then also $g_i(y_i) = 0$, which means that $\rho(x_i,y_i) \leqslant \nicefrac{\varepsilon}{2}$. Together $\rho(z,y_i) < \varepsilon$, and $y_i \in F$, which proves our claim.
\end{proof}

A result by Richman from \cite{fR02} characterises a well know classical result about functions of bounded variation.
\begin{Pro} If every uniformly continuous function on $[0, 1]$ of bounded variation\footnote{Here a function $f:[a,b] \to \RR$ has bounded variation if the set 
\[ S = \set{\sum_{i=1}^{n-1}  \abs*{f(x_i) - f(x_{i+1})}}{a = x_1 \leqslant x_2 \leqslant \dots \leqslant x_n = b}\] is bounded. A result by Bridges \cite{dB00c} shows that if one assumes that the supremum (in Bishop's sense) of the set $S$ exists   and that the function $f$ is strongly extensional (no continuity assumption), then it can be written as the sum of two increasing functions. This result was improved upon by Richman, who showed that one can work without the assumption of strong extensionality \cite{fR02}.
} is the difference of two increasing functions, then \WLPO holds. The converse holds in the presence of countable choice.
\end{Pro}
 
As the last result of this Section we have the following, which is the obvious analogue of Proposition \ref{Pro:LPO_equiv_extbij}. What is interesting though is that the function $f$ is injective in the sense that $f(x) \neq f(y) \rightarrow x \neq y$, and one would not expect to see that $x \neq y$ without the help of the \MP-part of \LPO.
\begin{Pro} 
\WLPO is equivalent to the existence of a  bijection\footnote{See the footnote in Proposition \ref{Pro:LPO_equiv_extbij} for the exact definition of bijection.} $f:[0,1) \to (0,1)$.
\end{Pro}
\begin{proof}
The proof will be omitted, since it is basically the same as the one of Proposition \ref{Pro:LPO_equiv_extbij}. One technical point worth pointing out is that one needs to use a variation of Ishihara's tricks \cite{hD12b}.
\end{proof}

  As a direct consequence of this we also have that the Cantor-Bernstein-Schr\"ader theorem implies \WLPO, even when restricted to functions between subsets of $\RR$, since $x \mapsto \frac{1+x}{2}$ and the identity are injections between $[0,1)$ and $(0,1)$.
\begin{Qu}
Which principle is the Cantor-Bernstein-Schr\"ader theorem equivalent to? It seems likely that the answer heavily depends on the precise formulation.
\end{Qu}

\subsection{The Rising Sun Lemma, and \texorpdfstring{\LPO and \WLPO}{LPO and WLPO}}
We can use Proposition \ref{Pro:WLPO-equivs} to show two more equivalences  of \LPO and \WLPO; namely (versions of) the Rising Sun Lemma, which can be used to prove the Hardy-Littlewood maximal inequality \cite{tT11}*{pp. 143}. 
\begin{Pro} \label{Pro:rising_sun}
\LPO is equivalent to the following statement. 

Consider  a continuous function $f:[0,1] \to \RR$, and let $E$ be the set
\[ E = \set{x \in [0,1]}{\ex{y > x}{f(y) > f(x) }} \ . \]  Then one can find a, at most countable, family of disjoint, open, non-empty intervals $I_n$ in $[0, 1]$  such that $E \cap (0,1) = \bigcup I_n$.
\end{Pro}
The name ``Rising Sun Lemma'' is due to imagining the graph of $f$ as a mountain range which the rising sun shines on to from the right. Viewed this way, the set $E$  consists of the areas which are in the shade. 

Before we prove the above proposition, we will make a couple of comments. 
\begin{itemize}
\item Since we assume \LPO (and therefore \MP) in the forward direction it doesn't make a difference whether we assume that our intervals are non-empty or inhabited. In fact, if $(a,b)$ is nonempty, then $\neg(b \leqslant a )$, and hence because of \MP $a<b$. Thus $\frac{b-a}{2} \in (a,b)$.
\item Unlike in many other results in Constructive Reverse Mathematics the Brouwerian counterexample part (i.e.\ the proof that the Rising Sun Lemma implies \LPO) is fairly straightforward, but the reverse, classical direction takes some effort. This is due to the fact that the standard proof (like the one in \cite{tT11}) uses the fact that classically ``any open subset U of $\RR$ can be written as the union of at most countably many disjoint non-empty open intervals, whose endpoints lie outside of U.'' However, this statement is equivalent to \LEM: Consider the set \[ U = \set{x \in (0,1)}{ x< \frac{1}{2} \lor \varphi} \ , \] for a syntactically correct statement $\varphi$. Now let  $U = \bigcup_{n \in I} I_n$, as above, where $I$ is either finite or countable. It is easy to see that $I$ must have (exactly) one element, say $i$ and $I_i = (a,b)$. If $b> \frac{1}{2}$ we must have $\varphi$ if $b <1 $ the assumption that $\varphi$ holds leads to a contradiction.
\item Notice that even classically we cannot assume that the $I_n$ are increasing or decreasing. To see this consider the piece-wise linear function $f:[0,1] \to [0,1]$ defined by
	  \[f(x) = \begin{cases} 
     -2x +1 & x \in \bracks*{0,\frac{1}{3}} \\
     x & x \in \bracks*{\frac{1}{3},\frac{2}{3}} \\
     -2x + 2&  x \in \bracks*{\frac{2}{3},1} \ . 
   \end{cases} \]
   For this function the set $E$ from the Rising Sun Lemma is $\bracks*{\frac{1}{6},\frac{2}{3}}$. Now define a functions $g,h:[0,1] \to [0,1]$ by placing a copy of $f$ in each of the squares $G_n$ and $H_n$ indicated in Figure \ref{Fig:rising_sun}.
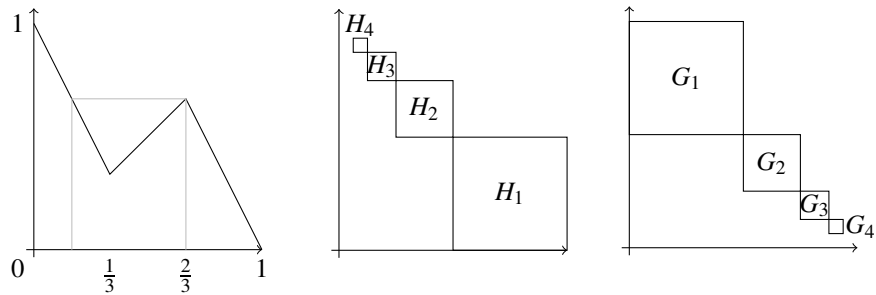
\begin{figure}[ht] 
\centering
\subfloat{\begin{tikzpicture}
\draw[->] (-0.1,0) -- (3,0);
\draw[->] (0,-0.1) -- (0,3.2);
\path (3,0) node[anchor=north] {$1$};
\path (1,0) node[anchor=north] {$\frac{1}{3}$};
\path (2,0) node[anchor=north] {$\frac{2}{3}$};
\path (0,0) node[anchor=north east] {$0$};
\path (0,3) node[anchor=east] {$1$};

\draw (0,3) -- (1,1) -- (2,2) -- (3,0);
\draw[gray!50] (2,0) -- (2, 2) -- (0.5,2) -- (0.5,0);
\end{tikzpicture}}
\qquad
\subfloat{\begin{tikzpicture}
\draw[->] (-0.1,0) -- (3,0);
\draw[->] (0,-0.1) -- (0,3.2);
\draw (1.5,1.5) rectangle node {$H_1$} (3,0);
\draw (0.75,2.25) rectangle node {$H_2$} (1.5,1.5);
\draw (0.375,2.625)  rectangle node {$H_3$} (0.75,2.25);
\draw (.1875,2.8125)  rectangle node[anchor=south] {$H_4$} (0.375,2.625);
\end{tikzpicture}}
\qquad
\subfloat{\begin{tikzpicture}
\draw[->] (-0.1,0) -- (3,0);
\draw[->] (0,-0.1) -- (0,3.2);
\draw (1.5,1.5) rectangle node {$G_1$} (0,3);
\draw (2.25,0.75) rectangle node {$G_2$} (1.5,1.5);
\draw (2.625,0.375)  rectangle node {$G_3$} (2.25,0.75);
\draw (2.8125,.1875)  rectangle node[anchor=west] {$G_4$} (2.625,0.375);
\end{tikzpicture}}
 \caption{An iterated construction.\label{Fig:rising_sun}}
\end{figure} 
It is clear that $g$'s $E$-set cannot be written as a sequence of open disjoint increasing intervals, and the same holds for $h$ and decreasing intervals. Obviously, we could iterate this construction to see that we cannot write $E$ as a decreasing union of increasing disjoint, open, non-empty intervals, or even further.
\end{itemize}
   
Finally, it is time to prove Proposition \ref{Pro:rising_sun}.
\begin{proof}
Assume the Rising Sun Lemma holds and let $a \in \RR$ be a real number. Now define a piece-wise linear function $f:[0,1] \to \RR$ by
	  \[f(x) = \begin{cases} 
     (2 - 2\abs{a})x  & x \in \bracks*{0,\frac{1}{2}} \\
      2\abs{a}x + 1 - 2\abs{a} & x \in \bracks*{\frac{1}{2},1} 
   \end{cases} \]
   
 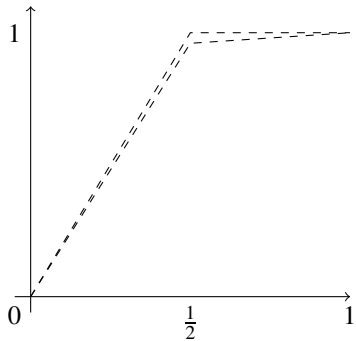
\begin{figure}[ht]
\centering
\begin{tikzpicture}[scale=0.7]
\draw[->] (-0.3,0) -- (6,0);
\draw[->] (0,-0.3) -- (0,5.5);
\path (6,0) node[anchor=north] {$1$};
\path (3,0) node[anchor=north] {$\frac{1}{2}$};
\path (0,0) node[anchor=north east] {$0$};
\path (0,5) node[anchor=east] {$1$};

\draw[dashed] (0,0) -- (3,5) -- (6,5);
\draw[dashed] (0,0) -- (3,4.8) -- (6,5);
\end{tikzpicture}
\caption{Depending on $a$ the function $f$ is either ``in the shadow'' or flat on $[\nicefrac{1}{2},1]$.}
\end{figure}
Notice that if $x \in \bracks*{\frac{1}{2},1}$, then $1-\abs{a} \leqslant f(x) \leqslant 1$.
Now let $E$ be the ``areas in the shade'' as in the Rising Sun Lemma, and assume that $E= \bigcup I_n$, where $I_n = (a_n,b_n)$ are open, nonempty, and disjoint. It is easy to see that the family $(I_n)_{n \geqslant 1}$ must consist of exactly one element, say $(a,b)$. Now either $b > \nicefrac{1}{2}$ or $b < 1$. In the second case the assumption that $\abs{a} >0$ leads to the contradiction that $E = [0,\frac{1}{2})$, so in that case $a = 0$. In the first case we have $z = \frac{b+\nicefrac{1}{2}}{2} \in E$. So by definition there exists $y > z$ such that $f(y) > f(z)$. But that means that 
\[ 1-\abs{a} \leqslant f(z) < f(y) \leqslant 1  \ , \] which implies that $\abs{a} > 0$. With both cases taken together we have decided \LPO.

Conversely let $f:[0,1] \to \RR$ be a continuous function. First notice that \LPO allows us to make the following decision. Since \LPO implies that $f$ is uniformly continuous (see \ref{Pro:LLPO_impl_UCT}) and therefore $S_x = \sup_{[x,1]} f$ exists, we can decide whether $f(x) = S_x$, or $f(x) < S_x$. In the second case, using Proposition \ref{Pro:WLPO-equivs}.\ref{Equ:WLPO4}, we can find $b_x = \inf \set{z \geqslant x }{ f(z) = S_x}$. Using Proposition \ref{Pro:WLPO-equivs}.\ref{Equ:WLPO4} again we can also find $a_x = \sup \set{z \leqslant x}{f(z) = S_x}$. 

Now let $r_n$ be an enumeration of all rationals in $[0,1]$ and consider 
\[ F = \set{n \in \NN}{f(r_n) < S_{r_n} \land \fa{i < n}{r_n \notin (a_{r_i},b_{r_i})}} \ .\] 
Using \LPO we see that $F$ is decidable and using \LPO again we can decide whether $F$ is finite or countable. By construction the intervals $((a_{r_n},b_{r_n}))_{n \in F}$ are non-empty and disjoint. We also have that if $x \in (a_{r_n},b_{r_n})$ for some $n$, then $f(x) < S_{r_n}$, since, by the definition of $b_{r_n}$ we cannot have $f(x) \geqslant S_{r_n}$. Using (an approximate version of) the intermediate value theorem we can find a point $y$ such that $x<y<b_{r_n}$ and $f(x)<f(y) < S_{r_n}$. 

It remains to show that if $x \in E$ there exists $n \in F$ such that $x \in (a_{r_n},b_{r_n})$. If $x \in E$, then there exists $y > x$ with $f(x) < f(y) < S_x$. Now find (again an approximate version of the intermediate value theorem suffices) a rational $r_n$ such that $x<r_n < y$ and $f(x)<f(r_n)<f(y)$. If for all $i < n$ we have $r_n \notin (a_{r_i},b_{r_i})$, then $x \in (a_{r_n},b_{r_n})$ and therefore $x \in \bigcup_{n \in F} (a_{r_n},b_{r_n})$. If there is $j < n$ with $r_n \in (a_{r_j},b_{r_j})$, then also $x \in (a_{r_j},b_{r_j})$ and therefore also in this case $x \in \bigcup_{n \in F} (a_{r_n},b_{r_n})$.
\end{proof}

The classification of the Rising Sun Lemma is very dependent on its precise formulation. If we had followed Tao's formulation \cite{tT11} more closely we get an equivalence to \WLPO. Luckily, we can reuse most of the proof above.

\begin{Pro} \label{Pro:rising_sun_WLPO}
\WLPO is equivalent to the following statement.

Consider  a continuous function $f:[0,1] \to \RR$.  Then one can find a, at most countable\footnote{This is one of the places we actually have to be tediously precise. In this case, by a ``at most countable'' set $A$ we mean that $A = \lambda^{-1}(\menge{0}) $  for an increasing binary sequence $\lambda$.}, family of disjoint, open, non-empty intervals $I_n$ in $[0, 1]$ such that
\begin{enumerate}
  \item $f(a_n) = f(b_n)$ unless $a_n = 0$, in which case $f(b_n) \geqslant f(a_n)$, and
  \item if $x \notin \bigcup I_n$ and $x >0$, then $f(x) \geqslant f(y)$ for all $y \geqslant x$.
\end{enumerate}
\end{Pro}  
\begin{proof}
We will only proof the ``counterexample direction'', since the converse direction is analogous to the one above.

So let $a \in \RR$ be a real number. Now define a piece-wise linear function $h:[0,1] \to \RR$ by
\[h(x) = \begin{cases} 
     \frac{3}{2}x  & x \in \bracks*{0,\frac{1}{3}} \ , \\
     \frac{1}{2}  & x \in \bracks*{\frac{1}{3},\frac{2}{3}} \ , \\
      \frac{1}{2}  + \abs{a}(3x-2) & x \in \bracks*{\frac{2}{3},1}  \ . 
   \end{cases} \]
   
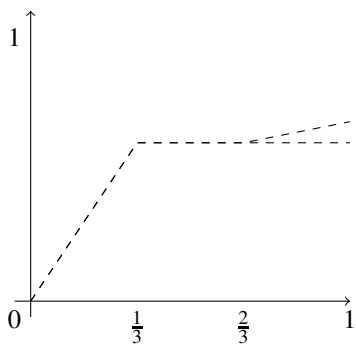
\begin{figure}[ht]
\centering
\begin{tikzpicture}[scale=0.7]
\draw[->] (-0.3,0) -- (6,0);
\draw[->] (0,-0.3) -- (0,5.5);
\path (6,0) node[anchor=north] {$1$};
\path (2,0) node[anchor=north] {$\frac{1}{3}$};
\path (4,0) node[anchor=north] {$\frac{2}{3}$};
\path (0,0) node[anchor=north east] {$0$};
\path (0,5) node[anchor=east] {$1$};

\draw[dashed] (0,0) -- (2,3) -- (4,3) -- (6,3);
\draw[dashed] (0,0) -- (2,3) -- (4,3) -- (6,3.4);
\end{tikzpicture}
\end{figure}

Again it is easy to see that the ``area in the shade'' consists of exactly one interval $(0,b)$. If $b < \frac{2}{3}$, then $\abs{a}>0$ leads to a contradiction, since then $E = (0,1)$. So we must have $a=0$. If $b> \frac{1}{3}$, then $a=0$ leads to a contradiction, since that in that case $E=(0,\frac{1}{3})$. Together we can decide whether $a=0$ or $\neg(a=0)$, which means that \WLPO holds. 
\end{proof}

\section{\texorpdfstring{\LLPO and \WKL}{LLPO and WKL}} \label{Sec:LLPO}
The final limited omniscience principle is the \define{lesser limited principle of omniscience}.
\begin{principle}[LLPO]{\LLPO} \label{PR:LLPO}
If $(a_n)_{n \geqslant 1}$ is a binary sequence with at most one $1$, then 
\[ \fa{n \in \NN}{a_{2n} = 0}  \lor \fa{n \in \NN}{a_{2n+1} = 0}  . \]
\end{principle}
It is also known as $\mathrm{SEP}$ in \cite{Troelstra1988a}. \index{SEP|see {\LLPO}}

\LLPO is often presented in the following form which is often actually used as its definition; for example in \cite{hI06a}. 
\begin{Pro} \label{Pro:LLPO_equiv_deMorg}
\LLPO is equivalent to one of the De Morgan's laws for simply existential statements. That is for $\alpha, \beta \in \CS$ 
\[ \lnot \left( \ex{n \in \NN}{\alpha(n) = 1} \land \ex{n \in \NN}{\beta(n) = 1} \right) \implies \lnot \left( \ex{n \in \NN}{\alpha(n) = 1}\right)  \lor \lnot \left(\ex{n \in \NN}{\beta(n) = 1} \right)  \]
\end{Pro}
\begin{proof}
Straightforward.	
\end{proof}

\begin{Pro} \label{Pro:equivs_of_LLPO}
The following are equivalent to \LLPO
\begin{enumerate}
\item \label{item:LLPO-1} $\fa{x \in \RR}{0 \leqslant x \lor x \leqslant 0}$
\item \label{item:IVT} (The intermediate value theorem) For every  non-discontinuous function $f:[0,1] \to \RR$ with $f(0) \leqslant 0 \leqslant f(1)$ there exists $x \in [0,1]$ such that $f(x) =0$
\item \label{item:LLPO-4} Every real number has a binary expansion.
\item If $x,y$ are real numbers such that $xy=0$, then either $x=0$ or $y=0$; that is $\RR$ is an integral domain. \label{item:LLPO-5}
\item \label{item:LLPO-6} If $x,y$ are real numbers, then $\menge{x,y}$ is a closed set.
\end{enumerate}
\end{Pro}
\begin{proof} 
All equivalencies  are very well known; see for example \cite{dB87} for the first 4 and \cite{mM88}*{Theorem 4.2} for the last one.
The equivalence $\LLPO \iff \ref{item:LLPO-5}$ is also proved in a more general form in Proposition \ref{Pro:LLPOn-equiv-real}. 
\end{proof}

\subsection{Completeness of Finite Sets}
In Proposition \ref{Pro:equivs_of_LLPO} we cited a paper by Mandelkern \cite{mM88} that among many other insights shows that \LLPO is equivalent to every two-element set of reals being closed.
That is constructively we cannot show that for any two reals $a,b$
\[ \overline{\menge{a,b}} = \menge{a,b}  \ .\]
Analysing Mandelkern's proof one might hope that constructively we have at least 
\[ \overline{\menge{a,b}} = \menge{a,b, \inf \menge{a,b}, \sup \menge{a, b} } \ .\]
However, as the next result shows this statement is still equivalent to \LLPO.
\begin{Pro}
   \LLPO is equivalent to the statement that for all $a,b \in \RR$ the set
   \[ \menge{a,b, \inf \menge{a,b}, \sup \menge{a, b} } \]
   is the closure of $\menge{a,b}$.	
\end{Pro}
\begin{proof}
	As mentioned above it is well known \cite{mM88}*{Theorem 4.2} that \LLPO implies (is equivalent to) that for all $a,b$ the set $\menge{a,b}$ is closed. 
	
	Conversely let $(a_n)_{n \geqslant 1}$ be a binary sequence with at most one $1$. Let $a = \sum \frac{a_n}{2^n}$ and define a sequence $(z_n)_{n \geqslant 1}$ in $\menge{0,a}$ by 
\[ z_n = \begin{cases}
 	0 & \fa{m < n}{a_m = 0} \\
 	0 & \ex{m<n}{a_m=1 \land m \text{ is odd}} \\
 	a & \ex{m<n}{a_m=1 \land m \text{ is even}} \ .
 \end{cases}
\]
Using the fact that $(a_n)_{n \geqslant 1}$ has at most one term equal to $1$, it is easy to see that $z_n$ is a Cauchy sequence converging to a point $z_\infty$. 
Now if $z_\infty \in \menge{a,b, \inf \menge{a,b}, \sup \menge{a, b}}$ we can make the following decisions:
\begin{description}
 \item[$z_\infty = 0$:] there cannot be an even $m$ such that $a_m = 1$, since in that case $a \neq 0$ and $z_n \to a$.
 \item[$z_\infty = a$:] there cannot be an odd $m$ such that $a_m = 1$, since in that case $a \neq 0$ and $z_n \to 0$.
 \item[$z_\infty = \inf \menge{a,b}$:] there cannot be an even $m$ such that $a_m = 1$, since in that case $a \neq 0 = \inf \menge{a,b} = 0$ and $z_n \to a$.
 \item[$z_\infty = \sup \menge{a,b}$:] there cannot be an odd $m$ such that $a_m = 1$, since in that case $a \neq 0 = \sup \menge{a,b} = a$ and $z_n \to 0$.
\end{description}
In any case we can decide whether $\fa{n \in \NN}{a_{2n}=0} \lor \fa{n \in \NN}{a_{2n+1}=0} $. So we have shown \LLPO.
\end{proof}

Actually we can even show the similar result that 

\begin{Pro}
   \LLPO is equivalent to the statement that for all $a,b \in \RR$  the closure of $\menge{a,b}$ is a finitely enumerable set with $3$ elements.		
\end{Pro}
\begin{proof}
	Again, as above \LLPO implies that $\menge{a,b}$ is closed, so the only interesting direction is the converse. Let $a_n$ and $a$ be as above.
	Since $\menge{0,a} \subset \overline{\menge{0,a}}$ we may assume that the closure of $\menge{0,a}$ is of the form $\menge{0,a,b}$. Now define $z_n$ as above, but also $z_n^\prime$ analogously, but with the odd and even case switched. If $z_\infty = 0$ or $z_\infty=a$ we are done as above. Similarly we are done if $z_\infty^\prime = 0$ or $z_\infty^\prime = a$. The only interesting and genuinely new case is that $z_\infty = z_\infty^\prime = b$. But in that case the assumption that there is $m$ with $a_m = 1$ leads to the contradiction $z_\infty \neq z_\infty^\prime$, both in the odd and the even case. Again in all cases we have decided $\fa{n \in \NN}{a_{2n}=0} \lor \fa{n \in \NN}{a_{2n+1}=0} $. So we have shown \LLPO.
\end{proof}

We conjecture that the above proposition can be generalised to any finite number and possibly the countable case. Notice that there exists a realisability model (based on infinite time Turing machines) in which there is a surjection $\NN \to \BS$ \cite{aB15}, which means that the closure of $\menge{a,b}$ is countable, but that in that model also \LPO and therefore \LLPO holds.

\subsection{\texorpdfstring{\WKL}{WKL}, Minima, and Fixed Points} \label{SSec:Minima}

The one principle in Constructive Reverse Mathematics that is possibly best known by non-constructivists is Weak K\H{o}nig's Lemma.
\begin{principle}[WKL]{\WKL} \label{PR:WKL}
Every infinite and decidable tree admits an infinite path.
\end{principle}
A decidable tree is a subset $T \subset \cS$ of all binary finite sequences which is \define{closed under restrictions}, that is if $u \in T$ and $v$ is a prefix of $u$, then also $u \in T$. A tree is infinite, if it is infinite as a set, which in this case is equivalent to containing arbitrarily long finite sequences. Finally, a tree \define{admits an infinite path}, if there exists an infinite binary sequence $\alpha \in \CS$ such that every finite prefix of $\alpha$ is in $T$.

It was, to our knowledge, Ishihara who first pointed out the equivalence of \LLPO and \WKL over intuitionistic logic and countable choice. If one is thinking in a resource sensitive way, then one can view \WKL as countable many versions of \LLPO \cite{Brattka:2010fk}. If one wants to factor out choice, then one can view \WKL as \LLPO plus a choice principle \cite{jB12}. 
\begin{Pro} \label{Pro:WKL_equiv_LLPO}
\WKL and \LLPO are equivalent.
\end{Pro}
\begin{proof}
This proof has been given many times; see for example \cite{jB12}. It relies on repeatedly using the fact that, given an infinite tree, \LLPO allows us to decide whether the left or the right subtree are infinite. We will repeat the proof for completeness sake, and since we are going to reuse the argument in similar situations further on. Let $T \subset \cS$. For $u \in \cS$ let $T_u$ denote the set $\set{w \in \cS}{u\ast w \in T}$; i.e.\ $T_u$ is the subtree ``below'' $u$. Now if $T$ is an infinite decidable tree, then we can define a binary sequence $(a_n)_{n \geqslant 1}$ by
\begin{align} 
a_{2n} = 0 \iff \ex{u \in 2^n}{0 \ast u \in T} \ , \\
a_{2n+1} = 0 \iff \ex{u \in 2^n}{1 \ast u \in T} \ . 
\end{align}
Notice that, since $T$ is infinite there cannot $n,m \in \NN$ such that $a_{2n}=1$ and $a_{2m+1}=1$. Hence by \ref{Pro:LLPO_equiv_deMorg} \LLPO  implies that either all even or all odd terms of $a_n$ are zero. That means that either $T_0$ or $T_1$ is infinite. Of course, both are also still decidable trees. Thus we can, using dependent choice, iteratively define a sequence $\alpha \in \CS$ such that $T_{\overline{\alpha}n}$ is a infinite, decidable tree for all $n \in \NN$. In particular $\overline{\alpha}n \in T$.
\end{proof}
\begin{Rmk} \begin{enumerate}
  \item Analysing the above proof we can see that \WLPO implies \WKL, only assuming unique choice. 
  \item  \WLPO is also equivalent to finding the left-most minimum.
\end{enumerate}
\end{Rmk}
\begin{proof}
\begin{enumerate}
  \item Clear.
  \item The only interesting part is the reverse direction. To that end let $(a_n)_{n \geqslant 1}$ be an increasing binary sequence. We can define a tree $T\subset \cS$ by $T = \set{1u }{u \in \cS} \cup \set{0w}{a_{\abs{w}+1} = 0}$. Now assume that $T$ admits an infinite path $\alpha \in \CS$. If $\alpha(0)=0$ we cannot have $n \in \NN$ with $a_n = 1$, since then $\overline{\alpha}(n) \notin T$. Hence, in that case, $\fa{n \in \NN}{a_n =0}$. If $\alpha(0)=1$ then the assumption $\fa{n \in \NN}{a_n=0}$ leads to a contradiction, since in that case, $0^n \in T$ for all $n \in \NN$, and therefore the left-most minimum $\alpha = \zero$. Together we can decide whether 
$\fa{n \in \NN}{a_n=0}$ or $\lnot \fa{n \in \NN}{a_n=0}$.  \qedhere
\end{enumerate}
\end{proof}

\begin{Pro} \label{Pro:WKL_Min}
The following are equivalent to \WKL
\begin{enumerate}
\item \label{Equiv:WKL0} Every uniformly continuous function $f:\CS \to \RR$ attains its minimum.
\item \label{Equiv:WKL0b} Every uniformly continuous function $f:X \to \RR$ on a compact metric space $X$ attains its minimum.
\item \label{Equiv:WKL1} Every uniformly continuous function $f:[0,1] \to \RR$ attains its minimum.
\item \label{Equiv:WKL3} (Brouwer's fix point theorem) Every uniformly continuous function $f:[0,1]^{2} \to [0,1]^{2}$ has a fixed point.
\item \label{Equiv:WKL4} (Peano's existence theorem for ordinary differential equations) If $A \subset \RR \times \RR$ and $r>0$ such that .
\end{enumerate}
\end{Pro}
\begin{proof}
That \ref{Equiv:WKL0} implies \ref{Equiv:WKL0b} follows that for every compact space $X$ there exists a surjection $G:\CS \to X$ \cite{Troelstra1988a}*{Corollary 4.4}.
Clearly \ref{Equiv:WKL0b} implies \ref{Equiv:WKL1}.
That \ref{Equiv:WKL1} implies \ref{Equiv:WKL0} follows from Lemma \ref{Lem:extendingfunctions}.
The proof of the equivalence of \ref{Equiv:WKL3} and \WKL can be found in \cite{mH13}. Finally the equivalence of \ref{Equiv:WKL4} and \WKL can be found in  \cite{dB12} (and relies on the equivalence of \WKL and  \ref{Equiv:WKL0b}).
\end{proof}

\begin{Rmk}
Using Proposition \ref{Pro:LLPO_impl_UCT} we can replace uniform by point-wise continuity in the above proposition.
\end{Rmk}

This last proposition above also allows the following heuristic.
\begin{Heu}
	Many fix point theorems are equivalent to \WKL.
\end{Heu}
Of course this is just a heuristic, which actually fails in some situations such as Proposition \ref{Pro:Hillam} or \ref{Pro:FANDfixp}.  It does, however, work for many well known ones such as Brouwer's fix point theorem mentioned above and many others \cite{mH13}. We can even give a ``proof'' for our heuristic. 
\begin{proof}
The reason that our heuristic is a good one is that often fix point theorems consider a point-wise continuous function $f$ on a compact space $X$, and one can construct approximate fix points with constructive, often discrete, arguments. Thus if we consider the function $g:X \to \RR$ defined by $g(x) = d(x,f(x))$ we can see that having approximate fix points translates into $\inf g =0$, and thus the above proposition means that \WKL implies the existence of a fix point. Since \WKL also implies \UCT (see Proposition \ref{Pro:LLPO_impl_UCT}) and therefore \FAND, this argument is still valid, even if our original function is only point-wise continuous, or if we need a mild version of the fan theorem to prove the existence of approximate fix points. 

Conversely most fix point theorems imply the intermediate value theorem as a very simple instance. Thus, by \ref{Pro:equivs_of_LLPO}, they imply \LLPO and therefore \WKL.
\end{proof}
As the ``proof'' above shows this also implies that if we add the additional assumption that there is uniformly at most one solution to a fix point theorem we often get a purely constructive theorem. This phenomenon is exhaustively treated in \cite{jB06c} and \cite{pS06}.

\subsection{Space-filling curves}
\begin{quote}
\textit{The content of this section might well be folklore among some constructivists. We do not claim to be the first to prove these results, but we are also not aware that they have been written down anywhere else.}
\end{quote}

One of the stranger objects in (classical) analysis are \define{space-filling curves}, which are continuous functions $[0,1] \to [0,1]^2$ that are surjective. The first such curve was discovered by Peano in 1890, and since then numerous variants have been described. One of these variants which is very similar but slightly easier than Peano's one, is a construction by Hilbert from 1891. We are not going to repeat the well-known ideas here, but simply include a pretty picture of the first 4 steps in the construction and otherwise refer to the exhaustive monograph \cite{hS94} for details. 

\newdimen\HilbertLastX
\newdimen\HilbertLastY
\newcounter{HilbertOrder}

\def\DrawToNext#1#2{%
   \advance \HilbertLastX by #1
   \advance \HilbertLastY by #2
   \pgfpathlineto{\pgfqpoint{\HilbertLastX}{\HilbertLastY}}
}

\def\Hilbert[#1,#2,#3,#4,#5,#6,#7,#8] {
  \ifnum\value{HilbertOrder} > 0%
     \addtocounter{HilbertOrder}{-1}
     \Hilbert[#5,#6,#7,#8,#1,#2,#3,#4]
     \DrawToNext {#1} {#2}
     \Hilbert[#1,#2,#3,#4,#5,#6,#7,#8]
     \DrawToNext {#5} {#6}
     \Hilbert[#1,#2,#3,#4,#5,#6,#7,#8]
     \DrawToNext {#3} {#4}
     \Hilbert[#7,#8,#5,#6,#3,#4,#1,#2]
     \addtocounter{HilbertOrder}{1}
  \fi
}

\def\hilbert((#1,#2),#3){%
   \advance \HilbertLastX by #1
   \advance \HilbertLastY by #2
   \pgfpathmoveto{\pgfqpoint{\HilbertLastX}{\HilbertLastY}}
   \setcounter{HilbertOrder}{#3}
   \Hilbert[1mm,0mm,-1mm,0mm,0mm,1mm,0mm,-1mm]
   \pgfusepath{stroke}%
}

\begin{figure}[ht]%
    \centering
    \subfloat[$n=1$]{\tikz[scale=18] \hilbert((0mm,0mm),1);}~~
    \subfloat[$n=2$]{\tikz[scale=6] \hilbert((0mm,0mm),2);}~~
    \subfloat[$n=3$]{\tikz[scale=2.6] \hilbert((0mm,0mm),3);}~~
    \subfloat[$n=4$]{\tikz[scale=1.2] \hilbert((0mm,0mm),4);}~~
    \caption{The first 4 steps in Hilbert's construction. (Latex code from \href{http://www.texample.net/tikz/examples/hilbert-curve/}{www.texample.com}).}
\end{figure}%

Technically speaking, both the Hilbert curve and the Peano curve are described as the uniform limit of uniformly continuous functions. It is not hard to see that even constructively we can follow Hilbert's or Peano's construction, and obtain uniformly continuous functions $H, P:[0,1] \to [0,1]^2$. One can also easily see that points in $[0,1]^2$ that have a binary expansion have a pre-image.  Therefore $P([0,1])$ and $H([0,1])$ are dense in $[0,1]^2$. However, we cannot prove that $H$ and $P$ are surjective. As a matter of fact: 
\begin{Pro}
\LLPO is equivalent to the statement that Hilbert's space-filling curve or Peano's space-filling curve is surjective.
\end{Pro}
\begin{proof}
We will only focus on $H$, since the case of $P$ can be treated similarly.  
\LLPO implies that every real number has a binary expansion (Proposition \ref{Pro:equivs_of_LLPO}). Thus, as already remarked above we can actually find the pre-image of any point in $[0,1]^2$.

 Conversely note that 
\begin{align*}
  H\left( \bracks*{0,\nicefrac{1}{4}} \right) & \subset \bracks*{0,\nicefrac{1}{2}} \times \bracks*{0,\nicefrac{1}{2}} \ , \\
  H\left(\bracks*{\nicefrac{1}{4},\nicefrac{3}{4}}\right) & \subset \bracks*{\nicefrac{1}{2},1} \times \bracks*{0,1} \ ,   \\
  H\left(\bracks*{\nicefrac{3}{4},1}\right) & \subset \bracks*{0,\nicefrac{1}{2}} \times \bracks*{\nicefrac{1}{2},1}  \ ,
\end{align*}
since these relations are satisfied at every stage of the construction of $H$ and preserved under limits. Now let $a \in \RR$ arbitrary. Without loss of generality $\abs{a} < \nicefrac{1}{2}$. Consider the point $p = (0,\nicefrac{1}{2}+a)$. If $H$ is surjective, then there exists $t$ such that $H(t) = p$. Now either $t < \nicefrac{3}{4}$ or $t > \nicefrac{1}{4}$. In the first case it is impossible that $a > 0$, since then $p \notin H([0,\nicefrac{3}{4}])$. Similarly in the second case it is impossible that $ a< 0$. Thus we can decide whether $a \leqslant 0$ or $a \geqslant 0$; hence \LLPO holds.
\end{proof}

Since  a space-filling curve cannot be injective classically it is also not injective constructively. The question of whether we can prove a stronger version of non-injectiveness for an arbitrary space-filling curve constructively appears to be a tricky one.
\begin{Qu}
	Is every space-filling curve non-injective? More precise:  if $f$ is a space-filling curve, then does there exist $x \neq y$ with $f(x) = f(y)$?
\end{Qu}

It seems feasible to alter the constructions of $H$ and $P$ slightly by using overlapping square decompositions to obtain constructively surjective space-filling curves. However, there is a more straightforward and cleaner construction, which follows an idea of Lebesgue.

\begin{Lem} \label{Lem:spacefilling}
There exists a space-filling curve.	
\end{Lem}
\begin{proof}
As we will see in a later section (Lemma \ref{Lem:Fp}) there exists a uniformly continuous, surjective function $F:\CS \to [0,1]$. If we compose this function with the uniformly continuous bijection $\varphi:\CS \to \left(\CS\right)^2$ defined by \[\varphi(\alpha_1 \alpha_2 \alpha_3 \dots) = (\alpha_1 \alpha_3 \dots , \alpha_2 \alpha_4 \dots)  \]
we get a uniformly continuous surjection $f:\CS \to [0,1]^2$. By Lemma \ref{Lem:extendingfunctions} we can extend this to a uniformly continuous function $\tilde{f}:[0,1] \to [0,1]^2$. Obviously $\tilde{f}$ is still surjective.
\end{proof}
\begin{Lem}
There exists a uniformly continuous surjection $f: \RR \to \RR^2$	
\end{Lem}
\begin{proof}
Straightforward with the previous lemma.	
\end{proof}

\begin{Cor}
In \RUSS there exists a continuous surjection $f:[0,1] \to \RR^2$.	
\end{Cor}
\begin{proof}
In \RUSS there exists an unbounded point-wise continuous surjective map $h:[0,1] \to \RR$  (see Proposition \ref{Pro:equivs_of_SS}). That map, composed with the one from the previous lemma gives us the desired result.
\end{proof}

We can use the above construction of a space-filling curve to partially answer a question posed by K.~Weihrauch, who asked whether in TTE \cite{Weihrauch:2000uq} there exist computable curves connecting pairs of diagonally opposite points of the unit square that do not intersect at a computable point. Such functions exist in \RUSS, however the original paper (1976 by S.N.~Manukyan) is in Russian and hard to obtain, so the details of the construction are not clear. Recently K.~Weihrauch has answered his own question and has shown that in TTE there cannot be a strong counterexample such as in \RUSS \cite{kW17}.

 The result below is transferrable to TTE and shows the weaker statement that that there cannot be an algorithm taking such curves as inputs and computing the point of intersection.
\begin{Pro}
\LLPO is equivalent to the following statement:
If $g,h:[0,1] \to [0,1]^2$	are point-wise continuous such that $g(0) =(0,0)$, $g(1) =(1,1)$, $h(0) =(0,1)$, and $h(1) =(1,0)$, then they intersect, that is there exist points $t_1,t_2$ such that $g(t_1) = h(t_2)$.
\end{Pro}
\begin{proof}
For one direction we will show that the statement implies the intermediate value theorem (see Proposition \ref{Pro:equivs_of_LLPO}.\ref{item:IVT}). So let $f:[0,1] \to \RR$ be continuous such that $f(0) <0$ and $f(1) > 0$. Without loss of generality we may assume that $f([0,1]) \subset [-1,1]$ and that $f(0)=-1$ and $f(1) = 1$. Now define $h,g:[0,1] \to [0,1]^2$ piecewise linearly by 
\[ h(x) = \begin{cases} 
 -\frac{3}{2}x+1 & \text{ if }	0 \leqslant x \leqslant \frac{1}{3} \\
 \frac{1}{2}  & \text{ if }	\frac{1}{3} \leqslant x \leqslant \frac{2}{3} \\
 -\frac{3}{2}x + \frac{3}{2}  & \text{ if }	\frac{2}{3} \leqslant x \leqslant 1 \\
 \end{cases}
\]
and
\[ g(x) = \begin{cases} 
 x & \text{ if }	0 \leqslant x \leqslant \frac{1}{3} \\
 \frac{1}{6}f(3x-1) + \frac{1}{2}  & \text{ if }	\frac{1}{3} \leqslant x \leqslant \frac{2}{3} \\
 x  & \text{ if }	\frac{2}{3} \leqslant x \leqslant 1  \ .\\
 \end{cases}
\]
 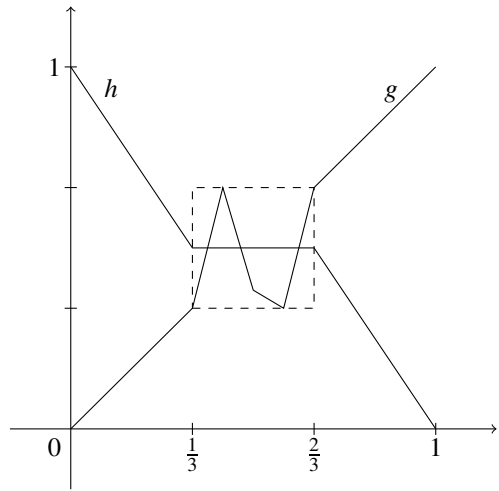
\begin{figure}[ht]
\centering
\begin{tikzpicture}[x=0.8cm,y=0.8cm]
\draw[->] (-1,0) -- (7,0);
\draw[->] (0,-1) -- (0,7);
\draw (2,0.1) -- (2,-0.1);
\draw (4,0.1) -- (4,-0.1);
\draw (6,0.1) -- (6,-0.1);
\draw (-0.1,2) -- (0.1,2);
\draw (-0.1,4) -- (0.1,4);
\draw (-0.1,6) -- (0.1,6);
\draw (0,6) -- (2,3) -- (4,3) -- (6,0);
\draw (0,0) -- (2,2) -- (2.5, 4) -- (3,2.3) -- (3.5, 2) -- (4,4) -- (6,6);

\draw[dashed] (2,2) -- (4,2) -- (4,4) -- (2,4) -- (2,2);

\path (6,0) node[anchor=north] {$1$};
\path (2,0) node[anchor=north] {$\frac{1}{3}$};
\path (4,0) node[anchor=north] {$\frac{2}{3}$};
\path (0,0) node[anchor=north east] {$0$};
\path (0,6) node[anchor=east] {$1$};

\path (0.4,5.5) node[anchor=base west] {$h$};
\path (5,5.5) node[anchor=base west] {$g$};

\end{tikzpicture}
\caption{A scaled copy of $f$ is placed in the middle third square of the unit square. The roots of $f$ are then linearly related to the points of intersection of $h$ and $g$.}
\end{figure}

It is easy to see that if $t \mapsto (t,h(t))$ and $t \mapsto (t,g(t))$ intersect, then they must do so for  $t \in [\frac{1}{3},\frac{2}{3}]$, which implies that $f(3t-1) = 0$. Thus we have found a root of $f$.

Conversely assume that \LLPO holds, and let $h$ and $g$ be as in the statement of the proposition. Then we can use the space-filling curve $\tilde{f}$ constructed in the previous lemma to consider the function  and define the function $M: [0,1] \to \RR$ by
\[ M(x) = d(h (\tilde{f}_1(x) ),g(\tilde{f}_2(x)) \ . \]
By construction $M$ is uniformly continuous. It is also easy to see that $\inf M = 0$. Thus by \ref{Pro:WKL_Min} there exists $x \in [0,1]$ such that $M(x) = 0$, and hence the two curves intersect.
\end{proof}

\subsection{The Greedy Algorithm}
A matroid $(E, \mathcal{F})$ consists of a finite set $E$ together with a collection $\mathcal{F}$ of subsets $E$, such that 
\begin{enumerate}[label=I.\arabic*]
\item $\emptyset \in \mathcal{F}$
\item $A \in \mathcal{F} \land B \subset A \implies B \in \mathcal{F}$
\item $A,B \in \mathcal{F} \land \abs{B}=\abs{A}+1 \implies \ex{x \in B}{ A \cup \menge{x} \in \mathcal{F}}$.
\end{enumerate}
In the following, we will also assume that all the sets in  $\mathcal{F}$ are decidable. 
A nice introduction to matroids is \cite{jO06}.
The standard example of a matroid are the linear independent subsets of a finite set of vectors in a vector space. Property (I.3) then is the Steinitz exchange lemma. 
This also justifies the terminology of calling a set $A\in \mathcal{F}$ independent. A maximal independent set $A$ is one that is independent and such that there is no $B \in \mathcal{F}$ with $A \subsetneq B$.
Matroids are also important since they provide the minimal setting\footnote{One can actually do without the  property (I2). Such structures are known as greedoids.} on which a Greedy algorithm yields an optimal solution. A Greedy strategy is  the most na\"\i{}ve strategy possible: build up your solution by always choosing the most tantalising looking element next. In the following we assume that there is also a weighting function $w : E \to \RR$. The goal is to find a maximal independent set $A \in \mathcal{F}$ that has minimal weight i.e.\ that is such that the sum \[ \sum_{a \in A} w(a) \]
is smaller than that for any other maximal independent set. 
\begin{Pro}
The following are equivalent:
\begin{enumerate}
\item \label{greedy1} LLPO 
\item \label{greedy2} If $\menge{x_{1}, \dots,x_n} \subset \RR$, then there exists $1 \leqslant i \leqslant n$ with $x_{i} \leqslant x_{j}$ for all $1 \leqslant j \leqslant n$.
\item \label{greedy3} The Greedy algorithm. I.e.\ if $(E, \mathcal{F})$ is a matroid and $w:E \to \RR$ a function, then there exists a maximal independent subset with a minimal weight.
\end{enumerate}
\end{Pro}
\begin{proof}
$(\ref{greedy1}) \implies (\ref{greedy2})$ can be proved using induction and \ref{Pro:equivs_of_LLPO}. 
$(\ref{greedy2}) \implies (\ref{greedy3})$: a maximal independent set with minimal weight is just constructed by the Greedy algorithm: 
\begin{algorithmic}[1]
\State Let $F_{0} = \emptyset$.
\For{$i \geqslant 0$}
\State Let $Z_{i} = \set{z_{i} \in E \setminus F_{i}}{F_{i} \cup \menge{z_{i}} \in \mathcal{F}}$.
\If{$Z_{i} = \emptyset$} terminate and \textbf{return} $F_{i}$ .
\Else ~\State choose $y \in Z_{i}$ such that $w(y) \leqslant  w(z_{i})$ for all  $z_{i} \in Z_{i}$ \State let $F_{i+1} = F_{i} \cup \menge{y}$
\EndIf
\EndFor
\end{algorithmic}
\bigskip
Notice, that in step 6 we need (\ref{greedy2}).
For $(\ref{greedy3}) \implies (\ref{greedy1})$ let $x,y \in \RR$. Then 
\[ \mathcal{M} = \left(\menge{1,2}, \menge{\emptyset, \menge{1} ,\menge{2}}\right)  \]
is a matroid. Now define $w: \menge{1,2} \to \RR$ by setting $w(1)=x$ and $w(2)=y$. If the Greedy algorithm works, then either $\menge{1}$ or $\menge{2}$ has minimal weight. But that means that either $x = w(1) \leqslant w(2) = y$ or $  w(2) = y \leqslant x = w(1)$.
\end{proof}
\begin{Rmk}
It is almost trivial to see that adapting the Greedy algorithm, we can, for every $\varepsilon > 0$  efficiently find a maximal independent set $A$ such that 
\[  \sum_{a \in A} w(a) - \varepsilon < \sum_{b \in B} w(b) \ . \]
\end{Rmk}

\subsection{Sharkovskii's Theorem}
Sharkovskii's\footnote{There are various alternative ways of spelling Sharkovskii. Among them Sharkovsky \cite{sE07}, Sarkovskii \cite{rD87}, Sarkovski \cite{jB12}.} Theorem from 1964 is, in our opinion, one of the most entertaining results in mathematics: simple to state yet utterly surprising. It concerns itself with a very rudimentary type of dynamical system, namely a point-wise continuous function $f$ acting on $\RR$. We are interested in fix points of $f$ and \define{(prime) $k$-periodic points} which are points $x$ such that $f^{k}(x) = x $ but $\lnot (f^{i}(x)= x) $. 
Let us now define a total order $\triangleright$ on the natural numbers as indicated below
\begin{multline*}
3  \triangleright 5 \triangleright 7 \triangleright 9 \triangleright \dots \triangleright 2 \cdot 3 \triangleright 2 \cdot 5 \triangleright 2 \cdot 7 \triangleright \dots 
2^{2} \cdot 3 \triangleright 2^{2} \cdot 5 \triangleright 2^{2} \cdot 7 \triangleright \phantom{2} \\\dots 2^{3} \cdot 3 \triangleright 2^{3} \cdot 5 \triangleright 2^{3} \cdot 7  \triangleright \dots 2^{n+1} \triangleright 2^n \triangleright \dots 8 \triangleright 4 \triangleright 2 \triangleright 1 \ .
\end{multline*}
Sharkovskii's Theorem states:
\begin{principle}[ST]{\textrm{ST}}
If $f$ has a $k$-periodic point and $k \triangleright \ell$, then $f$ has points of prime period $\ell$.
\end{principle}
In particular, if a function has a $3$-periodic point it has points of all periods.
The only non-constructive part in the usual proof \cite{rD87} is the following technicality, which is actually equivalent to \LLPO.
\begin{Pro} 
\LLPO is equivalent to the statement that if $f: [0,1] \to \RR$ is a uniformly continuous map and $[a,b]$ an interval such that $f([0,1]) \supset [a,b]$, then there exists $[c,d] \subset [0,1]$ such that $f([c,d]) = [a,b]$.
\end{Pro}
\begin{proof}
  To see that this principle implies LLPO let $(a_n)_{n \geqslant}$ be a binary sequence with at most one $1$. Let \[ a = \sum_{n \geqslant 1}
  \frac{(-1)^n a_n}{2^n} \] and define a piecewise linear function $f: [0,1] \to [0,1]$
  by \[ f(x) = \begin{cases}
    (3 +3a) x, & x \in \bracks*{0,\frac{1}{3}} \ ; \\
    -(3 + 6a)x +2 + 3a, & x \in \bracks*{\frac{1}{3}, \frac{2}{3}} \ ; \\
     (3+3a) x -2 -3a, & x \in \bracks*{\frac{2}{3},1} \ . \\
  \end{cases} \] \\
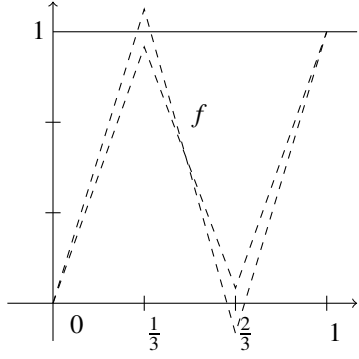
\begin{figure}[ht]
\centering
\begin{tikzpicture}[y=-1cm]
\draw[->] (3,7) -- (7.6,7);
\draw[->] (3.6,7.5) -- (3.6,3);
\draw (4.8,6.9) -- (4.8,7.1);
\draw (6,6.9) -- (6,7.1);
\draw (7.2,6.9) -- (7.2,7.1);
\draw (3.7,5.8) -- (3.5,5.8);
\draw (3.5,4.6) -- (3.7,4.6);
\draw (3.6,3.4) -- (7.6,3.4);
\draw[dashed] (3.6,7) -- (4.8,3.6) -- (6,6.8) -- (7.2,3.4);
\draw[dashed] (3.6,7) -- (4.8,3.1) -- (6,7.4) -- (7.2,3.4);
\draw (5.4,5.6) -- cycle;
\path (3.2,3.5) node[anchor=base west] {$1$};
\path (4.7,7.5) node[anchor=base west] {$\frac{1}{3}$};
\path (5.9,7.5) node[anchor=base west] {$\frac{2}{3}$};
\path (3.7,7.4) node[anchor=base west] {$0$};
\path (7.1,7.5) node[anchor=base west] {$1$};
\path (5.3,4.6) node[anchor=base west] {$f$};
\end{tikzpicture}
\caption{We think of $a$ as being very small. Depending on  $a$ $f$ crosses the lines $y = 0$ and $y=1$ or misses both.}
\end{figure} 
As $f([0,\frac{1}{3}]) \supset [0, \frac{2}{3}] $ and $f([\frac{2}{3},1]) \supset [\frac{1}{3},1] $ we see that $f([0,1]) \supset [0,1]$. So assume there exists an interval $[c,d] \subset [0,1]$ such that $ f([c,d]) = [0,1] $. Now either $c > 0 $ or $c < \frac{1}{6} $. In the first case $a \geqslant 0$ as the assumption $a < 0$ implies that for all $x \in [c,1]$ $f(x) > 0 $. In the second case either $d < 1$ or $d> \frac{5}{6} $. If $d < 1$ we can conclude as before that $ a \geqslant 0$. In the only case left that $d> \frac{5}{6} $ and $c < \frac{1}{6} $ we must have that $a \leqslant 0$ as the assumption $a > 0$ would imply that there are $x \in [c,d]$ such that $f(x) > 1$.
So we can decide whether $a \leqslant 0$ or $a \geqslant 0$. In the first case we must have that for all $n$ $\alpha_{2n} = 0$.
In the second case we must have that for all $n$ $\alpha_{2n+1} = 0$.

To see that the converse holds assume \LLPO holds. Furthermore let $f: [0,1] \to \RR$ be a uniformly continuous map and $[a,b]$ an interval such that $f([0,1]) \supset [a,b]$. Choose $a_0$ and $b_0$ such that $f (a_0) = a$ and $f(b_0) =b$. Without loss of generality $a_0 \leqslant b_0$.

We will define a sequence of nested intervals $([a_n,b_n])$ such that
  \begin{enumerate}
    \item $\abs*{a_{n+1} - b_{n+1}} \leqslant \frac{1}{2} \abs*{a_n - b_n}$
    \item $\sup \menge{ f \left( \bracks*{a_0, a_n} \right)} \leqslant b$ and $\sup \menge{ f \left(\bracks*{a_0, b_n}\right)} \geqslant b$
    \item $f(a_n) \leqslant b$ and $f(b_n) \geqslant b$.
  \end{enumerate}
So assume we have constructed $a_n$ and $b_n$. As we assume \LLPO either \[ \sup \menge{ f \bracks*{a_0, \frac{1}{2} (a_n + b_n)}} \geqslant b \text{ or } \sup \menge{ f \left(\bracks*{a_0, \frac{1}{2} (a_n + b_n)} \right)} \leqslant b \ . \] In the first case set $a_{n+1} = a_n $ and as we assumed \LLPO the minimum principle holds \cite{hI90} and so we can choose $b_{n+1} \in  (a_0,\frac{1}{2} (a_n + b_n))$ such that $f(b_{n+1}) \geqslant b$. In the second case set $b_{n+1} = b_n$ and choose $a_{n+1} \in \left(a_0, \frac{1}{2} (a_n + b_n) \right)$, such that $f(a_{n+1}) \leqslant b$. Both sequences converges to the same limit $d = a_\infty = b_\infty$ and furthermore $\sup \menge{f ([a_0, d])} \leqslant b$, $\sup \menge{f( [a_0, d])} \geqslant b$, $f(d) \leqslant b$ and $f(d) \geqslant b$. Hence $\sup \menge{f ([a_0, d])} = b$ and $f(d) = b$. A similar argument shows that there is a $c \in [a_0,d]$ such that $\inf \menge{f [c, d]} = a$ and $f(c) = a$. Hence $f [c,d] = [a,b]$.
\end{proof}
Of course this still leaves the possibility that there is a fully constructive proof of Sharkovskii's theorem. However, as the next Brouwerian counterexample shows such a proof does need to use \LLPO.

\begin{Pro}
For all natural numbers $n$ the following holds: If every (uniformly) continuous map that has a $3$-periodic point has a fixpoint, then \LLPO holds.
\end{Pro}
\begin{proof}
  Again let $\alpha_n$ be binary sequence with at most one $1$. Let \[a = \sum_{n \geqslant 1} \frac{(-1)^n \alpha_n}{2^n} \] and define a piecewise linear function $f: [0,1] \to [0,1]$
  by \[ f(x) = \begin{cases}
    3 x +\frac{1}{4} , & x \in \bracks*{0,\frac{1}{4}} \ ; \\
     (4a -2)x+ \frac{3}{2} - a, & x \in \bracks*{\frac{1}{4}, \frac{2}{4}} \ ; \\
     x+ a, & x \in \bracks*{\frac{2}{4}, \frac{3}{4}} \ ; \\
    (-3 -4 a)x + 3 +4a , & x \in \bracks*{\frac{3}{4},1} \ . \\
  \end{cases}   \]
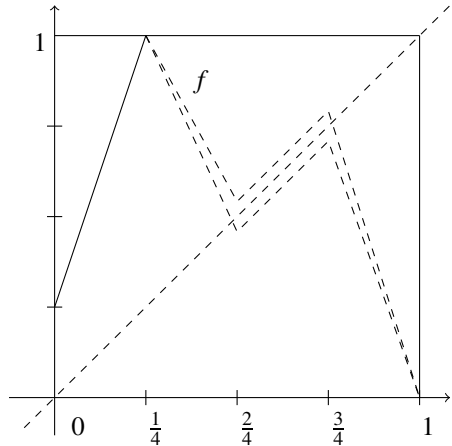
\begin{figure}[ht]
\centering
\begin{tikzpicture}[y=-1cm]
\draw[->] (3,7) -- (8.8,7);
\draw[->] (3.6,7.5) -- (3.6,1.8);
\draw (4.8,6.9) -- (4.8,7.1);
\draw (6,6.9) -- (6,7.1);
\draw (7.2,6.9) -- (7.2,7.1);
\draw (3.7,5.8) -- (3.5,5.8);
\draw (3.5,4.6) -- (3.7,4.6);
\draw (8.4,6.9) -- (8.4,7.1);
\draw (3.6,2.2) -- (8.4,2.2);
\draw (8.4,2.2) -- (8.4,7);
\draw (3.5,3.4) -- (3.7,3.4);
\draw[dashed] (3.2,7.4) -- (8.8,1.8);
\draw[dashed] (4.8,2.2) -- (6,4.4) -- (7.2,3.2) -- (8.4,7);
\draw (3.6,5.8) -- (4.8,2.2);
\draw[dashed] (4.8,2.2) -- (6,4.8) -- (7.2,3.6) -- (8.4,7);
\path (4.7,7.5) node[anchor=base west] {$\frac{1}{4}$};
\path (5.9,7.5) node[anchor=base west] {$\frac{2}{4}$};
\path (7.1,7.5) node[anchor=base west] {$\frac{3}{4}$};
\path (8.3,7.5) node[anchor=base west] {$1$};
\path (3.2,2.4) node[anchor=base west] {$1$};
\path (3.7,7.5) node[anchor=base west] {$0$};
\path (5.3,2.9) node[anchor=base west] {$f$};
\end{tikzpicture}
    \caption{We cannot decide whether the graph of $f$ between $\frac{2}{4}$ and $\frac{3}{4}$ lies above or below the diagonal.}
  \end{figure}
  $f(0) = \frac{1}{4}$, $f(\frac{1}{4}) = 1$  and $f(1)=0$ so $f$ has a 3--periodic point. If $f$ has a fixed point $x$, then
  either $x < \frac{3}{4}$ or $x > \frac{2}{4}$. In the first case $a \leqslant 0$, as the assumption
  $a > 0$ implies that $f(x) > x$ for all $x \in \bracks*{0,\frac{3}{4}}$. Then for all $\alpha_{2n+1} = 0$ for all $n$. Similar in the second case $a \geq
  0$ and $\alpha_{2n} = 0 $ for all $n$.
\end{proof}
Putting these two results together we get the following corollary.
\begin{Cor}
  Sharkovskii's Theorem is equivalent to LLPO.
\end{Cor}
\subsection{Graph colourings}
We assume that the reader is familiar with basic graph theoretic definitions. A nice little theorem by Erd\H{o}s and Bruijn \cite{rD05}*{Theorem 8.1.3}, reminiscent of the compactness theorem in logic. is
\begin{principle}[EBk]{$\textrm{EB}_{k}$} Let $G=(V,E)$ be a countable graph and $k \in \NN$. If every finite subgraph of $G$ is $k$-colourable, then so is $k$.
\end{principle}
We will show that the countable case is equivalent to \WKL. 
\begin{Pro}
$\textrm{EB}_{k}$ is equivalent to \WKL for every $k \geqslant 2$. 
\end{Pro}
\begin{proof}
To see that \WKL implies EB$_{k}$ is straightforward. Since we assume countability, we can write $V=\menge{v_{1}, v_{2}, \dots} $. Every $u \in k^{\ast}$ of length $m$ can be seen as a (not necessarily admissible) colouring of the subgraph $G[v_{1}, \dots, v_{m}]$ and every $\alpha \in k^{\NN}$ as a (not necessarily admissible) colouring of $G$. Now simply consider 
\[ T = \set{u \in \cS}{u \text{ is an admissible colouring of } G } \ .\]
$T$ is easily seen to be a tree and the assumption that every finite subset of $G$ has an admissible $k$-colouring translates into $T$ being infinite. By \WKL $T$ has an infinite path $\alpha$, which represents an admissible colouring of $G$.
Conversely, we will first show that EB$_{2}$ implies \LLPO.
So let $(a_n)_{n \geqslant 1}$ be a binary sequence with at most one term equal to $1$. Now, let $G=(\NN,E)$, with 
\[ E= \set{ \menge{n,n+2}}{ \fa{m \leqslant n}{a_{m} = 0}} \cup \set{ \menge{n,n+1} }{ a_n = 1 } \ . \]
That is, as long as $a_n=0$, $G$ can be drawn as two parallel lines. Should we hit $a_n=1$, we connect the upper and the lower strand. The way this graph is constructed if there is an odd (even) $n$ with $a_n=1$, then the vertices $1$ and $2$ are connected by a path of odd (even) length.
\begin{figure}[ht]
\centering
\begin{tikzpicture}
  \node [circle,fill=black,inner sep=1pt,label={above:$1$}] at (1,6) {};
  \node [circle,fill=black,inner sep=1pt,label={above:$3$}] at (2,6) {};
  \node [circle,fill=black,inner sep=1pt,label={above:$5$}] at (3,6) {};
	\draw (1,6) -- (2,6) -- (3,6);
	\draw[dotted] (3,6) -- (4,6) ;
  \node [circle,fill=black,inner sep=1pt,label={below:$2$}] at (1,5) {};
  \node [circle,fill=black,inner sep=1pt,label={below:$4$}] at (2,5) {};
  \node [circle,fill=black,inner sep=1pt,label={below:$6$}] at (3,5) {};
	\draw (1,5) -- (2,5) -- (3,5);
	\draw[dotted] (3,5) -- (4,5) ;
  \node [circle,fill=black,inner sep=1pt] at (6,7) {};
  \node [circle,fill=black,inner sep=1pt] at (7,7) {};
	\draw (6,7) -- (7,7) -- (7,6);
	\draw[dotted] (5,7) -- (6,7) ;
  \node [circle,fill=black,inner sep=1pt] at (6,6) {};
  \node [circle,fill=black,inner sep=1pt] at (7,6) {};
	\draw (6,6) -- (7,6) ;
	\draw[dotted] (5,6) -- (6,6) ;
  \node [circle,fill=black,inner sep=1pt] at (6,5) {};
  \node [circle,fill=black,inner sep=1pt] at (7,5) {};
  \node [circle,fill=black,inner sep=1pt] at (8,5) {};
	\draw (6,5) -- (7,5) -- (8,5);
	\draw[dotted] (5,5) -- (6,5) ;
  \node [circle,fill=black,inner sep=1pt] at (6,4) {};
  \node [circle,fill=black,inner sep=1pt] at (7,4) {};
	\draw (6,4) -- (7,4) -- (8,5) ;
	\draw[dotted] (5,4) -- (6,4) ;
\end{tikzpicture}
\caption{$G$ consists of two separate strands as long as $a_n=0$. If there is $n$ with $a_n=1$, then $G$ is one strand of even or odd length depending whether $n$ is odd or even.}
\end{figure}
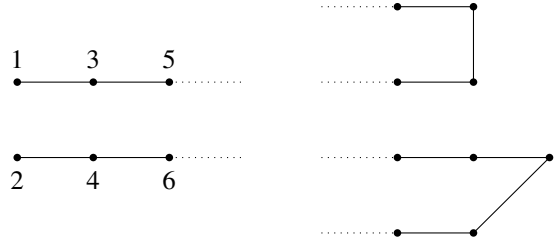
 It is easy to see that every finite sub-graph can be $2$-coloured. If all of $G$ can be coloured, then if the vertices $1$ and $2$ get the same colour there cannot be an even $n$ with $a_n=1$, which means $\fa{n \in \NN}{a_{2n}=0}$. Similarly if $1$ and $2$ are coloured with different colours, then $\fa{n \in \NN}{a_{2n+1}=0}$. 
For an arbitrary $k>2$ we simply add one copy of the complete graph $K_{n-2}$, and connect every one of its vertices with every vertex of $G$. Thus we ensure, that the ``original'' $G$ still gets coloured with exactly two colours.
If one is happy to use countable choice, then \LLPO implies \WKL and we are done. We can, however, adapt the construction above to show that EB$_{k}$ implies \WKL instance-wise and without countable choice. To this end assume that $T \subset \cS$ is a decidable infinite tree. We can fix a binary sequence $\cramped{(b^{(u)}_n)_{n \geqslant 1}}$ by 
\[ 
b^{(u)}_n = 1 \iff  \fa{w \in 2^n}{ u \ast w \notin T} \ . 
\]
 To turn the $b^{(u)}$ into sequences with at most one 1, we consider 
\[
d_{2n}^{(u)} = b_n^{(u0)} \text{ and } d_{2n+1}^{(u)} = b_n^{(u1)}
\]
and
\[
a_n^{(u)} = d_n^{(u)} - \max_{i < n} \menge{d_i^{(u)}} \ .
\]
If $T_u = \set{v \in \CS}{u \ast v \in T}$ is infinite, then there cannot be $n,m$ such that $b_n^{(u0)} =1$ and $b_m^{(u1)} =1$. Therefore if $T_u$ is infinite and $a^{(u)}_{2n} = 0$ for all $n \in \NN$, then $b^{(ui)}_n =0$ for all $n \in \NN$, and  if $a^{(u)}_{2n+1} = 0$ for all $n \in \NN$, then $b^{(ui)}_n =0$ for all $n \in \NN$. 

We can construct a graph $G^{(u)}$ for every $a^{(u)}$ as done above. The formal sum (disjoint union) of these graphs 
\[ H = \bigoplus_{u \in T} G^{(u)} \]
is also such that every finite subgraph can be coloured with $k$ colours. 

If there is an admissible colouring of $H$, then we can, as above, decide whether $a^{(u)}_n$ is zero for all even or for all odd terms, which means that if $T_u$ is infinite we can decide  whether $T_{u0}$ or $T_{u1}$ is infinite.

So we can recursively construct $\alpha$ such that $T_{\overline{\alpha}n}$ is infinite. In particular $\overline{\alpha}n \in T$ for all $n \in \NN$, which means we have found a path through $T$.
\end{proof}
\subsection{Variations of \texorpdfstring{\WKL}{WKL}} \label{Ssec:wklvariants}
Reminiscent of the way  \WWKL (Section \ref{sec:WWKL}) is a weakening of \FAND is the following weakening of \WKL.\footnote{This version was communicated to us by M.\ Hendtlass by Email. We do not know who deserves credit for proposing it.}
Let $k<1$. We now require that for each $n$ our tree not only have at least one sequence of length $n$, but for a certain percentage of all sequences of length $n$ to be in the tree.
\begin{principle}[WKLp]{\WKLp{k}}  \label{PR:WKLp}
If $T$ is a binary decidable tree such that 
\[ \frac{\abs*{\set{u \in T}{\abs{u} = n}}}{2^n} > k \ , \] then $T$ admits an infinite path.
\end{principle}
Of course, if $r<k<1$, then $\WKLp{r} \implies \WKLp{k}$, and \WKLp{k} is implied by \WKL. We can prove the following converse.
\begin{Pro}
For any $k \leqslant \frac{1}{2}$, we have $\WKLp{k} \implies$ \LLPO.
\end{Pro}
\begin{proof}
Let $(a_n)_{n \geqslant 1}$ be a binary sequence with at most one term equal to $1$. Now define a set $T \subset \cS$ by
\[ T = \set{ 0u }{ \fa{ m \leqslant \abs{u}}{a_{2m}=0 }} \cup \set{ 1u }{ \fa{ m \leqslant \abs{u}}{a_{2m+1}=0 }} \ . \] Then it is easy to see that $T$ is a tree and that 
$ | \set{ u \in T }{ \abs{u} = n } | /2^n \geqslant \frac{1}{2} > k$. Hence, by our assumptions, $T$ admits an infinite path $\alpha \in \CS$. Now if $\alpha(1)=0$, then there cannot be an even $n$ with $a_n = 1$.  Similarly if  $\alpha(1)=1$  there cannot be an odd $n$ with $a_n = 1$.
\end{proof}
So under the assumption of countable choice \WKLp{k} is actually equivalent \WKL, at least for $k < \frac{1}{2}$. It remains an open question whether the same holds for $k \geqslant \frac{1}{2}$. 
\begin{Qu}
To what principle is \WKLp{k} equivalent to for $k > \frac{1}{2}$?
\end{Qu}

Notice that this question seems to be related to the question whether \LLPOn{n} (which will be introduced in Section \ref{Sec:LLPOn}).

\subsection{Compactness of Propositional logic}

The language of first order logic is, as usual, defined inductively via
\begin{enumerate}
\item Every proposition symbol $A_{0},A_{1},A_{2}, \dots$ is a formula.
\item If $\alpha, \beta$ are formulas, then so is $\lnot \alpha$, $\alpha \land \beta$, $\alpha \lor \beta$, $\alpha \rightarrow \beta$.
\end{enumerate} 
It is straightforward to show that an truth assignment $\mathcal{B}$ from the set of all proposition symbols to the set of truth values $\menge{\mathbf{f}, \mathbf{t}}$ can be uniquely extended to one defined on all formulas. As usual we also assume that $\mathcal{B}(A_{0}) = \mathbf{f}$, and write $\bot$ instead of $A_{0}$. A truth assignment is called a \define{model} of a set of formulas $\Gamma$, when $\mathcal{B} (\alpha) = \mathbf{t}$ for all $\alpha \in \Gamma$. In this case we write $\mathcal{B} \vDash \Gamma$. Notice that in this propositional case we have  $\mathcal{B} \vDash \varphi \lor \lnot \varphi$ for any formula $\varphi$ and any $\mathcal{B}$.

If a set $\Gamma$ has a model it is called \define{satisfiable}. A basic result of logic, known under the name compactness, is that if every finite subset of a set $\Gamma$ is satisfiable, then so is $\Gamma$ itself. The proof generally uses K\"{o}nig's Lemma. It is thus not surprising that:

\begin{Pro} \WKL is equivalent to the statement that if every finite subset of a set $\Gamma$ is satisfiable, then so is $\Gamma$.
\end{Pro}
\begin{proof}
One direction is the usual proof, such as the one in \cite{rS95}*{Chapter III.1}, which apart from using \WKL, is actually constructive.

Conversely assume $T$ is an infinite decidable tree, so that there exists $u \in T$ with $\abs{u} = n$ for each $n$, and define 
\[ \alpha_n = \bigvee_{\substack{u \in T \\ \abs{u} = n}} \bigwedge_{1 \leqslant i \leqslant n} B_{u,i} \ , \]
where
\[B_{u,i} =  \begin{cases} 
A_{i} & \text{if } u(i) = 1 \\ 
\lnot A_{i} & \text{otherwise};
\end{cases}
\]
We are going to show that for every $m \in \NN$ the subset $\menge{\alpha_{1}, \dots, \alpha_{m}}$ of $\Gamma$ has a model. Since $T$ is infinite there exists $v \in T$ with $\abs{v}=m$. Define $\mathcal{B}^{\prime}$ by 
\[
\mathcal{B^{\prime}}(A_{i}) = \begin{cases}
\mathbf{t} & \text{if } v(i)=1 \text{ and } i \leqslant m  \\
\mathbf{f} & \text{otherwise}.
\end{cases}
\]
For all $j \leqslant i \leqslant m$, since$T$ is a tree $\overline{u}i \in T$, and therefore we have that $\mathcal{B^{\prime}}\left(B_{\overline{u}i,j} \right) = \mathbf{t}$.
Hence $\mathcal{B^{\prime}}(\alpha_{i}) = \mathbf{t}$ for all $i \leqslant m$, which means that $\mathcal{B^{\prime}}$ is a model of $\menge{\alpha_{1}, \dots, \alpha_{m}}$.

Now assume that $\mathcal{B}$ is a model of $\Gamma = \menge{\alpha_{1}, \alpha_{2}\dots}$. We claim that $\sigma \in \CS$ defined by 
\[ \sigma(i) = 1 \iff \mathcal{B}(A_{i}) = \mathbf{t} \  \]
is a path through $T$.
\end{proof}

\chapter{\MP and Below} \label{Ch:MP}
We will start this chapter by introducing some notation. If a number $x \in \RR$ is such that $\lnot \neg(0 < x)$ we say that $x$ is \define{almost positive} and write $0 \lessdot x$. Furthermore, we write  $y \lessdot x$ instead of $0\lessdot  x - y$ and $x \gtrdot y$ interchangeably with $y \lessdot x$. Note that since $x \leqslant y$  is, constructively, defined as $\lnot (y < x)$ \cite{dB85}, we have that 
\[  x \lessdot y \iff \lnot (y \leqslant x) \ . \]

Real numbers $x$ that satisfy 
\begin{equation} \label{Eqn:pspos}
\fa{y \in \RR}{0 \lessdot y \lor y \lessdot x}
\end{equation} are called \define{pseudo-positive}. If we apply the pseudo-positiveness property to a number $x$ itself the right disjunct is always ruled out and we get that $x$ is almost positive.

 So we have the following notions for a number being positive.
\[ \text{positive} \implies \text{pseudo-positive} \implies \text{almost positive} \implies \text{non-negative}  \ ,\]
or in symbols, for $x\in \RR$
\[ 0 < x \implies ( x  \text{ pseudo-positive} \implies) \quad 0 \lessdot x \implies 0 \leqslant x \ . \]

\begin{Lem}  \label{Lem:pos-order}
For all $x,y,z \in \RR$
\begin{enumerate}
\item \label{Lem:inequ1} $ \lnot \left( \neg(x > y) \land \neg(x = y) \land \neg(x < y) \right)$.
\item \label{Lem:inequ3} if $x \leqslant y$ and $y \lessdot z$, then $x \lessdot z$. (This implies that $\lessdot$ is transitive).
\item \label{Lem:inequ2} If $x = \max \menge{y,z}  $ and $\neg(x =y)$, then $x =z$.
\end{enumerate}

\end{Lem}
\begin{proof}
\begin{enumerate}
\item If $\neg(y > x)$ and $\neg(y < x)$, then $x=y$, since both the assumption that $x<y$ and $y <x$ lead to contradictions. But this in turn contradicts $\lnot (x = y)$.
\item Assume that $x \leqslant y$ and $y \lessdot z$. Furthermore assume that $z \leqslant x$. Then $z \leqslant y$, which is a contradiction. Hence $ \lnot (z \leqslant x)$, which is what we wanted to show.
\item By definition $x \geqslant z$. Assume that $x > z$. Then $y = x$: since either $y >x$ as well as $y<x$ leads to a contradiction. The first one since $x \geqslant y$, the second one as then we could find $r$ such that $x>r>z$ and $x > r > y$ and therefore $x> r \geqslant \max\menge{y,z} = x$. Altogether $x >z$ is a contradiction and therefore $z \leqslant x$, which means that $x=z$. \qedhere
\end{enumerate}
\end{proof}

\section{\texorpdfstring{\MP}{MP}} \label{Sec:MP}
Markov's principle is a weak form of double negation elimination.
\begin{principle}[MP]{\MP} \label{PR:MP}
Every almost positive number is positive: $x \gtrdot 0 \implies x > 0$.

\end{principle}
It was also called LPE\footnote{Limited Principle of Existence} \index{LPE|see {\MP}} by Mandelkern. It is generally accepted by recursive schools of constructive mathematics, where it embodies the strategy of an ``unbounded search''.
\begin{Pro}
The following are equivalent to \MP:
\begin{enumerate}
\item \label{MPequiv1} If $(a_n)_{n \geqslant 1}$ is a binary sequence such that $\lnot \left( \fa{n \in \NN}{a_n=0} \right)$, then $\ex{n \in \NN}{a_n=1}$.
\item \label{MPequiv2} $\fa{x\in\RR}{\lnot (x = 0) \implies x \neq 0}$. \label{Eqn:MPreal}
\item[2½] \label{MPequiv2b} $\fa{x\in\RR}{\lnot (x = 0) \implies \ex{n \in \NN}{n \abs{x} > 1}}$ (The Archimedean Property). 
\item \label{MPequiv3} If $A \subset \NN$ is countable and so is its complement $\overline{A}$, then $A$ is decidable.
\item \label{MPequiv4} Every weakly injective\footnote{A map is called weakly injective if $f(x)=f(y) \implies x =y$, as opposed to injective which means $x\neq y \implies f(x) \neq f(y)$. Notice that a map is weakly injective if and only if $\neg(x=y) \implies \lnot (f(x)=f(y)) $,  by the stability of equality. } map $f:[0,1] \to \RR$ is injective.
\item \label{MPequiv5} Every mapping from a  metric space into a metric space is strongly extensional.
\end{enumerate}
\end{Pro}
\begin{proof}
The equivalences between \MP, \ref{MPequiv1}, and \ref{MPequiv5} are from \cite{dB90} and \cite{hI06a}. 
To see the equivalence between \ref{MPequiv2} and \ref{MPequiv2b} note that, for $x$ we have
\[  x \neq 0 \iff \abs*{x} > 0 \iff \ex{n \in \NN}{\abs{x}> \frac{1}n} \iff \ex{n \in \NN}{n \abs{x}> 1} \ . \]

The equivalence between \ref{MPequiv2} and \MP is obvious. 

To see that \ref{MPequiv3} is equivalent to \ref{MPequiv1} consider a set $A \subset \NN$ such that $A = \menge{a_1, a_2, \dots}$ and $\Complement{A} = \menge{b_1, b_2, \dots}$. Given $m \in \NN$ we can consider the following binary sequence $c_n$ defined by
\[ c_n = \begin{cases} 0 & \text{if } m \notin \menge{a_1, b_1, \dots, a_n,b_n} \\ 1 & \text{otherwise} \ . \end{cases} \]	
Now we have $\lnot \fa{n \in \NN}{c_n =0}$, since the assumption $\fa{n \in \NN}{c_n =0}$ leads to the contradiction that $m \notin A \land m \notin \Complement{A}$. So we can use \MP to get $n$ such that $c_n = 1$, which immediately tells us that $m \in \menge{a_1, b_1, \dots, a_n,b_n}$ which allows us to decide whether $m \in A$ or not.
Conversely let $a_n$ be a binary sequence such that $\lnot \fa{n \in \NN}{a_n = 0}$. We may assume that $a_1 = 0$. Then the set $A = \set{2}{\ex{n \in \NN}{a_n = 1}} \cup \menge{ 1}$ is countable, as well as its complement. So we can decide whether $2 \in A$. But the assumption that $2 \notin A$ implies that $\fa{n \in \NN}{a_n = 0}$. Thus we must have $2 \in A$, which means $\ex{n \in \NN}{a_n = 1}$.

Finally we are going to show that \ref{MPequiv4} is equivalent to \ref{MPequiv2}. Notice that if $f$ is a weakly injective map and $x,y$ are such that $f(x) \neq f(y)$, then $\lnot (x = y)$ which implies that also $\lnot (x-y = 0)$. Thus, using \ref{MPequiv2} we have $x-y \neq 0$, which means that $x \neq y$. So $f$ is injective. Conversely let $x$ be such that $\neg(x = 0)$. Now consider the linear map $f(y) = xy$. It is easy to see that this map is weakly injective, for consider $y,z$ with $\lnot (y =z)$. Now assume $f(y) =f(z)$, which means $xy = xz$. But then the assumption that $y \neq z$ implies that $x =0$, which is a contradiction. So we must have $\lnot (y \neq z))$ which is equivalent to $y=z$. This contradicts our initial assumption that $\lnot (y =z)$, so we have proved $ \lnot (f(y) =f(z))$. By \ref{MPequiv4} $f$ is injective, and thus $0 = f(0) \neq f(1) = x$.
\end{proof}

\section{\texorpdfstring{\WMP}{WMP}} \label{Sec:WMP}
Weak Markov's Principle states that every pseudo-positive number is positive:
\begin{principle}[WMP]{\WMP} \label{PR:WMP}
For all $x \in \RR$, if 
\begin{equation} 
\fa{y \in \RR}{\left( y \gtrdot 0  \ \lor \ y \lessdot x \right)} \ ,
\end{equation}
or equivalently, if 
\begin{equation} 
\fa{y \in \RR}{\lnot (y \leqslant 0) \ \lor \  \lnot (y \geqslant x)} \ ,
\end{equation}

then $0< x$.
\end{principle}
This principle was also called WLPE\footnote{Weak Limited Principle of Existence}\index{WLPE|see {\WMP}}  and ASP\footnote{Almost Separating Principle}\index{ASP|see {\WMP}} by Mandelkern. U. Kohlenbach has shown that it does not hold in a certain intuitionistic formal system \cite{uK02}. This was improved upon by M.\ Hendtlass and B.\ Lubarsky who have recently given a topological model satisfying full IZF as well as dependent choice in which \WMP fails (see Corollary \ref{Cor:Countermodel_WMP}).

\begin{Pro}
The following are equivalent to \WMP:
\begin{enumerate}
\item Every pseudo-positive number is positive.
\item If $\neg(a=b)$ and $\menge{a,b}$ is complete, then $\abs{a-b}>0$.
\item \label{WMP-strongext} Every mapping from a complete metric space into a metric space is strongly extensional.
\item Every real-valued function which is non-decreasing and approximates intermediate values is point-wise continuous. 
\end{enumerate}
\end{Pro}

\begin{proof}
The proofs for these equivalences can be found in \cite{mM88}, \cite{hI91}, \cite{mM83}, respectively. 	
\end{proof}

The following lemmas provide a stepping stone to characterise \WMP in terms of elements of Cantor space.

\begin{Lem} \label{Lem:pseudposequiv}
A non-negative real number $x$ is pseudo-positive if and only if
\[ \fa{y \in \RR}{\neg(x=y) \lor \neg(y=0)} \ . \]
\end{Lem}
\begin{proof}
Since $a\lessdot b \implies \neg(a=b)$ one direction is clear. To prove the other direction let us first assume that $y \in [0,x]$, that is $y \geqslant 0$ and $y \leqslant x $. By our assumption either $\lnot (y = 0)$ or $ \lnot (x=y)$. In the first case $y \gtrdot 0$ by Lemma \ref{Lem:pos-order}. In the second case $y \lessdot x$ as well by Lemma \ref{Lem:pos-order}. Next consider $y \in (-\infty, x]$. Then $y^{\prime} = \max\menge{0,y} \in [0,x]$. Hence by what was just proved either $ y^{\prime} \gtrdot 0 $ or $ y^{\prime} \lessdot x $. In the first case, by part \ref{Lem:inequ2} of Lemma \ref{Lem:pos-order}, $y^{\prime} = y$ and therefore $y \gtrdot 0$.
In the second case, since $y \leqslant y^{\prime}$ also $y \lessdot x$ by part \ref{Lem:inequ3} of Lemma \ref{Lem:pos-order}. So we can perform the final step and consider an arbitrary $y \in \RR$. Consider a $x^{\prime} = \min \menge{x,y} \in (- \infty, x]$. With the previous step we know that either $x^{\prime} \gtrdot 0$ or $x^{\prime} \lessdot x$. In the first case, since $y \geqslant x^{\prime}$, by part \ref{Lem:inequ3} of Lemma \ref{Lem:pos-order}, we have $y \gtrdot 0$.
In the second case, by  part \ref{Lem:inequ2} of  Lemma \ref{Lem:pos-order}, we have $x^{\prime} = y$ and therefore $y \lessdot x$.  
\end{proof}

Even though it seems most of the work for the next proposition is contained in the last lemma, its proof is not at all that straightforward and, interestingly, we need countable choice in \emph{both} directions.
\begin{Pro}
\WMP is equivalent to the following statement. For every $\alpha \in \CS$ such that 
\begin{equation} \label{Eqn:psposvariant}
 \fa{\beta \in \CS}{\left( \neg(\alpha = \beta) \ \lor \ \lnot (\beta = 0) \right)}  
 \end{equation}
 there exists $n$ such that $\alpha(n)=1$.
\end{Pro}
\begin{proof}
Let us assume \WMP holds and let $\alpha \in \CS$ be as described. Without loss of generality we may assume that $\alpha$ is increasing. We claim that $x = \sum_{n \in \NN} \frac{\alpha(n)}{2^n}$ is pseudo-positive. To this end let $y \in \RR$ be arbitrary. Using countable choice construct a flagging sequence $(\lambda_n)_{n \geqslant 1}$, such that if $\alpha(n)=1$ and therefore $x > \frac{1}{2^{n+1}}$, then
\begin{align*}
\lambda_n = 0 & \implies \abs*{x-y} < \frac{1}{2^{n+2}} \ , \\ 
\lambda_n = 1 & \implies \abs*{y} > \frac{1}{2^{n+3}} \ . 
\end{align*}
Again, without loss of generality we may assume that $\lambda$ is increasing. Define $\beta\in \CS$ by 
\[ \beta(n)=\min\menge{\alpha(n),\lambda(n)}\]
If $y=0$, then $\beta=0 $ and if $y=x$, then $\beta = \alpha$.
So either $\neg(\beta = 0)$ or $\neg(\alpha=\beta)$. 

Conversely let $x$ be a pseudo-positive number. Using countable choice construct $\alpha$ such that
\begin{align*}
\alpha(n) = 0 & \implies \abs{x} < \frac{1}{2^n}  \ , \\ 
\alpha(n) = 1 & \implies \abs{x} > \frac{1}{2^{n+1}} \ . 
\end{align*}
We claim that $\alpha$ has the property \ref{Eqn:psposvariant}. To see this consider $\beta \in \CS$. 
Define a sequence $(y_n)_{n \geqslant 1}$
\[ 
y_n = \begin{cases}
x & \text{if } \fa{i \leqslant n}{\alpha(i) = \beta(i) } \\
0 & \text{otherwise} 
\end{cases}
 \ . \]
It is easy to see that $(y_n)_{n \geqslant 1}$ is a Cauchy sequence converging to a limit $y$.  Since $x$ is pseudo-positive by Lemma \ref{Lem:pseudposequiv} either $\lnot (0 = y)$ or $\lnot (x = y)$. 
In the first case $\beta = \alpha$, since otherwise $y=0$. Hence $\neg(\beta = 0)$ since it is easy to see that $\lnot (\alpha = 0)$.
In the second case $\neg(\alpha = \beta)$, since $\alpha=\beta$ implies that $y_n \to x$ and therefore $x=y$  a contradiction.
\end{proof}

As mentioned above \WMP is equivalent to the statement that every map on a complete metric space is strongly extensional. In fact, there are many specific spaces $X$ such that if every map defined on $X$ is strongly extensional, then \WMP holds. 

\begin{Lem}
Let $X$ be an arbitrary metric space.
Every function $f:X \to \RR$ is strongly extensional if and only if for every metric space $Y$ we have that every function $f:X \to Y$ is strongly extensional.	
\end{Lem}
\begin{proof}
	One direction is trivial. To prove the converse let $f:X \to Y$ be an arbitrary function and let $a,b \in X$ be such that $f(a) \neq f(b)$. Then $g(x) = d(f(a),f(x))$ defines a function $g:X \to \RR$. By assumption $g$ is strongly extensional, and $g(a) = 0 \neq g(b) = d(f(a),f(b))$. Hence $a \neq b$.
\end{proof}

\begin{Lem}
Assume $f:\Closure{\menge{a,b}} \to \RR$ is a function such that $f(a) \neq f(b)$. 
\begin{enumerate}
	\item If $0<p<q<1$, then we can, for any  $x \in \RR$, decide whether \[ x\geqslant p d(a,b) \text{ or } x \leqslant q d(a,b)  \ .\]
	\item For any $x \in [0,d(a,b)]$ and for any $n \in \NN$ there exists $0 \leqslant i \leqslant 2^n -2$ such that \[ x \in \bracks*{\frac{i}{2^n} d(a,b),\frac{i+2}{2^n} d(a,b)} \ .\]
	\item There exists a function $g:[0,1] \to \RR$ such that $g(0)=f(a)$ and $g(d(a,b)) = f(b)$.
\end{enumerate}	
\end{Lem}
\begin{proof}
\begin{enumerate}
	\item Use the axiom of countable choice to get a (increasing) sequence $\lambda_n$ such that 
\begin{align*}
	\lambda_n = 0 & \implies d(a,b) < \frac{1}{2n}  \ , \\
	\lambda_n = 1 & \implies d(a,b) > \frac{1}{2n+1} \ .
\end{align*}
Define a sequence $(x_n)_{n \geqslant 1}$ in $\menge{a,b}$ by the following algorithm. As long as $\lambda_n=0$ set $x_n=a$. If $\lambda_n=1$ for the first time we can decide whether $x > p d(a,b)$ or $x < q d(a,b) $. In the first case we set $x_n=a$ from then on, in the second case we set $x_n = b$ from then on. It is easy to check that $(x_n)_{n \geqslant 1}$ is Cauchy. Considering its limit $x_\infty$ we can check whether $f(x_\infty) \neq f(a)$ or $f(x_\infty) \neq f(b)$. In the first case the assumption that $x < pd(a,b)$ leads to the contradiction that $d(x_\infty,a) =0$ and therefore $x_\infty= a$. So $x \geqslant pd(a,b)$. Similarly in the second case $x \leqslant qd(a,b)$. 
\item This can be decided by a finite iteration of the previous part.
\item Given $x \in [0,1]$ we consider $z_x = \max \menge{x,d(a,b)}$. Using countable choice and the preceding part of the lemma we fix a sequence $h:\NN \to \NN$ such that $z_x \in \bracks*{\frac{h(n)}{2^n} d(a,b),\frac{h(n)+2}{2^n} d(a,b)}$. It is easy to see that the limit 
		\[ r_x= \lim_{n \to \infty} \frac{h(n)}{2^n} \in [0,1] \] 
		is independent of the choice of $h$. In this way we can define a function $g(x) = r_x f(b) + (1-r_x)f(a)$. Furthermore, it is easy to check that $r_0 = 0$ and $r_{d(a,b)} =1$ and therefore $g(0) = f(a)$ and $g(d(a,b))=f(b)$. \qedhere
\end{enumerate}		
\end{proof}

As a corollary of the first part we also get the following insight.
\begin{Cor}
	Assume $f:\overline{\menge{a,b}} \to \RR$ is a function such that $f(a) \neq f(b)$, then $d(a,b)$ is pseudo-positive.
\end{Cor}
 This has already been proved by Ishihara \cite{hI91}*{Lemma 4}. In the same paper he has also shown that \WMP is equivalent to the statement that every  every real-valued function on a complete metric space is strongly extensional. In the following we are going to extend this result.

\begin{Pro}
The following are equivalent to \WMP.
\begin{enumerate}
	\item \label{SE-all} Every real-valued function on a complete metric space is strongly extensional.
	\item \label{SE-BS} Every real-valued function on $\BS$ is strongly extensional.
	\item \label{SE-CS} Every real-valued function on $\CS$ is strongly extensional.
	\item \label{SE-NNast} Every real-valued function on $\Ninf$ is strongly extensional; where $\Ninf$ denotes the space of all increasing binary sequences. 
	\item \label{SE-01} Every real-valued function on $[0,1]$ is strongly extensional.
	\item \label{SE-ab} If $X= \Closure{ \menge{a,b}}$, then any function $f:X \to \RR$ is strongly extensional.
\end{enumerate}	
\end{Pro}
\begin{proof}
Obviously \ref{SE-all} implies \ref{SE-BS}. 

Assume $f:\CS \to \RR$ is a function with $f(\alpha) \neq f(\beta)$  If $\gamma \in \BS$, then let $\gamma^{01} \in \CS$ be the sequence $\gamma(n) = \min \menge{\gamma(n),1}$. Then $\gamma^{01}=\gamma$ for any $\gamma \in \CS$. So define $\tilde{f}(\gamma) = f(\gamma^{01})$. Altogether $\tilde{f}(\alpha) \neq \tilde{f}(\beta)$ and if we assume \ref{SE-BS} we get $\alpha \neq \beta$ and thus in this case also \ref{SE-CS} holds.

Now assume \ref{SE-CS} and let $f:\Ninf \to \RR$ be such that $f(\alpha) \neq f(\beta)$ for some $\alpha$ and $\beta$. For any $\gamma \in \CS$ let $\gamma^\uparrow$ denote the binary sequence defined by $\gamma^\uparrow (n) = \max_{i \leqslant n} \gamma(i)$. It is obvious that $\gamma^\uparrow \in \Ninf$ and that for $\gamma \in \Ninf$ we have $\gamma=\gamma^\uparrow$. Thus by setting $g(\gamma) = f(\gamma^\uparrow)$ we get a function $g:\CS \to \RR$ with $g(\alpha) \neq g(\beta)$. Hence $\alpha \neq \beta$ and \ref{SE-NNast} holds.

To see that \ref{SE-NNast} implies \ref{SE-ab} we need to use the axiom of countable choice to get a (increasing) sequence $\lambda_n$ such that 
\begin{align*}
	\lambda_n = 0 & \implies d(a,b) < \frac{1}{2n} \ , \\
	\lambda_n = 1 & \implies d(a,b) > \frac{1}{2n+1} \ .
\end{align*}
To a given sequence $\mu \in \Ninf$ we associate the sequence $(x^\mu_n)_{n \geqslant 1}$ in $\menge{a,b}$ by
\begin{align*}
	\lambda_n = 0 \land \mu_n =0 & \implies x_n = a \ , \\
	\lambda_n = 1 \lor \mu_n = 1 & \implies x_n = b  \ .
\end{align*}
It is easy to see that $x^\mu_n$ is a Cauchy sequence and thus converges to a limit $x^\mu_\infty$. Now we can consider the function $g:\Ninf \to \RR$ defined by $g(\alpha) = f(x^\alpha_\infty)$. Furthermore, it is easy to see that $x^{(\zero)}_\infty = a$ and $x^{\lambda}_\infty = b$. Hence $\lambda \neq \zero$ and thus $a \neq b$.

Next, we will show that \ref{SE-ab} implies \ref{SE-all}. So let $f:Y \to \RR$ be a function and $a,b \in Y$ such that $f(a) \neq f(b)$. Considering the restriction $f\restriction_{\Closure{\menge{a,b}}}$ we can therefore conclude that $a \neq b$.

Of course, \ref{SE-all} implies \ref{SE-01}, so if we can prove that \ref{SE-01} implies \ref{SE-ab} we are done. Unfortunately this is the trickiest part of the implications, however most of the work has been done in the preceding lemma. Given $f:X \to \RR$ with $f(a) \neq f(b)$ we can find a $g$ as in that lemma. By our assumption $g(0) \neq g(d(a,b))$, which means that $0 \neq d(a,b)$ which in turn means that $a \neq b$.
\end{proof}

\section{\MPv} \index{LLPE|see {\MPv}}
The \define{disjunctive version of Markov's principle} \MPv can be seen as an instance of de'Morgan's laws (see Section \ref{Sec:Pi01andSigma01}) as well as another weakening of \MP. It states that every almost positive number is pseudo-positive.

\begin{principle}[MPor]{\MPv} \label{PR:MPv}
For all $x\in \RR$
\[ 0 \lessdot x  \implies   \left( \fa{y \in \RR}{0 \lessdot y} \lor y \lessdot x  \right) \ .  \]
\end{principle}

This principle was also called LLPE\footnote{Lesser Limited Principle of Existence} by Mandelkern.  Berger, Ishihara, and Schuster have suggested WLLPO \index{WLLPO|see {\MPv}} as a name \cite{jB12}, since it can be viewed as a weakening of \LLPO. Finally in \cite{Troelstra1988a} it also appears under the name of SEP. \index{SEP|see {\MPv}} 

Notice that \MPv bears resemblance to the well-known statement 
\[ 0 < x  \implies   \left( \fa{y \in \RR}{0 < y} \lor y < x  \right) \ ,  \]
which is---in contrast---provable in Bishop style constructive mathematics \cite{dB85}*{Corollary 2.17}, or included as one of the axioms of constructive real analysis \cite{dB99b}

Before we state some equivalences of \MPv we would like to point out the fact that
\[ \MP \iff \MPv + \WMP  \ , \]
which is obvious, considering the following diagram illustrating which versions of Markov's principle allow one to ``upgrade'' between the different notions of positiveness of a real number.

\begin{center}
\begin{tikzpicture}[node distance=4 cm, auto]
  \node (P)  {\text{positive}};
  \node (PP) [right of=P] {\text{pseudo-positive}};
  \node (AP) [right of=PP] {\text{almost positive}};
  \draw[<-] (PP) to node {\MPv} (AP);
  \draw[<-] (P) to node {\WMP} (PP);
  \draw[->,bend left=25] (AP) to node {\MP} (P);
\end{tikzpicture}
\end{center}

\begin{Pro} \label{Pro:equivs_of_MPv}
The following are equivalent to \MPv:
\begin{enumerate}
\item \label{MPo:3} For any binary sequence $(a_n)_{n \geqslant 1}$ with at most one $1$ and $\lnot \fa{n \in \NN}{a_n = 0}$ either 
\[ \fa{n \in \NN}{a_{2n} = 0} \lor \fa{n \in \NN}{a_{2n+1} = 0} \ . \]
\item \label{MPo:4} If $\alpha$ is a binary sequence such that $\lnot \fa{n \in \NN}{\alpha(n)=0}$, then for any $\beta$ we have 
\[  \fa{\beta \in \CS}{\neg(\alpha = \beta) \lor \lnot (\beta = 0)}  \]
\item \label{MPo:8} For $\alpha,\beta \in \CS$ 
		 \[ \lnot (\fa{n \in \NN}{\alpha_n =0} \land \fa{n \in \NN}{\beta_n =0}) \implies \lnot \fa{n \in \NN}{\alpha_n =0} \lor \lnot \fa{n \in \NN}{\beta_n =0}  \ .\] 
		 (That is the De Morgan law for $\Pi^0_1$ formulae).
\item \label{MPo:5} For all $x,y \in \RR$ with $\neg(x=y)$ we have $\lnot (x \leqslant y) \lor \lnot (y \leqslant x)$

\item \label{MPo:1} For all $x,y \in \RR$ with $\neg(x=y)$ we have $x \leqslant y \lor y \leqslant x$ 
\item[\theenumi ½] \label{MPo:1b} For all $x \in \RR$ with $\neg(x=0)$ we have $x \leqslant 0 \lor 0 \leqslant x$
\item \label{MPo:2} For all $x,y,z \in \RR$ we have $\lnot (x=y) \implies \neg(x=z) \lor \neg(y=z)$.
\item \label{MPo:6} For all $x,y \in \RR$ with $x \gtrdot y$, $\menge{x,y}$ is a closed subset of $\RR$
\item \label{MPo:7} For any $x\in \RR$ with $\abs{x} \gtrdot 0 $ either $\abs{x}=x$ or $ \abs{x} = -x$.
\end{enumerate}
\end{Pro}
\begin{proof} 
Most of these equivalences to \MPv (\ref{MPo:3},\ref{MPo:6},\ref{MPo:7}) are taken from \cite{mM88}. The equivalences between $\ref{MPo:1} ,\ref{MPo:2}, \ref{MPo:8}$ are due to Josef Berger\footnote{currently unpublished, talk at CiE2012 in Cambridge}. The equivalence of \ref{MPo:1b}½ and \ref{MPo:1} is due to the fact that $\lnot(x=y) \iff \lnot(x-y = 0)$.
 The proof of $\ref{MPo:1} \implies \ref{MPo:2}$, surprisingly, seems to necessitate countable choice.  The equivalence that $\ref{MPo:5} \iff \ref{MPo:1}$ is simply because $\neg(x \leqslant y) \iff x \geqslant y$ given that $\lnot (x = y)$.
Clearly, \ref{MPo:8} implies \ref{MPo:3}.
To see that $\ref{MPo:3} \implies \ref{MPo:4}$ let $\alpha$ be a binary sequence such that $\lnot \fa{n \in \NN}{\alpha_n=0}$, and $\beta \in \CS$ be arbitrary. Now define $\gamma \in \CS$ by
\begin{align*}
	\gamma_{2n}=1 & \implies \alpha_n \neq \beta_n  \ , \\
	\gamma_{2n+1}=1 & \implies \beta_n = 1 \ .	
\end{align*} 
We have that $\lnot (\fa{n \in \NN}{\gamma_n=0})$, since the assumption that $\fa{n \in \NN}{\gamma_n=0}$ implies that $\alpha = \beta$ and $\beta = \zero $, which contradicts the assumption that $\lnot (\alpha = \zero)$.
Now consider $\gamma^\prime$ defined by 
\[ \gamma^\prime(n) = \gamma(n) - \max_{i<n} \menge{ \gamma(i)} \ ,\]
so that $\gamma^\prime$ contains at most one $1$. Furthermore we also have $\neg(\gamma^\prime = \zero)$. So applying \ref{MPo:3} we get that either $\fa{n \in \NN}{\gamma^\prime_{2n}=0}$ or $\fa{n \in \NN}{\gamma^\prime_{2n+1}=0}$. In the first case the assumption that $\beta = \zero$ leads to the following contradiction. Then $\gamma_{2n+1} =0$ for all $n \in \NN$, which means that we must have $\lnot (\fa{n \in \NN}{\gamma_{2n}=0})$. But this contradicts $\fa{n \in \NN}{\gamma^\prime_{2n}=0}$. In the second case the assumption that $\alpha = \beta$ leads to the following contradiction. Then $\gamma_{2n} =0$ for all $n \in \NN$ which means that we must have $\lnot (\fa{n \in \NN}{\gamma_{2n+1}=0})$. But this contradicts $\fa{n \in \NN}{\gamma^\prime_{2n+1}=0}$. So we have decided 
\[ \lnot (\alpha = \beta) \lor  \lnot (\beta = \zero) \ .\]

To see that $\ref{MPo:4} \implies \ref{MPo:8}$ let $\alpha, \beta$ be a binary sequences such that \[ \lnot (\fa{n \in \NN}{\alpha_n =0} \land \fa{n \in \NN}{\beta_n =0}) \ . \]
Now consider the sequences 
\[ \gamma = \alpha_0 \beta_0 \alpha_1 \beta_1 \alpha_2 \beta_2 \dots \]
and 
\[ \mu = 0 \beta_0 0 \beta_1 0 \beta_2 \dots \ . \]
Then we can apply \ref{MPo:4} to decide whether $\lnot (\mu = \zero)$ which means that $\lnot (\beta = \zero)$ or $\lnot (\gamma = \mu)$ in which case $\lnot (\alpha = \zero) $.
 Thus we have shown \ref{MPo:8}.
\end{proof}

\begin{Rmk}
	From part \ref{MPo:1b}½ of the above proposition we immediately get that \LLPO $\implies$ \MPv
\end{Rmk}

As \LLPO is equivalent to the intermediate value theorem, it is straightforward to adapt that proof and get the following \MPv-variant.
\begin{Pro}
\label{Pro:IVTunique} \MPv is equivalent to the statement that for every  non-discontinuous function $f:[0,1] \to \RR$ which is such that $f(0)f(1) \leqslant 0$ there is at most one root in the sense that 
\[ x \neq y \implies \lnot (f(x) = f(y) = 0) \ , \]
 there exists $x \in [0,1]$ such that $f(x) =0$
\end{Pro}
\begin{proof}
This is almost the same as the equivalence of IVT and \LLPO. Instead of interval halving, we have to trisect the interval. For the converse direction we will show that \ref{MPo:1b} of Proposition \ref{Pro:equivs_of_MPv} holds. Take $a$ such that $\neg(a=0)$. We can  define a piecewise linear function $f$ connecting the points  $(0,-1)$, $(\nicefrac{1}{3},a)$, $(\nicefrac{2}{3},a)$, and $(1,1)$. It is easy to see that if $t \neq s$ then $\neg(f(t) = f(s) = 0)$, for assume that there is $t \neq s$ with $f(t) = f(s) = 0$. Then $a=0$, since if $a \neq 0$ there must be a unique root. But $\neg (a =0)$, so we have our desired contradiction. Now, as usual, if $f$ has a root $f(z) = 0$, then we can check whether $f(z) < \nicefrac{2}{3}$ or $f(z) > \nicefrac{1}{3}$. In the first case we cannot have $a < 0$ and in the second case we cannot have $a > 0$. Together $x \leqslant 0 \lor x \geqslant 0$.
\end{proof}

As we have seen above, \WKL is equivalent to \LLPO. Because of the similarity of \LLPO and \MPv it is not surprising that we can find a weakening of \WKL that is equivalent to the latter. 
\begin{principle}[WKL!!!]{\WKLppp} If $T$ is an infinite decidable tree that
\begin{enumerate}
\item has at most one path in the weak sense that if $\alpha$ and $\beta$ are paths through the tree, then $\alpha=\beta$, 
\item and for every infinite subtree it is \emph{impossible not to} admit an infinite path, \label{Enu:wkluniunicond}
\end{enumerate}
 then it actually has an infinite path.
\end{principle}
Notice that our notion of \WKLppp differs only slightly from the one of J.\ Moschovakis, who first considered it \cite{jM13}. In her version, named \WKLpp,  condition (\ref{Enu:wkluniunicond}) is missing. She also showed that 
\[ \WKLpp \iff \lnot \neg \WKL \land \MPv \ . \]
By, vaguely speaking, moving  Condition (\ref{Enu:wkluniunicond}) from $\lnot \neg \WKL$ into \WKLppp, we can strengthen this result: in this way we get a version of weak K\H{o}nig's Lemma that holds in recursive models!
\begin{Pro}
\MPv is equivalent to \WKLppp.
\end{Pro}
\begin{proof}
The proof is essentially the same as the one for Proposition \ref{Pro:WKL_equiv_LLPO}. However it is worth noticing that this one  only requires unique choice. Consider a tree $T$ satisfying the conditions of \WKLppp. As in the proof of Proposition \ref{Pro:WKL_equiv_LLPO} define
\begin{align*} 
a_{2n} = 0 \iff \ex{u \in 2^n}{0 \ast u \in T} \ , \\
a_{2n+1} = 0 \iff \ex{u \in 2^n}{1 \ast u \in T}  \ .
\end{align*}
Notice that, since $T$ is infinite there cannot $n,m \in \NN$ such that $a_{2n}=1$ and $a_{2m+1}=1$. However it is also not possible that $a_n =0$ for all $n \in \NN$, since that would imply (second condition) that it is impossible for $T_0$ and $T_1$ not to admit infinite paths. Thus it is impossible that there not two infinite and not-equal paths through $T$, which would contradict the first condition of \WKLppp.

Hence by \ref{Pro:equivs_of_MPv} \MPv implies that either all even or all odd terms of $a_n$ are zero. That means that either $T_0$ or $T_1$ is infinite. Also both subtrees satisfy the first and second condition of  \WKLppp and are decidable and infinite. Thus we can, using dependent choice, iteratively define a sequence $\alpha \in \CS$ such that $T_{\overline{\alpha}n}$ is a infinite, decidable tree satisfying both conditions of \WKLppp for all $n \in \NN$. In particular $\overline{\alpha}n \in T$.

Conversely consider a binary sequence $(a_n)_{n \geqslant 1}$ with at most one $1$ and such that $\lnot \fa{n \in \NN}{a_n = 0}$. Now consider the decidable tree $T$ defined by
\[ u \in T \iff \ex{n \in \NN}{u = 0^n \land \fa{i \leqslant n}{a_{2i} = 0 }} \lor \ex{n \in \NN}{u = 1^n \land \fa{i \leqslant n}{a_{2i+1} = 0 }} \ . \]
It is straightforward to see that $T$ satisfies the conditions of \WKLppp. Thus, by \WKLppp, it admits an infinite path $\alpha$. Now either $\alpha(1)=0$, in which case we have $\alpha = \zero$ and therefore $\fa{n \in \NN}{a_{2n} =0}$ or $\alpha(1)=1$, in which case we have $\alpha = \one$ and therefore $\fa{n \in \NN}{a_{2n+1} =0}$.  Hence \MPv holds.
\end{proof}
\begin{Rmk}
Condition (\ref{Enu:wkluniunicond}) cannot be weakened by removing ``for every infinite subtree''.
\end{Rmk}
\begin{proof}
Let us assume that there exists a Kleene tree $K$ (for example in \RUSS). So there exists $u_n \in K$ with $\abs*{u_n}=n $. Furthermore let $(a_n)_{n \geqslant 1}$ be an increasing binary sequence.

Define a set $T \subset \cS$ by $u \in T$ if and only iff
\begin{equation*}
 \left( a_n=0 \land \left( u =1^n \lor \ex{w \in T}{ u=0w}  \right)  \right) \lor   \left(a_n=1 \land \left(u=u_{m}1^{n-m}   \lor \ex{w \in T}{u = 1^{m}w  }  \right) \right)
\end{equation*} 
where $\abs{u}=n$ and $m$ is the smallest number $i \leqslant n$ such that $a_{i}=1$.
The idea is that as long as $a_n=0$, $T$ looks like a copy of $K$ below $0$ and just a single path $\one$ on the right hand side. If we ever hit an $m$ such that $a_{m}=1$ we continue only one path on the left hand side and glue a copy of $K$ on the path $1^{m}$ on the right hand side of the tree. Now, clearly, $T$ is an infinite tree. It is also impossible not to admit an infinite path (if $\fa{n \in \NN}{a_n=0}$ there is an infinite path, and if $\ex{n \in \NN}{a_n=1}$ there is an infinite path; furthermore, by the double negation translation, constructively  we have $\lnot \neg (\fa{n \in \NN}{a_n=0} \lor \ex{n \in \NN}{a_n=1}$). Finally, it is also easy to see that $T$ has at most one path. However if $T$ actually admits a path $\alpha$ we can check whether $\alpha(1)=0$, in which case $\lnot \fa{n \in \NN}{a_n=0}$, or $\alpha(1)=1$, in which case $\fa{n \in \NN}{a_n=0}$. Thus \WLPO holds, which contradicts the existence of a Kleene tree.
\begin{figure}[ht]
\centering
\tikzset{external/export next=true}
\begin{tikzpicture}[y=0.80pt,x=0.80pt,yscale=-0.5,xscale=0.6, inner sep=0pt, outer sep=0pt]
  \path[draw=black,line width=0.6pt] (252,82) -- (192,22) -- (132,82) -- (102,142);
  \path[draw=black,line width=0.6pt]  (132,82) -- (162,142);
  \path[draw=black,line width=0.6pt] (82,192) -- (102,142) -- (122,192);
  \path[draw=black,line width=0.6pt] (162,142) -- (182,192) -- (172,232);
  \path[draw=black,line width=0.6pt] (112,232) -- (122,192) -- (132,232);
  \path[draw=black,line width=0.6pt] (82,192) -- (92,232);
  \path[draw=black,line width=0.6pt] (252,82) -- (282,142) -- (302,192) -- (312,232);
  \path[draw=black,dash pattern=on 1.6pt off 0.8pt, opacity=0.220,line
    join=miter,line cap=butt,miter limit=4.00,line width=0.6pt]
    (42,242) .. controls (93,3) and (156,2) .. (222,242);
  \path[draw=black,line width=0.6pt] (543,81) -- (483,21) -- (423,81) -- (393,141);
  \path[draw=black,line width=0.6pt] (423,81) -- (453,141);
  \path[draw=black,line width=0.6pt] (373,191) -- (393,141) -- (413,191);
  \path[draw=black,line width=0.6pt] (453,141) -- (473,191);
  \path[draw=black,line width=0.6pt] (463,231) -- (473,191);
  \path[draw=black,line width=0.6pt] (403,231) -- (413,191) -- (423,231);
  \path[draw=black,line width=0.6pt] (373,191) -- (383,231);
  \path[draw=black,line width=0.6pt] (543,81) -- (573,141) -- (593,191) -- (603,231);
  \path[fill=black] (190,267) node[above right] (text3073) {$\vdots$};
  \path[fill=black] (490,267) node[above right] (text3115) {$\vdots$};
  \path[draw=black,line width=0.6pt] (625,323) -- (595,383);
  \path[draw=black,line width=0.6pt] (625,323) -- (655,383)-- (675,433) -- (665,473);
  \path[draw=black,line width=0.6pt] (595,383) -- (615,433);
  \path[draw=black,line width=0.6pt] (585,473) -- (575,433) -- (595,383);
  \path[draw=black,line width=0.6pt] (605,473) -- (615,433) -- (625,473);
  \path[draw=black,dash pattern=on 1.60pt off 0.80pt,opacity=0.220,line
    join=miter,line cap=butt,miter limit=4.00,line width=0.80pt]
    (536.4034,484.3617) .. controls (586.9381,244.8980) and (650.5344,244.1294) ..
    (716.3896,484.3617);
  \path[draw=black,line width=0.6pt] (430.9939,322) -- (462,382) -- (482,432) -- (492,472);
  \path[draw=black,line width=0.6pt] (431,322) -- (439,284);
  \path[draw=black,line width=0.6pt] (610,282) -- (625,325);
  \path[fill=black] (393,330) node[above right] (t2) {$u_m$};
  \path[fill=black] (54.066982,229.97491) node[above right] (t3) {$K$};
  \path[fill=black] (549,470) node[above right] (t4) {$K$};
\end{tikzpicture}
\caption{As long as $a_n=0$ the tree consists of a copy of $K$ on the left side and a single path on the right. If there is $m$ with $a_{m}=1$ from then on a copy of $K$ is attached to the single path on the right and another single path to one of the elements $u_{m}$ on the left side.}
\end{figure}
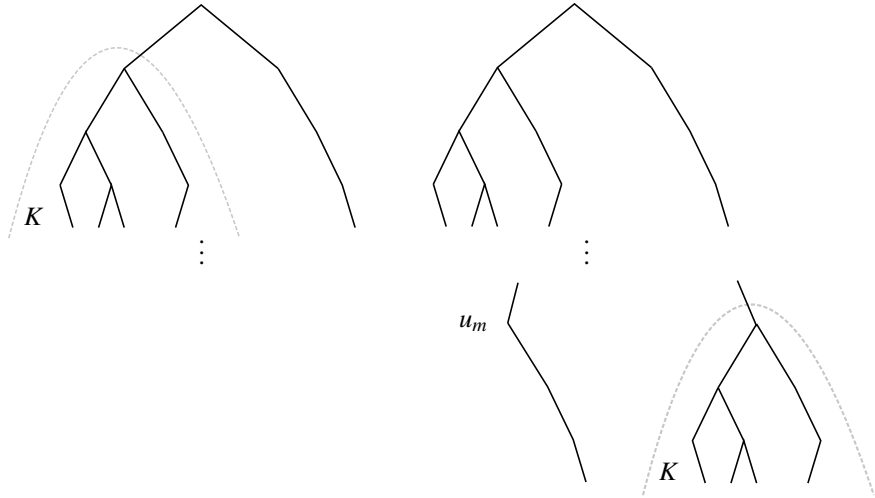
\end{proof}

\section{\IIIa}

The following principle, which is weaker than \MPv has been investigated in \cite{hI15} and \cite{uK15}. It can be seen as \LEM restricted to $\Delta_0^1$ formulas (also see Section \ref{Sec:Pi01andSigma01}). 

\begin{principle}[IIIa]{\IIIa} \label{PR:IIIa} 
\[ \fa{ \alpha, \beta}{ ( \alpha \neq 0 \iff \beta = 0) \implies \alpha \neq 0 \lor  \alpha = 0 } \ . \]
\end{principle}

There are a few equivalent formulations.
\begin{Pro}
The following are equivalent to \IIIa:
\begin{enumerate}
\item \( \fa{ \alpha, \beta}{ (\alpha \neq 0 \iff \beta = 0) \implies \lnot \alpha = 0 \lor  \alpha = 0 }\).
\item \( \fa{ \alpha, \beta}{ (\alpha = 0 \iff \beta \neq 0) \implies \lnot \alpha = 0 \lor  \alpha = 0 } \)
\end{enumerate}
\end{Pro}
\begin{proof}
See \cite{hI15}.	
\end{proof}

In \cite{hI15} it is shown that \MPv implies \IIIa, that this cannot be derived in certain formal systems, that itself \IIIa cannot be derived in certain formal systems, and that it is independent from \WMP. In Section \ref{Sec:topmodels} we will show that it fails in the topological model over $\BS$ (Corollary \ref{Cor:BS_nval_IIIa}), thus showing that it is not provable in $\mathrm{IZF}+\mathrm{DC}$, and that it is not implied by \WMP. 

It is clear, that, using countable choice, we have the following equivalences.
\begin{Pro}
The following are equivalent to \IIIa.
\begin{enumerate}
\item \( \fa{ x,y \in \RR}{ (x > 0 \iff y = 0) \implies x > 0 \lor  y=0}\).
\item \( \fa{ x,y \in \RR}{ (x > 0 \iff y = 0) \implies \lnot (x = 0) \lor  y=0}\).
\end{enumerate}
\end{Pro}
There are no other, more interesting, equivalences known, at this stage.

\chapter{The Fan Theorems} \label{Ch:fan}
As a way to re-capture the unit interval's compactness---that is cover compactness---which was lost when rejecting the law of excluded middle, L.E.J.\ Brouwer made generous use of the fan theorem. Since he also made free use of the principle of continuous choice the complexity of the sets involved did not make a difference to his mathematics---as we will see in Section \ref{Sec:Fan_collapse}. In the absence of continuous choice we do, however, have to make some careful distinctions. This explains why we talk about the fan theorems in the plural, since we are going to distinguish between different versions, all of which are interesting in constructive (reverse) mathematics. What all fan theorems have in common is that they enable one to conclude that a given bar is uniform. 
Here a \define{bar} is a subset $B \subset \cS$ that ``bars'' every infinite path in $\CS$, that is 
\[  \fa{\alpha \in \CS}{\ex{n \in \NN}{\overline{\alpha}n \in B}} . \]
A bar is called \define{uniform} if this ``barring'' occurs before a uniform height $N$, that is if
\[ \ex{N \in N}{\fa{\alpha \in \CS}{\ex{n \leqslant N}{\overline{\alpha}n \in B}}} \] 
If a bar $B$ is \define{closed under extensions}, that is if 
\[ u \in B \implies \fa{w \in \cS}{u \ast w \in B} \ , \]
then it is uniform if and only if 
\[ \ex{N \in N}{\fa{\alpha \in \CS}{\overline{\alpha}N \in B}} \ . \] 

The difference between the fan theorems lies in the required complexity of the bar $B$. This ranges from the very strongest requirement---decidable---to no restriction on the bar at all. Here, a set $S$ is decidable, if, for every $x$, we can decide whether
\[ x \in S \lor x \notin S \ .\]
Of course, in the absence of the law of excluded middle, this might not always be possible. A weaker requirement than decidability is the notion of a $c$-bar.
A bar $C \subset 2^ \ast $ is called a \define{$c$-{bar}}, if there exists a decidable set $C' \subset 2^ \ast $ such that 
	\[ u \in C \iff \fa{w \in \cS}{\left(u \ast w \in C'\right)} \ . \] 
A bar $B \subset 2^ \ast $ is called \define{$\Pi_1^0$-bar}, if there exist a countable family of decidable sets $(B_n)_{n \geqslant 1}$ that are all closed under extension such that \[B = \bigcap_{n \geqslant 1} B_n \ .  \] 
The $\Pi^0_n$-nomenclature allures to the arithmetical hierarchy in recursion theory. 
We can now formally state the four versions of the fan theorem, that are going to be of interest to us.
\begin{principle}{Fan theorems} \label{PR:FAN} $ $ \\[0.5em]
\begin{tabular}{ll}
	\FAND: & Every decidable bar is uniform. \\[0.1em]
	\FANc: & Every $c$-bar is uniform. \\[0.1em]
	\FANP: & Every $\Pi^0_1$-bar is uniform. \\[0.1em]
	\FANf: & Every bar is uniform.
\end{tabular}
\end{principle}\index{FAND@\FAND}\index{FANc@\FANc}\index{FANP@\FANP}\index{FANfull@\FANf}
It is mostly trivial that 
\[ \FANf \implies \FANP  \implies \FANc \implies \FAND  \ , \]
with the exception of the middle implication which does require a modicum of trickery. It follows indirectly from \cite{hD08b}*{Proposition 4.1.6}, but we can also give a direct construction.
\begin{Pro}
	For every $c$-bar $C$ there exists is a $\Pi^0_1$-bar $B$ such that $C$ is a uniform bar only if $B$ is.
\end{Pro}
\begin{proof}
Let $C \subset 2^ \ast $ be a $c$-bar. So there exists a decidable set $C' \subset 2^ \ast $ such that 
	\[ u \in C \iff \fa{w \in \cS}{\left(u \ast w \in C^\prime \right)} \ . \]
Now, for every $n \in \NN$ define decidable sets $B_n \subset \cS$ by
\[ B_n = \set{u \in \CS}{\fa{w \in \cS}{\abs{w} \leqslant n - \abs{u} \implies w \in C^\prime}} \ . \] 
The antecedent ``$\abs{w} \leqslant n - \abs{u} $'' ensures that $B_n$ is closed under extensions (this is the modicum of trickery that we referred to earlier). The set 
\[  B = \bigcap_{n \geqslant 1} B_n \]
is a bar. For let $\alpha \in \CS$ be arbitrary. Since $C$ is a bar there exists $m$ such that $\overline{\alpha}m \in C$, which, by definition means that $\overline{\alpha}m \ast w \in C^\prime$ for all $w \in \cS$. So this is the case, in particular, for any $n$ and any $w$ such that $\abs{w} \leqslant n - m$. Hence $\overline{\alpha}m \in B_n$ for all $n$, which means that $\overline{\alpha}m \in B$. Thus $B$ is a $\Pi^0_1$-bar. 

Now if $B$ is uniform there exists $M$ such that \[ \fa{\alpha \in \CS}{\overline{\alpha}M \in B} \ ,\]
since $\Pi^0_1$-bars are closed under extensions. But now let $\alpha \in CS$ and $w \in \cS$ be arbitrary. Since $\overline{\alpha}M \in B$, in particular, $\overline{\alpha}M \in B_{M+\abs{w}}$. Since $\abs{w} \leqslant (M + \abs{w}) - M$, we have $\overline{\alpha}M \ast w \in C^\prime$, and since $w$ was arbitrary that means that $\overline{\alpha}M \in C$. Thus $C$ is a uniform bar.
\end{proof}

\FAND, and therefore all of the fan theorems, fail in \RUSS, since there one can construct a Kleene tree, that is an infinite decidable binary tree that blocks every infinite path \cite{aB06}. The complement of such a tree is by itself a ready-made counterexample to \FAND.
One implication in analysis that the existence of a Kleene tree has (see Section \ref{Sec:KT}) is the existence of a Specker sequence, that is a sequence of real numbers in a compact set $X$ that is bounded away from every point of $X$.
To distinguish whether such a strange object exists in a space or not we say that a subspace $X$ of a metric space $Y$ satisfies the \define{anti-Specker property relative to $Y$} if 
\begin{principle}[ASXY]{\AS{X}{Y}} \label{PR:AS}
 any sequence in $Y$ that is eventually bounded away from any point in $X$ is eventually bounded away from the entire set $X$.
\end{principle}
Here a sequence $(x_n)_{n \geqslant 1}$ in a metric space $(Y,d)$ is \define{eventually bounded away} from a point $x$ (a set $X$), if there exists a natural number $N$ such that $d(x_n,x) > 2^{-N}$ for all $n \geqslant N$ (and for all $x \in X$). It is easily seen that \AS{X}{Y} holds, whenever $X$ is a subspace of $Y$ and satisfies the Heine-Borel property.\footnote{Depending on the variety of constructive mathematics, the class of spaces satisfying the Heine-Borel property may, of course, be drastically smaller than in classical mathematics. Nevertheless, maybe surprising, even in \BISH there are infinite spaces that satisfy the Heine-Borel property: a prototypical example of such a space is the closure of  $\set{2^{-n}}{n \in \NN}$} 
Notice that for any two one-point extensions $Y$ and $Y^\prime$ of a space $X$ the principles $\AS{X}{Y}$ and $\AS{X}{Y^\prime}$ are equivalent. So in that case we can combine them into a generic principle labeled $\AS{X}{1}$. Notice that some early papers on the topic have used a more refined notation: what is labeled \AS{}{} there is what we label \AS{[0,1]}{\RR}, and which is equivalent to \AS{[0,1]}{1}.
\section{\texorpdfstring{Linking $\CS$ and $[0,1]$}{Linking Cantor Space and [0,1]}} \label{Sec:linking-cantor-and-unitint}
In the following we want to establish strong links between $\CS$ and $[0,1]$. 
We will first adapt Cantor's middle third set construction for our purposes.
First consider a fixed $p \in (0,1)$. 
Let 
\[ I^{p}_{u} = [a^{p}_{u},b^{p}_{u}] = \bracks*{ (1-p) \sum_{n \leqslant \abs{u}} p^{n-1}u(n),  (1-p) \sum_{n \leqslant \abs{u}} p^{n-1}u(n) + p^n} \ . \]
Furthermore, let $I^{p}_{()}$ be the unit interval. It is easy to see that $I^{p}_{u0}$ is the left $p$ of $I^{p}_{u}$, and $I^{p}_{u1}$ is the right $p$ of $I^{p}_{u}$. With this notation the Cantor middle third set is
\[ \mathcal{C}^{p} = \bigcap_{n \geqslant 0} \bigcup_{u \in 2^n} I^{p}_{u} \ , \]
where $p=\frac{1}{3}$. It is easy to see that $F^{p}:\CS \to [0,1]$ defined by
\[ F^{p}(\alpha) = (1-p) \sum_{n \geqslant 1} p^{n-1}\alpha(n)\] is well-defined. It has many nice properties.
\begin{Lem} \label{Lem:Fp} $ $
\begin{enumerate}
\item   For $\alpha,\beta \in \CS$ and $n \geqslant 1$ we have \[ \overline{\alpha}n = \overline{\beta}n \implies \abs*{F^{p}(\alpha) - F^{p}(\beta)} \leqslant  p^n \ . \] In particular, $F^{p}$ is uniformly continuous. 
\item \label{Lem:Var-Bish} (Variation on Bishop's Lemma) \\ 
Assume $0< p <\frac{1}{2}$. For all $x \in [0,1]$ there exists $\alpha \in \CS$ such that  
 \begin{equation} \label{Eqn:variationbish}
   \fa{ \beta <  \alpha}{ F^{p}(\beta) < x} \quad \text{ and } \quad  \fa{  \alpha < \beta}{ x< F^{p}(\beta) } \ .
   \end{equation}
In particular, $F^{p}$ is injective and
\begin{equation} \label{Eqn:OI-BL}
d(x,F^{p}(\alpha)) > 0 \implies d(x,F^{p}(\CS)) >0 \ . 
\end{equation}

 \item  \label{Fp_surj} Assume $\frac{1}{2}<p<1$. Then for every $u \in \cS$ and $y \in I_{u}$ there exists $\alpha \in \CS$ such that $\overline{\alpha}\abs{u} = u$ and $F^{p}(\alpha) = y$; in particular $F^{p}$ is surjective.
\end{enumerate}
\end{Lem}

\begin{proof}
\begin{enumerate}
\item Assume $\alpha,\beta \in \CS$ and $n \in \NN$ such that $ \overline{\alpha}n = \overline{\beta}n$. Then
\begin{align*}
 \abs*{F^{p}(\alpha) - F^{p}(\beta)}  & \leqslant  (1-p) \sum_{i \geqslant 1} p^{i-1}\abs*{\alpha(i) - \beta(i)} \\ & \leqslant  (1-p) \sum_{i > n} p^{i} \\ & =  (1-p) p^n \sum_{i \geqslant 0} p^{i} \\
 & =  \frac{1-p}{1-p} p^n = p^n \ .  
 \end{align*}

\item Consider $0< p <\frac{1}{2}$. We will define subsets $J^{p}_{u}$ and positive numbers $\varepsilon^{p}_n$ such that 
\begin{equation}\label{Eqn:partitionintervala}  J^{p}_{u} = J^{p}_{u0} \cup J^{p}_{u1} \ ,
\end{equation}
and 
\[ I^{p}_{u} \subset J^{p}_{u} \quad \text{ as well as } \quad I^{p}_{u} + \varepsilon^{p}_n< J^{p}_{v} < I^{p}_{w}- \varepsilon^{p}_n \ , \]
for $u<v<w$ of   length $n$; according to the lexicographic order.
First we will introduce (partial) functions $\mathrm{next}:\cS \to \cS$ and $\mathrm{prev}:\cS \to \cS$, mapping $u \in \cS$ to the element of the same length succeeding and preceding it, again in the lexicographical order.
Next, set 
\begin{itemize}
\item $c^{p}_{0^{\abs{u}}} = 0$,  $d^{p}_{1^{\abs{u}}} = 1$,
\item $c^{p}_{u} = a_{u} - \frac{2}{3} (a^{p}_{u} - b^{p}_{\mathrm{prev}(u)})  $, and
\item $d^{p}_{u} = b^{p}_{u} + \frac{1}{3} (a^{p}_{\mathrm{next}(u)} - b^{p}_{u} )$ for $0^{\abs{u}} \neq u \neq 1^{\abs{u}}$.
\end{itemize}
It is easy to see that the intervals $J^{p}_{u} = [c^{p}_{u},d^{p}_{u}]$ have the desired properties. \\[1em]
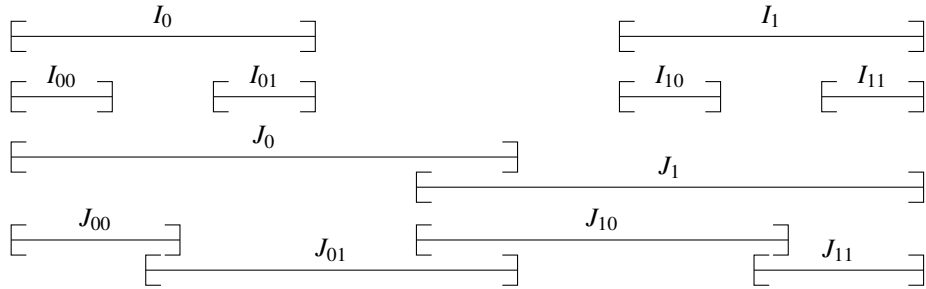
\begin{figure}[h]
\centering
\begin{tikzpicture}
\foreach \xa / \xb / \y / \lab in 
{0 / 4 / 7.6 /  $I_{0}$, 8 / 12 / 7.6 / $I_{1}$, 0 / 1.33 / 6.8 / $I_{00}$, 2.66 / 4 / 6.8 / $I_{01}$, 8 / 9.33 / 6.8 / $I_{10}$, 10.66 / 12 / 6.8 / $I_{11}$, 0 / 6.66 / 6 /  $J_{0}$, 5.33 / 12 / 5.6 / $J_{1}$,  0 / 2.22 / 4.9 /  $J_{00}$, 1.77 / 6.66 / 4.5 / $J_{01}$, 5.33 / 10.22 / 4.9 /  $J_{10}$, 9.77 / 12 / 4.5 / $J_{11}$}
{
\path ({\xa+0.5*(\xb-\xa)},\y) node [above] {\lab};
\draw (\xa,\y) -- (\xb,\y);
\draw ({\xa+0.2},{\y+0.2}) -- (\xa,{\y+0.2}) -- (\xa,{\y-0.2}) -- ({\xa+0.2},{\y-0.2});
\draw ({\xb-0.2},{\y+0.2}) -- (\xb,{\y+0.2}) -- (\xb,{\y-0.2}) -- ({\xb-0.2},{\y-0.2});
}
\end{tikzpicture}
\caption{Schematic of the first couple of intervals for $p=\frac{1}{3}$.}
\end{figure}

Given $x \in [0,1]$, we can, using dependent choice and Property \ref{Eqn:partitionintervala}, find $\alpha \in \CS$ such that 
\[ \fa{n \in \NN}{x \in J^{p}_{\overline{\alpha} n}} \ .\]
Now let $\beta \in \CS$ be such that $\beta < \alpha$, which means there exists $n \in \NN$ such that $\overline{\beta}n < \overline{\alpha}n$. Since $F(\beta) \in I^{p}_{\overline{\beta}n}$ and $x \in I^{p}_{\overline{\alpha}n}$, by \ref{Eqn:variationbish} we have $F(\beta)< x$. The case $\alpha<\beta$ can be treated analogously. 
To see that $F^{p}$ is injective let $\gamma \neq \beta$. Now let $\alpha$ be as constructed above for $x=F^{p}(\beta)$. Then the assumption that $\alpha \neq \beta$ leads to a contradiction: if $\alpha < \beta$ we get $ F^{p}(\beta) < x = F^{p}(\beta)$ and if $\beta < \alpha$ we get $ F^{p}(\beta)= x < F^{p}(\beta)$. Therefore $F^{p}(\alpha) = F^{p}(\beta)$, and therefore, again by \ref{Eqn:variationbish}, also $F^{p}(\gamma) \neq F^{p}(\beta)$.
Finally assume that $d(x,F^{p}(\alpha)) > 0$. Since $F^{p}$ is uniformly continuous by the first part of the proposition, there exists $n$ such that 
\[\overline{\beta}n=\overline{\alpha}n \implies d(F^{p}(\beta),F^{p}(\alpha)) < \frac{d(x,F^{p}(\alpha))}{2} \ . \]  
Now for any $\gamma \in \CS$ either $\overline{\gamma}n \neq \overline{\alpha}n$, or $\overline{\gamma}n = \overline{\alpha}n$. In the first case either  $\overline{\gamma}n < \overline{\alpha}n$ or $\overline{\gamma}n > \overline{\alpha}n$, but in both cases $d(x,F^{p}(\beta)) > \varepsilon^{p}_n$. In the second case 
\[ d(x,F^{p}(\beta)) > \abs*{d(F^{p}(\beta),F^{p}(\alpha)) - d(x,F^{p}(\alpha))} > \frac{d(x,F^{p}(\alpha))}{2} \ .  \]
In any case \[ d(x,F^{p}(\beta))>  \min \menge{\frac{d(x,F^{p}(\alpha))}{2}, \varepsilon^{p}_n} > 0 \ . \]
\item Let $\frac{1}{2}<p<1$. Notice that 
\begin{equation}\label{Eqn:partitionintervalb}  I^{p}_{u} = I^{p}_{u0} \cup I^{p}_{u1} \ .
\end{equation}
Similar to above, given $x \in [0,1]$, we can, using dependent choice and Property \ref{Eqn:partitionintervalb}, find $\alpha \in \CS$ such that 
\[ \fa{n \in \NN}{x \in I^{p}_{\overline{\alpha} n}} \ .\]  It is easy to see that therefore $F^{p}(\alpha) = x$, which means we have shown surjectivity.  \qedhere
\end{enumerate}
\end{proof}

\begin{Rmk}
Notice that for $p=\frac{1}{2}$ the function $F^{p}$  cannot be shown to be  surjective, constructively, since that would be a restatement of the fact that every real number $x\in [0,1]$ has a binary expansion and therefore equivalent to \LLPO (see Proposition \ref{Pro:equivs_of_LLPO}). 
\end{Rmk}

We are now in the position to prove a technical but utterly central lemma. 

\begin{Lem} \label{Lem:extendingfunctions} 
Assume $0<p<\frac{1}{2}$. If $f:\CS \to \RR$ is a point-wise continuous function, then there exists a point-wise  continuous function $\widetilde{f}:[0,1] \to \RR$ such that
\[  f = \widetilde{f} \circ F^{p} \]
Moreover, $\widetilde{f}$ is  uniformly continuous, (bounded, has a positive infimum, attains its infimum) if $f$ is (does).
\end{Lem}
\begin{proof}

The idea of the embedding were simple if we assumed classical logic: for elements $F^{p}(\alpha)$ in the Cantor set $\mathcal{C}^{p}\subset [0,1]$ we obviously want to set $\widetilde{f}(F^{p}(\alpha))= f(\alpha)$. If an element is not in $\mathcal{C}^{p}$ there exist unique $\alpha, \beta \in \CS$ such that $F^{p}(\alpha) < x < F^{p}(\beta)$ and $\mathcal{C}^{p} \cap \left( F^{p}(\alpha), F^{p}(\beta) \right) = \emptyset$. For such an $x$ we set 
\begin{equation} \label{Eqn:extfunc} \widetilde{f}(x) = t_{\alpha,\beta}(x) =  f(\alpha) + \frac{x-F^{p}(\alpha)}{ F^{p}(\beta) - F^{p}(\alpha)}f(\beta) \ ; 	
\end{equation}

that is we simply interpolate linearly between $F^{p}(\alpha)$ and $F^{p}(\beta)$. However, constructively, we can, of course, not show that $x \in [0,1]$ is either in the Cantor set or is in its metric complement. Nevertheless one can, surprisingly, give a construction of $\widetilde{f}$ without the use of \LEM. The construction can be found in \cite{dB07}. 

Since $\mathcal{C}^p \cup -\mathcal{C}^p$ is dense in $[0,1]$ it easily follows from Equation \eqref{Eqn:extfunc} that all the properties listed above are preserved.
\end{proof}

It is well known (\cite{Troelstra1988}*{Chapter 7 Corollary 4.4}) that constructively every compact, that is a complete and totally bounded, space is the uniformly continuous image of $\CS$. Below, we need a slightly stronger version of this result.

\begin{Def}
A map $g:X \to Y$	between to metric spaces $X$ and $Y$ is called \define{uniformly surjective} if 
\begin{equation} \label{Equ:unif_surj}
\fa{\varepsilon >0}{\ex{ \delta >0}{ \fa{x,y \in Y}{d(x,y)<\delta \implies \ex{\alpha,\beta \in Y }{ d(\alpha,\beta)<\varepsilon \land x=g(\alpha) \land y = g(\beta)}}   }} 
\end{equation}
\end{Def}

\begin{Lem} \label{Lem:cmpact-image-of-2N}
For $p > \frac{1}{2}$ the function $F^p:\CS \to [0,1]$ is uniformly surjective. 
\end{Lem}
\begin{proof}
Let $n \in \NN$ be arbitrary. Since $p > \nicefrac{1}{2}$, neighbouring intervals $I_u$ with $u \in 2^n$ are overlapping by $(2p-1)p^n$. Now  consider the slightly smaller intervals 
   \[ I^\prime_u = [a_u + \delta , b_u -\delta] \]
   where $\delta = \nicefrac{(2p-1)p^n}{3}$, as a special case we also set $I^\prime_{0^n} = [0,b_{0^n}]-\delta$ and $I^\prime_{1^n} = [a_{1^n}+\delta,1]$. By our choice of $\delta$ we still have 
   \[ [0,1] = \bigcup_{u \in 2^n} I^\prime_u \ . \]
   Consider $x,y \in Y$  such that $d(x,y)<\delta$. First, find $u \in 2^n$ such that $x \in I^\prime_u$. Since $d(x,y)<\delta$ we have $y \in I^\prime_u$, and also by  construction, $x \in I^\prime_u \subset I_u$. So using Lemma \ref{Lem:Fp}.\ref{Fp_surj} we can find $\alpha, \beta$ with $\overline{\alpha}n = \overline{\beta}n = u$, and $F(\alpha) =x$ and $f(\beta) = y$.
\end{proof}
More generally we have:
\begin{Pro} \label{Pro:compXunifsurj}
If $X$ is a compact set, then it is the uniformly surjective image of a uniformly continuous function $f: \CS \to X$.
\end{Pro}
\begin{proof}
This follows from \cite{Troelstra1988}*{Proposition 7.4.3.vi}.	
\end{proof}

It is worth taking the time for a little detour via constructive analysis here. We remind the reader that even classically the continuous image of complete space is not necessarily complete. However the continuous image of a cover-compact space is again cover-compact. Now constructively we traditionally use totally boundedness + completeness as a useful notion of compactness in metric spaces. Unfortunately neither totally boundedness, nor completeness are preserved under point-wise continuous maps. If we switch to uniformly continuous maps at least totally boundedness is preserved, but  completeness  is still not (the graph of $\frac{1}{x}$ is complete, but its image under the projection onto the $x$-axis is not). As the standard recursive counterexample (cf.\ Proposition \ref{Pro:KT_equiv}) shows we cannot even rule out the existence of a uniformly continuous function $f:[0,1] \to \RR$ such that $f([0,1]) = (0,1]$. But the situation is actually much worse since even a function as well behaved and canonical as $F^p$ with $p > \nicefrac{1}{2}$ doesn't have the property of mapping complete sets to complete sets.
\begin{Lem} 
If, for $p > \nicefrac{1}{2}$, the function $F^p:\CS \to [0,1]$ maps complete sets to complete sets (i.e.\ is a closed map), then \LLPO holds.
\end{Lem}
\begin{proof}
Let $x \in \RR$ be arbitrary. We may assume, without loss of generality that $\abs*{x}< \frac{1}{2}$. Let $x^+ = \max \menge{x,0}$ and $x^- = \max \menge{-x,0}$. Consider 
\[ A = (F^{p})^{-1} \left(\bracks*{0, \frac{1}{2} - x^+} \right) \cap \menge{0} \ast \CS  \]
and
\[ B = (F^{p})^{-1} \left(\bracks*{\frac{1}{2} + x^-,1 }\right) \cap \menge{1} \ast \CS  \ . \]
Both sets are closed and by Lemma \ref{Lem:Fp}.\ref{Fp_surj} we have $F^p(A) = [0, \frac{1}{2} - x^+]$ and $F^p(B) = [\frac{1}{2} + x^-,1]$. It is easy to see that $\frac{1}{2} \in \Closure{F^p(A \cup B)}$. However, assume that there exists $\alpha \in A \cup B$ with $F^p(\alpha) = \frac{1}{2}$. If $\alpha(0) = 0$, we cannot have $x > 0 $, since then $F(\beta) < \frac{1}{2}$. Similarly, if $\alpha(0)=1$ we cannot have $x< 0$. Thus we can decide $x \leqslant 0$ or $x \geqslant 0$, which means \LLPO holds.
\end{proof}
\begin{Qu}
Is there a general principle that ensures that complete sets are mapped to complete sets? The above example suggests that connectedness may play a role.	
\end{Qu}

\begin{Def}
We call a sequence $(x_n)_{n \geqslant 1}$ \define{tail-located} if the distances $d(x,\menge{x_n, x_{n+1}, \dots})$  exist for every $n \geqslant 1$ and $x \in \RR$.
\end{Def}

\begin{Lem} \label{Lem:seq-to-bar}
If $(x_n)_{n \geqslant 1}$ is a sequence of real numbers that is eventually bounded away from every point in $[0,1]$, then there exists a $c$-bar $B \subset \cS$ such that
\begin{enumerate}
\item $B$ is uniform, if $(x_n)_{n \geqslant 1}$ is eventually bounded away from $[0,1]$,
\item $B$ misses arbitrarily long finite sequences, that is 
\[ \fa{n \in \NN}{\ex{v \in \cS}{v \notin B \land \abs{v} \geqslant n}} \ , \] if $x_n \in [0,1]$ infinitely often, and
\item  if $(x_n)_{n \geqslant 1}$ is tail-located, then $B$ is decidable.
\end{enumerate}
\end{Lem}
\begin{proof}
Assume that $(x_n)_{n \geqslant 1}$ is a sequence that is eventually bounded away from every point in $[0,1]$.
Notice first, that since $(x_n)_{n \geqslant 1}$ is eventually bounded away from $0$ and $1$ we can, for a tail but without loss of generality for the whole sequence, decide whether $x_n \in [0,1]$ or not.
 Since $F^{\nicefrac{1}{2}}(\CS)$ is dense in $[0,1]$ we can, using countable choice, find a sequence $w_n \in \cS \cup \menge{\omega}$ such that 
\[w_n \in \cS \implies x_n \in [0,1] \land \abs*{F^{\nicefrac{1}{2}}(w_n) - x_n} < 2^{-n} \ , \] and
\[ w_n = \omega \implies x_n \notin [0,1] \ .\]
Define 
\[ v_n = \begin{cases}
\overline{w_n}n & \text{ if } w_n \in \cS \land \abs*{w_n} \geqslant n \ , \\
w_n 0^{n-\abs*{w_n}} & \text{ if } w_n \in \cS \land  \abs*{w_n} < n , \\
\omega & \text{ if }w_n = \omega \ .
\end{cases}
 \]
that is if $w_n$ is not $\omega$ either chop of $w_n$ after $n$ places or fill it up to length $n$ with zeroes. Then, by part 1 of Lemma \ref{Lem:Fp} \[w_n \in \cS \implies \abs*{F^{\nicefrac{1}{2}}(v_n) - F^{\nicefrac{1}{2}}(w_n)} \leqslant 2^{-n}  \] for all $n \in \NN$. Hence 
\begin{equation} \label{Eqn:tailloc}	
w_n \in \cS \implies \abs*{F^{\nicefrac{1}{2}}(v_n) - x_n} < 2^{-n+1} \ 
\end{equation}
for all $n \in \NN$. Since $\abs*{v_n} = n$ for all $n$, the set  $D=\menge{v_{1}, v_{2}, \dots} \cap \cS$ is decidable. Therefore the set \[ B = \set{u \in \cS}{\neg\ex{w \in \cS}{u \ast w \in D} } \] is a $c$-set. It is also a bar. To see this let $\alpha \in \CS$ be arbitrary. There exists $M,N \in \NN$ such that $\abs*{F^{\nicefrac{1}{2}}(\alpha) - x_n} > 2^{-M}$ for all $n \geqslant N$. Let $K = \max\menge{N,M}+2$. We are going to show that for all $w \in \cS$ we have $\overline{\alpha}K\ast w \notin D$. For assume there was such a $w$ with $\overline{\alpha}K\ast w \in D$ In this case $\overline{\alpha}K \ast w = v_{j}$ for some $j \geqslant K$ (for $j = K+\abs{w}$ to be precise). But then
 \begin{align*}
 2^{-M} & < \abs*{F^{\nicefrac{1}{2}}(\alpha) - x_{j}}   \\
 & \leqslant  \abs*{F^{\nicefrac{1}{2}}(v_{j}) - x_{j}} +  \abs*{F^{\nicefrac{1}{2}}(v_{j}) - F^{\nicefrac{1}{2}}(\alpha)} \\
 & \leqslant  2^{-j+1} + 2^{-K} \\
 & \leqslant 2^{-K+1} +2^{-K} \\
 & \leqslant 2^{-K+2} \leqslant 2^{-M} \ ;
 \end{align*}
a contradiction. Thus $\overline{\alpha}K\ast w \notin D$ for all $w \in \cS$ and therefore $\overline{\alpha}K \in B$. 

\begin{enumerate}
\item Assume now that $B$ is uniform. That is there $N \in \NN$ such that $\overline{\alpha}N \in B$ for all $\alpha \in \CS$. Now there cannot be $n \geqslant N$ such that $x_n \in [0,1]$, since in that case $w_n \in \cS$ and therefore $v_n \in D$. Since $\abs*{v_n} = n \geqslant N$ this would be a contradiction.
\item By construction, if $(x_n)_{n \geqslant 1}$ is in $[0,1]$ infinitely often, then $v_n \in \cS$ infinitely often. This implies that $v_n \notin B$ infinitely often. Since, furthermore, $\abs*{v_n} =n$ we are done.
\item Assume that $(x_n)_{n \geqslant 1}$ is tail-located, and let $u \in \cS$ be arbitrary. Let $a,b$ be the endpoints of $F^{\nicefrac{1}{2}}(B_u)$ that is $a = F^{\nicefrac{1}{2}}(u)$ and $b =F^{\nicefrac{1}{2}}(u\ast 1 \ast 1 \ast \dots)$.\footnote{That is with the notation of Section \ref{Sec:linking-cantor-and-unitint} $a=a^{\nicefrac{1}{2}}_u$ and $b=b^{\nicefrac{1}{2}}_u$.}
Choose $N$ and $\delta>0$ such that $x_n$ is bounded away from $a$ and $b$ by $\delta$, that is such that
\[  \fa{n \geqslant N}{\abs*{x_n - a} > \delta \land \abs*{x_n - b} > \delta} \ .  \]
Choose $M$ such that $2^{-{M+1}} < \delta$, and let $K = \max\menge{M,N}$
 Since 
\[ F^{\nicefrac{1}{2}}(B_u) = \set{F^{\nicefrac{1}{2}}(u \ast w)}{w \in \cS} \] is compact (as the uniform image of a compact set) the distance 
\[\rho = d\left(F^{\nicefrac{1}{2}}(B_u),  (x_n)_{n \geqslant K}\right) \] exists. 
Now either $\rho > 2^{-(K+1)}$ or $\rho < 2^{-K}$. In the first case there cannot be $n > K$ with  $u = \overline{v_n}\abs{u}$, since that would imply that  $v_n \in \cS$, which implies that $w_n \in \cS$ and therefore
\[ \abs*{F^{\nicefrac{1}{2}}(v_n) - x_n} < 2^{-n+1} < \delta  \ .\] 
This in turn implies that $d(x_n,F^{\nicefrac{1}{2}}(B_u)) < 2^{-(K+1)}$ which is a contradiction. Thus we only need to check $v_1, \dots , v_K$ to check whether $u \in B$ or not.
In the second case there must be $n$ such that $v_n \in cS$ and $F^{\nicefrac{1}{2}}(v_n) \in [a+\delta,b-\delta]$. This ensures that $u = \overline{v_n}\abs{u}$, which means that $u \in B$. 
Altogether we have shown that $B$ is decidable. \qedhere
\end{enumerate}
\end{proof}

\begin{Lem} \label{Lem:bar-to-seq}
If $B$ is a $c$-bar, then there exists $(x_n)_{n \geqslant 1}$  a sequence of real numbers that is eventually bounded away from every point in $[0,1]$  such that
\begin{enumerate}
\item if $(x_n)_{n \geqslant 1}$ is eventually bounded away from $[0,1]$, then $B$ is uniform,
\item  if $x_n \in [0,1]$ infinitely often, then there are arbitrarily large $w$ such that $w \notin B$, and
\item if $B$ is decidable and closed under extensions, then $(x_n)_{n \geqslant 1}$ is tail-located.
\end{enumerate}
\end{Lem}
\begin{proof}
Assume that $B \subset \cS$ is a $c$-bar, and $C^{\prime}$ is a decidable set as in the definition. 
Let $\eta: \NN \to \cS$ be a bijection. In particular, that means that we have
\begin{equation} \label{Eqn:tail-loc-seq}
\fa{n \in \NN}{\ex{m \in \NN}{i \geqslant m \implies \abs*{\eta(i)} \geqslant n}} \ .
\end{equation}
We may also assume that $i \leqslant j \implies \abs*{\eta(i)} \leqslant \abs*{\eta(j)}$.\footnote{The function mapping a natural number $n \geqslant 1$ to the string of its binary expansion without the leading $1$ would be a suitable bijection.}
Define 
\begin{equation*}
x_n = 
\begin{cases}
F^{\nicefrac{1}{3}}(\eta(n)) & \text{if }   \eta(n) \notin C^{\prime} \\
2 & \text{otherwise.}
\end{cases}
\end{equation*} 
We want to show that this sequence is eventually bounded away from every $x\in [0,1]$. To this end choose $\alpha \in \CS$ as in Lemma \ref{Lem:Fp}.2. Since $B$ is a bar there exists $n \in \NN$ such that $\overline{\alpha}n \in B$, which means, that for all $w \in \cS$ we have $\overline{\alpha}n \ast w \in C^{\prime}$. Furthermore, let $m$ be as in Equation \eqref{Eqn:tail-loc-seq}.
For every $i \geqslant m$ we  have $\overline{\alpha}n \neq \overline{\eta(i)}n$ whenever $\eta(i) \notin C^{\prime}$ (notice that $\abs*{\eta(i)} \geqslant n$), which implies that 
\[ \abs*{F^{\nicefrac{1}{3}}(\eta(i)) - F^{\nicefrac{1}{3}}(\alpha)} > 3^{-(n+1)}  \ ,\] for such $\eta(i)$. Since $x_{i}=2$ when $w_{i} = \omega$ we get that 
\[ \abs*{x_{i} - F^{\nicefrac{1}{3}}(\alpha)} > 3^{-(n+1)}  \ ,\] for all $i \geqslant m$.
 Now either $\abs*{F^{\nicefrac{1}{3}}(\alpha) - x} < 3^{-(n+2)} $ or $\abs*{F^{\nicefrac{1}{3}}(\alpha) - x} > 3^{-(n+3)} $. In the first case 
\[ \abs*{x_{i} - x} > 3^{-n+1}  \]
for all $i \geqslant m$. In the second case by Lemma \ref{Lem:Fp}.2 we have that $d(x,F^{\nicefrac{1}{3}}(\CS)) > \delta$ for some $\delta > 0$. In both cases $\abs*{x_{i} - x} > \min \menge{\delta, 3^{-(n+2)}}$ for all $i \geqslant m$, so $(x_n)_{n \geqslant 1}$ is eventually bounded away from every $x \in [0,1]$. 

So let us tackle the three numbered assertions.
\begin{enumerate}
\item Assume that $(x_n)_{n \geqslant 1}$  is eventually bounded away from the entire set $[0,1]$; say  $x_n\notin [0,1]$ for all $n \geqslant M$ on. 
Set $K = \max\set{\abs*{\eta(j)}}{1 \leqslant j < M}$.
Then for all $v \in \cS$ such that $\abs{v} \geqslant K$  we must have $v \notin C^{\prime}$, since otherwise $w_{\eta^{-1}(v)} \in \cS$, and therefore $x_{\eta^{-1}(v)} \in [0,1]$. Hence for all $\alpha \in \CS$ we have $\overline{\alpha}K \in B$; that is $B$ is uniform.
\item If $x_n \in [0,1]$ infinitely often we must have $w_n \in \cS$ infinitely often. This in turn implies that $w_n \notin C^{\prime}$ infinitely often. Since  $C^{\prime} \subset B$ we have that there are arbitrarily large $w$ with $w \notin B$.
\item  It suffices to show that the sets 
 \[A_n = \set{\eta(i)}{i \geqslant n \land \eta(i) \notin B} \]
 are either empty or totally bounded as subsets of $\CS$ , since 
 \[ \menge{ x_n, x_{n+1}, \dots} = F^{\nicefrac{1}{3}}(A_n) \cup \menge{2} \ , \]
 and the uniformly continuous image of a totally bounded set is also totally bounded  and therefore located
\cite{dBlV06}*{Proposition 2.2.6 and 2.2.9}. 
Choose 
\[ M=1 + \set{\abs*{\eta(i)}}{1 \leqslant i < n} \ , \]
which ensures that 
\begin{equation} \label{Eqn:eta-prop}
\abs*{\eta(j)} \geqslant M \implies j \geqslant n \ .
\end{equation}
Furthermore for any $m$ we can choose $k_{m} \geqslant k$ large enough such that 
\begin{equation} \label{Eqn:eta-prop2}
 k_{m} \geqslant \max\set{i \in \NN}{\abs*{\eta(i)} = m} \ . 
\end{equation}
Set
\[ F_{n,m}=\set{\eta(i)}{k_{m} \geqslant i \geqslant n \land \eta(i) \notin B } \ .\] 
The sets $F_{n,m}$ are, for $m \geqslant M$,  a finite $2^{-m}$-approximation of $A_n$: for let $i \geqslant n$ with $\eta(i) \notin B$ and let $j$ be such that $\eta(j) = \overline{\eta(i)} m$. By \ref{Eqn:eta-prop} and \ref{Eqn:eta-prop2} we have $n \leqslant j \leqslant k_{m}$. Since $B$ is closed under extensions $\eta(j) \notin B$. Hence $\eta(j) \in F_{n,M}$ and $d(\eta(j),\eta(i)) \leqslant 2^{-m}$. 

The same argument shows that  $A_n$ is empty if and only if $F_{n,M}$ is. \qedhere
\end{enumerate}
\end{proof}

\section{\texorpdfstring{\WWKL}{WWKL}} \label{sec:WWKL}
We start with a principle that is a weakening of \FAND. 
The so called \define{weak weak K\"{o}nig's lemma} (\WWKL) plays a role in Simpson style reverse mathematics \cite{sS90}. The name is somewhat misleading since it is not resembling weak K\"{o}nig's lemma (see Section \ref{Sec:LLPO}) but rather its contrapositive i.e.\ the fan theorem.\footnote{Of course in Simpson style reverse mathematics which is based on classical logic this distinction is unimportant.} \WWKL does not assure that a decidable bar is uniform, that is that there is a level at which all sequences of that length have been barred, \WWKL only assures that the ratio of all the sequences of a given length that are in the bar over all sequences of a given length tends to $0$.

\begin{principle}[WWKL]{\WWKL} \label{PR:WWKL} If $B \subset \cS$ is a decidable bar that is closed under extensions, then \[ \lim_{n \rightarrow \infty} \frac{\abs*{\set{ u \notin B }{ \abs{u} = n}}}{2^n} = 0 \ .\] 
\end{principle}
In constructive reverse mathematics two publications have involved \WWKL. Their main results are combined in the next proposition.
\begin{Pro}
The following are equivalent to \WWKL
\begin{enumerate}
\item Every positive, uniformly continuous function $f: [0,1] \to \RR$ satisfies the following property: For any $\varepsilon >0 $, there exists $\delta > 0$ such that 
\[ \mu (\set{x}{f(x) < \delta}) \] is defined and 
\[ \mu (\set{x}{f(x) < \delta}) < \varepsilon \ .\]
\item Vitali's theorem: Let $\varepsilon > 0$ be arbitrary. If $\mathcal{V}$ is a countable Vitali cover of $[a,b]$, then there exists a finite set $\menge{I_{1}, \dots,I_{m}}$ of pairwise disjoint intervals of $\mathcal{V}$ such that \[ \mu \left( [a,b] \setminus \bigcup_{i=1}^{m} I_{i} \right) < \epsilon \ .\footnote{Where by $\mu$ we denote the usual measure, which we can obviously define, at least for sets that are finite unions of intervals.}
\]
\end{enumerate}
\end{Pro}
\begin{proof}
The first equivalence can be found in \cite{tN08}, the second one in \cite{hD12a}.
\end{proof}
It seems reasonable to conjecture that WWKL has more equivalents in measure theory, which at the moment has not received much attention in constructive analysis.

\begin{Pro}
\WWKL is equivalent to the following ``weaker'' version for every $k>0$
\begin{principle}[WWKL(k)]{\WWKL(k)} If $B \subset \cS$ is a decidable bar that is closed under extensions, then \[ \frac{ \abs*{\set{ u \notin B}{\abs{u} = n }}}{2^n} < k \]
eventually. 
\end{principle}
\end{Pro}
\begin{proof}
We will show that if $\WWKL(k)$ holds, then also $\WWKL(k^{2})$. This then shows that $\WWKL(k)$ is independent of the choice of $k$, by a suitable number of iterations. So let $B$ be a decidable bar that is closed under extensions. A first application of $\WWKL(k)$ yields a natural number $N$ such that 
 \[ \frac{\abs*{\set{ u \notin B }{\abs{u} = N }}}{2^n} < k \ .\]
 Let $u_{1}, \dots, u_{m} \in 2^n$ be all the finite sequences with $u_{1}, \dots, u_{m} \notin B$; so we have $m < k 2^n$. It is clear that for all of these $u_{1}, \dots, u_{m}$ the sets
 \[ B^{(i)} = \set{w \in \cS}{ u_{i}w \in B }\]
are again decidable bars that are closed under extensions. So applying $\WWKL(k)$ $m$ times we can find $N_{i}$ such that 
 \[ \frac{ \abs*{\set{ u \notin B^{(i)} }{\abs{u} = N_i}}}{2^{N_{i}}} < k \ .\]
Set $M = \max\menge{N_{1}, \dots, N_{m}}$. Since all $B^{(i)}$ are closed under extensions this implies that 
\[  \abs*{ \set{ u \notin B^{(i)} }{\abs{u} = M }}  < k2^{M} \ . \]
Now, if $u \in 2^{N+M}$ is such that $u \notin B$, then it must be of the form $u = u_{i}w$ with $w \notin B$ and $w \in 2^{M}$.  Hence 
 \begin{align*}
 \abs*{\set{u \in 2^{N+M}}{ u \notin B }}
 & =  \abs*{ \bigcup_{i=1}^{m} \set{ u_{i}w \in 2^{N+M} }{ w \notin B^{(i)}}} \\
 & =  \sum_{i=1}^{m}\abs*{\set{ w \in 2^{M} }{ w \notin B^{(i)}}} \\
  & < m k2^{M} <  k2^nk2^{M} = k^{2}2^{N+M} \ ,    \end{align*}
and we are done.
\end{proof}

As conjectured above \WWKL is likely to play an important role in measure theory. One could, of course, also consider the principles which are weakenings of \FANc and \FANP (and even \UCT) in the same way that \WWKL is a weakening of \FAND.
\begin{Qu}
What are the following principles equivalent to? \\ If $B \subset \cS$ is a $c$-bar (or $\Pi^0_1$-bar), then \[ \lim_{n \to \infty} \frac{\abs*{\set{ u \notin B}{ \abs*{u} = n } } }{2^n} = 0 \ . \]
eventually. 
\addtocounter{Que}{1}
\end{Qu}

\section{\texorpdfstring{\FAND}{FANÎ}}  \label{Sec:FAND}
As mentioned in the introduction \FAND deserves the prominence of being involved in the first ``proper'' equivalence of constructive reverse mathematics. It is also a fairly robust statement, as the next lemma shows.

\begin{Lem} \label{Lem:dec_bar_closedunderextensions}
If $B$ is decidable bar, then there exists a decidable bar $B^{\prime}$ that is closed under extensions, such that $B$ is uniform only if $B^{\prime}$ is.
\end{Lem}
\begin{proof}
If $B$ is a decidable bar, then it is easy to see that 
\[ B^{\prime} = \set{u \in \cS}{\ex{n \leqslant \abs{u}} {\overline{u}n} \in B} \]
is also a decidable bar that has the required properties.
\end{proof}
Another interesting result is the following lemma. Surprisingly it has, to our knowledge, not found many uses apart from \cite{dB87}*{Chapter 6.2}.

\begin{Lem} \label{countbar}
If $B$ is a countable bar, then there exists a decidable bar $B'$ which is uniform only if $B$ is.
\end{Lem}
\begin{proof}
  Let $B =\menge{b_n}_{n \geqslant 1} $ a countable bar. Then it is easy to see that
  \[
  B^{\prime} = \set{u \in \cS}{\ex{ i \in \NN}{ \ex{ v \in 2^*}{ \left( i \leqslant \abs{u} \land u \ast v = b_{i} \right)}}}.
  \]
 has the desired properties.\footnote{The full proof can be found in \cite{hD08b}*{Lemma 4.1.1}.}
\end{proof}

\begin{Pro} \label{Pro:Fand-equivalents}
The following are equivalent to  \FAND
\begin{enumerate}
\item (\POS). \label{PR:POS} Every uniformly continuous, positive-valued function $f:[0,1] \to \RR^{+}$ has a positive infimum.
\item \label{fand-equiv-pos2} Every uniformly continuous, positive-valued function $f:\CS \to \RR^{+}$ has a positive infimum.
\item \label{fand-equiv-pos3} Two compact subsets $A,B$ of a metric space that are such that $d(a,b)>0$ for all $a\in A$ and $b \in B$ are a  positive distance apart.
\item \label{fand-equiv-pos4} The Heine-Borel theorem for compact metric spaces and countable coverings with open balls.
\item \label{fand-equiv-dini} Dini's theorem: If $(f_n)_{n \geqslant 1}:[0,1] \to \RR$ is an increasing sequence of uniformly continuous functions converging point-wise to a uniformly continuous $f:[0,1] \to \RR$, then the convergence is uniform.
\end{enumerate}
\end{Pro}
\begin{proof} 
The equivalence of \FAND to \POS is, as already mentioned, the first result in \CRM and is proved in \cite{wJ84} and \cite{dB87}*{Section 6.2}. The equivalence of \ref{PR:POS} and \ref{fand-equiv-pos2} follows from Lemmas \ref{Lem:extendingfunctions} and \ref{Lem:Fp}.
The equivalence of \FAND to \ref{fand-equiv-pos4} is in \cite{hD08b}*{Section 4.2}. 

It is easy to see that \ref{fand-equiv-pos3} implies \ref{PR:POS}, by applying the latter to the graph of a uniformly continuous, positive-valued  function $f$ and the unit interval $[0,1] \times \menge{0}$ as subsets of the Euclidean plane. To see that \ref{fand-equiv-pos2} implies \ref{fand-equiv-pos3} let $A,B$ be subsets such that $d(a,b)>0$ for all $a \in A$ and $b \in B$. By \cite{Troelstra1988}*{Chapter 7 Corollary 4.4} there exists surjective, uniformly continuous functions $g_1:\CS \to A$ and $g_2:\CS \to B$. Then the function $h: \CS \to \RR$ defined by \[ h(\alpha) = d(g_1(\alpha^e),g_2(\alpha^o)) \ , \]
where,  $\alpha^e = \alpha_0 \alpha_2 \alpha_4 \dots$ and $\alpha^o = \alpha_1 \alpha_3 \alpha_5 \dots$. Then $h$ is easily seen to be continuous, and such that $h(\alpha)> 0$ for all $\alpha \in \CS$. Applying \ref{fand-equiv-pos2} implies that there exists $\varepsilon>0$ such that $h(\alpha)>\varepsilon$ for all $\alpha \in \CS$. Since $g_1$ and $g_2$ are surjective also $d(a,b)>\varepsilon$ for all $a \in A$ and $b \in B$.

The equivalence to Dini's theorem is proved in \cite{jB06b}.
\end{proof}


\begin{Pro}\label{Pro:FANDfixp} \FAND is equivalent to the following statement:
Consider a uniformly continuous $f:[0,1] \to \RR$ such that $0<f(x) <x$ for all $x \in (0,1)$. Now let $x_0 \in (0,1)$ be arbitrary and consider the sequence defined by $x_{n+1} = f(x_n) = f^n(x_0)$. The sequence $(x_n)_{n \geqslant 1}$ converges to $0$.\footnote{This is  Exercise K8---one of many in a very entertaining classical book on analysis \cite{tK04}.}
\end{Pro}
\begin{proof}
Clearly $(x_n)_{n \geqslant 1}$ is strictly decreasing. Now assume \FAND and let $\epsilon > 0$ be arbitrary. We want to show that there is $n \in \NN$ such that $x_n < \varepsilon$. The function $g:[\varepsilon, x_1] \to \RR$ defined by \[g(x) = x-f(x) \] is uniformly continuous. It is also such that $g(x) > 0$ for all $x \in [\varepsilon, x_1]$. So by \POS, which as we have seen above is equivalent to \FAND, there exists $\delta >0$ such that $g(x)>\delta$ for all $x \in [\varepsilon, x_1]$. Now choose $n$ such that $1-n\delta < \varepsilon$. Then clearly $x_n < \varepsilon$ and we are done.

Conversely consider a decidable bar $B \subset \CS$. Let $u_n$ be an enumeration of all elements of $B$ such that $u_n[:-1] = \overline{u_n}(\abs*{u_n}-1) \notin B$. Notice that either $u_n = 0 \ast \dots \ast 0$ or we can find its immediate ``left'' neighbour $u_{L(n)}$ that is there exists $w$ such that $u_n=w\ast 1 \ast 0^k$ and $u_{L(n)} = w \ast 0 \ast 1^\ell$ for some $k,\ell \geqslant 0$. We may assume that $u_1 = 1 \ast \dots \ast 1$. Now define a function $f(0,1) \to \RR$ by interpolating linearly between the points \[ (F(u_n),F(u_{L(n)}) ) \ , \]
where $F = F^\frac{1}{2}$ as defined in Section \ref{Sec:linking-cantor-and-unitint}, which simply maps every binary sequence onto the real with the matching binary expansion. Then $f$ satisfies the condition above. Thus we can consider the sequence $x_n$, again defined as above, where $x_0 = F(u_1) < 1$. Let $k$ be such that $u_k = 0 \ast \dots \ast 0$. If $x_n$ converges to $0$, then there must be $N$ such that $f^N(x_0) < 2^{- (\abs*{u_k} +1) }$. It is clear that the sequences $u_1, u_{L(1)} \dots, u_{L^N(1)} = u_k $ describe a cover of $\CS$. In other words their maximum length is a uniform bound for $B$.
\end{proof}

\begin{Pro}
\FAND is equivalent to the statement that every tail-located sequence in $\RR$ that is eventually bounded away from every point in $[0,1]$ is eventually bounded away from the entire set. 
\end{Pro}
\begin{proof}
Immediate consequence of Lemmas \ref{Lem:seq-to-bar} and \ref{Lem:bar-to-seq}.
\end{proof}

Next, we extend the results of Berger in \cite{jB05b}. First we need a lemma, whose proof idea is based on the proof of the main result in the paper mentioned.
\begin{Lem} \label{Lem:Fanandunifcont}
For every $f:\CS \to \RR$ that has a continuous (functional) modulus of continuity and for every $e \in \NN$, there exists a decidable bar $B$ that is uniform only if 
there exists $N \in \NN$ such that 
\[ \overline{\alpha}N = \overline{\beta}N \implies \abs*{ f(\alpha) - f(\beta)} < 2^{-e} \ . \]
\end{Lem}
\begin{proof}
Say $f:\CS \to \RR$ has a continuous modulus that is there exists a continuous function $\mu:\CS \times \NN \to \NN$ such that 
\[ \fa{\alpha, \beta \in \CS}{\fa{e > 0}{ \left( \overline{\alpha}\mu(\alpha,e) =  \overline{\beta}\mu(\alpha,e) \implies \abs*{ f(\alpha) - f(\beta)}  < 2^{-e}\right)}} \ . \]
Now let $e \in \NN$ be arbitrary and define 
\[ B = \set{ u \in \cS}{ \mu(u \ast \zero ,e  ) \leqslant \abs{u}   }  \ .\]
Clearly $B$ is decidable. It is also a bar: for let $\alpha \in \CS$ be arbitrary. Since $\mu$ is continuous itself there exists $M$ such that $\mu(\alpha,e) =\mu(\beta,e) $ for all $\beta \in \CS$ with $ \overline{\alpha}M = \overline{\beta}M$. For  $M^{\prime} = \max \menge{M,\mu(\alpha,e)}$ we must have 
\[ \mu(\overline{\alpha}M^{\prime}  \ast \zero ,e) = \mu(\alpha,e) \leqslant M^{\prime} \ , \] and therefore $\overline{\alpha}M^{\prime} \in B$. 

Next, assume $B$ is uniform. So there exists $N \in \NN$ such that 
\[ \fa{\alpha \in \CS}{ \ex{n \leqslant N}{\overline{\alpha}n \in B} } \ .\]
So for any $\alpha \in \CS$ there is $n \leqslant N$ such that $\mu(\overline{\alpha}n\ast \zero,e ) \leqslant n \leqslant N$, and hence for all $\beta \in \CS$ with $\overline{\alpha}N=\overline{\beta}N$ we have $\abs*{f(\alpha) - f(\beta)} < 2^{-e}$.
\end{proof}

\begin{Pro} \label{FAND-equiv} The following are equivalent to \FAND
\begin{enumerate}
\item \label{FAND-equiv-cont1} Every continuous function $f:X \to Y$, with a continuous modulus of continuity is uniformly continuous,  where $X$ is a compact and $Y$ an arbitrary metric space.
\item  \label{FAND-equiv-cont2} Every continuous function $f:\CS \to \RR$ with a continuous modulus of continuity is uniformly continuous.
\item  \label{FAND-equiv-cont3} Every continuous function $f:\CS \to \NN$ with a continuous modulus of continuity is uniformly continuous.
\end{enumerate}
\end{Pro}
\begin{proof}
It is clear that \ref{FAND-equiv-cont2} implies \ref{FAND-equiv-cont3} and that  \ref{FAND-equiv-cont1} implies  \ref{FAND-equiv-cont2}. In \cite{jB05b} it is shown that if $B$ is a decidable bar, then 
\[ f(\alpha) = \min\set{n \in \NN}{\overline{\alpha}n \in B} \]
is a continuous function with itself (!) as a continuous modulus of continuity. Furthermore it is shown that if $f$ is uniformly continuous then $B$  is uniform. Hence \ref{FAND-equiv-cont3} implies \FAND. Next, with the help of Lemma \ref{Lem:Fanandunifcont}, we can show that \FAND implies \ref{FAND-equiv-cont2}. 

So we are done if we can show that  \ref{FAND-equiv-cont2} implies  \ref{FAND-equiv-cont1}. To this end let $f:X \to Y$ be a uniformly continuous map with a continuous modulus of continuity $\mu: X \times \NN \to \NN$ such that
\[ \fa{x,y \in X}{\fa{e \in \NN}{ d(x,y)<2^{-\mu(x,e)} \implies d(f(x),f(y))<2^{-e}   } }  \ .\] 
Let  $g: \CS \to X$ the mapping from Proposition \ref{Pro:compXunifsurj}. Using countable choice we can therefore construct $\tau: \NN \to \NN$ such that 
\[ \fa{\alpha, \beta \in \CS}{\fa{n \in \NN}{\overline{\alpha}\tau(n) =  \overline{\beta}\tau(n) \implies d(g(\alpha),g(\beta)) < 2^{-n}}}  \ . \]

 Furthermore, for $\alpha \in \CS$ let $\alpha^{e} = \alpha(0)\alpha(2)\alpha(4) \dots$ and $\alpha^{o} = \alpha(1)\alpha(3)\alpha(5) \dots$. Now consider the function $h :\CS \to \RR$ defined by 
\[h(\alpha) = d(f(h(\alpha^{e})),f(h(\alpha^{o}))) \ .\]
It is straightforward to show that $\eta:\CS \times \NN \to \NN$ defined by 
\[ \eta(\alpha,n) =  2 \tau (\mu (g(\alpha^{e}),n) )     \] is a continuous modulus of continuity for $h$. By \ref{FAND-equiv-cont2} it is also uniformly continuous. So for an arbitrary $\varepsilon >0$ there exists $n$ such that $\overline{\alpha}n=\overline{\beta}n$ implies $\abs*{h(\alpha) - h(\beta)} < \varepsilon$. By Lemma \ref{Lem:cmpact-image-of-2N} there exists $\delta >0$ such that for  $x,y \in X$ with $d(x,y)< \delta$ there exist $ \alpha,\beta \in \CS$ such that 
\[ \overline{\alpha}n = \overline{\beta}n \land x=g(\alpha) \land y = g(\beta) \ .  \]
Now define $\gamma$ and $\gamma^{\prime}$ by 
\[ \gamma= \alpha(0)\beta(0)\alpha(1)\beta(1)\dots \]
and
\[ \tau= \alpha(0)\alpha(0)\alpha(1)\alpha(1)\dots \ . \]
That way $\gamma^{e} = \tau^{e} = \tau^{o} = \alpha$ and $\gamma^{o} = \beta$.
Furthermore $\overline{\gamma}n = \overline{\tau}n$, and hence 
\begin{align*}
 \varepsilon & > \abs*{h(\gamma) - h(\tau)}  \\
 & = \abs*{ d(f(h(\gamma^{e})),f(h(\gamma^{o}))) - d(f(h(\tau^{e})),f(h(\tau^{o})))} \\
 & =  \abs*{ d(f(h(\alpha)),f(h(\beta))) - d(f(h(\alpha)),f(h(\alpha)))} \\
 & = d(f(h(\alpha)),f(h(\beta))) = d(f(x),f(y)) \ . 
\end{align*}
That means that $\delta$ is a modulus of uniform continuity for $\varepsilon$, and we are done.
\end{proof}

\begin{Def}
We will call a function $f:[0,1] \to \RR$ \/ \define{fully located} if, $f([a,b])$ is located for every $a<b$.
\end{Def}
Classically every function is fully located, and constructively every uniformly continuous function $[0,1] \to \RR$ is fully located, however the converse does not hold, as is shown in Proposition \ref{Pro:KT-fullyloc}, which will be a consequence of the following lemmas.
\begin{Lem} $ $ 
\begin{enumerate}
\item Assume $f:[0,1] \to \RR$ is fully located. Then $f$ is point-wise continuous if it is sequentially continuous.
\item If $f(x) > 0$ for all $x \in [0,1]$, then $f$ is fully located if and only if $\frac{1}{f}$ is.
\end{enumerate}

\end{Lem}
\begin{proof}
\begin{enumerate}
\item Let $x \in [0,1]$ and $\varepsilon>0$ be arbitrary. 
Using our locatedness assumption, we can decide for every $n \in \NN$ whether $\ex{y \in B_{2^{-n}}(x)}{\abs*{f(y)-f(x)} < \varepsilon} $ or whether $\fa{y \in B_{x}(2^{-n})}{\abs*{f(y)-f(x)} > \frac{\varepsilon}{2}}$. So using countable choice we can construct a binary sequence $(\lambda_n)_{n \geqslant 1}$ and a sequence of reals $(x_n)_{n \geqslant 1}$ such that $\abs*{x_n-x}<\frac{1}{2^n}$ and
\begin{align*}
\lambda_n=0 & \implies \abs*{f(x_n) -f(x)} > \frac{\varepsilon}{2}  \ , \\
\lambda_n=1 & \implies \fa{y}{\abs*{x-y}<\frac{1}{2^n} \implies \abs*{f(y) - f(x)}< \varepsilon} \ .
\end{align*}
Since $x_n$ converges to $x$ and $f$ is sequentially continuous there exists $N$ such that for all $n \geqslant N$ we have $\abs*{f(x_n) -f(x)} < \frac{\varepsilon}{2}$. This in turn implies that $\lambda_n = 1$, which immediately yields the desired property. Hence $f$ is point-wise continuous.
\item Straightforward. \qedhere
\end{enumerate}
\end{proof}

\begin{Lem} \label{Lem:fullyloc}
For every decidable bar $B$ there exists a point-wise continuous, fully located function $f:[0,1] \to \RR$ such that $f$ is bounded if and only if $B$ is uniform; and vice versa.
\end{Lem}
\begin{proof}
First, start with a decidable bar $B$ that is without loss of generality closed under extension. We are going to adapt the construction in \cite{dB87}*{Theorem 2.2.7}. Even though, given a bar, we could also define a function $f:\CS \to \NN$ from it and extend it to a function $\hat{f}:[0,1] \to \RR$, proving that if $f$ is fully located, then $\hat{f}$ is seems tedious, so we are going to go a different route. Since $B$ is decidable so is $T =\lnot B$, and since $B$ is closed under extensions $T$ is closed under restriction, i.e.\ a tree. Furthermore $T$ does not admit infinite paths. Now consider the set 
\[ S = \set{(F^{\nicefrac{1}{3}}(u),2^{-\abs{u}}) \in \RR^{2}}{u \in T} \ . \]
It is easy to see that $S$ is totally bounded and therefore located. Hence the function $f: [0,1] \to \RR$ defined by $f(x) = d((x,0),S)$ exists and is uniformly continuous. 
So, since $[a,b]$ is totally bounded, its image under $f$ also totally bounded and therefore located. Just as in  \cite{dB87}*{Theorem 2.2.7} the fact that $T$ does not admit infinite paths implies that $f$ is positively valued. Furthermore if $\inf f > 0$, then $B$ is uniform. By the previous  lemma $\frac{1}{f}$ is fully located, and is such that it is bounded if $B$ is uniform.

Conversely assume that $f:[0,1] \to \RR$ is fully located, and point-wise continuous. Hence, using countable choice, we can fix a decidable set $B \subset \cS$ such that 
\begin{align*}
u \notin B & \implies \ex{x \in I_{u}}{f(x) > 2\abs{u}} \ , \\
u \in B & \implies \fa{x \in I_{u}}{f(x)  <2 \abs{u}+1}  \ ;
\end{align*}
where  $I_{u}$ are the intervals $I_{u}^{\nicefrac{2}{3}}$ defined above. We want to show that $B$ is a bar, so let $\alpha \in \CS$ be arbitrary. Since $f \circ F^{\nicefrac{2}{3}}$ is point-wise continuous there exists $N \in \NN$ such that $\abs*{f \circ F^{\nicefrac{2}{3}}(\alpha) - f \circ F^{\nicefrac{2}{3}}(\beta)} < 1$ whenever $\overline{\alpha}N = \overline{\beta}N$. Now choose $M > N$ such that $f \circ F^{\nicefrac{2}{3}}(\alpha) < M$. Now assume that $\overline{\alpha} M \notin B$. Then, by the choice of $B$ there exists $y \in I_{\overline{\alpha}M}$ such that $f(y) > 2M$. Since $F^{\nicefrac{2}{3}}$ is surjective on $I_{\overline{\alpha}M}$ (Lemma \ref{Lem:Fp}) there exists $\beta \in \CS$ such that $F^{\nicefrac{2}{3}}(\beta) = y$ and $\overline{\beta}M = \overline{\alpha}M$.  But that means that \[
f \circ F^{\nicefrac{2}{3}}(\alpha)+ \abs*{f \circ F^{\nicefrac{2}{3}}(\alpha) - f \circ F^{\nicefrac{2}{3}}(\beta)} > f \circ F^{\nicefrac{2}{3}}(\beta)  > 2M
\]
which implies that $f \circ F^{\nicefrac{2}{3}}(\alpha) > 2M -1$; but this is a contradiction to $f \circ F^{\nicefrac{2}{3}}(\alpha)  < M \leqslant 2M-1$. 

By definition $B$ is also closed under extensions. Therefore, if $B$ is uniform there exists $N$ such that $\overline{\alpha}N \in B$ for all $\alpha \in \CS$. Since $[0,1] \bigcup_{u \in 2^n} I_{u}$ that means that $f(x) < 2N+1$ for all $x \in [0,1]$.
\end{proof}

\begin{Cor} \label{Cor:FAND_equiv_fullyloc} \FAND is equivalent to every sequentially continuous, fully located function $f:[0,1] \to \RR$ being bounded.
\end{Cor}

\begin{Rmk}
In \RUSS there exists a point-wise continuous, fully located function $f:[0,1] \to \RR$  which fails to be bounded.
\end{Rmk}
\begin{proof}
See Proposition \ref{Pro:KT-fullyloc}.
\end{proof}

\section{\texorpdfstring{\FANc}{FANc}}

In \cite{jB06} in order to answer the question what kind of fan theorem is equivalent to the uniform continuity theorem for functions $\CS \to \NN$ Berger introduced the notion of a $c$-bar. Not much later, the resulting principle \FANc turned out to be equivalent to the ``anti-Specker''-principle, which was conceived independently from  Berger's work \cite{jB07}. 

\begin{Pro} \label{Pro:FANc-equivs}
The following are equivalent to \FANc.
\begin{enumerate}
\item \AS{[0,1]}{\RR} 
\item \AS{[0,1]}{1}
\item \AS{\CS}{1} 
\item \label{Equiv:FANc2} For every compact $X$ and every one-point extension $Y$ \AS{X}{Y} holds.
\item \label{Equiv:FANc3} Every point-wise continuous function $f:\CS \to \BS$ is uniformly continuous.
\item \label{Equiv:FANc4} Every point-wise continuous function $f:\CS \to \NN$ is uniformly continuous.
\item \label{Equiv:FANc5} Every point-wise continuous function $f:\CS \to 2$ is uniformly continuous.
\item \label{Equiv:FANc6} Every point-wise equi-continuous sequence of mappings of $[0,1]$ into $\RR$ is uniformly sequentially equi-continuous.
\item \label{Equiv:FANc7} Every point-wise continuous mapping of $[0, 1]$ into a metric space is uniformly sequentially continuous.
\end{enumerate}
\end{Pro}
\begin{proof}
Even though the equivalence $\cramped{\FANc \iff  \AS{[0,1]}{\RR}}$ has been shown in \cite{jB07} it also follows from Lemmas  \ref{Lem:seq-to-bar} and \ref{Lem:bar-to-seq}. For the equivalence between \AS{[0,1]}{\RR} and \ref{Equiv:FANc2} see \cite{hD09}, for the one between \AS{[0,1]}{\RR} and \AS{[0,1]}{1} see \cite{hD09b}, and for the one between \FANc and \AS{\CS}{1} see \cite{hD08b}. The equivalence between \ref{Equiv:FANc4} and \FANc~has been shown by Berger in \cite{jB06}. 
Clearly $\ref{Equiv:FANc3} \implies \ref{Equiv:FANc4} \implies \ref{Equiv:FANc5}$. To see that also $\ref{Equiv:FANc5} \implies \ref{Equiv:FANc3}$ let $f:\CS \to \BS$ be a point-wise continuous function and $\varepsilon >0$ arbitrary. First choose $N$ such that $2^{-N}< \varepsilon$. Now define $g: \CS \to \menge{0,1}$ by
\begin{equation*}
 	g(\alpha) = 
 	\begin{cases} 
 		0 & \text{ if }\fa{i\leqslant N}{f(\alpha^e)(i) = f(\alpha^o)(i)} \\
		1 &  \text{ otherwise,} 
	\end{cases}
\end{equation*}
where, as before, $\alpha^e = \alpha_0 \alpha_2 \alpha_4 \dots$ and $\alpha^o = \alpha_1 \alpha_3 \alpha_5 \dots$. 
The function $g$ is continuous, since for an arbitrary $\alpha$ there exists $M$ such that for arbitrary $\beta$ we have 
\[ \abs*{f(\overline{\alpha^e}M \ast \beta) - f(\alpha^e)} < 2^{-N} \text{ and } \abs*{f(\overline{\alpha^o}M \ast \beta) - f(\alpha^o)} < 2^{-N} \ , \] which means that for all $i \leqslant N$
\[ f(\overline{\alpha^e}M \ast \beta)(i) = f(\alpha^e)(i) \text{ and } f(\overline{\alpha^o}M \ast \beta)(i) = f(\alpha^o)(i) . \]
So either $f(\alpha^o)(i) = f(\alpha^e)(i)$ for all $i \leqslant N$ or not. In the first case $g(\alpha) = 0$ and in the second case $g(\alpha) = 1$. But by the choice of $M$, also for any $\beta$ we have $g(\overline{\alpha}(2M) \ast \beta) = g(\alpha)$, and hence $g$ is point-wise continuous. By our assumption it is therefore uniformly continuous, which means we can find an $M$ as above which is independent of $\alpha$, which immediately gives us uniform continuity of $f$.

Finally, the equivalences between \AS{[0,1]}{\RR}, \ref{Equiv:FANc6}, \ref{Equiv:FANc7} are proved in \cite{dB09b}.
\end{proof}

Notice that, unlike in the case of variations of \UCT, \AS{}{} and many other principles it is not possible to replace $\CS$ with $[0,1]$ in \ref{Equiv:FANc4} and \ref{Equiv:FANc5} above, since a point-wise continuous function $f:[0,1] \to \NN$ is necessarily constant.

\begin{Pro}[Variant of \POS]
\FANc~is equivalent to every point-wise continuous function $f:\CS \to \NN$ being bounded.
\end{Pro}
\begin{proof}
	Clearly, this statement is implied by Part \ref{Equiv:FANc4} of the above proposition and therefore follows from \FANc. Conversely we can directly see that it implies \FANc. Given a $c$-bar $B$ with $C \subset \cS$ such that 
	\[  u \in B \iff \fa{w \in \cS}{\left(u \ast w \in C\right)} \ , \] 
	we can define $f:\CS \to \NN$ by 
	\[ f(\alpha) = \max \set{n \in \NN}{\overline{\alpha}n \notin C}  \ .\footnote{Notice that $f(\alpha) = \min \set{n \in \NN}{\overline{\alpha}n \in B}$ would not be well-defined, unless $B$ would actually be a decidable bar. }\]
	The fact that $B$ is a $c$-bar and $C$ is decidable ensures that this is a well-defined and point-wise continuous function. If $f$ is bounded by $N$, then this $N$ also gives us a uniform bound for $B$.
\end{proof}

This last equivalence also allows us to proof another version of Dini's theorem, which was first proved in \cite{jB09c}*{Theorem 19}.\footnote{To be precise, in that paper, the authors prove an equivalence of Dini's theorem to \ref{Pro:FANc-equivs}.\ref{Equiv:FANc4}. A direct proof can be found in the freely available \cite{jB18b}*{Proposition 8}}
\begin{Pro}[version of Dini's theorem] \label{Pro:FANc_equiv_Dini}
\FANc~is equivalent to every decreasing sequence of point-wise continuous functions $(f_n)_{n \geqslant 1}$ converging point-wise to $0$. Then the convergence is uniform.
\end{Pro}
\begin{proof}
One direction follows from the fact that \FANc implies \FAND, that \FANc implies that a point-wise continuous function is uniformly continuous (see \ref{Equiv:FANc4} of \ref{Pro:FANc-equivs}), so we can use Proposition \ref{Pro:Fand-equivalents}.\ref{fand-equiv-dini}.

For the other direction, by the previous proposition, we only need to show that a function $f:\CS \to \NN$ is bounded. Now define $f_n : \CS \to \NN$ by $f_n(\alpha) = \max\menge{f(\alpha) - n,0}$. Clearly $f_n \geqslant f_{n+1}$. Since $f$ is point-wise continuous (which in this case even means locally constant), it is locally bounded. Therefore $f_n$ is eventually zero, locally, which means that $f_n \to 0$ point-wise. Now assume the convergence is uniform. That means there is $N$ such that $f_n < 1$. But that means that there cannot be $\alpha \in \CS$ with $f(\alpha) > N+1$, since then $f_N(\alpha) = \max\menge{f(\alpha)-N,0 } > 1$.
\end{proof}

\section{\texorpdfstring{\UCT}{UCT}} \label{Sec:UCT}

The uniform continuity theorem is the standard first step in numerous theorems of classical analysis. It states:
\begin{principle}[UCT]{\UCT} \label{PR:UCT}
Every point-wise continuous function $f:[0,1] \to \RR$ is uniformly continuous.
\end{principle}

Its importance in shaping the development of constructive mathematics cannot be overrated. It is not provable with purely constructive methods (i.e.\ in \BISH), which led Brouwer to the notion of bar induction, in order to be able to prove it. Bishop sneakily avoided such additional assumptions by simply building it into his definition of a continuous function. 

It is surprising, that  \UCT for real-valued functions on $[0,1]$ suffices to show the most general conceivable one, as proved in \cite{dB07}
\begin{Pro} \label{Pro:Equiv_UCT}
\UCT is equivalent to the statement that every point-wise continuous map $X \to Y$ on a compact metric space $X$ and into an arbitrary metric space $Y$ is uniformly continuous.
\end{Pro}
Also in \cite{dB07}, and maybe even more surprising, it is shown it is also equivalent to some weaker versions
\begin{Pro} \label{Pro:more_Equiv_UCT}
\UCT is equivalent to the following
\begin{enumerate}
  \item Every point-wise continuous function $f:[0,1] \to \RR$ is bounded. (That is $[0,1]$ is pseudo-compact.)
  \item Every point-wise continuous function $f:\CS \to \RR$ is bounded.
  \item Every point-wise continuous mapping of $[0,1]$ into $\RR$ is  integrable.
\end{enumerate}
\end{Pro}

Notice that the construction to prove the last equivalence in the above proposition relies on creating a function that has ``very high peaks.'' An interesting question is therefore: 

\begin{Qu}
Is there a version of the fan theorem equivalent to the following statement: 
every bounded point-wise continuous function $f:[0,1] \to \RR$ is integrable.
\end{Qu}

It is folklore that the following holds.
\begin{Pro} \UCT is equivalent the statement that every point-wise continuous $f:[0,1] \to \RR$ can be uniformly approximated with polynomials.
\end{Pro}
\begin{proof}
	Of course, if $f$ can be uniformly approximated with polynomials, then $f$ is uniformly continuous, since polynomials are. Conversely the only non-constructive step in the normal proof that uses Bernstein polynomials \cite{gH94}*{Section 4.2.2} is to conclude that $f$ is uniformly continuous, which is exactly what \UCT enables us to do.
\end{proof}

Although the following proposition seems like a harmless variant of the previous one, it is actually very different: Notice that for the converse direction the functions are not assumed to be uniformly continuous. Somehow, just the fact that there are countably many of them suffices.

\begin{Pro} \label{Pro:UCT01sep}
	\UCT is equivalent to $C([0,1])$ being separable. \\ More precise: \UCT is equivalent to the statement that there exists a sequence of point-wise continuous functions $f_n:[0,1] \to \RR$ such that for all $\varepsilon > 0$ and all point-wise continuous $f:[0,1] \to \RR$ there exists $n \in \NN$ with 
	\[  \fa{x \in [0,1]}{\abs*{f_n(x)-f(x)} < \varepsilon} \ . \]	
\end{Pro}
\begin{proof}
Obviously \UCT implies that $C([0,1])$ is separable, just as in the classical proof. The interesting part is the proof of the converse: Consider $(f_n)_{n \geqslant 1}$ dense in $C([0,1])$. By Proposition \ref{Pro:Equiv_UCT} it suffices to show that every point-wise continuous $g:[0,1] \to \RR$ is bounded. To this end consider the functions $g_n$ defined by
\[ g_n(x) = \max \menge{0, 1- \abs*{g(x)-n}} \ . \] These have the property that $g_n(x)>0 \iff g(x) \in (n-1,n+1)$. Thus the function $G:[0,1] \to \RR$ defined by
\[ G(x) = \sum_{n \in \NN} g_n(x) \left( \max\menge{f_1(x), \dots f_n(x) } + 1\right) \]
is well-defined, and continuous. So there exists $M \in \NN$ such that 
\[ \fa{x \in [0,1]}{\abs*{f_M(x)-G(x)} <1 } \ .\]
Now there cannot be $x$ such that $g(x) > M$, since in that case $g_M(x) = 1$ and $g_i(x) = 0$ for all $i \neq M$. Hence $G(x) = \max\menge{f_1(x), \dots f_n(x) } + 1 \geqslant f_M(x) + 1$, which is a contradiction. Thus $g(x) \leqslant M$ for all $x \in [0,1]$ and we are done.
\end{proof}
\begin{Cor}
In \RUSS, $C([0,1])$ is not separable.
\end{Cor}

As mentioned in Section \ref{Sec:FAND} \FAND is equivalent to Dini's theorem formulated for a sequence of uniformly continuous functions $[0,1] \to \RR$. It is not too surprising that we can also prove the following version.

\begin{Pro*}[{\textbf{\theThm{}½}} version of Dini's theorem] \label{Pro:Qu6answered} \UCT is equivalent to the statement that 
if $(f_n)_{n \geqslant 1}:[0,1] \to \RR$ is a decreasing sequence of point-wise continuous functions converging point-wise to a point-wise continuous $f:[0,1] \to \RR$, then the convergence is uniform.\footnote{This answers Question 6 in the original version.} 
\end{Pro*}
\begin{proof}
	One direction is trivial since \UCT implies \FAND, and allows us to replace point-wise continuity by uniform continuity in Proposition \ref{Pro:Fand-equivalents}.
	
	Conversely, as mentioned in Proposition \ref{Pro:more_Equiv_UCT}, it suffices to show that a point-wise continuous $f:[0,1]\to\RR$ is bounded.  Now consider the sequence of functions $(f_n)_{n \geqslant 1}$ defined by \[ f_n(x) = \max\menge{f(x) - n, 0} \ .\] It is easy to see that this sequence is decreasing. Since a point-wise continuous function is locally bounded, it is also easy to see that we have $f_n(x) = 0$ eventually (this is actually the case on an entire neighbourhood of $x$). So $f_n(x) \to 0$ point-wise. Now assume that $f_n \to 0$ uniformly. Then there exists $N \in \NN$ such that $\abs*{f_N(x)} < 1$. We have $f(x)\leqslant N+1$ for any $x \in [0,1]$: for assume that $f(x) > N+1$, or  that equivalently $f(x) - N > 1$. Then $f_N(x) =\max\menge{f(x) - N,0}  > 1$; a contradiction.
\end{proof}

In \cite{hD11b} the following variations were considered in the context of differential equations.
\begin{Pro} The following are equivalent to \UCT
\begin{enumerate}
  \item (BUCT). Every bounded, continuous function $f: [0,1] \to \RR$ is uniformly continuous.
  \item (LUCT). Every continuous function $f: [0,1] \to \RR$ is locally uniformly continuous.
\end{enumerate}
\end{Pro}

\section{\texorpdfstring{\FANP and \FANst}{FAN-Î  and FANst}}
We will start this subsection with an attempt to remove some confusion about principles running under the name \FANP. Above we defined a $\Pi_{1}^{0}$-bar to be a bar that is the intersection of decidable sets that are closed under extension. Equivalently one often finds the following definition for a $\Pi_{1}^{0}$-bar $B$. There exists a decidable set $S \subset \cS \times \NN$ such that
\begin{enumerate}
\item $u \in B \iff \fa{n \in \NN}{(u,n) \in S}$
\item \label{stablebarcond} If $(u,n) \in S$, then for any $w \in \cS$ also $(u \ast w,n) \in S$.
\end{enumerate}
However, some authors\footnote{Including ourself.} have also referred to bars only satisfying the first condition as  $\Pi_{1}^{0}$-bars. In the following we will call these bars \define{stable bars}. Notice that stable bars are exactly bars that are the complement of a countable set. Of course another complexity of a bar leads to another fan theorem.
\begin{principle}[FANst]{\FANst} \label{PR:FANst}
Every stable bar is uniform.
\end{principle}
Trivially, 
\[ \FANf \implies \FANst \implies \FANP \ , \]
but the converses seem unlikely. The unproven (and probably false) assumption that \FANst~and \FANP are equivalent can be explained by Lemma \ref{Lem:dec_bar_closedunderextensions} taken together with the following observation.
\begin{Pro}
If a stable bar $B$ is closed under extensions, then $B$ is a $\Pi_{1}^{0}$-bar.
\end{Pro}
\begin{proof}
This proof is based on one by J.\ Berger.\footnote{Personal Correspondence, 18th Mar 2011.} Notice that a similar construction is also used in the proof of Theorem 2 of \cite{hD08} as well as in the proof of Proposition 2 of \cite{hD09} and in the proof of Proposition \ref{Pro:Equiv-fanp-equipos}. 
Assume that  the bar $B$ is the intersection of the decidable sets $B_n$. Now define 
\[ \widetilde{B}_n = \set{u \in \cS}{\fa{w}{\left( \abs{w} \leqslant n -\abs{u} \implies u \ast w \in B_n \right)} }  \ . \]
By  construction  every $\widetilde{B}_n$ is closed under extensions. We want to show that \[ B = \bigcap_{n \geqslant 1} \widetilde{B}_n \ . \]
To see that ``$\subset$'' holds let $u \in B$, which means that $u \ast w \in B$ for all $w \in \cS$. This implies that for all $w \in \cS$ and all $n \in \NN$ we have $u\ast w \in B_n$ and in particular that $u \ast w \in \widetilde{B}_n$. 
For the other direction ``$\supset$'' consider $u$ such that $u \in \bigcap_{n \geqslant 1}\widetilde{B}_n$, and consider an arbitrary $m \in \NN$. If $m \leqslant \abs{u}$, then write $u=\overline{u}m \ast w$. For this $w$ we have $\abs{w} \leqslant m - \abs{u}$, and therefore, since $u \in \widehat{B}_{m}$, also $u \in B_{m}$. If $m > \abs{u}$, then taking the $\varepsilon$, the empty sequence, as $w$ we get $u \ast w = u \in B_{m}$. Altogether $u \in B_{m}$ for any  $m$ and therefore $u \in \bigcap_{n \geqslant 1} B_n = B$.
\end{proof}
With the help of this proposition it seems as if we can show that for every stable bar there exists an equivalent $\Pi_{1}^{0}$ bar. However, the missing step is to prove that for every stable bar there exists an equivalent stable bar that is closed under extensions. The proof of Lemma \ref{Lem:dec_bar_closedunderextensions} does not generalise to this case, even if it seems on first glance easy to adapt the proof somehow. Notice that it would be enough to show that the closure (under extensions) of a stable bar is stable.

Our first \FANP-equivalence is a sequential version of \ref{Pro:Fand-equivalents}.\ref{PR:POS}.

\begin{Pro} \label{Pro:Equiv-fanp-equipos}
\FANP is equivalent to the following statement: \\
Every  equi-continuous, and equi-positive,\footnote{We call a sequence of functions $(f_{k})_{k \geqslant 1}:X \to \RR$ \define{equi-positive} if for every $x \in X$ there exists $\varepsilon>0$ with $f_{k}(x)> \varepsilon$ for all $k \in \NN$.} sequence of functions $(f_{k})_{k \geqslant 1}: \CS \to \RR$ is such that there exists $\delta>0$ such that 
\[  f_{k}(\alpha) > \delta \text{ for all } \alpha \in \CS \text{ and } k\in \NN \ . \]

\end{Pro}
\begin{proof}
First, assume that \FANP holds. Using countable choice fix decidable sets $B_{n,k} \subset \cS$ such that
\begin{align}
\label{Eqn:stabbardef} u \notin B_{n,k} & \implies \ex{w \in \cS}{\abs{w}+\abs{u} \leqslant n \land f_{k}(u*w) < 2^{-\abs{u}}} \ , \\
 u \in B_{n,k} & \implies \fa{w \in \cS}{\abs{w}+\abs{u} \leqslant n \implies f_{k}(u*w) > 2^{-(\abs{u}+1)}} \ .
\end{align}
We will show that $B = \bigcap_{n,k \in \NN} B_{n,k} $ is a  bar. To this end let $\alpha \in \CS$ be arbitrary. Since $(f_{k})_{k \geqslant 1}$ is equi-positive there exists $N$ such that $f_{k}(\alpha) > 2^{-(N-1)}$ for all $k \in \NN$. Now choose $M \in \NN$ such that 
\[ \fa{\beta}{\overline{\beta}M =\overline{\alpha} M \implies \abs*{f_{k}(\beta) -  f_{k}(\alpha)} < 2^{-N}} \ .  \]
Let $K= \max \menge{N,M}$. For any $k \in \NN$ and $w \in \cS$ we get  that
\begin{align*}
 f_{k}(\overline{\alpha}K \ast w )   
 & > f_{k}(\alpha ) - \abs*{f_{k}(\overline{\alpha}K \ast w) - f_{k}(\alpha)}  \ , \\
 & > 2^{-(N-1)} - 2^{- N} \\
 & > 2^{-N} > 2^{-K} \ . 
\end{align*}
Hence $\overline{\alpha}K \in B_{n,k}$, since case \ref{Eqn:stabbardef} is ruled out. Therefore $B$ is a bar. By construction it is also closed under extensions and we can apply \FANP, to get $L \in \NN$ such that $u \in B$ for all $u \in 2^{L}$. So for all $\alpha \in \CS$ we get, by continuity, $f_{k}(\alpha) \geqslant 2^{-(L+1)}$, which concludes one direction of the proof.

Conversely, let $B=\bigcap_{k \in \NN} B_{k}$ a stable bar, with decidable sets $B_{k}$ that are closed under extensions. Define $f_{k}: \CS \to \RR$ by  
\[ f_{k}(\alpha) = 2^{- \min \set{\ell \in \NN}{\overline{\alpha}\ell \in B_{k}}} \ .\]
We will show that $(f_n)_{k \geqslant 1} $ is equi-continuous and equi-positive.
To this end let $\alpha \in \CS$ be arbitrary. Since $B$ is a bar there exists $n \in \NN$ such that $\overline{\alpha}n \in B$, which means that for all $k \in \NN$ we can find the minimal $n_{k} \leqslant n$ such that $\overline{\alpha}n_{k} \in B_{k}$. This, in turn, implies that if $\overline{\beta}n = \overline{\alpha}n$, then $\overline{\beta}n_{k} = \overline{\alpha}n_{k}$ and therefore
\[ f_{k}(\alpha) = f_{k}(\beta) = 2^{-n_{k}} \geqslant 2^{-n}> 0 \ . \] Thus we have both shown equi-positivity and equi-continuity. By our assumptions there exists $N \in \NN$ such that $f_{k}(\alpha) > 2^{-N}$ for all $k \in \NN$ and $\alpha \in \CS$. This translates back, for all $\alpha \in \CS$, into the existence of $m \leqslant N$ such that $\overline{\alpha}m \in B_{k}$. Since the sets $B_{k}$ are closed under extensions also $\overline{\alpha}N \in B_{k}$ for all $n \in \NN$ and $\alpha \in \CS$. Therefore $\overline{\alpha}N \in B$  for all $\alpha \in \CS$, which means $B$ is uniform.
\end{proof}

\begin{Pro} The following are equivalent to \FANP
\begin{enumerate}
\item \label{bizzarevariant} Every locally constant function $f:X \to Y$ is uniformly locally constant; where $X$ is a compact and $Y$ an arbitrary metric space.
\item \label{bizzare} Every locally constant function $f:\CS \to \RR$ is uniformly locally constant.
\end{enumerate}
\end{Pro}
\begin{proof} 
The equivalence of \FANP to \ref{bizzare} is shown in \cite{Ber06}.
The direction $\ref{bizzarevariant} \implies \ref{bizzare}$ is trivial. To see that the converse is valid let $f:X \to Y$ be a locally constant function. As usual, we can find a surjective and uniformly continuous function $g:\CS \to X$ \cite{Troelstra1988}*{Chapter 7 Corollary 4.4}. 
Now consider the function $h:\CS \to \RR$ defined by:
\[ h(\alpha) = d(f(\alpha^e),f(\alpha^o)) \ ,\]
where $\alpha^e$ is the sequence of all even terms of $\alpha$ and $\alpha^o$ the one of all odd ones (compare the proof of Proposition \ref{FAND-equiv}). Since $f$ is locally constant so is $h$. By our assumption that means that $h$ is globally constant, so there exists $N$ such that $\overline{\alpha}N = \overline{\beta}N$ implies that $h(\alpha) = h(\beta)$. Now given $\alpha, \beta$ with $\alpha N^\prime =\beta N^\prime$ where $N^\prime \geqslant N/2$ consider  $\gamma = \alpha_0 \beta_0 \alpha_1 \beta_1 \dots$ and $\eta=\alpha_0 \alpha_0 \alpha_1 \alpha_1 \dots$. By that construction $\overline{\gamma}N = \overline{\eta}N$, so $h(\gamma) = h(\eta)$.
So \[ 0 = d(f(\alpha),f(\alpha)) = d(f(\eta^e),f(\eta^o)) = h(\eta) = h(\gamma) = d(f(\gamma^e),f(\gamma^o)) =  d(f(\alpha),f(\beta))  \ .\]
In other words $f(\alpha) = f(\beta)$. Hence $f$ is globally constant.
\end{proof}
\begin{Rmk}
Notice that for functions $f:\CS \to \NN$ the notion of (uniformly) locally constant coincides with the one for (uniformly) point-wise continuous. For such functions the statement that ``every locally constant function is uniformly locally constant'' is therefore equivalent to \FANc.
\end{Rmk}

Maybe the most interesting equivalence is the following, involving a notion well known from the Arzela-Ascoli-Theorem in analysis. A proof can be found in \cite{hD08}.
\begin{Pro}
\FANP is equivalent to the statement that every equi-continous sequence of functions $[0,1] \to \RR$ is uniformly equi-continuous.
\end{Pro}
It seems feasible to also prove this theorem for functions functions of type $\CS \to \RR$ and for functions $f:X \to Y$, where $X$ is an arbitrary compact and $Y$ is an arbitrary metric space.
This is a common theme throughout a lot of the results in this chapter; namely that $\CS$ and $[0,1]$ are often interchangeable without changing the classification. However, with the anti-Specker principle we have to be a bit careful. Here the space that $\CS$ or $[0,1]$ are embedded in makes a difference, as was shown in \cite{hD09}.
\begin{Pro} \label{Pro:FANP-equiv-ASXY}
 	The following are equivalent to \FANP.
	\begin{enumerate}
		\item \AS{[0,1]}{\RR^{2}}
		\item \AS{X}{Y} for any compact subspace $X$ of a metric space $Y$
	\end{enumerate}
\end{Pro}

\section{\texorpdfstring{\FANf}{FANfull}}
There is, maybe surprisingly, hardly any equivalence to the full, unrestricted, fan theorem. The only one, which is almost a kind of restatement, is that $\CS$ is cover-compact.
\begin{Pro} The following are equivalent
\begin{enumerate}
	\item \label{EquiFanf1} \FANf 
	\item \label{EquiFanf2} Every compact (totally bounded and complete) metric space is cover compact.
	\item \label{EquiFanf3} $\CS$ is cover compact.
\end{enumerate}

\end{Pro}
\begin{proof}
Trivially \ref{EquiFanf2} implies \ref{EquiFanf3}. To see that \ref{EquiFanf3} implies \ref{EquiFanf1} let $B$ be an arbitrary bar. The sets \[ B_u =  \set{\alpha \in \CS}{\overline{\alpha}\abs{u} = u} \] form an open cover of $\CS$. If this cover can be refined to a finite cover $B_{u_1}, \dots , B_{u_n}$, then $\max \menge{u_1, \dots,u_n}$ is a uniform bound for $B$. 
Finally, assume \ref{EquiFanf1} and let $X$ be totally bounded and complete. We can find a continuous $g:\CS \to X$ \cite{Troelstra1988}*{Chapter 7 Corollary 4.4}.
Now if $(U_i)_{i \in I}$ is an open cover of $X$ let
\[ B= \set{u \in \CS}{\ex{i \in I}{\fa{\beta \in \CS}{u \ast \beta \in g^{-1}(U_i)}}} \ . \]
For every $\alpha \in \CS$ there exists $i \in I$ such that $g(\alpha) \in U_i$, or equivalently $\alpha \in g^{-1}(U_i)$. Since $g^{-1}(U_i)$ is open there exists $n$ such that $\overline{\alpha}n \ast \beta \in g^{-1}(U_i)$ for all $\beta \in \CS$. So $B$ is a bar. Applying \FANf yields $n$ such that $u \in B$ for all $u \in 2^n$. That means we can find finitely many $i_u$ such that $(g^{-1}(U_{i_u}))_{u \in 2^n}$ covers $\CS$. But since $g$ is surjective this means that $(U_{i_u})_{u \in 2^n}$  cover $X$. 
\end{proof}
Given the similarities between them we might be tempted to try and replace  $\CS$ with $[0,1]$ in the above proposition. However, the work in \cite{iM84} shows that $[0,1]$ being cover compact is strictly weaker than \FANf. This is a very subtle distinction since in the case of a countable cover with basic opens the cover compactness of $[0,1]$ is  equivalent to the one of $\CS$ (and both are equivalent to \FAND \cite{hD08b}).

\begin{landscape}

\section{Comparing the Fan Theorems}
The differences between the fan theorems are overall very subtle and often confusingly minute. Since many of the results in this section are variations of each other, we hope that the following table might highlight some of the differences. 
\newlength{\tabwid}
\setlength{\tabwid}{4.13cm}
\begin{center}\hbadness=10000
{\rowcolors{1}{}{mylightgray}
\begin{tabular}{p{2.1cm}|p{\tabwid}|p{\tabwid}|p{\tabwid}|p{\tabwid}} \toprule
 & \FAND & \FANc &  \UCT & \FANP{} \\ \midrule
Uniform continuity & Every pwc.\ fully located $f:[0,1] \to \RR$ is uc.\ \newline Every $f:[0,1] \to \RR$ with a pwc.\ modulus of cont.\ is uc. & Every pwc.\ $f:\CS \to \NN$ is uc.\ \newline  Every pwc.\ $f:[0,1] \to \RR$ is usc.\ & Every pwc.\ $f:[0,1] \to \RR$ is uc.\ \newline Every pwc.\ $f:\CS \to \RR$ is uc.\ &\\ \midrule
Anti-Specker & \AS{[0,1]}{\RR} for tail-located sequences & \AS{[0,1]}{\RR},   \AS{[0,1]}{1},  \AS{\CS}{1}, & & \AS{[0,1]}{\RR^2} \\ \midrule
Boundedness statements & \POS & Every pwc.\ $f:\CS \to \NN$ is bounded. & Every pwc.\ $f:\CS \to \RR$ or $f:[0,1] \to \RR$ is bounded. & \\ \midrule 
Equicontinuity & & Every ec.\ $f_n:[0,1] \to \RR$ is usec.\ & & Every ec.\ $f_n:[0,1] \to \RR$ is uec.\ \\ \midrule constancy & & Every locally const.\ $f:\CS \to \NN$ is globally const.\ &  & Every locally const.\ $f:\CS \to \RR$ is globally const.\ \\ \bottomrule
\end{tabular}}
\end{center}
The following abbreviations are used:
\begin{multicols}{2}
\begin{description}
  \item pwc: point-wise continuous
  \item uc: uniformly continuous 
  \item usc: uniformly sequentially continuous
  \item ec: equi-continuous (i.e.\ point-wise equi-continuous)
  \item uec: uniformly equi-continuous
  \item usec: uniformly sequentially equi-continuous
\end{description}
\end{multicols}
\noindent We can recognise the following themes.
\begin{itemize}
  \item \FANc is involved whenever sequential continuity or sequences of real numbers are involved.
  \item \FANP is involved in sequences of functions.
  \item Introducing locatedness assumptions reduces principles to \FAND.
\end{itemize}

\end{landscape}

\chapter{\texorpdfstring{\BDN}{BDN} and Below} \label{Ch:BDN}
\todo{Future work: include \cite{hI08}}
Together with  Weak Markov's principle (Section \ref{Sec:WMP}), \BDN occupies a special and rare place in constructive mathematics: it is accepted in \CLASS, in \INT, as well as in \RUSS. However, it is not accepted in \BISH. 
One might think that after reading Proposition \ref{Pro:BDN-equivs}, which says that \BDN is equivalent to every sequentially continuous function on a separable metric space being point-wise continuous, one  fully understands this principle. However \BDN does appear in some quite unexpected places, and has an entire zoo of weaker principles below it.

\section{\texorpdfstring{\BDN}{BDN}} 
A subset S of $\NN$ is \define{pseudobounded} if 
\[\lim_{n \to \infty} \frac{s_n}n = 0\]
for each sequence $(s_n)_{n \geqslant 1}$ in $S$. This is equivalent to assuming that for every sequence $(s_n)_{n \geqslant 1}$ in $S$ we have
\[ s_n < n\] eventually.

 Trivially, every bounded set is pseudobounded, but the converse is more subtle and the content of our next principle.
\begin{principle}[BDN]{\BDN} \label{PR:BDN}
Every inhabited, countable, pseudobounded subset of $\NN$  is bounded.
\end{principle}
One can readily show that when dealing with a decidable set one can  establish the consequence:
\begin{proof}
  Let $S$ be a inhabited, decidable and pseudobounded subset of $\NN$. Choose $s_0 \in S$ and
  define a sequence  by \[s_{n+1} = \begin{cases} n+1, & n+1 \in S; \\ s_0, & n+1 \notin S \ . \end{cases} \] As $S$ is pseudobounded there is $M$ such that $s_i < i$ for all $i \geqslant M$. We claim that $S \subset \menge{1 , \dots, M} $. For assume that there is a $k \in S$ such
  that $k > M $. Then $k = s_k < k$ a contradiction.
\end{proof}
F.\ Richman has done a wonderful job of showing that there are many alternative equivalent notions for pseudobounded \cite{fR09}. The most surprising one may be that a subset $A$ of $\NN$ is pseudobounded if and only if every nonempty subset of $A$ that is detachable from $\NN$ is finite.

\begin{Rmk*}
	The inhabitedness assumption in \BDN can be disregarded without loss of generality, by considering the set $A \cup \menge{0}$ instead of $A$.
\end{Rmk*}

The reason for H.\ Ishihara to investigate \BDN was the question whether constructively there is a difference between sequential and point-wise continuity. The following result is a slight improvement of his original result \cite{Ishihara1992}, since there he showed that for every pseudobounded and countable set there exists a separable space and a sequentially continuous function such that if that function is point-wise continuous, then the set is bounded. Our construction $\BDN \iff \ref{BDNequiv3}$ is not only, arguably, simpler it also shows that it suffices to concentrate on Baire space.\footnote{Since every separable metric space is a continuous image  of Baire space one could probably get the same result by lifting the original equivalence. However, that would involve not only showing that every separable metric space is the continuous image of Baire space \cite{Troelstra1988}, but also that this embedding has some sort of openness property.}
\begin{Pro} The following are equivalent to \BDN. \label{Pro:BDN-equivs}
\begin{enumerate}
\item  \label{BDNequiv2} Every sequentially continuous map on a separable metric space is point-wise continuous.
\item   \label{BDNequiv3} Every sequentially continuous map $f:\BS \to \NN$ is point-wise continuous.

\item   \label{BDNequiv3.1} Every sequentially continuous map $f:\BS \to \NN$ is locally bounded.
\end{enumerate}
\end{Pro}
\begin{proof}
$\BDN \iff \ref{BDNequiv2}$ is shown in \cite{Ishihara1992}. Moreover, it is clear that \ref{BDNequiv2} $\implies$ \ref{BDNequiv3}, and that \ref{BDNequiv3} $\implies$ \ref{BDNequiv3.1}, so it remains to show that \ref{BDNequiv3.1} $\implies$ \BDN. Let $A =\menge{a_n}$ be a countable, pseudobounded subset of the natural numbers. We may assume that $a_n$ is non-decreasing. Notice that for all $\alpha \in \BS$ the sequence $\widehat{\alpha}$ defined by $\widehat{\alpha}(n) = \sum_{i=1}^n\alpha(i)+1$ is strictly increasing. In particular $\widehat{0} = \id$. Conversely, for every strictly increasing $\beta \in \BS$ there exists $\gamma$ with $\widehat{\gamma} = \beta$, but we will not use this fact. \\
Since $A$ is pseudobounded, the function $f: \BS \to \NN$ defined by 
\[ f(\alpha) = \max \set{n \in \NN}{a_{\widehat{\alpha}(n)} > n} \] 
is well-defined. It is also easy to see that $f$ is strongly extensional and that \LPO implies that it is sequentially continuous. Hence, using a corollary to Ishihara's tricks \cite{hD12b}*{Corollary 3.4.}, we get that $f$ is sequentially continuous. (Alternatively, sequential continuity also follows from the uniform sequential continuity of $f$ as proved directly below.) Our assumptions then imply that $f$ is locally bounded. So there exist $N, K $  such that, if $\overline{\alpha}N = 0$, then $f(\alpha) \leqslant K$. That means that if we take $M = \max \menge{N ,K}$ we have that if $\overline{\alpha}M = 0$, then $f(\alpha) \leqslant M$.  \\
We will show that there cannot be $i > M$ with $a_{i} > M +1$. If there were we could consider $\beta \in \CS$ defined by
\[  \beta = \underbrace{0 \dots 0}_{M} (i-M) 0 \dots \ ,\] so that 
 $\overline{\beta}M = 0$ and $\beta(M+1) = i-M$. This way $\widehat{\beta}(M+1) = i$, which means that $a_{\widehat{\beta}(M+1) } = a_i > M+1$ and therefore $f(\beta) \geqslant M+1$.  But this gives a contradiction, since $\overline{\beta}M = 0$, which means that $f(\beta) \leqslant M$, by our choice of $M$. 
That means that $\max \menge{a_1, \dots, a_M, M+1}$ is an upper bound of $A$ and therefore \BDN holds.
\end{proof}

\begin{Pro} The following are equivalent to \BDN. \label{Pro:BDN-unif-equivs}
\begin{enumerate}
\item \label{BDNequiv4} Every uniformly sequentially continuous mapping of a separable metric space into a metric space is uniformly continuous.
\item   \label{BDNequiv4.1} Every uniformly sequentially continuous map $f:\BS \to \NN$ is uniformly continuous.
\item   \label{BDNequiv4.2} Every uniformly sequentially continuous map $f:\BS \to \NN$ is locally bounded.
\end{enumerate}
\end{Pro}
\begin{proof}
In \cite{dB05c} it is shown that \BDN is equivalent to \ref{BDNequiv4}. 
	Clearly $\ref{BDNequiv4} \implies \ref{BDNequiv4.1}$, and $\ref{BDNequiv4.1} \implies \ref{BDNequiv4.2}$, so it remains to show that the latter implies \BDN. To this end let $A =\menge{a_n}$ be a countable, pseudobounded subset of the natural numbers,  and define $f$ in the same way as in the proof of Proposition \ref{Pro:BDN-equivs}. We want to show that $f$ is uniformly sequentially continuous. Let $\alpha_n$ and $\beta_n$ be sequences such that $d(\alpha_n, \beta_n) \to 0$ as $n \to \infty$. More specificly, with the help of countable choice, we can choose a modulus $\mu: \NN \to \NN$ such that $\overline{\alpha_i}n = \overline{\beta_i}n$ for all $i \geqslant \mu(n)$.
Now, for every $n$ we can decide whether  there is an $i$ with
\[ \mu(n) \leqslant i < \mu(n+1) \land f(\alpha_i) \neq f(\beta_i)  \]
or not. In the first case there must be $k_n$ such that $a_{k_n} > n$: for assume there is such $i$ with $f(\alpha_i) \neq f(\beta_i)$ such that $\overline{\alpha_i}n = \overline{\beta_i}n$. Then either $f(\alpha_i) > n$ or $f(\beta_i)>n$, since if both were $\leqslant n$ we would have $f(\alpha_i) \neq f(\beta_i)$. But that means that there is $\ell > n$ such that either 
\[ a_{\widehat{\alpha_i}(\ell)} > \ell > n \text{ or } a_{\widehat{\beta_i}(\ell)} > \ell > n \ . \] In the first case let $k_n = \widehat{\alpha_i}(\ell)$, in the second case let $k_n = \widehat{\beta_i}(\ell)$.

Now define
\[ \gamma_n = \begin{cases}
	k_n & \text{in the first case} \\
	n & \text{otherwise}.
\end{cases}  \]
Since $A$ is pseudobounded there exists $N$ such that for all $n \geqslant N$ we have $a_{\gamma(n)} \leqslant n$. But that means that for all $n \geqslant N$ we cannot be in the first case, which means that we must have $f(\alpha_n) = f(\beta_n)$. So $f$ is uniformly sequentially continuous. 

That means that if can apply \ref{BDNequiv4.2} to get that $f$ is locally bounded, and continue as in the proof of Proposition \ref{Pro:BDN-equivs}.
\end{proof}

We can also show that $\ref{BDNequiv3.1}$ of Proposition \ref{Pro:BDN-equivs} implies $\ref{BDNequiv4.1}$ of Proposition \ref{Pro:BDN-unif-equivs} directly without referring to \BDN, by using the following lemma. This might seem like a fairly obscure point to make, however the construction itself is quite interesting and the result also works for $\CS$.

\begin{Lem} \label{Lem:seq_unif_cont}
	If $f:\BS \to \NN$ (or $\CS \to \NN$) is uniformly sequentially continuous, then there exists a $g:\BS \to \NN$ (respectively $\CS \to \NN$) such that \begin{itemize}
  \item $g$ is sequentially continuous, and
  \item if $g$ is locally bounded at $\zero$, then $f$ is uniformly continuous.
\end{itemize}
\end{Lem}
\begin{proof}
We will only treat the case  $\BS$; the case of $\CS$ can be treated analogously. 
We first need to introduce some notation. We assume that we have fixed a bijection $\varphi : \NN^2 \to \NN$.
Given $\gamma \in \BS$, we can slice up $\gamma$ into countably many sequences and turn it into a double sequence $(\curlywedgedownarrow \gamma) :  \NN \times \NN \to \NN$
 defined by \[ (\curlywedgedownarrow \gamma) (m,n) = \gamma(\varphi(m,n)) \ . \]
We can also zip a double indexed sequence  $\sigma :\NN \times \NN \to \NN $ up into a single indexed sequence $(\curlyveedownarrow \sigma) \in \BS$ by setting
\[ (\curlyveedownarrow \sigma) (n) = \sigma(\varphi^{-1}_1(n),(\varphi^{-1}_2(n)) \ .  \]
These operations are inverses, so we have
\[ \curlywedgedownarrow \curlyveedownarrow \sigma = \sigma \text{ and }  \curlyveedownarrow  \curlywedgedownarrow \gamma = \gamma \ . \]

Again, consider an arbitrary $\gamma \in \BS$, and define two sequences of sequences $\alpha^{(n)}$ and $\beta^{(n)}$ by 
\[ \alpha^{(n)} = \curlywedgedownarrow \gamma (2n, \cdot), \quad \text{and} \quad \beta^{(n)} = \overline{\curlywedgedownarrow \gamma (2n, \cdot)}(n) \ast \curlywedgedownarrow \gamma (2n+1, n+1) \ast \gamma (2n+1, n+2) \ast \dots \ .\]
These sequences are constructed such that 
\[ \overline{\alpha^{(n)}}(n) = \overline{\beta^{(n)}}(n) \ , \]
that means $d(\alpha^{(n)},\beta^{(n)}) \to 0$. Since $f$ is uniformly sequentially continuous $f(\alpha^{(n)}) = f(\beta^{(n)})$ eventually. That means the number 
\begin{equation}
	\abs*{\set{n \in \NN}{f(\alpha^{(n)}) \neq f(\beta^{(n)}) }}
\end{equation} 
exists. So we can define a function $g: \BS \to \NN$ that maps $\gamma$ to that number. It is easy to see that $g$ is strongly extensional and that it is point-wise continuous provided \LPO holds. Thus, using a corollary to Ishihara's tricks \cites{hD12b,hI91} we can conclude that $g$ is sequentially continuous. 

Assume that $g$ is locally bounded at $\zero$. So there exists $N,K$ such that for all $\alpha \in \BS$ we have $g(0^N \ast \alpha) < K$.
Since $\varphi$ is bijective there exists $M$ such 
\[   \set{ \varphi^{-1}_1(i) }{ i \leqslant N}   \ \]
is bounded by $2M$. 

We claim that $L = M+K$ is a modulus of constancy for $f$. Let $\mu, \eta$ such that $\overline{\mu}(L) = \overline{\eta}(L) $, but $f(\mu) \neq f(\eta)$. Now consider the following sequence of sequences:
\[ \underbrace{\zero,\zero, \dots, \zero}_{2M}, \underbrace{\mu, \eta, \mu, \eta, \dots, \mu, \eta}_{2K}, \zero, \zero, \dots   \ .\]
If we combine these with $\curlyveedownarrow$ into a sequence $\vartheta$  we have the following properties:
\begin{enumerate}
  \item $\overline{\vartheta}(N) = 0 \dots 0$, and
  \item $\curlywedgedownarrow \vartheta (2i, \cdot ) = \mu$ for $ M < i \leqslant L $
  \item $\curlywedgedownarrow \vartheta (2i+1, \cdot ) = \nu$ for $ M < i \leqslant L $.
\end{enumerate} 
This is set up such that $g(\vartheta) = K$, which is a contradiction to $K$ being a local bound at $\zero$. That means that for any $\mu, \eta \in \BS$ such that $\overline{\mu}(L) = \overline{\eta}(L) $, we have $f(\mu) = f(\eta)$.
\end{proof}

From \cite{dB98}
\begin{Pro}
The following are equivalent to \BDN
\begin{enumerate}
\item If $T$ is a nonzero bounded linear mapping of a separable Hilbert space $H$ into itself such that $T^{\ast}$ exists and $\mathrm{ran}(T)$ is complete, then $T$ is open.
\item Every one-one self-adjoint sequentially continuous linear mapping from a Hilbert space onto itself is bounded.
\end{enumerate}
\end{Pro}
In \cite{dB11} a weakened form of the usual Cauchy condition is considered. There a sequence $(x_n)_{n \geqslant 1}$ in a metric space $(X,d)$ is called \emph{almost Cauchy}, if for any strictly increasing $f,g: \NN \to \NN$ \[ d(x_{f(n)},x_{g(n)}) \to 0\] as $n \to \infty$. A closer analysis of this and equivalent conditions can be found in \cite{hD12c}. 
The main theorem proved in \cite{dB11} is the following.
\begin{Pro} \label{Pro:BDN-almost-Cauchy}
\BDN is equivalent to the statement that every almost Cauchy sequence in a semi-metric space is Cauchy.
\end{Pro}
Here a semi-metric space is a metric space not necessarily satisfying the triangle-inequality condition. Notice that the statement is not provable in \BISH for metric spaces (see \cite{hD12c} referring to \cite{bL11}).
Another equivalence of \BDN concerns a lesser known fact of analysis. It together with its proof can be found in \cite{dB10}
\begin{Pro}
\BDN is equivalent to the statement that for all conjugate exponents $p$ and $q$, if $a$ is any sequence of complex numbers such that $\sum_{n=1}^{\infty} a_nx_n$ converges for each $x \in l_{p}$, then $\sum_{n=1}^{\infty} \abs*{a_n}^{q}$ has bounded partial sums.
\end{Pro}
Lastly, we would like to mention that in \cite{hI02} it is shown that \BDN is equivalent to two important spaces being complete.

\section*{\textbf{\thesection ½} \BD}\addcontentsline{toc}{section}{\texorpdfstring{\thesection ½ \BD}{\thesection ½ BD}}
At the same time as investigating \BDN H.~Ishihara also introduced the stronger principle \BD which drops the assumption of the set being countable. This also explains the nomenclature: ``$\mathrm{BD}$'' is short for boundedness while the $\mathrm{N}$ in \BDN stands for countable.\footnote{H.~Ishihara has also mentioned that he prefers a simply typed $\mathrm{N}$ rather than a boldface $\NN$, in this principle's name.}
\begin{principle}[BD]{\BD} \label{PR:BD}
Every inhabited, pseudobounded subset of $\NN$  is bounded.
\end{principle}

We have the following counterpart to Propositions \ref{Pro:BDN-equivs} and \ref{Pro:BDN-unif-equivs}.

\begin{Pro} The following are equivalent to \BD. \label{Pro:BD-equivs}
\begin{enumerate}
\item  \label{BDequiv2} Every  map that is sequentially continuous at a point $x$  is point-wise continuous at $x$.\footnote{Note that in Proposition \ref{Pro:BDN-equivs} we need the sequential continuity not just at a point, but everywhere.}
\item \label{BDequiv4} Every uniformly sequentially continuous mapping of a metric space into a metric space is uniformly continuous.
\item \label{BDequiv5} Every uniformly sequentially continuous mapping of a metric space into a metric space is point-wise continuous.
\end{enumerate}
\end{Pro}
\begin{proof}
The equivalence of \BD to \ref{BDequiv2} has already been shown in \cite{Ishihara1992}, where \BD was first introduced. The equivalence of \BD to \ref{BDequiv4} and \ref{BDequiv5} is from \cite{dB05c}.
\end{proof}

The strength of \BD has, to our knowledge, not been considered, apart from Lubarsky's topological model \cite{rL12}, in which \BD fails.

It is clear, that \BD holds classically. 
One can make the almost trivial observation that in the presence of Kripke's Scheme (see Section \ref{Sec:KS_and_PFP}) \BDN $\implies$ \BD (with the help of Proposition \ref{Pro:KS_equiv_Ncount}). That means that \BD also holds in some interpretations of \INT, however, the exact state of \KS in \INT is somewhat unclear. There is also a less problematic way: In \cite{Ishihara1992} it is shown that \WCN implies \BDN, which means it holds in \INT, and we can use basically the same proof for \BD.

\todo{Future Work: link up with CC and so on}
\begin{principle}[WCN]{\WCN} \label{PR:WCN}
The principle \define{weak continuity for numbers} is
\begin{multline*}
 \fa{\alpha \in \CS}{\ex{i \in \NN}{A(\alpha,i)}} \\ \implies \left( \fa{\alpha \in \CS}{\ex{i,j \in \NN}{\fa{\beta \in \CS}{ A(\overline{\alpha}(i) \ast \beta,j)  }}}  \right)
 \end{multline*}
\end{principle}

\begin{Pro}
	\WCN implies \BD.
\end{Pro}
\begin{proof}
	Assume $S$ is an inhabited, pseudobounded subset of $\NN$. Without loss of generality, we may assume that $0 \in S$. Now define a predicate $A \subset \CS \times \NN$ by
	\[ A(\alpha, j) \equiv \left(\fa{n \in \NN}{\alpha_n \in S} \right) \implies \left(\fa{n \geqslant j}{\alpha_n \leqslant j} \right) \ .
	 \]
	 The pseudoboundedness assumption translates into $\fa{\alpha \in \CS}{\ex{i \in \NN}{A(\alpha,i)}}$.  So we can use \WCN to get $i,j$ for $\alpha = \zero$ such that $\fa{\beta \in \CS}{A(0^i\ast \beta, j)}$. We claim that $M = \max \menge{i,j}$ is an upper bound for $S$. For assume there is $N > M$ such that $N \in S$. Then we have $A(\gamma , j)$ for $\gamma = 0^M \ast N \ast N \ast \dots$. For that sequence $\gamma$ we have $\fa{n \in \NN}{\gamma_n \in S}$, so $\fa{n \geqslant j}{\gamma_n \leqslant j}$.  In particular this is the case for $n=N$, which means $j \leqslant M < N = \gamma_N  \leqslant j$; a contradiction. 
\end{proof}

\section{Below \texorpdfstring{\BDN}{BDN}} \label{Sec:belowBDN}

Between, approximately, 2007 and 2010 a couple of statements were considered by researchers working in \CRM for which a proof in \BISH could not be found, but that were all implied by \BDN. Naturally a considerable amount of time was spent trying to prove that they were in fact equivalent to \BDN. As it turned out around 2011 for most of them it could be shown that that could not be done, but that also they could not be proved within \BISH. These statements, and some additional ones, are the following:

\begin{Pro} \label{Pro:BDN-impl}
\BDN \emph{implies} the following statements.
\begin{enumerate}
\item \label{BDN-impl12} If $X$ is separable, then $\AS{X}{1}$ implies \AS{X}{Y} for any $Y \supset X$.
\item \label{BDN-impl1} $\FANc$ is equivalent to $\FANP$.
\item \label{BDN-impl1p} $\FANc$ is equivalent to $\UCT$.
\item \label{BDN-impl2} The anti-Specker property is closed under products in the sense that if $\AS{X}{1}$ and $\AS{Y}{1}$ holds, then also $\AS{X \times Y}{1}$ holds. 
\item \label{BDN-impl3} The anti-Specker property implies pseudo-compactness. That is, if $\AS{X}{1}$ and $X$ is separable, then any point-wise continuous function $f:X \to \RR$ is bounded.
\item \label{BDN-impl3.5} The anti-Specker property implies totally boundedness That is, if $\AS{X}{1}$ and $X$ is separable, then $X$ is totally bounded.
\item \label{BDN-impl4} A converse of the Riemann permutation theorem: if $(a_n)_{n \geqslant 1}$ is a sequence of real numbers such that for all $\sigma$ permutations the series
\[ \sum_{n = 1}^{\infty} a_{\sigma(n)} \] converges, then $(a_n)_{n \geqslant 1}$ converges absolutely.
\item \label{BDN-impl5}  Every almost Cauchy sequence in a metric space is Cauchy. 

\end{enumerate}
Furthermore, \ref{BDN-impl12} implies both \ref{BDN-impl1} and \ref{BDN-impl3}, \ref{BDN-impl1} implies \ref{BDN-impl1p}, and \ref{BDN-impl3} and \ref{BDN-impl3.5} are equivalent and they also imply \ref{BDN-impl1p}. Finally, \ref{BDN-impl5} implies \ref{BDN-impl4}. 
\end{Pro}
\begin{proof}
To see that \ref{BDN-impl12} follows from \BDN let $(x_n)_{n \geqslant 1}$ be a dense sequence in $X$; without loss of generality we can assume that every point is repeated infinitely often in the sequence $(x_n)_{n \geqslant 1}$. Furthermore, consider  a sequence $(z_n)_{n \geqslant 1}$ in $Y$ that is bounded away from every point in $X \subset Y$.
Now define the set
\[ A= \set{n \in \NN}{\ex{i,j \geqslant n}{d(z_i,x_j) < 2^{-n}}} \cup \menge{0} \ .\]
The set $A$ is easily seen to be countable; and, unsurprisingly, we are going to show that it is also pseudobounded. To this end let $a_n$ be a sequence in $A$ and with the use of countable choice define a sequence $(y_n)_{n \geqslant 1}$ in $X \cup \menge{\omega}$\footnote{We choose a metric on $X \cup \menge{\omega}$ that is equivalent to the one on $X$  and is such that $d(x,\omega) \geqslant  1$.} such that 
\begin{align*}
   a_n < n &\iff y_n = \omega  \ , \\
   a_n \geqslant n &\iff y_n = x_j, \ j \geqslant n, \ \ex{i \geqslant n}{d(z_i,x_j) < 2^{-n}} \ .
\end{align*}

The sequence $(y_n)_{n \geqslant 1}$ is eventually bounded away from every point in $x$: for let $x$ in $X$ be arbitrary. Since $(z_n)_{n \geqslant 1}$ is eventually bounded away from $x$ there exists $N$ such that 
\begin{equation} \label{Eqn:BDN-ASusw}
d(x,z_k) > 2^{-N} \text{ for all } k > N \
\end{equation}
 We claim that $d(x,y_k) > 2^{-(N+2)}$ for all $k > N$. For assume $d(x,y_k) < 2^{-(N+1)}$. Then we must have $a_k \geqslant k$, and therefore $y_k = x_j$ for some $j \geqslant k$ and $d(z_i,x_j) < 2^{-k}$ for some $i \geqslant k$. Hence
\[ d(x,z_i) \leqslant d(x,x_j) + d(x_j,z_i) = d(x,y_k) + d(x_j,z_i) < 2^{-(N+1)} + 2^{-k}  \leqslant 2^{-N} \ ,\]
but this is a contradiction to \ref{Eqn:BDN-ASusw}. Hence we have shown that the sequence $(y_n)_{n \geqslant 1}$ is eventually bounded away from $x$. Since we assumed \AS{X}{1} it is therefore bounded away from the entire set $X$, which means there must exist $N$ such that $y_n = \omega$ for all $n \geqslant N$ and therefore $a_n < n$ for all $n \geqslant N$. Hence $A$ is bounded, say by a number $M$. Now there cannot be an $x \in X$ and $i \geqslant M$ with $d(x,z_i) < 2^{-(M+2)}$, since otherwise by the density we can find $j \geqslant M$ with $d(x_j,x) < 2^{-(M+2)}$, which would imply that $d(x_i,x_j) < 2^{-(M+1)}$ and therefore $M+1 \in A$, which is a contradiction. So for all $x \in X$ and $i \geqslant M$ we have $d(x,z_i) \geqslant 2^{-(M+2)}$, which means that $(z_n)_{n \geqslant 1}$ is eventually bounded away from the set $X$.

That \ref{BDN-impl1} follows from \ref{BDN-impl12} is an immediate consequence of Propositions \ref{Pro:FANc-equivs} and \ref{Pro:FANP-equiv-ASXY}.

To see that \ref{BDN-impl12} implies \ref{BDN-impl3} let $f:X \to \RR$ be a point-wise continuous function and $(x_n)_{n \geqslant 1}$ a dense sequence in $X$. We may assume that $f(x) > 1$ for all $x \in X$. Hence $g(x)=1/f(x)$ is well-defined and also point-wise continuous. Furthermore we have that $g(x) > 0$ for all $x \in X$. Hence, together with the continuity, the sequence $((x_n,g(x_n)))_{n \geqslant 1}$ in $X \times \RR$ is eventually bounded away from every point in $X$ (embedded in $X \times \RR$ as $X \times \menge{0}$). Since we assume \AS{X}{1} and therefore by assumption also \AS{X}{Y}, this sequence is also eventually bounded away from the entire set $X \times \menge{0}$. So there exists $M$ such that $d((x,0),(x_i,g(x_i))) > 2^{-M}$ for all $i \geqslant M$ and $x \in X$. In particular, $\abs*{g(x_i)} > 2^{-M}$ for all $i \geqslant M$, which is equivalent to $f(x_i) < 2^{M}$. Since $(x_n)_{n \geqslant 1}$ is dense and $f$ continuous $f(x) \leqslant 2^{M}$ for all $x \in X$, and thus $f$ is bounded.

The equivalence of \BDN to the statements \ref{BDN-impl2} and \ref{BDN-impl3} were proved in \cite{dB10d} and \cite{hD09b} respectively.

To see that \ref{BDN-impl3}  and \ref{BDN-impl3.5} are equivalent notice, for one direction, that pseudo-compactness implies totally boundedness \cite{dB76}*{Theorem 2}. For the other direction notice that if $f:X \to \RR$ is point-wise continuous, then $f(X)$ is also separable and that $\AS{f(X)}{1}$ follows from $\AS{X}{1}$, whence $f(X)$ is totally bounded and therefore, in particular, bounded (also see\ \cite{hD09b}*{Theorem 11}).

To see that \ref{BDN-impl12} implies \ref{BDN-impl2} let $z_n$ be a sequence in $X \times Y \cup \menge{\omega}$, where $\AS{X}{1}$ and $\AS{Y}{1}$. Let $y_n$ be the sequence in $Y \cup \menge{\omega}$ that is, basically, the projection of $z_n$ onto $Y$. More precisely:
\[ y_n = \begin{cases} \pi_Y(z_n) & \text{if } z_n \in Y \\ \omega & \text{if } z_n = \omega \ . \end{cases}  \]
We want to show that $y_n$ is bounded away from every point in $Y$. So let $y \in Y$ be arbitrary. By \ref{BDN-impl12} we have \AS{X}{Z}, where we think of $X$ being embedded in $X \times Y \cup \menge{\omega}$ by $x \mapsto (y,x)$. In other words, we consider $X \times \menge{y} \subset X \times Y \cup \menge{\omega}$. Since $z_n$ is bounded away from every point in $ X \times \menge{y} $ there exists $\delta > 0$ and $N \in \NN$ such that $d(z_n, (x,y)) > \delta$ for all $x \in X$ and $n \geqslant N$. Hence $d_{Y\cup \menge{\omega}}(y_n, y) > \delta $ for all $n \geqslant N$. So $y_n$ is bounded away from every point in $Y$. By \AS{Y}{1} that means that there exists $K$ such that $y_n = \omega$ for all $n \geqslant K$, but that implies, by construction of $y_n$, that $z_n = \omega$ for all $n \geqslant K$.

The implications \ref{BDN-impl1} implies \ref{BDN-impl1p}, and \ref{BDN-impl3.5} implies \ref{BDN-impl1p} are trivial.

 Finally, the proof of the equivalence of \BDN and \ref{BDN-impl4} is sketched in \cite{hD12b} and that \BDN implies \ref{BDN-impl5} is simply Proposition \ref{Pro:BDN-almost-Cauchy} restricted to metric spaces and without the converse. The proof that \ref{BDN-impl5} implies \ref{BDN-impl4} can be found in \cite{hD18b}.
\end{proof}

Since in every topological model (Section \ref{Sec:topmodels}) \UCT holds (consequence of \cite{mF79}*{Theorem 3.2}) and there is a model in which \BDN fails \cite{rL12}, Statement \ref{BDN-impl1} cannot imply \BDN. In \cite{bL11b} it is shown that it cannot be proved in \BISH. That \ref{BDN-impl2}, \ref{BDN-impl4}, and \ref{BDN-impl5} are strength-wise also between \BDN and unadorned \BISH was shown in \cite{bL11}.  It is unknown, whether this also holds for \ref{BDN-impl12} and \ref{BDN-impl3}, but it seems likely.

Another natural principle which falls into the same category is
\begin{principle}[wBDN]{\wBDN} \label{PR:wBDN} Every sequentially continuous mapping $f:\CS \to \NN$ is point-wise continuous.
\end{principle}

\begin{Pro} \label{Pro:wBDN_equiv}
	\wBDN is equivalent to the statement that every sequentially continuous mapping $f:\CS \to \NN$ is locally bounded. \end{Pro}
\begin{proof}
One direction is clear. Conversely, consider $f:\CS \to \NN$ and let $\alpha \in \CS$. Notice that on complete spaces sequential continuity implies strong extensionality and so we can, for every $\beta \in \CS$, if $f(\beta) \neq f(\alpha)$ find the minimal $n$ such that $\alpha(n) \neq \beta(n)$. So define 
	\[ g(\beta) = \begin{cases} n & \text{where $f(\beta) \neq f(\alpha)$ and $n$ is as above} \\ 0 & \text{otherwise}. \end{cases} \]
	
Similar to proofs above, it is easy to see that $g$ is strongly extensional and that it is point-wise continuous provided \LPO holds. Thus, using a corollary to Ishihara's tricks \cites{hD12b,hI91} we can conclude that $g$ is sequentially continuous. 
	
 Now if $g$ is locally bounded around $\alpha$ there exists $N,K$ such that $\overline{\alpha}N = \overline{\beta}N$, then $g(\beta) \leqslant K$. If we use $M = \max{N,K}$ we get that if $\overline{\alpha}M = \overline{\beta}M$, then $g(\beta) \leqslant M$. But that means that there cannot be any $\beta$ such that $\overline{\alpha}M = \overline{\beta}M$ such that $f(\beta) \neq f(\beta)$, since in that case $g(\beta) > M$. Thus $f$ is point-wise continuous.

\end{proof}
\begin{Rmk} \label{Rmk:wBDN}
	It does not seem to be possible to extend the above result analogously to Proposition \ref{Pro:BDN-equivs} and show that \wBDN is equivalent to sequentially continuous functions $\CS \to \RR$ being point-wise continuous. Moreover, it even seems impossible to get an equivalence  analogous to the one of Proposition \ref{Pro:BDN-unif-equivs}.
\end{Rmk}
\begin{Pro} \label{Pro:wwBDN}
	\wBDN implies that every uniformly sequentially continuous mapping $f:\CS \to \NN$ is uniformly continuous.
\end{Pro}
\begin{proof}
		 This follows easily from the next lemma, or alternatively, from Lemma \ref{Lem:seq_unif_cont}. 
\end{proof}
\begin{Lem} \label{Lem:useqc_pwc_impl_uc}
	If $f$ is a function  $\CS \to \NN$, and $f$ is both uniformly sequentially continuous and point-wise continuous, then $f$ is uniformly continuous.
\end{Lem}
\begin{proof}
	Let $u_n$ be an enumeration of all finite binary sequences.
	 Then the set of all sequences $\alpha_n = u_n \ast \zero $ is dense in $\CS$. We will also consider the related sequences $\beta_n = u_n \ast 1 \ast \zero$. Since for all $m$ there exists $N$ such that $\abs*{u_n} \geqslant m$ for all $n \geqslant N$ (a fact which holds for any enumeration, but in particular for the obvious one enumerating the sequences by length and then lexicographically), we have that $d(\alpha_n, \beta_n) \to 0$. Since $f$ is uniformly sequentially continuous we also have that $d(f(\alpha_n),f(\beta_n)) \to 0$, which means that there exists $K$ such that $f(\alpha_n) = f(\beta_n)$ for $n \geqslant K$. Now choose $L = \max\menge{ \abs*{u_1}, \dots , \abs*{u_K}}$.
	 We claim that if $\alpha \in \CS$ arbitrary, then $f(\overline{\alpha}L \ast 0\ast \dots ) = f(\overline{\alpha}L \ast \beta ) $ for all $\beta \in \CS$; which means that $f$ is uniformly continuous. For assume there exists $\alpha, \beta \in \CS$ such that $f(\overline{\alpha}L \ast 0\ast \dots ) \neq f(\overline{\alpha}L \ast \beta ) $. 
	 Since $f$ is point-wise continuous there exists $M$ such that \[ f(\overline{\overline{\alpha}L \ast \beta} M \ast \zero ) = f(\overline{\alpha}L \ast \beta )  \ .\]
	 We must have $M > L$, so for $M^\prime = M - L$ we have 
	 \[ f(\overline{\alpha}L \ast \overline{\beta} M^\prime \ast \zero ) = f(\overline{\alpha}L \ast \beta ) \neq f(\overline{\alpha}L \ast \zero ) \ .\]
By checking all finitely many prefixes of $\overline{\beta} M^\prime$ we can find $w \in \cS$ such that 
	 \[ f(\overline{\alpha}L \ast w \ast \zero) \neq f(\overline{\alpha}L \ast w \ast 1 \ast \zero)  \ . \]
	 Let $ k $ be such that $u_k = \overline{\alpha}L \ast w$. Since $\abs*{u_k} \geqslant  L$ we must have $k \geqslant K$ and therefore
	 $f(\alpha_k)  = f(\beta_k)$. But this is a contradiction, since
	 \[ f(\alpha_k) = f(\overline{\alpha}L \ast w \ast \zero)  \neq f(\overline{\alpha}L \ast w \ast 1 \ast \zero) = f(\beta_k) \ . \]
	 
Altogether $L$ is a modulus of constancy and therefore for uniform continuity of $f$.
\end{proof}

It seems worth noting that the above proof does not work for functions defined on $\BS$.

It is unknown if \wBDN is equivalent to, or implies, or is implied by any of the statements of Proposition \ref{Pro:BDN-impl}.
\begin{Qu}
Are there any other equivalences to \wBDN? Can the issues outlined in Remark \ref{Rmk:wBDN} be resolved?	
\end{Qu}

We finish this section with a diagram summing up all these statements and their known interactions. Here a double headed arrow means that the implication is strict. A double arrow ending in $\top$ indicates non-provability in \BISH. The numbers are the ones from Proposition \ref{Pro:BDN-impl}.
\begin{center}
\tikzset{external/export next=true}
\begin{tikzpicture}[>=stealth,shorten >=1pt, node distance=1.5cm]
    \node (BDN)  {\BDN};
    \node[below left of=BDN] (C12) {\ref{BDN-impl12}};
    \node[below left of=C12] (C1) {\ref{BDN-impl1}};
    \node[below of=C1] (C1p) {\ref{BDN-impl1p}};

    \node[below of=C12] (C3) {\ref{BDN-impl3}, \ref{BDN-impl3.5}};
    \node[below of=BDN] (C2) {\ref{BDN-impl2}};
    \node[right of=C2] (C4) {\ref{BDN-impl5}};
    \node[below of=C4] (C5) {\ref{BDN-impl4}};
    \node[right of=C4] (C6) {\wBDN};
    \node[below = 3cm of C2] (BISH) {$\top$};
    \draw[->] (BDN) -- (C12);
    \draw[->] (C12) -- (C1);
    \draw[->] (C12) -- (C3);
    \draw[->] (C12) -- (C2);

    \draw[->>] (C1p) -- (BISH);
    \draw[->] (C1) -- (C1p);
    \draw[->] (C3) -- (C1p);

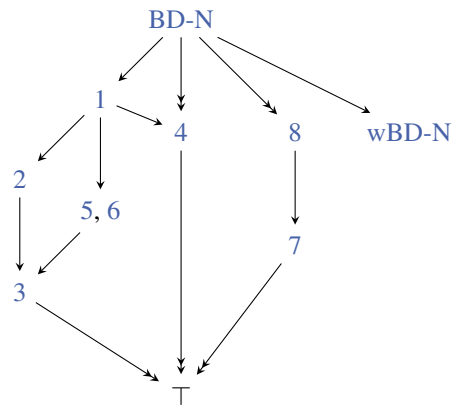
    \draw[->>] (BDN) -- (C2);
    \draw[->>] (C2) -- (BISH);

    \draw[->>] (BDN) -- (C4);
    \draw[->] (C4) -- (C5);
    \draw[->] (BDN) -- (C6);

    \draw[->>] (C5) -- (BISH);
\end{tikzpicture}
\end{center}

\chapter{The recursive side} 
\label{Ch:recside}
\section{Introduction} \label{Sec:rec-Intro}
Let us assume that $K \subset \cS$ is a decidable tree (i.e.\ it is closed under restriction) that does not admit infinite paths; that is
\[ \fa{\alpha \in \CS}{\ex{n \in \NN}{\overline{\alpha}n \notin K}} \ .\]
Then the complement  $B \subset \cS$ has the following properties:
\begin{itemize}
\item $B$ is decidable, since $K$ is decidable,
\item $B$ is a bar, since $K$ does not admit infinite paths, and
\item $B$ is closed under extension, since $K$ is a tree and therefore closed under restriction.
\end{itemize}
In other words $B$ is a decidable bar that is closed under restriction.
Conversely, the complement of such a bar is a decidable tree that  does not admit infinite paths.\footnote{Note that the complement of a decidable set is again decidable.} Now in \RUSS (and other recursive varieties of constructive mathematics) there exists a Kleene tree, which is a tree having the above properties that is also infinite as a set. The latter condition is actually equivalent to  
\[ \fa{n\in \NN}{\ex{u \in T}{\abs*{u} \geqslant n}} \ .\] With the above observation we  see that the complement of a Kleene tree fails to be a uniform bar and therefore provides  a  counterexample to \FAND. This gives rise to the idea of considering  ``Anti-Fan''-principles.
\begin{principle}[Anti-FAN]{Anti-\FAN$\square$}
There exists a $\square$-bar $B$ such that there is a sequence $(u_n)_{n \geqslant 1}$ with 
\begin{itemize}
\item $\abs*{u_n} \geqslant n$ and 
\item $u_n \notin B$.
\end{itemize}
\end{principle}
As already pointed out in \cite{hD08b}*{Proposition 4.5.2.} Anti-\FANf~is actually equivalent to Anti-\FANc: for if $B$ is an arbitrary bar and $(u_n)_{n \geqslant 1}$ as above, then the complement $B^{\prime}$ of $\menge{u_{1}, u_{2}, \dots}$ is actually a $c$-bar and a superset of $B$. Furthermore, by construction, $u_n \notin B^{\prime}$. Therefore there are actually only two Anti-Fan principles to consider: Anti-\FAND and Anti-\FANc. 

In a similar spirit we can also define \AWWKL for $k \in (0,1)$: 
\begin{principle}[Anti-WWKL]{\AWWKL($k$)}
There exists a decidable bar $B$ that is closed under extension such that 
\[
 \fa{n\in \NN}{\frac{\abs*{\set{u \notin B}{\abs{u} = n }}}{2^n} > k} 
\]
\end{principle}

As we will see in the following section this principle is actually independent of the choice of $k$, and we will therefore simply refer to it as \AWWKL. 

So we have the following  hierarchy of recursive principles:

\[ \AWWKL (\SinC) \implies \AFAN_{\Delta} (\KT) \implies \AFAN_{c} (\SS) \ . \]

In the following sections we are going to show that most recursive counterexamples in analysis are equivalent to one of these three, which will also lead to better names (already bracketed).

Although we are going to cite details later on we would like to point out that W.\ Veldmann has considered equivalences to the existence of a Kleene tree \cite{wV11}.

\section{Singular Covers}

By  slightly extending, in an obvious way,  the observation at the start of Section \ref{Sec:rec-Intro} we can see that  Anti-\WWKL\!$(k)$ is equivalent to
\begin{principle}[KTs(k)]{\KTs{k}}
There exists a decidable binary tree $T$ such that 
\begin{equation} \label{Eqn:majority}
 \fa{n\in \NN}{\frac{\abs*{\set{u \in T}{\abs{u} = n }}}{2^n} > k}  \ , 
\end{equation}
and
\begin{equation} \label{Eqn:treeblock} \fa{\alpha \in \CS}{\ex{n \in \NN}{\overline{\alpha}n \notin T} } \ . 
\end{equation}
\end{principle}

\begin{Lem}
\[ \KTs{k} \implies \KTs{(2-k)k} \]
\end{Lem}
\begin{proof}
For the following we are going to introduce some notation for slicing sequences.\footnote{Basically be the Python programming language list notation.} If $u \in \CS$, $\abs*{u} \geqslant m \geqslant n$, then 
$u[n \co m]$ is the sequence $u(n)u(n+1) \cdots u(m)$. Furthermore, $ \pylist{u}n{} = \pylist{u}n{\abs*{u}}$ and $\pylist{u}{}n = \pylist{u}{1}n$. Finally, $\pylist{u}n{-m} = \pylist{u}n{\abs*{u} - m}$ and in particular $\pylist{u}{}{-1}$ is the sequence $u$ cut short by the last element. Naturally, apart from the last two, these notations also make sense for infinite sequences.  

Let $T$ be a decidable Kleene Tree with property \ref{Eqn:majority}. Define \[ K_n = \abs*{\set{u \in T}{\abs{u}=n}} \cdot 2^n \ . \] By \ref{Eqn:majority} we know that $K_n \geqslant k2^n $ for every $n \in \NN$. The set $B$ of all sequences $u \in \cS$ such that $u \notin T$, but $u[:-1] \in T$ is decidable. Notice that for every $u \in \cS$ either $u \in T$ or there exists at exactly one $n$ such that $u[\co n] \in B$, which means that
\[ S = T \cup \bigcup_{u \in B} \set{ u \ast v }{v \in T} \ . \]
 is a disjoint union. The idea behind this construction is to attach to every leaf of $T$ another copy of $T$.

As it contains $T$ this set $S$ is still an infinite tree. It also blocks every path: for let $\alpha \in \CS$ be arbitrary. Since $T$ blocks $\alpha$  there is $n_{1} \in \NN$ such that $\pylist{\alpha}{}{n_{1}}  \in B$. Since also $\pylist{\alpha}{n_{1}+1}{}  $ gets blocked there exists $n_{2}$ such that $\pylist{\alpha}{n_{1}+1}{n_{2}}  \in B$, which means that $\pylist{\alpha}{}{n_{2}}  \notin S$. 

Last, we need to count the nodes at a level $n$ in $S$.  There are finitely many $u_{1}, \dots, u_{m} \in B$ such that $\abs*{u_{i}} \leqslant n$ and for every $w \in 2^n$ either $w \in T$, or there exists a unique $i$ such that $u_{i} \sqsubseteq w$. For every $1 \leqslant i \leqslant m$ there are  $K_{n-\abs*{u_{i}}}$  Elements $w \in 2^n$ with $u_{i} \sqsubseteq w$ and $\pylist{w}{\abs*{u_{i}}+1}{}  \in T$, which means there $k_{n-\abs*{u_{i}}}$ Elements $w \in S$ with $u_{i} \sqsubseteq w$.  
Furthermore there are exactly $2^n-K_n$ sequences of length $n$ starting with a sequence from $B$.
Altogether:
\begin{align*}
\abs*{\set{u\in S}{\abs{u}=n}} & = \abs*{ \set{u \in 2^n}{u \in T}} + \sum_{i=1}^{m} \abs*{\set{w \in 2^n}{u_{i} \sqsubseteq w \land \pylist{w}{\abs*{u_{i}}+1}{}  \in T} }  \\
& = K_n + \sum_{i=1}^{m} K_{n-\abs*{u_{i}}} \\
& \geqslant K_n + k \sum_{i=1}^{m} 2^{n-\abs*{u_{i}}} \\
& = K_n + k (2^n-K_n) \\
& = (1-k)K_n+k2^n \\
& \geqslant (1-k) k2^n +k2^n = (2-k)k2^n
\end{align*}
But that is exactly the inequality we wanted.
\end{proof}
\begin{Cor}
$\KTs{k}$ is independent of the choice of $k >0$.
\end{Cor}
\begin{proof}
For any $k \in (0,1)$ the sequence defined by $x_{0}=k$, and $x_n=(2-k)k$ converges to $1$. Therefore for $y_n = 1- x_n$ we have $y_{n+1} = (y_n)^{2}$. That shows that $y_n$ converges to $0$.
\end{proof}
We can actually view a tree $T$ that does not admit infinite paths  as an open cover of Cantor space. If we collect all of the finite sequences $w_n$ that are just barely not in $T$, i.e.\ $w_n \notin T$ but $\overline{w_n}(\abs*{w_n} -1) \in T$, then these give us basic open sets $U_n = U_{w_n}$ where  $U_{w} = \set{\alpha \in \CS}{\overline{\alpha}\abs*{w} = w }$. These form a cover of Baire space because  for every $\alpha \in \CS$ we can find $n$ such that $w_n$ is a prefix of $\alpha$. Moreover, this cover has a small measure in the sense that 
\[ \sum_{i=1}^n \mu(U_n) < k \ , \] 
where $\mu(U_{w}) = 2^{-\abs*{w}}$.\footnote{We are not getting into deep waters of constructive measure theory here and are not saying that $\mu$ is a measure in the traditional sense.} This seems paradoxical, since $\mu(\CS) = 1$. However, such a strange topological behaviour is well-known in \RUSS \cite{Beeson1985}*{Theorem 6.1}: an $\alpha$-\define{singular cover}  is a sequence of intervals $(J_n)_{n \geqslant 1}$ with rational endpoints such that 
\begin{enumerate}[label=SC\arabic*]
\item \label{SC1} $\sum_{i=1}^n \abs*{J_n} < \alpha$ for all $n$,
\item any two $J_n$ are disjoint or have only an endpoint in common, 
\item \label{SC3}  for any $x \in [0,1]$ there exists $n,m \in \NN$ such that the intervals $J_n=[a_n,b_n]$ and $J_{m}=[a_{m},b_{m}]$ are such that $a_{m}=b_n$ and $x \in [a_n,b_{m}]$.
\end{enumerate}

\begin{Rmk}
Any  cover satisfying \ref{SC1} cannot be refined to a finite one.\footnote{Regardless of whether by cover we mean open cover or one in the sense of \ref{SC3}.}
\end{Rmk}

After all this build-up the following proposition is no surprise.
\begin{Pro}
The following are equivalent
\begin{enumerate}
\item \label{Equiv:KTs1} \KTs{k} for any $0\leqslant k<1$
\item \label{Equiv:KTs2} For  $0 < \alpha < 1$ there exists a sequence of open intervals $(J_n)_{n \geqslant 1}$ covering $[0,1]$ and satisfying \ref{SC1}
\item \label{Equiv:KTs3} The existence of an $\alpha$-singular cover of $[0,1]$ for any $0 < \alpha < 1$
\end{enumerate}
\end{Pro}
\begin{proof}
Let $1 > \alpha > 0$ be arbitrary. Now let $T$ be a tree as in \KTs{1-\alpha/2}, and let $(w_n)_{n \geqslant 1}$ be an enumeration of all the finite sequences such that $w_n \notin T$, but $\overline{w_n}(\abs*{w_n-1}) \in T$. 

Define, similarly as in Section \ref{Sec:linking-cantor-and-unitint} inductively for every $u \in \cS$ intervals $I_u$ such that 

\begin{itemize}
  \item $I_u = I_{u0} \cup I_{u1}$
  \item $\abs*{I_u} = \frac{1}{2^{\abs*{u}}} + \frac{\alpha}{2^{2\abs*{u}+1}}$
\end{itemize}

We claim that $J_n = I_{w_n}$ is the desired sequence of open intervals.

So let $x \in [0,1]$ be arbitrary. Using dependent choice we can construct a sequence $\alpha \in \CS$ such that $x \in I_{\overline{\alpha}n}$ for all $n \in \NN$. Now, since $T$ does not admit infinite paths, there exists $m$ such that $w_m = \overline{\alpha}M$ where $M = \abs*{w_m}$. This means that 
\[ x \in I_{w_m} \subset \bigcup_{n \in \NN} I_{w_n}  \ .\]
Now let $n \in \NN$ be arbitrary and choose $N = \max \menge{\abs*{w_1}, \abs*{w_2}, \dots, \abs*{w_n}}$. First notice that 
\[ \frac{1}{2^{\abs*{w_i}}} = \sum_{\substack{u \in 2^N \\ w_i \sqsubseteq u}} \frac{1}{2^N} \]
and that for such $u$ we also have $u \notin T$. 
 We have, 
\begin{align*}
	\sum_{i =1}^n \abs*{I_i} &  = \sum_{i =1}^n \left( \frac{1}{2^{\abs*{w_i}}} + \frac{\alpha}{2^{2\abs*{w_i}+1}} \right) \\ & \leqslant  \sum_{i =1}^n \frac{1}{2^{\abs*{w_i}}} + \sum_{u \in \cS} \frac{\alpha}{2^{2\abs*{u}+1}} \\ & =  \sum_{i =1}^n \frac{1}{2^{\abs*{w_i}}} + \sum_{n \in \NN} 2^n \frac{\alpha}{2^{2n+1}} \\ & =  \sum_{i =1}^n \frac{1}{2^{\abs*{w_i}}} + \alpha \sum_{n \in \NN} \frac{1}{2^{n+1}} 	\\ & \leqslant  \sum_{i =1}^n \sum_{\substack{u \in 2^N \\ w_i \sqsubseteq u}} \frac{1}{2^N} +  \frac{\alpha}{2}  \\ & \leqslant \sum_{\substack{u \in 2^N \\u \notin T}} \frac{1}{2^N} +  \frac{\alpha}{2} \\ & \leqslant 2^N (1-(1-\alpha/2))\frac{1}{2^N} +\frac{\alpha}{2} = \alpha \ .
\end{align*}
Thus we have shown that \ref{Equiv:KTs1} implies \ref{Equiv:KTs2}. 
Now assume that \ref{Equiv:KTs2} holds and let $0< \alpha <1$ be arbitrary. We will first show that we may assume that the intervals in \ref{Equiv:KTs2} have rational endpoints. To start, we may assume, Without loss of generality, that $\alpha$ is rational. Now let $I_n = (a_n,b_n)$ a sequence of open intervals covering $[0,1]$ and satisfying \ref{SC1} for $\alpha/2$. Now choose rationals $a^\prime_n$ and and $b_n^\prime$ such that $a_n^\prime \leqslant a_n$, $b_n \leqslant b_n^\prime$, $\abs*{a_n^\prime - a_n} < \frac{\alpha}{2^{n+2}}$, and $\abs*{b_n^\prime - b_n} < \frac{\alpha}{2^{n+2}}$. Then $I^\prime_n = (a_n^\prime, b_n^\prime)$ are obviously still a cover of $[0,1]$ and for any $n \in \NN$
\[ \sum_{i=1}^n \abs*{I_n^\prime} \leqslant \sum_{i=1}^n \abs*{I_n} + 2\sum_{i=1}^n \frac{\alpha}{2^{i+2}} \leqslant \nicefrac{\alpha}{2} + \nicefrac{\alpha}{2} \ . \] It is clear that by shrinking  the intervals $I_n^\prime$ and possibly cutting up into subintervals we can obtain a singular cover (with the same constant $\alpha$). Thus \ref{Equiv:KTs3} follows from \ref{Equiv:KTs2}. 

Finally let $1>k\geqslant 0$ and assume that we have an $\alpha$-singular cover $(J_n)_{n \geqslant 1}$, where $\alpha =1 -k$. Let $I_u$ be the intervals defined in Section \ref{Sec:linking-cantor-and-unitint} for $p = \nicefrac{1}{2}$. Now define a set $T$ by
\[ T  = \set{u \in \cS}{I_u \not \subset \bigcup_{i=1}^{\abs*{u}} J_i  }  \  .\]
Since $J_i$ have rational endpoints the set $T$ is decidable. It is also closed under restriction, since, for $w \sqsubseteq u$ and $u \in T$ we have $I_u \subset I_w$ and \[ \bigcup_{i=1}^{\abs*{w}} J_i  \subset \bigcup_{i=1}^{\abs*{u}} J_i \ . \] Therefore, the assumption that $I_w \subset  \bigcup_{i=1}^{\abs*{w}} J_i$ implies $I_u \subset  \bigcup_{i=1}^{\abs*{u}} J_i$; a contradiction. Thus $I_w \not \subset  \bigcup_{i=1}^{\abs*{w}} J_i$, which means that $w \in T$. Hence $T$ is a tree. 
Now, consider $\alpha \in CS$ arbitrary. By \ref{SC3} there exists $n,m$ such that $b_n = a_m$ and $F^{\nicefrac{1}{2}}(\alpha) \in [a_n, b_m]$. Since all the intervals in a singular cover are proper $\abs*{b_m - a_n} > 0$, which implies that there exists $N$ such that $I_{\overline{\alpha}N} \subset [a_n, b_m]$. Thus, for $K = \max\menge{N,n,m}$ we have $I_{\overline{\alpha}K} \subset \bigcup_{i=1}^{K} J_i$, which implies $\overline{\alpha}K \notin T$. Hence $T$ does not admit infinite paths.

So the only item left to consider is to show that $T$ satisfies Equation \eqref{Eqn:majority}. First notice that by definition of $T$ for $n \in \NN$
\[ 
\sum_{\substack{u \notin T \\ \abs*{u} = n}} \abs*{I_u} \leqslant \sum_{i=1}^n \abs*{J_i} \ .
\]
So 
\[ \sum_{\substack{u \in T \\ \abs*{u} = n}} \frac{1}{2^{\abs{u}}} = 1 - \sum_{\substack{u \notin T \\ \abs{u} = n}} \frac{1}{2^{\abs{u}}} = 1- \sum_{\substack{u \notin T \\ \abs{u} = n}} \abs*{I_u}  =  1 - \sum_{i=1}^n \abs*{J_i} \geqslant 1- \alpha = k  \ . \]
Thus Equation \eqref{Eqn:majority} is fulfilled and therefore \KTs{k}.
\end{proof}

\begin{Cor}
	If there exists a $\alpha$-singular cover for a specific  $0< \alpha < 0$, then there exists a $\alpha^\prime$-singular cover for any $0< \alpha^\prime < 0$.
\end{Cor}

Thus the following principle is independent of the choice of $\alpha$.
\begin{principle}[SC]{\SinC} \label{PR:SinC}
There exists a singular cover (for some $0< \alpha < 0$).
\end{principle}

\section{Kleene Trees} \label{Sec:KT}
As already mentioned in the introduction of Chapter \ref{Ch:fan}, in some schools of constructive mathematics there exists a Kleene tree. In this section we want to investigate equivalences of the existence of such an object. 
\begin{principle}[KT]{\KT} \label{PR:KT}
There exists a decidable binary tree $T$ such that 
\[ \fa{n\in \NN}{\ex{u \in T}{\abs{u} \geqslant n}} \]
and
\[ \fa{\alpha \in \CS}{\ex{n \in \NN}{\overline{\alpha}n \notin T} } \ . \]
\end{principle}
In the presence of countable choice\footnote{Actually countable choice is too much. } we can also assume that there is a sequence $(u_n)_{n \geqslant 1}$ with $\abs*{u_n}=n$ and $u_n \in T$. 
The principle \KT was also named \textrm{Anti-FT$_{\Delta}$} in \cite{hD08b}*{Section 4.5} for obvious reasons. There, also, parts of the following proposition are proved.
\begin{Pro} \label{Pro:KT_equiv}
  The following are equivalent to  \KT.
  \begin{enumerate}
    \item \label{dj} There exist two compact subsets $A,B$ of a metric space such that 
    \[\fa{ a \in A}{\fa{b \in B}{ d(a,b) > 0 }} \] 
    but 
    \[ \inf \set{ d(x,y)}{x \in A, y \in B} = 0 .\]
    \item \label{pos} There is a uniformly continuous map $f:\CS \to \RR^+$ with $\inf f(\CS) =     0$.
    \item \label{pos2} There is a uniformly continuous map $f:[0,1] \to \RR^+$ with $\inf f([0,1]) =     0$.
    \item There exists an countable open cover of $\CS$ of basis sets that does not admit a finite subcover.
  \end{enumerate}
\end{Pro}

\begin{proof} These equivalences can be proved with the same arguments as in the proof of Proposition \ref{Pro:Fand-equivalents}. 
\end{proof}

\begin{Pro} \label{bij-cant-baire}
\KT is equivalent to   the existence of a homeomorphism $\varphi: \CS \to \BS $ with a continuous modulus of continuity given by a function $\mu:\CS \times \NN \to \NN$; i.e.\ 
    \[ \fa{\alpha \in \CS, n \in \NN}{\overline{\beta}\mu(\alpha,n) = \overline{\alpha}\mu(\alpha,n)  \implies \overline{\varphi(\beta)}n = \overline{\varphi(\alpha)}n } \ . \]
\end{Pro}

\begin{proof}
One direction is well known \cite{Beeson1985}*{IV.13}, at least for the inverse mapping $\BS \to \CS$, and will be easy to adapt it to our purposes.  Assume that \KT holds, so assume that $T$ is a decidable infinite tree that blocks ever infinite path. In particular we can find $\menge{u_{1}, u_{2}, \dots }$ an injective enumeration of all $u \notin T$ such that  $\overline{u}(\abs{u}-1) \in T$. In particular, for every $\alpha \in \CS$ there exist unique $k_{\alpha},m_{\alpha} \in \NN$ such that $\overline{\alpha}k_{\alpha} = u_{m_{\alpha}}$. Now define  $(\alpha_n)_{n \geqslant 1}$ recursively by
\begin{align*}
\alpha_{0} & = \alpha  \ , \\
\alpha_{n+1} & = \alpha_n[k_{\alpha_n} :] \ ;
\end{align*}
that is we keep chopping off unique $u_{m}$ prefixes of $\alpha$. Now $\varphi:\CS \to \BS$ can be succinctly defined by
\[ \varphi(\alpha)(n)  = m_{\alpha_n} \ . \] 
We still need to show that $\varphi$ is surjective, injective, and has a continuous modulus if continuity. The easiest is surjectivity: if $\gamma \in \BS$, then 
\[ \varphi( u_{\gamma(0)} \ast u_{\gamma(1)} \ast u_{\gamma(2)} \ast \dots ) = \gamma \ . \] 
It is easy to see, that by the way it is constructed $\varphi$ is injective.
To see that $\varphi$ has a continuous modulus of continuity let $e \in \NN$ and $\alpha \in \CS$ be arbitrary. Let $N = \sum_{i=0}^{e} k_{\alpha_i}$. Then
\[  \alpha = u_{m_{\alpha_0}} \ast u_{m_{\alpha_1}} \ast \dots \ast u_{m_{\alpha_e}} \ast \alpha_{e+1}  \ ,  \]
and for any $\beta$ such that $\overline{\alpha}N = \overline{\beta}N$ we have that \[  \beta = u_{m_{\alpha_0}} \ast u_{m_{\alpha_1}} \ast \dots \ast u_{m_{\alpha_e}} \ast \beta_{e+1}  \ .  \]
That means that for $i \leqslant e$ we have that $\varphi(\alpha)(i)=\varphi(\beta)(i)$. Hence $\mu(\alpha,e) = \sum_{i=0}^{e} k_{\alpha_i}$ is a modulus of continuity, which by construction only depends on a finite initial prefix of $\alpha$ and therefore is itself continuous.

Conversely let $\varphi : \CS \to \BS$ be a homeomorphism between  Cantor and Baire space with $\mu$ as stated above. Now define 
\[ T = \set{u \in \cS}{\fa{i \leqslant \abs{u}}{ \mu(\overline{u}i \ast \zero ,1) > \abs*{i}}}  \ . \]
By definition $T$ is decidable and closed under restriction. We claim that  it does not admit infinite paths, but is infinite.
To see that $T$ does not admit infinite paths let $\alpha \in \CS$ arbitrary. Then there exists, since $\mu$ is continuous itself, $M$ such that $\mu(\alpha,1) = \mu(\overline{\alpha}M \ast \beta,1)$ for all $\beta \in \CS$. Now let $N = \max \menge{M,\mu(\alpha)}$. Then the assumption that $\overline{\alpha}N \in T$ implies that $\mu(\overline{\alpha}N \ast \zero,1) > N$. But 
\[ N \geqslant \mu(\alpha,1) = \mu(\overline{\alpha}N \ast \zero,1) > N \] a contradiction and hence $\overline{\alpha}N \notin T$.

Assume that $n$ bounds the height of $T$  that is that $\overline{\alpha}n \notin T$ for all $\alpha \in \CS$. That means that  $\mu(\overline{\alpha}n \ast \zero) \leqslant n$ for $\alpha \in \CS$, which means that $\varphi(\overline{\alpha}n\ast \zero)(1) = \varphi(\alpha)(1)$. This in term implies that $\varphi(\alpha)(1) \leqslant \max_{u \in \cS:\abs{u}=n} \menge{\varphi(u \ast \zero)(1)}$. So $\varphi$ cannot be surjective; a contradiction. We conclude that $T$ is infinite.
\end{proof}

\begin{Pro} \label{Pro:KT-fullyloc} \KT is equivalent to the existence of a point-wise continuous, fully located function $f:[0,1] \to \RR$ that fails to be uniformly continuous.
\end{Pro}
\begin{proof}
This follows from Lemma \ref{Lem:fullyloc}.
\end{proof}

\section{Specker Sequences}
In the seminal article \cite{Specker1949} E.\ Specker showed that in recursive mathematics there exists an increasing, computable sequence of rationals $(r_n)_{n \geqslant 1}$ in $[0,1]$ that does not converge to a computable number. More than that, he showed that it  does not converge to a computable number in the strong sense that it is computably eventually bounded away from every computable real. 
Interpreted in \RUSS, this shows that the following principle holds there.
\begin{principle}[SS]{\SS} \label{PR:SS}
	There exists a sequence $(x_n)_{n \geqslant 1}$ in $[0,1]$ that is bounded away from every point in $[0,1]$.
\end{principle}

It is worth pointing out that we have not forgotten about the fact that Specker's original sequence is rational, as the next lemma clarifies. The fact that Specker's sequence is increasing is discussed in Section \ref{Sec:incss}.
\begin{Lem}
If  $(x_n)_{n \geqslant 1}$ is a sequence in $[0,1]$ that is bounded away from every point in $[0,1]$, then there exists a sequence of rationals with the same property.
\end{Lem}
\begin{proof}
Straightforward.	
\end{proof}

\begin{Pro} \label{Pro:KT_impl_iSS}
\KT implies that there exists an increasing Specker sequence of rationals.
\end{Pro}
\begin{proof}
The construction, naturally, is a variation of the one in Lemma \ref{Lem:bar-to-seq}.
Let $T$ be a Kleene tree. Now construct a sequence $(w_n)_{n \geqslant 1} \in T$ such that for all $n \in \NN$
\begin{enumerate}
\item $\abs*{w_n} = n$, and
\item $\left( u \in T \land \abs{u} =n \right) \implies w_n \leqslant_{\mathrm{lex}} u$.
\end{enumerate}
Notice that then $w_n0 \leqslant_{\mathrm{lex}} w_{n+1}$.
Define $x_n = F^{\nicefrac{1}{3}}(w_n)$. Then $(x_n)_{n \geqslant 1}$ is an increasing sequence of rationals in $[0,1]$. We want to show that it is eventually bounded away from every $x\in [0,1]$. To this end choose $\alpha \in \CS$ as in Lemma \ref{Lem:Fp}.2. Since $T$ blocks every infinite path there exists $N \in \NN$ such that $\overline{\alpha}N \notin T$. Hence, for every $n \geqslant N$ we must have $\overline{\alpha}N \neq \overline{w_n}N$, which means that 
\[ \abs*{F^{\nicefrac{1}{3}}(w_n) - F^{\nicefrac{1}{3}}(\alpha)} > 3^{-(N+1)}  \ .\] Now either $\abs*{F^{\nicefrac{1}{3}}(\alpha) - x} < 3^{-(N+2)} $ or $\abs*{F^{\nicefrac{1}{3}}(\alpha) - x} > 3^{-(N+3)} $. In the first case 
\[ \abs*{F^{\nicefrac{1}{3}}(w_n) - x} > 3^{-N+1}  \]
for all $n \geqslant N$. In the second case by Lemma \ref{Lem:Fp}.2 we have that $d(x,F(\CS)) > \delta$ for some $\delta > 0$. In both cases $\abs*{x_n - x} > \min \menge{\delta, 3^{-(N+2)}}$ for all $n \geqslant N$.
\end{proof}

In the following we want to weaken the requirement of decidability on a Kleene tree. Just as a decidable tree is the complement of a decidable set that is closed under extensions, we define a \define{$c$-tree} to be the complement of a $c$-set that is closed under extensions. More formally a tree $T \subset \cS$ is a $c$-tree, if there exists a decidable set $D \subset \cS$ such that 
\[ u \in T \iff \ex{w \in \cS}{u \ast w \in D}  \ . \]
A $c$-Kleene tree is then simply  a Kleene tree, that is a $c$-tree instead of a decidable one. Obviously every Kleene tree is a $c$-Kleene tree.

\begin{Pro} \label{Pro:equiv_SS} The following are equivalent.
\begin{enumerate}
\item \label{enu-ss} There exists a Specker sequence in $[0,1]$.
\item \label{KTc} There exists a $c$-Kleene tree. 
\item \label{surj-cant-natural} There exists a continuous surjection $\varphi: \CS \to \NN$
\item \label{surj-cant-baire} There exists a continuous surjection $\Phi: \CS \to \BS$. 
\end{enumerate}
\end{Pro}
\begin{proof}
The equivalence $\ref{enu-ss} \iff \ref{KTc}$ follows from Lemmas \ref{Lem:seq-to-bar} and \ref{Lem:bar-to-seq} and the above mentioned fact that a $c$-Kleene tree is just the complement of a $c$-bar and vice-versa.

To see that \ref{KTc} implies \ref{surj-cant-natural} let $T$ be a $c$-Kleene tree. Now consider the function $g:\CS \to \NN$ defined by 
\[g(\alpha) = \min \set{n \in \NN}{\overline{\alpha}n \notin T} \ .  \] 
The fact that $T$ does not admit infinite paths ensures that this is well-defined. It is also easily seen to be point-wise continuous. Unfortunately it might not be a surjection. However, the image $R=g(\CS)$ is decidable and infinite, which means we can find a surjection $h:R \to \NN$. So the composition $\varphi=h \circ g$ is the desired function.

Next, will now show that \ref{surj-cant-natural} implies \ref{KTc}: for assume $\varphi : \CS \to \NN$ is a continuous surjection. Let $\alpha_n = \varphi^{-1}(n) $. Using continuity and (unique) countable choice, we can find  $w_n \in \cS$ such that $\varphi(w_n) = n$. Without loss of generality we may also choose $w_n$ such that $\abs*{w_n} \geqslant n$, so that $\menge{w_{1},w_{2},\dots}$ is a decidable subset of $\cS$. We want to show that 
\[ u \in T \iff \ex{w \in \cS}{uw \in \menge{w_{1},w_{2},\dots}}  \]
is a $c$-Kleene tree. It is clear that it is a $c$-tree by definition, and that it is infinite, since $w_n \in T$ and $\abs*{w_n} \geqslant n$ for all $n \in \NN$. So it remains to show that $T$ does not admit any infinite path. To this end let $\alpha \in \CS$ be arbitrary. Set \[ m = \varphi(\alpha) \ , \] and, by continuity, choose $M$ such that $\overline{\alpha}M = \overline{\beta}M$ implies that $\varphi(\beta) = m$; which implies that for all such $\beta$ and $i>m$ we must have $\overline{\alpha_{i}}M \neq \overline{\beta}M$. Set $k = \max\menge{M, \abs*{w_{1}}, \dots, \abs*{w_{m}}} +1$. Now assume that $\overline{\alpha}k \in T$; that is that there is $w \in \cS$ and $j \in \NN$ such that $\overline{\alpha}k \ast w = w_{j}$. Since $\abs*{\overline{\alpha}k \ast w} \geqslant k $ we must have $j \geqslant k $. But for all $j \geqslant k$ also $j >m,M$ and we have  \[ \overline{w_{j}}M  = \overline{\alpha_{j}}M \neq  \overline{\alpha}M  =  \overline{\overline{\alpha}k \ast w}M \]  and therefore $w_{j} \neq \overline{\alpha}k \ast w $. That means that $\overline{\alpha}k \notin T$.

It is clear that \ref{surj-cant-baire} implies \ref{surj-cant-natural}. Conversely assume $\varphi:\CS \to \NN$ is a continuous surjection. Now consider a bijective pairing function $\pi = (\pi_{1},\pi_{2}):\NN \to \NN^{2}$. Define $\Phi:\CS \to (\NN \to \NN)$ by
\[ \Phi(\alpha)(m) = \varphi(\alpha \circ ( \pi^{-1}(m,\cdot)))  \ . \]
We want to show that $\Phi$ is surjective. So let $\gamma \in \BS$ be arbitrary. Using countable choice and continuity we can find $w_n \in \cS$ such $\varphi(\alpha_n) = \gamma(n)$, where $\alpha_n=w_n \ast \zero$. Now for $\beta\in \CS$ defined by 
\[ \beta(i) = \alpha_{\pi_{1}(i)}(\pi_{2}(i))  \]
we have that $\beta \circ ( \pi^{-1}(m,\cdot))  = \alpha_{m}$ and therefore $\Phi(\beta)(m) = \varphi(\alpha_{m}) = \gamma(m)$. Thus $\Phi(\beta) = \gamma$, which means $\Phi$ is surjective. It is also straightforward to show that it is continuous.
\end{proof}

The last proposition together with Proposition \ref{bij-cant-baire} raises an interesting question. 
\begin{Qu}
 Is the existence of a bijection $\BS \to \CS$ equivalent to \SS or \KT (or neither)?
\end{Qu}
Analysing the proofs of the two propositions mentioned it seems as if \KT is too strong an assumption, whereas \SS is just barely not enough.\footnote{A good candidate might actually be \iSS as in Section \ref{Sec:incss}.}

\begin{Pro} \label{Pro:equivs_of_SS}
The following are equivalent
\begin{enumerate}
\item \label{Equiv:SS1} There exists a Specker sequence in $[0,1]$.
\item \label{Equiv:SS2} There exists a point-wise continuous function $f:[0,1] \to \RR$ that is unbounded.
\item \label{Equiv:SS3} There exists a surjection $f:[0,1] \to \RR$
\item \label{Equiv:SS4} There exists a point-wise continuous function $f:[0,1] \to [0,1]$ that is not uniformly continuous.
\item \label{Equiv:SS5} There exists a point-wise continuous function $f:[0,1] \to \RR^{+}$ with $\inf f =0$.
\end{enumerate}
\end{Pro}
\begin{proof}
The equivalences between \ref{Equiv:SS1}, \ref{Equiv:SS2}, and \ref{Equiv:SS4} are proved in \cite{hD08b}*{Proposition 4.5.2}. It is also clear that the proof between \ref{Equiv:SS1} and \ref{Equiv:SS2} there can be extended to show the equivalence of \ref{Equiv:SS1}, \ref{Equiv:SS3}. Finally, if $f:[0,1] \to \RR^+$ is a function with $\inf f =0$, then $1/f$ is well-defined, point-wise continuous and unbounded. Conversely if $f:[0,1] \to \RR$ is point-wise continuous and unbounded, then $h= 1 / \max\menge{\abs*{f}, 1}$ is also point-wise continuous and such that $h(x)>0$ for all $x \in [0,1]$ but $\inf h = 0$. Thus \ref{Equiv:SS2} $\iff$ \ref{Equiv:SS5}.
\end{proof}

\chapter{Relationships Between the Principles} \label{Ch:relations}
\section{Basic Relations} \label{Sec:WKL_impl_UCT}
It has long been known that \WKL implies \FAND \cite{hI06}. In \cite{jB07b} Berger has shown that it also implies \FANc. This result in turn was again slightly improved upon in \cite{hD12} where we showed that it also implies \UCT. For completeness' sake we will include the proof here.

\begin{Lem}
\LLPO / \WKL implies that the image of a sequentially continuous map $f:\CS \to \RR$  is order located.
\end{Lem}
\begin{proof}
Let $f:\CS \to \RR$ be a map and $a<b$ two arbitrary real numbers. Using countable choice and \LLPO\footnote{This application of \LLPO is actually not necessary, since the proof still works with a slightly modified definition of $\mu$ (namely $\mu(1) \implies f(u) > b - (b-a)2^{-\abs{u}} $ ). It is, however, convenient and since \LLPO is used later in the proof, nothing seems to be gained from dropping it here. }  fix $\mu : \cS \to \menge{0,1}$ such that
\begin{align*}
\mu(u) = 0 & \implies f(u \ast \zero ) \leqslant b \ , \\
\mu(u) = 1 & \implies f(u \ast  \zero) \geqslant b \ .
\end{align*}

Let $\preceq$ denote the decidable order on $\cS$ such that
\[ u \preceq v \iff \abs{u} < \abs{v} \lor ( \abs{u}= \abs{v} \land u \leq_{Lex} v ) \ , \] where $\leq_{Lex}$ is the lexicographic order.
For convenience and readability's sake we define decidable predicates $\Lambda$ and $ F$ on $2^{*}$  by 
\begin{align*}
\Lambda(u)  & \iff \fa{v \in 2^{*}}{ \; \left( v \preceq u \implies \mu (v) = 0 \right)} \ , \\
F(u) & \iff \mu(u) =1 \land  \left( \fa{ v \in \cS }{ \; \mu(v) = 1 \implies u \preceq v} \right)  \ .
\end{align*}
In words, if we imagine searching for a finite binary sequence $u$ with $\mu(u) = 1$, then as long as we have not been successful $\Lambda$ holds and $F$ only holds for the first such sequence. 
 Now define a decidable set $T$ by
\[ T = \set{u \in \cS}{\Lambda(u)} \cup \set{u \in \cS}{\ex{v} {(F(v)  \land u = v\ast 0 \ast \dots \ast 0)}} \ . \]
Notice that for every $n \in \NN$ either $\mu(u) = 0$ for all $u \in 2^{*}$ with $\abs{u} \leqslant n$ or there exists $v \in 2^{*}$ such that $\abs{v} \leqslant n$ and $F(v)$ holds.

Therefore, if $u \in T$ either $\Lambda(u)$ holds or there exists $v$ with $F(v)$ such that $u = v * 0 * \dots * 0$. In the first case also $\Lambda(\overline{u} (n)) $ for all $1 \leqslant n \leqslant \abs{u}$, whence $\overline{u} (n) \in T$. In the second case and if $1 \leqslant n < \abs{v}$, then $\overline{v}(n) = \overline{u}(n)$ and therefore $\Lambda(\overline{u}(n))$ holds which in return implies that $\overline{u} n \in T$. Finally in the second case and if $ \abs{v} \leqslant n \leqslant \abs{u}$, then $\overline{u}n = v * 0 * \dots * 0$ which means that $\overline{u}n \in T$.
This shows that $T$ is closed under restriction and hence is a tree. It also contains elements of arbitrary lengths and therefore, by \WKL, this tree admits an infinite path $\alpha \in \CS$. Now either $f(\alpha) > a$ and we are done, or $f(\alpha) < b$. We claim that in this second case $\mu(u) =0$ for all $u \in 2^{*}$. For assuming that there exists $u \in \cS$ with $\mu(u) =1$,  we can find $v \preceq u$ such that $F(v)$ holds, which means that $\alpha = v * 0 * \dots$ by the construction of $T$. But this leads to the contradiction \[ b \leqslant f(v \ast \zero) = f (\alpha) < b \ . \] By sequential continuity, the fact that $f(u \ast \zero) \leqslant b$ for all $u \in 2^{*}$ now implies $f(\alpha) \leqslant b$ for all $\alpha \in \CS$.
\end{proof}

\begin{Lem}
\FANc implies that an order located image of a point-wise continuous map $f:\CS \to \RR$ is bounded.
\end{Lem}
\begin{proof}
Define a sequence $(\alpha_n)_{n \geqslant 1}$ in $\CS \cup \menge{\omega} $ such that
\begin{align*}
  \alpha_n \in \CS  & \implies f(\alpha_n) > n-1 \ , \\
  \alpha_n = \omega & \implies \fa{ \alpha \in \CS}{ f(\alpha) \leqslant n} \ .
\end{align*}
Since $f$ is point-wise continuous it is locally bounded. Hence $(\alpha_n)_{n \geqslant 1}$ is eventually bounded away from every point in $\CS$. In \cite{hD08b} it is shown that \FANc implies what is know as the anti-Specker property for $\CS$: namely that every sequence in $\CS \cup \menge{\omega} $ that is bounded away from every point in $\CS$ is eventually bounded away from the entire set. So there exists $N$ such that $\alpha_n = \omega$, which means that $f$ is bounded.
\end{proof}

\begin{Pro} \label{Pro:LLPO_impl_UCT}
\LLPO /\WKL implies \UCT \ .\footnote{M.~Hendtlass has given a direct proof of this result (unpublished). There is also a generalisation \cite{mH13}*{Proposition 30}.} 
\end{Pro}
\begin{proof} J. Berger's work in \cite{jB07b} shows that $\LLPO \implies \FANc$. So combining the previous lemmas we get that under the assumption of \LLPO every point-wise continuous map $f:\CS \to \RR$ is bounded. Moreover, this last statement was shown to be an equivalent of \UCT in 
\cite{dB07}.
\end{proof}

This result enables us to replace ``uniformly continuous'' by ``point-wise continuous'' thereby improving the well known characterisation of $\WKL$ \cite{hI90}. 
\begin{Cor}
\WKL is equivalent to the statement that every point-wise continuous map on a compact space attains its minimum.
\end{Cor}
More general, this means that we can replace ``uniformly continuous'' by ``point-wise continuous'' in any of the equivalences to \LPO, \WLPO, and \LLPO in Chapter \ref{Ch:OmniPr}.

 It is also clear that \WLPO implies \FANst, so we get the following diagram of implications:
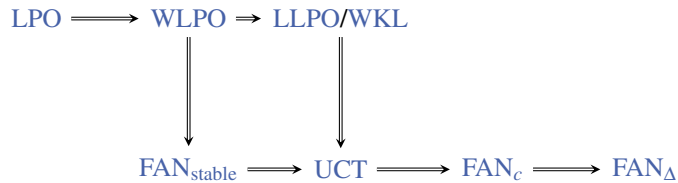
\begin{figure}[h]
  \begin{tikzpicture}[>=stealth,shorten >=1pt, node distance=2cm]
    \node (A)  {\LPO};
    \node[right of=A] (B) {\WLPO};
    \node[right of=B] (C) {\LLPO/\WKL};
    \node[below of=B] (D) {\FANst};
    \node[right of=D] (E) {\UCT};
    \node[right of=E] (F) {\FANc};
    \node[right of=F] (G) {\FAND};
    \draw[double,->] (A) -- (B);
\draw[double,->] (B) -- (C);
\draw[double,->] (B) -- (D);
\draw[double,->] (C) -- (E);
\draw[double,->] (D) -- (E);
\draw[double,->] (E) -- (F);
\draw[double,->] (F) -- (G);
  \end{tikzpicture}
\end{figure}

There is a possibility to improve on these results:
\begin{Qu}
Does \WKL imply \FANst or \FANP?
\end{Qu}

\section{Kripke's Schema and the Principle of Finite Possibility} \label{Sec:KS_and_PFP}
Kripke's Schema, which goes back to Kreisel (see historical notes \cite{Troelstra1988a}*{Ch apter 4, 10.6}) and Myhill \cites{jM68, jM66}, states that 
\begin{principle}[KS]{\KS} \label{PR:KS}
for every  statement $\varphi$ there exists a binary sequence $(a_n)_{n \geqslant 1}$ such that \begin{equation} \label{Eqn:KS} \varphi \iff  \ex{ n \in \NN}{ a_n=1}  \ .\end{equation}
\end{principle}
This amounts to saying that every statement is simply existential. The philosophy behind it is that, if there is a proof for $\varphi $, somebody will eventually find it. Kripke's Schema captures  the idea of Brouwer's creating subject and to formalise the rejection of Markov's principle see Proposition \ref{Pro:WPFP+MP_equiv_LPO}. 
A weakening of \KS which restricts the complexity of the formulas $\varphi$ is the Principle of Finite Possibility (\PFP). It was introduced by Mandelkern in \cite{mM82} (also named the  Principle of Inverse Decision (PID) in \cite{mM88}).  
\begin{principle}[PFP]{\PFP} \label{PR:PFP}
\ For every binary sequence $(a_n)_{n \geqslant 1}$ there exists a binary sequence $(b_n)_{n \geqslant 1}$ such that 
\begin{equation*} 
 \fa{n \in \NN}{a_{n } = 0}  \iff \ex{ n \in \NN}{ b_n = 1} \ . 
 \end{equation*}
\end{principle}
To reject \MP the following weaker version  is in fact sufficient
\begin{principle}[WPFP]{\WPFP} \label{PR:WPFP} For every binary sequence $(a_n)_{n \geqslant 1}$ there exists a binary sequence $(b_n)_{n \geqslant 1}$ such that 
\begin{equation} \label{Eqn:WPFP}
 \fa{n \in \NN}{ a_{n } = 0}  \iff \lnot \fa{n \in \NN}{b_n = 0} \ . 
 \end{equation}
\end{principle}
Trivially 
\[ \KS \implies \PFP \implies \WPFP  \ .\]
\WPFP was also named WPID\index{WPID|see \WPFP} (weak principle of Inverse Decision) by Mandelkern, who, to our knowledge, also first suggested it. The  ``rejection'' of \MP using \WPFP is of course only convincing if one agrees with the latter and disagrees with \LPO.
\begin{Pro} \label{Pro:WPFP+MP_equiv_LPO}
$\WPFP + \MP \iff \LPO$
\end{Pro}
Note that a proof of this fact has already been sketched in \cite{mM82}*{page 258}. It will, however, also follow from Proposition \ref{Pro:MPvWPFPequivWLPO} below.

The next proposition improves on an example from \cite{mM82}*{Example 3}, in which it is shown that the second part of the following implies \PFP.
\begin{Pro} \label{Pro:KS_equiv_Ncount}
\KS is equivalent to the statement that every inhabited subset of $\NN$ is countable.
\end{Pro}
\begin{proof}
Let $\varphi$ be any syntactically correct closed statement and let 
\[ A = \set{x \in \NN}{x = 0 \lor \left( x=1 \land \varphi \right)}  \ . \]
Then $A$ is inhabited, since $0 \in A$. If $A$ is countable there exists a, without loss of generality binary, sequence $(a_n)_{n \geqslant 1}$ such that $A = \set{a_n}{n \in \NN}$. In particular, Equation \eqref{Eqn:KS} holds.
Conversely, if $A$ is an inhabited subset of $\NN$ and Kripke's schema holds, then for every $m\in \NN$ there exists a binary sequence $(a_{m,n})_{n \geqslant 1}$ such that 
\begin{equation*} 
m \in A  \iff  \ex{ n \in \NN}{ a_{m, n}=1}  \ .
\end{equation*} 
Using countable choice we can write this as a double sequence. Now, using a standard bijection $\pi: \NN \to \NN^{2}$,  it is easy to see that the function
\[ f(n) = \begin{cases} \pi(n)_{1} & \text{if } a_{\pi(n)} =1 \\
a & \text{otherwise, }
\end{cases}
\] where $a$ is an arbitrary fixed element of $A$, enumerates $A$.
\end{proof}
Exercise 1.17 \cite{dBlV06}*{page 21} requires one to fill in the (essentially above proposition) easy proof that the statement that ``every inhabited subset of $\NN$ is countable'' implies \PFP.
It is almost obvious, that $\WPFP + \LLPO  \iff \WLPO$.
Although on first glance \WPFP seems to be the perfect puzzle piece bridging the gap between \LLPO and \WLPO it is somewhat to strong as the following result shows. We remind the reader that \LLPO is strictly stronger than \MPv.
\begin{Pro}
$\WPFP + \MPv \iff \WLPO$.\footnote{This has also been proven in \cite{mM88} and was also pointed out in \cite{mH13}*{Proposition 25}.} \label{Pro:MPvWPFPequivWLPO}
\end{Pro}
\begin{proof} 
Let $(a_n)_{n \geqslant 1}$ be a binary sequence and let $(b_{n})_{n \geqslant 1}$ be as in \ref{Eqn:WPFP}. Then 
\[ \lnot \fa{n \in \NN}{ \left( a_{n} = 0 \land b_{n} = 0  \right)} \ . \] Thus by MP$^{\lor}$ either $a_{n} = 0$ for all $n \in \NN$ or $b_{n} = 0$ for all $n \in \NN$. In the latter case  it is impossible that $a_{n}=0$ for all $n \in \NN$. Thus \WLPO holds.

Conversely, it is clear that \WLPO implies \MPv. To see that it also implies \WPFP let $(a_{n})_{n \geqslant 1}$ be a binary sequence. By \WLPO either $a_{n} =0$ for all $n \in \NN$ or there it is impossible that   $a_{n}=0$ for all $n \in \NN$. In the first case set $(b_{n})_{n \geqslant 1} \equiv 1$, in the second case set $(b_{n})_{n \geqslant 1} \equiv 0$. It is clear that the so defined sequence satisfies Equation \eqref{Eqn:WPFP}.
\end{proof}

\begin{Pro*}[\textbf{\theThm{}$\nicefrac{1}{2}$}] \label{Pro:PFP_and_IIIa_impl_WLPO}
	$\PFP+\IIIa \implies \WLPO$.
\end{Pro*}
\begin{proof}
Straightforward.
\end{proof}

The following can be found in \cite{rL12b}.
\begin{Pro}
\KS is equivalent to the following
\begin{enumerate}
\item Every open subspace of a separable metric space is separable.
\item Every open subset of a separable metric space is a countable union of open balls.
\end{enumerate}
\end{Pro}

There are also some  equivalences of \KS involving the countability of subsets of $\NN$ in \cite{pS08}.

Finally, even though it might seem like \PFP and \WPFP are rather ``ad hoc'' principles among many others, the following result shows that it is not the case, and similar looking principles are not similarly interesting.
\begin{bareprinciple}\label{PR:coPFP} \label{PR:coWPFP}
	(\coPFP) \quad  For every binary sequence $(a_{n})_{n \geqslant 1}$ there exists a binary sequence $(b_{n})_{n \geqslant 1}$ such that 
\begin{equation*} 
 \ex{n \in \NN}{a_{n} = 1}  \iff \fa{n \in \NN}{b_{n} = 0} \ . 
 \end{equation*}
\\ 
	(\coWPFP) \quad For every binary sequence $(a_{n})_{n \geqslant 1}$ there exists a binary sequence $(b_{n})_{n \geqslant 1}$ such that 
\begin{equation*} 
 \lnot \fa{n \in \NN}{a_{n} = 0}  \iff \fa{n \in \NN}{ b_{n} = 0} \ . 
 \end{equation*}
\end{bareprinciple}
\begin{Pro} $ $
\begin{enumerate}
  \item $\coPFP \iff \LPO$.
  \item $\coWPFP \iff \WPFP$.
\end{enumerate}
\end{Pro}
\begin{proof}
\todo{Future check: \coPFP and \coWPFP proof from arithmetic hierarchy paper}
\begin{enumerate}
\item We will first show that \coPFP implies \MP. Assume $(a_n)_{n \geqslant 1}$ is such that $\lnot \fa{n\in \NN}{a_n = 0}$. By \coPFP there exists $(b_n)_{n \geqslant 1}$ such that 
\[ \ex{n \in \NN}{a_n =1 } \iff \fa{n \in \NN}{b_n = 0} \ . \] 
Since 
\begin{align*}
\lnot \fa{n\in \NN}{a_n = 0} & \implies \lnot \lnot \fa{n \in \NN}{b_n = 0} \\ & \iff \fa{n \in \NN}{b_n = 0} \\
& \iff \ex{n \in \NN}{a_n = 1} 
\end{align*}
\MP holds. 

The rest is pretty much the well known argument from recursion theory that shows that  if a set $A \subset \NN$ and its complement $\overline{A}$ are recursively enumerable, then it is decidable.
So let $(a_n)_{n \geqslant 1}$ be arbitrary. By co-\PFP there exists $(b_n)_{n \geqslant 1}$, such that 
\[ \ex{n \in \NN}{a_n =1} \iff \fa{n \in \NN}{b_n = 0} \ . \]
Let $c_n = \max_{i \leqslant n} \menge{a_i,b_i}$. It is easy to see that $\lnot \fa{n \in \NN}{c_n = 0}$. My \MP that means that $\ex{n \in \NN}{c_n = 1}$, which in turn, allows us to decide whether $\ex{n \in \NN}{a_n =1}$ or $\ex{n \in \NN}{b_n =1}$, but in that second case $\fa{n \in \NN}{a_n =0}$. Hence \LPO holds.

Conversely it is straightforward to see that \LPO implies \coPFP.
\item This is simply the case since, if $(a_n)_{n \geqslant 1}$ and $(b_n)_{n \geqslant 1}$ are such that
\[
\fa{n\in \NN}{a_n = 0} \iff  \lnot \fa{n \in \NN}{b_n = 0}
\]
then by negating that relationship we get
\begin{align*}
\lnot \fa{n\in \NN}{a_n = 0} & \iff  \lnot \lnot \fa{n \in \NN}{b_n = 0} \\ & \iff \fa{n \in \NN}{b_n = 0} \ . \qedhere
\end{align*}
\end{enumerate}
\end{proof}

\section{Collapsing the Fan Theorems} \label{Sec:Fan_collapse}

As we have seen in Proposition \ref{Pro:BDN-impl} under the assumption of \BDN we have
$\FANc \implies \FANP$. This result was first proved in \cite{hD08b}. It is still interesting to give a direct proof, in order to understand the subtle difference between \FANc and \FANP. J.\ Berger has given such a proof, however has not published it. The following proof of $\FANc \implies \FANP$ follows basically Berger's argument.
\begin{Pro}
\BDN implies that $\FANc \implies \FANP$.
\end{Pro}
\begin{proof}
Assume $B$ is a bar and such that there are decidable $B_n$ such that $B = \cap B_n$, and is such that $B$ is closed under extensions.
Hence the set
\[A = \set{n \in \NN}{\ex{u \in \cS}{\abs{u}=n \land u \notin B}} \]
is countable. We want to show that it is also pseudobounded. To this end let $(a_n)_{n \geqslant 1}$ be a sequence in $A$. By definition there exists a sequence $(w_n)_{n \geqslant 1}$ in $\cS$ such that $\abs*{w_n}=a_n$ and $w_n \notin B$. The complement $C$ of 
\[ \set{w_n}{a_n \geqslant n} \] 
is decidable. We claim that 
\[ D = \set{u \in \cS}{\fa{v \in \cS}{u \ast v \in C}} \]
is a bar. But we know that $\overline{C} \subset \overline{B}$, and since $C$ is decidable we can reverse this inclusion to see that $B \subset C$. This immediately tells us that $C$ is a bar that is closed under extensions, which in turn implies that $D$ is a bar. Applying \FANc we get a uniform bound $N$ for $D$. Now there cannot be $i \geqslant N$ such that $a_i \geqslant i$: in that case we would have $w_i$ such that $\abs*{w_i} \notin C$ and $\abs*{w_i} \geqslant N$. But then also $w_i \notin D$, which is a contradiction to $N$ being a uniform bound for $D$. Altogether we get that $A$ is pseudobounded. Now we can use \BDN to conclude that $A$ is bounded, which immediately translates into $B$ being uniform.
\end{proof}
\begin{Rmk}
 Notice that the previous proof relies on the assumption that $B$ is closed under extension and therefore does not work for \FANst.
\end{Rmk}

Since \BDN is a very weak principle already \FANP and \FANc are intuitively ``close''. Since \UCT lies between these two, it is, still intuitively speaking, even closer to \FANc. As a matter of fact the author has spend considerable time during his PhD, trying to prove $\FANc \implies \UCT$ and has to admit to sometimes believe to actually having found such a proof. However, all these proofs turned out to be faulty, and often relied on implicitly using the following harmless looking choice principle (majorised choice).

\begin{principle}[MC]{\MC} \label{PR:MC}
	For every point-wise continuous function $f:\CS \to \RR$ there exists a point-wise continuous function $\tilde{f}:\CS \to \NN$ such that \[ \fa{\alpha \in \CS}{f(\alpha) \leqslant \tilde{f}(\alpha)} \ .\]
\end{principle}
\begin{Pro}
   $\UCT \iff \FANc + \MC$
\end{Pro}
\begin{proof}
One direction is clear. Conversely it suffices to show (Proposition \ref{Pro:more_Equiv_UCT}) that any point-wise continuous $f:\CS \to \RR$ is bounded. So consider such a function $f$. By \MC there exists a point-wise continuous $\tilde{f}:\CS \to \NN$ that bounds $f$. By \FANc this function is uniformly continuous (Proposition \ref{Pro:FANc-equivs}). That means that $\tilde{f}$ is bounded, and therefore so is $f$.
\end{proof}

\begin{Qu}
	Is there a (topological) model in which \MC fails?
\end{Qu}

As already mentioned, Brouwer never distinguished between different version of the fan theorem, since he assumed the principle of continuous choice. Following the notation of \cite{dB87} continuous choice can be split into two parts. \todo{Future work: Sort out mess here. CC1 is KLST, CC2 implies CC1? See 209 of TvD ConI. What is the relationship to \WCN}

\begin{principle}[CC]{\CC} \label{PR:CC} Continuous choice. \begin{description}
\item[CC(1)] \label{PR:CC1} Any function from $\BS \to \NN$ is point-wise continuous.
\item[CC(2)] \label{PR:CC2} If $P \subset \BS \times \NN$, and for each $\alpha \in \BS$ there exists $n \in \NN $ such that $(\alpha,n) \in P$, then there is a function $f : \BS \to \NN$ such that $(\alpha,f(\alpha)) \in P$ for all $\alpha \in \BS$.
\end{description}
\end{principle}

\begin{Pro}
Under the assumption of \CC we have $\FAND \implies \FANf$.
\end{Pro}
\begin{proof}
Let $B \subseteq \cS$ be an arbitrary bar. Let $P=\set{(\alpha,n)}{\overline{\alpha}n \in B}$ and apply \CC(2). So there is a function $f:\CS \to \NN$ such that $\overline{\alpha}f(\alpha) \in B$ for all $\alpha$. Now define 
\[ B^\prime = \set{u \in \cS}{f(u \ast 0 \dots) \leqslant \abs{u}} \ . \]
$B^\prime$ is decidable. It is also a bar, since for $\alpha \in \CS$ we have, by \CC(1) that $f$ is continuous at $\alpha$; that is there exists $N$ such that $\overline{\alpha}N=\overline{\beta}N$ implies that $f(\alpha)=f(\beta)$. Now let $M=\max \menge{N,f(\alpha)}$. Then
\[ f(\overline{\alpha} M \ast \zero) = f(\alpha) \leqslant M \ . \] So $\overline{\alpha} M \in B^\prime$. 

Applying \FAND we get a uniform bound $K$ for $B^\prime$. It is easy to see that this bound $K$ is also a uniform bound for $B$.
\end{proof}

\begin{Pro}
Under the assumption of \KS we have $\FAND \implies \FANf$.
\end{Pro}
\begin{proof}
If $B$ is an arbitrary bar, then, using \KS, there exists, for every $u \in \cS$ a sequence $(a^u_n)_{n \geqslant 1}$ such that 
\[ u \in B \iff \ex{n \in \NN}{a^u_n =1} \ .\]
So $B$ is countable (notice that any bar is inhabited, since $\overline{\zero }n$ must be in $B$ for some $n$). By Lemma \ref{countbar} there exists a decidable bar $B^\prime$ that is uniform only if $B$ is. Hence, using \FAND, is enough to give us a uniform bound for $B$.
\end{proof}

\begin{Pro}
Under the assumption of \PFP we have $\FAND \implies \FANst$.
\end{Pro}
\begin{proof}
The proof is similar to the one above. Starting with a stable bar $B$ we can use \PFP to conclude that $B$ is countable. Using Lemma \ref{countbar} and \FAND again, we can conclude that $B$ is uniform.
\end{proof}

Summarising all these implications in a diagram we get:

\begin{center}
\begin{tikzpicture}[node distance=2 cm, auto]
  \node (FAND)  {\FAND};
  \node (FANc) [right of=FAND] {\FANc};
  \node (UCT) [right of=FANc] {\UCT};
  \node (FANP) [right of=UCT] {\FANP};
  \node (FANst) [right of=FANP] {\FANst};
  \node (FANfull) [right of=FANst] {\FANf};

\draw[->] (FANc) to node {} (FAND);
\draw[->] (UCT) to node {} (FANc);
\draw[->] (FANP) to node {} (UCT);
\draw[->] (FANst) to node {} (FANP);
\draw[->] (FANfull) to node {} (FANst);

  \draw[->,bend right=35] (FANc) to node {\MC} (UCT);
  \draw[->,bend left=25] (FANc) to node {\BDN} (FANP);
  \draw[->,bend right=25] (FAND) to node {\PFP} (FANst);
  \draw[->,bend left=25] (FAND) to node {\CC/\KS} (FANfull);

\end{tikzpicture}
\end{center}

\section{Other Implications}

\begin{Pro} \label{Pro:WMP+WLPO_equiv_LPO}
\WMP implies that \LPO and \WLPO are equivalent.
\end{Pro}
\begin{proof}
Consider $x \in \RR$, such that $x \geqslant 0$. By \WLPO we know that either $x=0$ or $\neg(x=0)$; or with the notation introduced in Chapter \ref{Ch:MP} whether $x=0$ or $0 \lessdot x$. In the second case we can, using \WLPO again, for any $y \in \RR$ decide whether
\[ z=0  \lor \neg(z=0)  \ , \] 
where $z= \max \menge{x-y, 0}$. In the first case $z \leqslant 0$, which implies that $x \leqslant y$. Since $0 \lessdot x$, also $0 \lessdot y$ (Lemma \ref{Lem:pos-order}). In the second case $0 \lessdot z$, which implies that $y \lessdot x$. Altogether we can decide, for arbitrary $y \in \RR$ whether
\[ 0 \lessdot y \lor y \lessdot x \ . \]
So, using \WMP, we can conclude that $x>0$; and hence \LPO holds.
\end{proof}
\begin{Rmk}
Therefore \WMP also implies that  \nLPO and  \nWLPO are equivalent. Here's an analytical proof for the latter:
 We only need to show that \nLPO implies \nWLPO. So assume \nLPO and $\WLPO$. Hence there exists a function $f:[0,1]$ such that $f(0) = 0$ and $f(x) =1$ for $\lnot (x = 1)$. By \WMP this function is strongly extensional. Using Ishihara's tricks \cite{hD12b} \LPO holds; a contradiction.
\end{Rmk}

A principle that collapses \WLPO and \LLPO is \PFP.
\begin{Pro}
\WPFP implies that \WLPO and \LLPO are equivalent.
\end{Pro}
\begin{proof}
Let $(a_n)_{n \geqslant 1}$ be a binary sequence with at most one $1$. By \WPFP there exists a binary sequence $(b_n)_{n \geqslant 1}$, such that 
\[ \fa{n \in \NN}{a_n=0} \iff \lnot (\fa{n \in \NN}{b_n =0}) \ . \]
We may also, without loss of generality assume that $(b_n)_{n \geqslant 1}$ contains at most one $1$. It is, furthermore easy to see that it cannot be the case that both sequences contain a $1$. So we can use \LLPO to decide whether $\fa{n \in \NN}{a_n=0}$ or $\fa{n \in \NN}{b_n=0}$. In the second case $\lnot (\fa{n \in \NN}{a_n=0})$ and hence \WLPO holds. 
\end{proof}

\begin{Pro}
\CC(1) implies \BDN.
\end{Pro}
\begin{proof}
A simple consequence of Proposition \ref{Pro:BDN-equivs}.
\end{proof}

\section{The Big Picture} \label{Ch:BigPicture}
As a handy overview of the relationship between most of the principles discussed we include the following diagram. Dotted lines indicate contradictions.
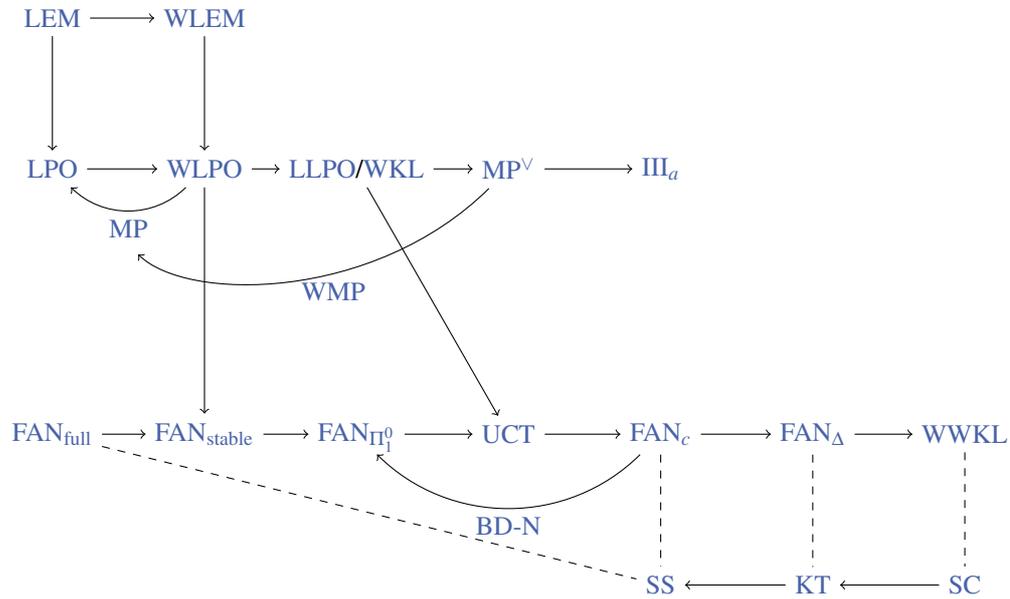
\begin{figure}[h!]
\tikzset{external/export next=true}
\begin{tikzpicture}[node distance=2 cm, auto]
  \node (LEM) {\LEM};
  \node (WLEM) [right of=LEM] {\WLEM};
  \node (LPO) [below of=LEM] {\LPO};
  \node (WLPO) [right of=LPO] {\WLPO};
  \node (LLPO) [right of=WLPO] {\LLPO/\WKL};
  \node (MPv) [right of=LLPO] {\MPv};
  \node (IIIa) [right of=MPv] {\IIIa};

  \draw[->] (LEM) to node {} (WLEM);
  \draw[->] (LEM) to node {} (LPO);
  \draw[->] (LPO) to node {} (WLPO);
  \draw[->] (WLEM) to node {} (WLPO);
  \draw[->] (WLPO) to node {} (LLPO);
  \draw[->] (LLPO) to node {} (MPv);
  \draw[->] (MPv) to node {} (IIIa);

  \draw[->,bend left=45] (WLPO) to node {\MP} (LPO);
  \draw[->,bend left=45,shorten >=1.25cm] (MPv) to node {\qquad \qquad \WMP} (LPO);

    \node[below  = 3 cm   of WLPO] (FANst) {\FANst};
    \node[left of=FANst] (FANf) {\FANf};
    \node[right of=FANst] (FANP) {\FANP};
    \node[right of=FANP] (UCT) {\UCT};
    \node[right of=UCT] (FANc) {\FANc};
    \node[right of=FANc] (FAND) {\FAND};
    \node[right of=FAND] (WWKL) {\WWKL};

    \node[below of=FANc] (SS) {\SS};
    \node[below of=FAND] (KT) {\KT};
    \node[below of=WWKL] (SC) {\SinC};

  \draw[->] (SC) to node {} (KT);
  \draw[->] (KT) to node {} (SS);

  \draw[dashed] (WWKL) to node {} (SC);
  \draw[dashed] (FAND) to node {} (KT);
  \draw[dashed] (FANc) to node {} (SS);
  \draw[dashed] (FANf) to node {} (SS);

  \draw[->] (FANf) to node {} (FANst);
  \draw[->] (FANst) to node {} (FANP);
  \draw[->] (FANP) to node {} (UCT);
  \draw[->] (UCT) to node {} (FANc);
  \draw[->] (FANc) to node {} (FAND);
  \draw[->] (FAND) to node {} (WWKL);
  \draw[->,bend left=45] (FANc) to node {\BDN} (FANP);

  \draw[->] (LLPO) to node {} (UCT);
  \draw[->] (WLPO) to node {} (FANst);
\end{tikzpicture}
\end{figure}

\chapter{Separating Principles} \label{Ch:separating}
\section{The Big Three} \label{Sec:varieties}
The easiest and most convenient way to see that principle $A$ does not imply principle $B$, or more general, that theorem $T$ is  \emph{not provable} in \BISH is to show that theorem $T$ is \emph{false} in classical mathematics (\CLASS), Brouwer's intuitionism (\INT), or in Markov's recursive school of mathematics (\RUSS). 
The view that all three of these varieties of mathematics can be seen as models of \BISH has slowly evolved over the years until it was cemented by the publication of \cite{dB87}. Of course, we use the term model somewhat loosely here, since none of \BISH, \CLASS, \INT, or \RUSS are fully formalised systems.\footnote{Is countable choice accepted in \BISH? Is Kripke's Schema part of \INT? Is the continuum hypothesis true in \CLASS?} Nevertheless, experience has shown that any proof given in \BISH can be immediately interpreted and is acceptable by a practitioner of any of the three varieties. We should stress that one should not make the mistake of characterising \BISH as the intersection of these three varieties, since there are statements namely \BDN, \WMP, and possibly others, that are acceptable in \CLASS, \INT, and \RUSS, but for which there is no compelling reason to accept them in \BISH.
The three varieties are all well understood so they immediately show many separations between the principles we have discussed so far. For example \MP does not imply \FAND, since the first one is true in \RUSS, but the second one is false. Many more separations can be read off the diagram below.

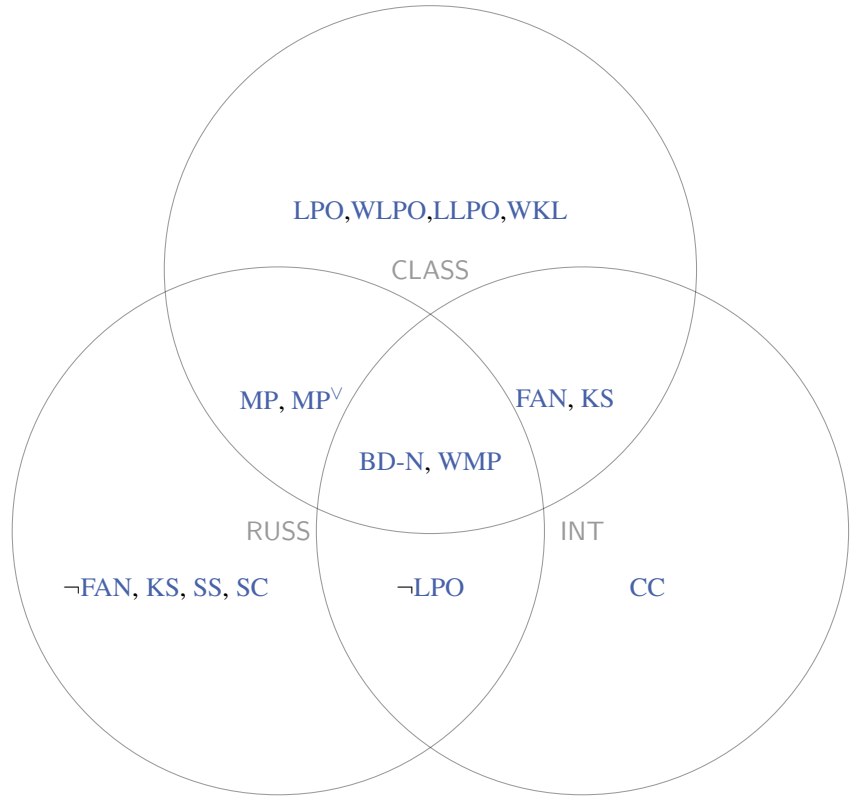
\begin{center}
\tikzset{external/export next=true}
\begin{tikzpicture}
  \tikzset{venn circle/.style={draw,circle,minimum width=7cm,opacity=0.4}}
  \node [venn circle] (A) at (0,0) {\RUSS};
  \node [venn circle] (B) at (60:4cm) {\CLASS};
  \node [venn circle] (C) at (0:4cm) {\INT};
  \node[left] at (barycentric cs:A=1/2,B=1/2 ) {\MP, \MPv}; 
  \node[below=0.5cm] at (barycentric cs:A=1/2,C=1/2 ) {$\lnot \LPO$};   
  \node[right] at (barycentric cs:B=1/2,C=1/2 ) {\FAN, \KS};   
  \node[below] at (barycentric cs:A=1/3,B=1/3,C=1/3 ){\BDN, \WMP};
  \node[above=0.5cm] at (barycentric cs:B=1/2){\LPO,\WLPO,\LLPO,\WKL};
  \node[below right = 0.5cm and 0.5cm] at (barycentric cs:C=1/2 ) {\CC};   
  \node[below left = 0.5cm and 0cm] at (barycentric cs:A=1/2 ) {$\neg$\FAN, \KS, \SS, \SinC};   

\end{tikzpicture}  
\end{center}

\section{Topological and Heyting-valued Models} \label{Sec:topmodels}
Topological models are a natural setting to interpret formalised intuitionistic theories. By ``intuitionistic'' we mean theories using intuitionistic logic; it is worth noting though, that topological models also have a distinct intuitionistic flavour a'la Brouwer. For example they all validate \FANf~(Proposition \ref{Pro:TopmodFanf}).
Even though they have a long history starting with several publications around 1970 \cites{dS68,dS70,rG84,mF79}, according to our personal judgement, there is no good introduction available (the subsection in van Dalen's chapter on intuitionistic logic in \cite{goble2001blackwell} is probably the most accessible text). This section will not remedy this situation; its aim is to give a decent overview and starting point for researchers working in constructive reverse mathematics. To poach a famous booktitle the theme for this section is ``Topological models for the working constructivist.''

\emph{Notice that, for this section, we work with classical logic in the meta-theory.}\footnote{This is almost unavoidable for our purposes. For a lot of omniscience principles we have the rule of thumb that if the principle holds in the model it also holds in the meta-theory. If one was only interested in principles from \INT one could potentially develop topological models in \BISH (or \INT). Notice, however, that then one would have to deal with problems such as Lemma \ref{Lem:topmodels_impl} not being available.}

\subsection{Propositional Logic} 
The basic idea of topological models is to use open sets the truth values. As usual, the propositional case is a lot easier and cleaner to deal with than the predicate case.

A topological model for propositional intuitionistic logic consists of a topological space $(T,\tau)$, and a function $\ext{\cdot}$ that maps all propositional symbols of the underlying language to elements of $\tau$ and is such that $\ext{\bot} = \emptyset$. We can then extend this function to one defined on all propositional formulas by setting
\begin{align*}
 \ext{A \land B} &= \ext{A} \cap \ext{B} \ ,\\
 \ext{A \lor B} &= \ext{A} \cup \ext{B} \ , \\
 \ext{A \rightarrow B} &= \Interior{\Complement{\ext{A}} \cup \ext{B}}  \ .\footnote{Note that we use Halmos' notation $\Complement{A}$ to denote the  complement of a set $A$.}
\end{align*}
We assume, as usual, that negation is defined as $\lnot \varphi \equiv \varphi \rightarrow \bot$, so
\[ \ext{\lnot \varphi} = \Interior{\Complement{\ext{\varphi}}} \ .\]
We say that $T$ \define{models} $\varphi$, if $\ext{\varphi} = T$. In this case we use the notation $T \Vdash \varphi$. As common, the interpretation $\ext{\cdot}$ should be clear from the context and is therefore omitted from the notation. 
\begin{Lem} \label{Lem:topmodels_impl}
$T \Vdash \varphi \rightarrow \psi $ if and only if we have $\ext{\varphi} \subset \ext{\psi}$
\end{Lem}
\begin{proof}
Straightforward. Notice, however, that we need \LEM to prove that $\ext{\varphi} \subset \ext{\psi}$ implies $T \Vdash \varphi \rightarrow \psi $.
\end{proof}

One can easily show that we have soundness, that is if $\varphi$ is a propositional formula that is derivable in intuitionistic logic from a set of propositional formulae $\Gamma$, i.e.\ $\Gamma \vdash_i \varphi$, then $T \Vdash \varphi$ for any topological space $T$ that models every formula in $\Gamma$. Therefore, we can use topological models to show the unprovability of many statements.

\begin{Pro}
\LEM and \WLEM are not derivable in intuitionistic logic.
\end{Pro}
\begin{proof}
Let $T=\RR$ with the usual topology. Furthermore let $\ext{\cdot}$ be such that $\ext{P}=(-\infty,0)$. Then $\ext{\lnot P}= (0,\infty)$ and therefore $\ext{\lnot \neg P} = (-\infty,0)$. However $\RR \neq \ext{\lnot P \lor \lnot \neg P} = (-\infty,0) \cup (0,\infty)$. 
\end{proof}

This space $T$ is actually not the simplest model showing that \WLEM is not derivable in intuitionistic logic. Such a model must contain at least three points: in fact, if a space only contains two points there are only three topologies: the trivial one, the discrete one, and the Sierpinski space $\Sigma=\left(\menge{0,1}, \menge{\emptyset, \menge{1}, \menge{0,1}}  \right)$. The trivial and the discrete one satisfy \LEM, the Sierpinski one validates \WLEM but not \LEM. Hence $\Sigma$ shows that \WLEM does not imply \LEM.

The simplest topology not validating \WLEM is $ T_2 = (  \menge{1,2,3} , \menge{  \emptyset, \menge{1}, \menge{2},  \menge{1,2},\menge{1,2,3} }  )$. Then, if $P$ is a propositional symbol such that $\ext{P} = \menge{1}$ we get that \[ \ext{\lnot P} = \Interior{\Complement{\menge{1}}} = \menge{2} \] and similarly  \[\ext{\lnot \neg P} = \Interior{\Complement{\menge{2}}} = \menge{1} \ . \]  Thus \[ T_2 \nVdash \lnot P \lor \lnot \neg P \ . \]

In this sense we can talk about $\RR$ or $\Sigma$ as a topological model being a \emph{counterexample} to \LEM. However, we have to be careful that there are two ways a statement can fail to hold. We will say that $T$ is a \define{weak counterexample} for $\varphi$ if $T \nVdash \varphi$. It is a \define{strong counterexample} if $T \Vdash \lnot \varphi$. Since we can prove $\vdash_i \neg(\varphi \land \lnot \varphi)$ we will never be able to find a strong counterexample to \LEM or \WLEM. Nevertheless for weaker principles such as \LPO we can, as we will see, find topological models that are strong counterexamples.

\subsection{Predicate Logic}
To extend the topological interpretation to predicate logic we also need some universe $\mathcal{U}$ to interpret constants and variables. A predicate $\ext{P(x_1, \dots, x_k)}$ should be mapped to a function $\mathcal{U}^k \to \tau$.
It is convenient for $d \in \mathcal{U}$ and a formula $\varphi(x)$ to simply write $\ext{\varphi(d)}$ for the truth value of $\ext{\varphi(x)}$ evaluated at $d$.

Since we want to interpret equality, we also assume that $\ext{x=y}$ maps to a function $\mathcal{U}^2 \to \tau$, such that 
\[ \ext{a=b} = \ext{b=a} \]
and
\[ \ext{a=b} \cap \ext{b=c} \subset \ext{a=c} \ . \]
Now all predicate symbols $P(x_1, \dots, x_k)$ should be mapped to a function $\ext{P(x_1, \dots, x_k)}: \mathcal{U}^k \to \tau$ such that equality is respected; that is
\[ \ext{x=y} \cap \ext{P(x)} = \ext{x=y}\cap \ext{P(y)}  \ .\]

\begin{Rmk}
We can extend every predicate evaluation that is defined on $\mathcal{V} \subset \mathcal{U}$ and that respects the equality.
\end{Rmk}
\begin{proof}
Assume that $k \in \NN$ is fixed. Now extend $\ext{\varphi(\vec{x})} : \mathcal{V}^k \to \tau$  to a function $\mathcal{U}^k \to \tau$ by setting
\[\ext{\varphi(\vec{b})} = \bigcup_{\vec{a} \in \mathcal{V}^k} \left( \ext{\vec{a} = \vec{b}} \cap \ext{\varphi(\vec{a})} \right)  \ .\]
It is routine to check that this definition is in fact an extension and that it respects equality.
\end{proof}

Finally we can give the interpretation of the quantifiers:
\begin{align*}
\ext{ \ex{x}{\varphi(x)}} &= \bigcup_{d \in \mathcal{U}} \ext{\varphi(d)} \ , \\
\ext{ \fa{x}{\varphi(x)}} &= \Interior{\bigcap_{d \in \mathcal{U}} \ext{\varphi(d)}} \ .
\end{align*}
All of this naturally extends to first order predicate logic with types. 
\emph{Since it will not affect the following, we will not distinguish between having a logic with different types or having some sort of set theory.}

Again one can easily show soundness, by induction over deductions.
Notice that in general we do not have an existence property, that is we do not have that if $T \Vdash \ex{x}{\varphi(x)}$, then there exists $d \in \mathcal{U}$ such that $T \Vdash \varphi(d)$.

\subsection{Topological Models of Arithmetic and of Analysis}

Naturally, we want to consider models that validate, at least, Heyting arithmetic. To be precise, we assume that our language contains a constant $0$ and function symbols $s, +, \cdot$ and that a model validates the axioms
\begin{multicols}{2}
\begin{enumerate}[label=H\arabic*]
\item $\fa{x}{\lnot (x = s(0))}$
\item $\fa{x,y}{s(x)=s(y) \rightarrow x=y}$
\item $\fa{x}{x+0 = x}$
\item $\fa{x,y}{x+s(y) = s(x+y)}$
\item $\fa{x}{x \cdot 0 = 0}$
\item $\fa{x,y}{x \cdot s(y) = x \cdot y + y}$
\item $ \left(\varphi(0) \land \fa{x}{\varphi(x) \rightarrow \varphi(x+1) } \right) \rightarrow \fa{x}{\varphi(x)}$.
\end{enumerate}
\end{multicols}
Here the last axiom is, as usual, actually an axiom schema for all statements $\varphi$. And, of course, all variables are of a fixed natural number type.

We have a natural embedding of the external natural numbers into the language of the model, via \[ n \mapsto s^n(0) \ . \] To distinguish the external and the internal natural numbers we will denote the embedded ones by $\breve{n}$. Internally there might be more natural numbers than externally: Given any set $S$ equipped with the discrete topology, it is easy to check that the functions $S \to \NN$ with the obvious interpretation form a model of Heyting arithmetic. So, for example, in the case that $S = \menge{0,1}$ the identity function  is a natural number with name, say $c$, which is  not of the form $s^n(0)$. That is there is no $n$ such that $S \Vdash \breve{n} = c$. 
 Of course locally we have \[ \menge{0} \Vdash c = \breve{0}  \text{ and } \menge{1} \Vdash c = \breve{1} \ , \] so these new numbers are not really adding any interesting new behaviour. 
 There might, however, also be more internal numbers locally---just think of any non-standard model. We would like to exclude this non-standard case. To be precise, we would like to exclude the case that the internal and the external numbers differ locally---if we assume non-standard numbers in the meta-theory we also have these in the model. One way of ensuring this is by assuming that our model is full in the sense that we have have names for all predicates that we can define on the topological space.

\begin{Pro}  \label{Pro:ext_nat}
Consider $(T,\tau,\ext{\cdot})$ a model of Heyting arithmetic and assume that there exists a predicate $N(x)$ where $x$ is of type $\NN$ such that 
  \[ \ext{N(x)} = \bigcup_{n \in \NN} \ext{x = \breve{n}} . \]

Then, if $t \in \ext{\varphi(x)}$, where $x$ is of type $\NN$, there exists $n$ such that $t \in \ext{\varphi(\breve{n})}$. In other words, we can export natural numbers locally.

\end{Pro}
\begin{proof} It is enough to 
use natural induction (H7) to show that 
\[ T \Vdash x \in \NN \rightarrow N(x) \ . \]
The  case $x = 0$ is obviously fine. So let $x$ be an arbitrary variable symbol of the natural number type, and let $t \in T$ be arbitrary. 

By the induction hypothesis there exists a neighbourhood $U$ of $t$ and an external natural number $n$ such that $U \Vdash x = \breve{n}$. So we also have  $U \Vdash s(x) = s(\breve{n})$. If we can show that $T \Vdash s(\breve{n}) = \breve{n+1}$ we are done. But that holds by natural induction (in the metatheory) using H2. 
\end{proof}

\emph{From now on we are only going to consider topological models, that are models of HA and that satisfy the conclusion of the previous proposition. The full models introduced in the next subsection are such models.}

Let us move  on to models of the real numbers. That is our language contains constants $0,1$ binary function symbols $+, \cdot$, unary function symbols $-, ^{-1}$ and a binary relation $>$ such that the usual axioms are satisfied (for a detailed list of axioms see \cite{dB99b}).

If we have a model $T$ of the real numbers in this sense any internal real can be thought of as a continuous function $T \to \RR$.

\begin{Pro} \label{Pro:int_ext_reals}
Consider $(T,\tau,\ext{\cdot})$ a model of Heyting arithmetic and of real numbers.
 If $x$ is of type $\RR$, then there exists a continuous function $f:T \to \RR$ such that 
  \begin{equation} \label{Eqn:ext_reals}
  	\ext{\breve{r} < x< \breve{s} } = \set{t \in T}{r < f(t) <s}  \ ,
  \end{equation} 
  for all $r, s \in \QQ$.
  Therefore also, in particular, 
\begin{align*}
\ext{x>0} & = \set{t}{f(t) > 0} \ , \\
\ext{x=0} & =  \Interior{\set{t}{f(t) = 0}} \ , \\
\ext{x \geqslant 0} & =  \Interior{\set{t}{f(t) \geqslant 0}} \ , \\ 
\ext{\neg(x=0)} & = \Interior{\Complement{\Interior{\set{t}{f(t) = 0}}}}  \ .
\end{align*}
\end{Pro}
\begin{proof}
First, in order to define $f$, we will show that for $t_0 \in T$ the set 
\[ D_{t_0} = \set{q \in \QQ}{t_0 \in \ext{\breve{q} < x}} \ ,\] 
is a Dedekind cut.
\begin{itemize}
  \item $D_{t_0}$ is bounded and inhabited: since 
\[ T \Vdash \ex{q, p \in \QQ}{q < x< p }  \ , \]
we can use Proposition \ref{Pro:ext_nat} to export the rationals (as pairs of natural numbers) locally. That is there exists a neighbourhood $U$ of $t_0$ and $q, p \in \QQ$\footnote{Here and below we abuse our notation slightly by reusing the same name for external and internal variables.} such that 
\[ U \Vdash \breve{q} < x< \breve{p}  \ . \]
Clearly $q \in D_{t_0}$ and $p \notin D_{t_0}$.
\item $D_{t_0}$ is downward closed: if $q \in D_{t_0}$ and $p < q$ then  
$T = \ext{\breve{p} < \breve{q}}$ and therefore \[ t_0 \in \ext{\breve{q} < x} \subset \ext{\breve{p} < x} \ . \] Thus $p \in D_{t_0}$.
\item Even though we are working with a classical meta-theory it seems worth pointing out that, similarly, we can also show that $D_{t_0}$ is order located.
\end{itemize}

Now consider the function $f$ that maps every $t \in T$ to the real given by the Dedekind cut $D_t$. We want to show that this function is point-wise continuous. It suffices to show that, given $\varepsilon > 0$, we can find an open neighbourhood $U$ of $t_0$ such that 
\[ \fa{t,t^\prime \in U, q \in f(t)}{\ex{p\in f(t^\prime) }{ \abs*{q-p} < \varepsilon}} \ . \]
So let $\varepsilon >0$ and $t_0 \in T$ be arbitrary. We may, without loss of generality, assume that $\varepsilon \in \QQ$. Since
\[ T \Vdash \ex{r \in \QQ}{x-\breve{\varepsilon} < r < x }  \ , \]
there exists an open neighbourhood $U$ of $t_0$ and $r \in \QQ$ such that 
\[ U \Vdash x-\breve{\varepsilon} < \breve{r} < x   \ . \]
Now let $t,t^\prime \in U$ and $q \in f(t)$. The latter implies that  $V \Vdash \breve{q} <  x$ for some open neighbourhood $V$ of $t$ and therefore 
$V \Vdash \breve{q} - \varepsilon < \breve{r}$.
Since we are dealing with imported reals that means that $q - \varepsilon < r$. Now set $p = q-\varepsilon$. Then $p < r $, which means that $U \Vdash \breve{p} < \breve{r} < x$, and therefore $p \in f(t^\prime)$. 

We will show \ref{Eqn:ext_reals} by showing that the left hand side set is a subset of the right hand one and vice versa. 
So let $t \in \ext{\breve{r} < x < \breve{s}}$. Similar to above, we can find a neighbourhood $U$ of $t$ and $p,q \in \QQ$ such that $U \Vdash \breve{r} < \breve{p} < x < \breve{q} < s$. Then, in terms of Dedekind cuts  $p \in f(t)$ and $q \notin f(t)$ for all $t \in U$ and therefore, seen as reals $r < f(t) < s$. 

Conversely, let $t_0 \in T$ such that $r < f(t_0) < s$. By continuity, we may choose a neighbourhood $U$ of $t_0$ and $q \in \QQ$ such that $r  < f(t) < q < s$ for all $t \in U$. (The asymmetry in needing $q$ on the right hand side, but no rational on the left comes from the fact that Dedekind cuts are somewhat asymmetric). Since $r < f(t)$ on $U$ that means that for any $t \in U$ we have $r \in D_{t_0}$, which in turn means that $ U \subset \ext{\breve{r} < x}$. Similarly we can show that for any $t \in U $  we have $t \notin  \ext{ x < \breve{q}}$. Thus 
\[ U \subset \ext{\lnot ( \breve{q} < x)} \subset \Complement{\ext{\breve{q} < x}} \subset \ext{x < \breve{s}} ; \]
because we have assumed that $T$ is a model for the constructive reals and therefore \[ T = \ext{ \breve{q}< x \ \lor \ x < \breve{s}} =  \ext{\breve{q}< x} \ \cup \  \ext{x < \breve{s}} \ . \] 
Together $t_0 \in \ext{\breve{r} < x< \breve{s} } $.
\end{proof}

At this point we have not assumed that we also have the converse, that is that for every continuous real-valued function $f:T \to \RR$ we have a constant symbol which has this real as its representative. However, we will do so in the next section. To our knowledge, models that are not full in this sense have not been considered as models of analysis.

\subsection{The Full Model and Countable Choice} \label{Sec:Choice}

The commonly used models are the ``full'' ones, whose existence is guaranteed by the following.

\begin{Pro}
For any topological space $(T,\tau)$ there exists a model such that
\begin{enumerate}
  \item It is a model of IZF (therefore also of CZF, HA, and the real numbers).
  \item For every $V \in \tau$ there is a proposition $P_V$ such that $\ext{P_V} = V$.

\end{enumerate}
 In particular, that means that we can use Propositions \ref{Pro:ext_nat} and \ref{Pro:int_ext_reals}, which imply that any internal real number $x$ is represented by a continuous function $f_x: T \to \RR$. Fullness also guarantees the converse, namely that 
  for every continuous function $f: T \to \RR$ there exists a name for a real number $x_f$ such that $f_{x_f} = f$.
\end{Pro}
\begin{proof}
	See the section ``Topological Models'' in \cite{mH15}.
\end{proof}

We should note that the first point does not cover any form of choice. As a matter of fact, even countable choice often fails in topological models. We therefore need to be a bit careful when talking about, for example, \LPO being satisfied, since we need to specify whether we mean the sequential version $\LPO_\sigma$ or the one for the real numbers $\LPO_\RR$. We have that $\LPO_\RR$ implies $\LPO_\sigma$, but the converse is only true in the presence of countable choice, as we can see by considering the topological model over the reals. There  $\LPO_\RR$ fails: Let $z$ be the real given by the identity function. Then
\[ \ext{\neg(z = 0)} = \Interior{\Complement{\ext{z=0}}} = \Interior{\Complement{ \Interior{\menge{0}}}} = \Interior{\RR} = \RR \ ,  \]
 but
 \[\ext{z \neq 0} = (-\infty,0) \cup (0, \infty)  \ . \]
 On the other hand if $m$ is an (internal) natural number, then,  it is represented by a point-wise continuous function $f_m:\RR \to \NN$. But since the only such functions are the constant ones we have that for every name of a natural number $m$ there exists an external natural number $n$ such that $\RR \Vdash m = \breve{n}$. Extending this argument to binary sequence we get that the internal binary sequences are exactly the external ones. Thus $\LPO_\sigma$ holds, since we are working with a classical metatheory.
 
 Together we have 
 \[  \RR \Vdash \LPO_\sigma \text{ but } \RR \nVdash \LPO_\RR  \ , \]
 and therefore 
 \[  \RR \nVdash \ACC \ . \]

For the rest of the section we will keep the distinction between the real and the sequential version of our principles. Notice that for almost all principles the real version is stronger than the sequential one, and different sequential ones are equivalent to each other, as well as are the real ones. 

\subsection{Reverse Reverse Mathematics} \label{SubS:ReverseReverse}

In the following we give some characterisation of properties of topological spaces, such that the full model satisfies certain principles. This is actually quite a natural question to consider, and these kind of results are very helpful to find custom separations of principles. We have---not entirely seriously---named this area ``Reverse Reverse Mathematics.'' The assumption that we are working with full models is essential for this approach.

For example, it is easy to characterise the spaces that satisfy \LEM.
\begin{Pro} \label{Pro:disc_top_LEM}
A topological space $(X,\tau)$ validates $\LEM$ if and only if every open set is also closed.
\end{Pro}
\begin{proof}
One direction is trivial. Conversely let $U \in \tau$. Since we consider \emph{full} models, there exists $P$ such that $\ext{P} = U$. Since $X = \ext{P \lor \lnot P} = \ext{P} \cup \ext{\lnot P}$,  the complement $ \Complement{U} = \ext{\lnot P}$ of $U$ is open, whence $U$ is closed.
\end{proof}

\begin{Cor}
For the Sierpinski space $\Sigma=\left(\menge{0,1}, \menge{\emptyset, \menge{1}, \menge{0,1}}  \right)$, we have $\Sigma \nVdash \LEM$ but also $\Sigma \Vdash \LPO_\RR$. Hence $\Sigma \nVdash \KS_\RR$.
\end{Cor}
\begin{proof}
To see that $\Sigma$ satisfies $\LPO_\RR$, simply note that all continuous functions $f:\Sigma \to \RR$ are constant.\footnote{As M.~Hendtlass has pointed out, actually, every finite topological model satisfies \LPO.}	
\end{proof}

\begin{Cor*}[\textbf{\theThm ½}]
If $(X, \tau)$ does not have an accumulation point then $X \Vdash \LEM$.
\end{Cor*}
\begin{proof}
	If no point is an accumulation point, then $\tau$ is the discrete topology, which means every set is clopen.
\end{proof}

\begin{Rmk}
The Sierpinski space validates $\Sigma \Vdash \WLEM$.
\end{Rmk}
\begin{proof}
Assume $\varphi$ is an arbitrary formula. If either $\ext{\varphi}=\emptyset$ or $\ext{\varphi} = \Sigma$ we have $\Sigma \Vdash \lnot \varphi \lor \lnot \neg \varphi$. But also in the case that $\ext{\varphi} = \menge{1}$ we have that $\ext{\lnot \varphi} = \Interior{\menge{0}} = \emptyset$ and therefore $X \Vdash \lnot \neg \varphi$.
\end{proof}

We remind the reader that a topological space $X$ is called \define{functionally Hausdorff}, if for any two points $x,y$ there exists a continuous function $f:X \to \RR$ such that $f(x) = 0$ and $f(y) = 1$. Any such space is necessarily Hausdorff. Conversely every normal, Hausdorff (i.e.\ a T4) space is functionally Hausdorff.\footnote{Slightly more obscure: a completely regular, T1 space---a so called Tychonoff space---is functionally Hausdorff.}

\begin{Pro} \label{Pro:topspace_omin_or_notMPv}
Assume $(T,\tau)$ is a first countable, functionally Hausdorff space, that has an accumulation point. Then  $T \nVdash \LLPO_\RR$.
\end{Pro}
\begin{proof} 
Consider an accumulation point $t_{0} \in T$ and choose a countable base of neighbourhoods $(U_n)_{n \geqslant 1}$ of $t_0$. For every $n \in \NN$ choose $t_n \in U_n$ such that $t_n \neq t_0$. Since $T$ is functionally Hausdorff, for every $n \in \NN$ there exists a continuous function $h_n:T \to [0,1]$ such that $h_n(t_0) = 0$ and $h_n(t_n) = 1$.

Now the function
\[ H(x) = \sum_{n \in \NN} \frac{h_n(x)}{2^n} \ , \] 
is well-defined and continuous. It also satisfies $H(t_0) = 0$ and $H(t_n) \geqslant \frac{1}{2^n}$ for all $n \in \NN$. Since $t_n \in U_n$ and the latter forms a neighbourhood base we have $H(t_n) \to 0$. We may assume, without loss of generality (by switching to an appropriate subsequence), that $H(t_n) > H(t_{n+1})$. Now choose reals $s_n$, $t_n$ such that 
\[ H(t_{n+1}) < s_n < r_n < H(t_n)  \ .\]

Let $f: \RR \to \RR$ be the piece-wise linear function that is $0$ for $x \leqslant 0$ and is otherwise such that 
$f (x) = x $ for $x \in \bracks*{r_{n+1},s_{n}}$ whenever $n$ is even, and $f(x) = -x$ for $x \in \bracks*{r_{n+1},s_{n}}$ whenever $n$ is odd. An illustration of $f$ can be found in Figure \ref{Fig:funch1}.
\begin{figure} 
\begin{center}
\tikzset{external/export next=true}
\begin{tikzpicture}
	\begin{axis}
		[
			every extra x tick/.style={
				grid=none,
				tick0/.initial=1.9em,
		        tick1/.initial=-0.9em,
				xticklabel style={font=\footnotesize,
					yshift=\pgfkeysvalueof{/pgfplots/tick\ticknum},
				},
			},
			axis x line=middle,
			axis y line=middle,
			enlarge x limits=0  0.1,
			enlarge y limits=0.1,
			xtick={ 5,6,7},
			xticklabels={ $s_n$, $r_n$, $\ s_{n-1}$},
			ytick=\empty,
			yticklabels={},
			xticklabel style = {yshift=0ex},
			extra x ticks={4.5, 6.5 },
		extra x tick labels={$f(t_{n+1})$, $f(t_n)$ }]	
		
		\addplot[very thick] coordinates 
		{ (0.25,-0.25)  (0.3125,-0.3125)  (0.375,0.375)  (0.4375,0.4375)   (0.5,-0.5)  (0.625,-0.625)  (0.75,0.75)  (0.875,0.875)   (1,-1)  (1.25,-1.25)  (1.5,1.5)  (1.75,1.75)   (2,-2)  (2.5,-2.5)  (3,3)  (3.5,3.5)  (4,-4)   (5,-5)  (6,6)  (7,7)  (7.5,0)} node[right,pos=0.9] {$f$};;
		\addplot[thick] coordinates { (-1,0)  (0,0)};
		\addplot[dotted] coordinates {(6.5,-1.2)  (6.5,6.5)};
		\addplot[dotted] coordinates {(6,0)  (6,6)};
		\addplot[dotted] coordinates {(7,0)  (7,7)};
		\addplot[dotted] coordinates {(4.5,0)  (4.5,-4.5)};
		\addplot[dotted] coordinates {(5,-1.2)  (5,-5)};
		\addplot[gray] coordinates {(0,0)  (7,7)} node[above = 0.2em,pos=0.7] {$\id$};	
		\addplot[gray] coordinates {(0,0)  (7,-7)} node[below = 0.2em,pos=0.7] {$-\id$};	
	\end{axis}
\end{tikzpicture} 
\end{center}
\caption{A sketch of the function $f$.}
\label{Fig:funch1}
\end{figure}
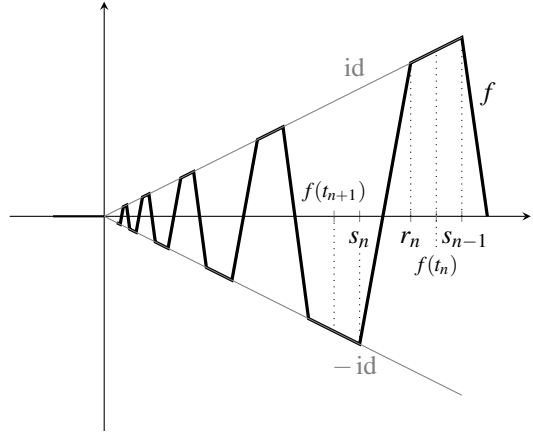
Consider the real $x_{g}$ given by $g = f \circ H$. 
 Notice that for every $n \geqslant 1$ we have that $t_n \in H^{-1}((r_n,s_{n-1}))$. So $t_{0} \notin \Interior{\set{t \in T}{g(t) \geqslant 0}}$ and $t_{0} \notin \Interior{\set{t \in T}{g(t) \leqslant 0}}$, and therefore 
\[ T \nVdash \lnot (x_{g}=0) \rightarrow (x_{g} \leqslant 0 \lor x_{g} \geqslant 0) \ . \qedhere \]
\end{proof}

Fred Richman has shown that \WMP holds in the sheaf model on every metric space \cite{fR02b}---a result which translates to this being true for every topological model over a metric space. To be precise we should also mention that we consider \WMP in its real form:
\begin{quote}
For all $x \in \RR$, if 
\begin{equation*} 
\fa{y \in \RR}{\lnot \neg (y > 0)  \lor \lnot \neg (y < x)}
\end{equation*}
then $0< x$.
\end{quote}

We can extend Richman's result mentioned above, using a very similar construction as in the proof of the previous proposition. Interestingly enough, we do not need to assume that the space is functionally Hausdorff, since we get a suitable real-valued function that we can adapt ``for free.''

\begin{Pro}
Assume $(T,\tau)$ is first countable. Then $T \Vdash \WMP_\RR$.
\end{Pro}
\begin{proof}
Let the internal real $x$ be given by the continuous function $f:T \to \RR$. 
Now let $t_0 \in T$ be a point such that 
\begin{equation} \label{Eqn:WMP_top_models} t_{0} \in \ext{\fa{y \in \RR}{\lnot \neg (y > 0)  \lor \lnot \neg (y < x)} } \ . \end{equation}
We want to show that $t_{0} \in \ext{0<x}$. Now if $f(t_{0}) > 0$ we are done, so it suffices to exclude the case that $f(t_{0}) \leqslant 0$. A sub-case that is then easily excluded is that there is an open neighbourhood $U$ of $t_{0}$ such that $f(t^{\prime}) \leqslant 0$ for all $t^{\prime} \in U$, since then $U \Vdash x \leqslant 0$, but for $y=x$ the antecedent of \WMP  implies $ \lnot \left( x \leqslant 0 \right)$.  
So, finally, we can assume that there is a sequence of points $t_{n}$ such that $f(t_{n}) > 0$ and $t_{n} \in U_{n}$; where $U_{n}$ is a base of open neighbourhoods of $t_{0}$. Since $f$ is continuous we have $f(t_n) \to 0$. By using an appropriate subsequence we may assume, without loss of generality, that $f(t_n) > f(t_{n+1})$.
Now choose reals $s_n,r_n$ such that 
\[ f(t_{n+1}) < s_n < r_n < f(t_n) \ . \]
Let $h: \RR \to \RR$ be the piece-wise linear function that is $0$ for $x \leqslant 0$ and is otherwise such that 
$h (x) = x $ for $x \in \bracks*{r_{n+1},s_{n}}$ whenever $n$ is even, and $h(x) = 0$ for $x \in \bracks*{r_{n+1},s_{n}}$ whenever $n$ is odd.
An illustration of $h$ can be found in Figure \ref{Fig:funch2}.

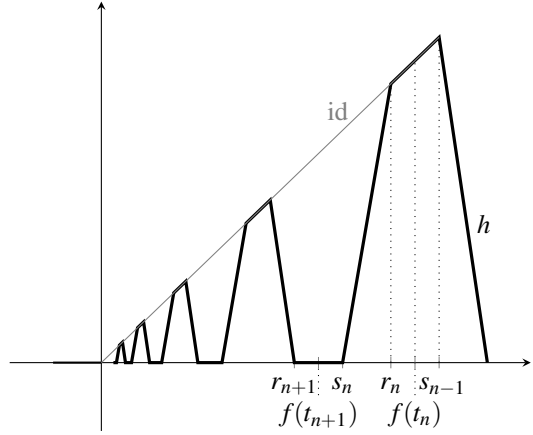
\begin{figure}
\begin{center}
\tikzset{external/export next=true}
\begin{tikzpicture}
	\begin{axis}
		[
			every extra x tick/.style={
				grid=none,
				xticklabel style={
					yshift=-0.9em
				},
			},
			axis x line=middle,
			axis y line=middle,
			enlarge x limits=0  0.1,
			enlarge y limits=0.1,
			xtick={4, 5,6,7},
			xticklabels={$r_{n+1}$, $s_n$, $r_n$, $\ s_{n-1}$ },
			ytick=\empty,
			yticklabels={},
			xticklabel style = {yshift=0ex},
			extra x ticks={4.5, 6.5 },
		extra x tick labels={$f(t_{n+1})$, $f(t_n)$ }]	
		
		\addplot[very thick] coordinates 
		{ (0.25,0)  (0.3125,0)  (0.375,0.375)  (0.4375,0.4375)   (0.5,0)  (0.625,0)  (0.75,0.75)  (0.875,0.875)   (1,0)  (1.25,0)  (1.5,1.5)  (1.75,1.75)   (2,0)  (2.5,0)  (3,3)  (3.5,3.5)  (4,0)   (5,0)  (6,6)  (7,7)  (8,0)} node[right,pos=0.9] {$h$};;
		\addplot[thick] coordinates { (-1,0)  (0,0)};
		
		\addplot[dotted] coordinates {(6.5,-0.7)  (6.5,6.5)};
		\addplot[dotted] coordinates {(6,0)  (6,6)};
		\addplot[dotted] coordinates {(7,0)  (7,7)};

		\addplot[dotted] coordinates {(4.5,-0.7)  (4.5,0)};
		\addplot[gray] coordinates {(0,0)  (7,7)} node[above = 0.2em,pos=0.7] {$\id$};	
	\end{axis}
\end{tikzpicture} 
\end{center}
\caption{A sketch of the function $h$.}
\label{Fig:funch2}
\end{figure}

Since $f$ is continuous, $W_n = f^{-1}((r_{n+1},s_n))$ is an open neighbourhood of $t_{n+1}$. So if $n$ is even, there is an open neighbourhood around $t_n$ (namely $W_n$) such that $h \circ f$ is equal to $f$, and if  $n$ is  odd, then $h \circ f$ is zero on an entire open neighbourhood around $t_n$. (It is crucial that all this happens on a neighbourhood and not just at $t_n$). Let $y$ be the real given by $h \circ f$.

Then $W_{2n+1} \subset \ext{\neg (y > 0)} = \Interior{\set{t \in T}{f(t) \leqslant 0}}$ for all $n$. 
That means, in particular, that $t_{2n+1} \notin \ext{\lnot \neg (y > 0)}$. That, in turn, implies that $t_0 \notin \ext{\lnot \lnot (y > 0)}$, since otherwise there exists $N$ such that $t_{2N+1} \in U_{2N+1} \subset \ext{\lnot \lnot (y > 0)}$.  Similarly $t_{0} \notin \ext{\lnot \neg (y<x)}$. Together we have the desired contradiction to \eqref{Eqn:WMP_top_models}.
\end{proof}

This nicely fits in with a result of Lubarsky and Hendtlass who showed in \cite{mH15} that there is a topological model satisfying \LLPO, but not \LPO, which means that that model also does not satisfy \WMP. 

Their space $X_{U}$ is defined to consist of $\NN$ and an added point $\omega$, where a base of the topology is given by
\begin{itemize}
\item $\menge{n}$ for $n \in \NN$; \item $\menge{\omega} \cup A$ for $A \in \mathcal{U}$; where $\mathcal{U}$ is an ultrafilter on $\NN$. 
\end{itemize}
In particular every set that doesn't contain $\omega$ is open.

We also assume that $\mathcal{U}$ is non-principal, i.e.\ it contains the  Fr\'{e}chet filter of cofinite subsets of $\NN$.

\begin{Pro} $ $
\begin{enumerate}
\item $X_{U} \nVdash \MP_\RR$
\item $X_{U} \Vdash \WLEM$
\end{enumerate}
\end{Pro}
\begin{proof}
\begin{enumerate}
  \item Consider the real $x$ given by $f(n) = \frac{1}{2^{n}}$ and $f(\omega) = 0$. $f$ is continuous since our filter is non-principal and therefore contains sets $A \in \mathcal{U} $ such that $A \cap\menge{1, \dots,n} = \emptyset $. Since $\Interior{\set{t}{f(t) =0}}= \Interior{\menge{\omega}} = \emptyset$ we conclude that $\ext{\neg(x=0)} = X_{U}$. But $\omega \notin \ext{x \neq 0}$, since $f(\omega)=0$. Hence $X_{U} \nVdash \neg(x=0) \implies x \neq 0$. 
  \item Consider an arbitrary $\varphi$. We distinguish three cases:
  \begin{itemize}
  \item $\ext{\varphi}  \subset \NN$. Then either $\ext{\varphi} \in \mathcal{U}$ or $ \NN \setminus \ext{\varphi} \in \mathcal{U}$. In the first case $\ext{\lnot \varphi} = \Interior{\Complement{\ext{\varphi}}} = \emptyset$ and therefore $\ext{\lnot \neg \varphi} = X_{U}$. In the second case $\ext{\lnot \varphi} = \Interior{\Complement{\ext{\varphi}}} = \menge{\omega} \cup (\NN \setminus \ext{\lnot \varphi})$ and $\ext{\lnot \neg \varphi} = \ext{\varphi}$.  In both cases we get $X_{U} \Vdash \lnot \varphi\lor \lnot \neg \varphi$. 
  \item If $\ext{\varphi} = \menge{\omega} \cup A$, then $A \in \mathcal{U}$. Therefore  $\ext{\lnot \varphi} =  \NN \setminus A$ and $ \ext{\lnot \neg \varphi} = \ext{ \varphi}$. So also in this case $X_{U} \Vdash \lnot \varphi\lor \lnot \neg \varphi$.  \qedhere 
  \end{itemize}
\end{enumerate}
\end{proof}

\begin{Cor} \label{Cor:Countermodel_WMP}
$X_{U} \nVdash \WMP_\RR$.
\end{Cor}
\begin{proof}
Since $\MP \iff \MPv \land \WMP$ and $\WLEM \implies \MPv$.
\end{proof}

\begin{Cor} 
$ \nvdash \WLPO \implies \MP_\RR$.
\end{Cor}
\begin{proof}
Since  $\WLEM \implies \WLPO$.
\end{proof}

\begin{Rmk}
The space $X_{U}$ is not first countable.
\end{Rmk}
\begin{proof}
Assume that $U_{n}$ is a base of open neighbourhoods of $\omega$. We may assume that $U_{n} \supset U_{n+1}$. Since $\mathcal{U}$ contains cofinite, and since $\mathcal{U}$ does not contain the empty set, sets we can find a strictly increasing sequence $y_{n}$ of natural numbers and a strictly increasing $f:\NN \to \NN$, such that $y_{n} \in U_{f(n)}$ and $y_{n} \notin U_{f(n+1)}$. Now consider the set $A = \set{y_{2n}}{n \in \NN}$. Either $A \in \mathcal{U}$ or $A \notin \mathcal{U}$. We will only treat the first case, since the second one can be handled similarly. Then $\menge{\omega} \cup A$ is open and should hence contain $U_{N}$ for some $N \in \NN$. Consider $M$ such that $f(2M+1) > N$. Since $y_{2M+1} \in U_{f(2M+1)} \subset U_{N} \subset A$, we get a contradiction.
\end{proof}

The following is actually a special case of \cite{mF79}*{Theorem 3.2}, where it is shown that in any spatial topoi $\CS$ is (cover) compact. Nevertheless, in the spirit of this section, a direct proof seems well fitted.
\begin{Pro} \label{Pro:TopmodFanf}
In any topological model \FANf~holds.
\end{Pro}
\begin{proof}
Let $T$ be a topological model such that 
\[ T \Vdash  \textrm{``$B$ is a bar''} \ . \]
Now consider an arbitrary $t \in T$.\footnote{This is the exact point, at which this argument fails for certain Heyting valued models; namely for some which are not spatial.} Define an external set $B_t \subset \cS$ by 
\[ u \in B_t \iff t \in \ext{ u \in B}  \ .\]
$B_t$ is a bar, since we can internalise any $\alpha \in \CS$ and know that 
\[ t \in \ext{ \ex{n \in \NN}{\overline{\alpha}n \in B}} \ .\] We can externalise this $n$ locally by Proposition \ref{Pro:ext_nat}, so there exists $n \in \NN$ such that 
\[ t \in \ext{ \overline{\alpha}n \in B} \ .\]
Externalising again, we can conclude that $\overline{\alpha}n \in B_t$, which means that $B_t$ is a bar. Using the Fan theorem in the meta-theory we can therefore find a uniform bar $N_t$ for $B_t$. It is straightforward to show that 
\[ t \in \ext{\text{``$N_t$ is a uniform bound for $B$''}}  \ .\]
Altogether
\[T = \bigcup_{n \in \NN} \ext{\text{``$N$ is a uniform bound for $B$''}} \ , \]
which means that 
\[T = \ext{\text{``$B$ is uniform''}} \ . \qedhere \]
\end{proof}

\begin{Pro}
If $(X,\rho)$ is a metric space, then the topological model over $X$ satisfies the following real number version of \KS: for all statements $\varphi$ there exists $x$ such that 
\[ X \Vdash \varphi \iff x > 0  \ .\]
\end{Pro}
That means, that if $X$ also validates countable choice, then $X \Vdash \KS$.
\begin{proof}
This is the case, since if $U$ is an inhabited open set, we can write it as
	\[ U = \set{t}{f(t) > 0} \ , \]
where $f(t) = \inf_{u \in U} \rho(t,u)$. If $U$ is empty, we can simply use $f = 0$.
\end{proof}

\begin{Rmk}
M.~Hendtlass has communicated\footnote{unpublished work} to us the result that $X$ satisfies (the sequence version of) \KS  if and only if every open set in $X$ is a countable union of clopens. This provides an alternative way of seeing that $\BS$ satisfies \KS.
\end{Rmk}

\begin{Cor} \label{Cor:BS_nval_IIIa}
	$\BS \nVdash \IIIa$.
\end{Cor}
\begin{proof}
Since $\PFP \vdash \IIIa \implies \WLPO$, and $\BS \Vdash \KS$, as mentioned above, and therefore $\BS \Vdash \PFP$, but $\BS \nVdash \WLPO$.
\end{proof}

\todo{Future Work: Write Subsection more models}

\subsection{Overview}
The following quick overview might be useful to compare the big three varieties and three topological models, discussed above.
\begin{center}	
{\rowcolors{1}{}{mylightgray}
\begin{tabular}{lcccccccc}
\toprule
 & \LEM & \WLEM &  \LPO & \WLPO & \LLPO & \MP & \WMP & \MPv \\ \midrule
\CLASS & \cmark & \cmark & \cmark & \cmark & \cmark & \cmark & \cmark & \cmark \\
\INT & \xmark & \xmark & \xmark & \xmark & \xmark & \xmark & \cmark & \xmark \\
\RUSS & \xmark & \xmark  & \xmark & \xmark & \xmark & \cmark & \cmark & \cmark \\
$\RR$\footnote{Note that, since this space does not validate $\mathrm{DC}$ this row refers to the real-number versions the principles, where applicable.} & \xmark & \xmark & \xmark & \xmark & \xmark & \xmark & \cmark & \xmark  \\
$\Sigma$ &  \xmark & \cmark & \cmark & \cmark & \cmark & \cmark & \cmark & \cmark \\
$X_{U}$ &  \xmark & \cmark & \xmark & \cmark & \cmark & \xmark & \xmark & \cmark \\
\bottomrule
\end{tabular}}
\end{center}
We conclude this section by pointing out that topological models have been used to give models that 
\begin{itemize}
  \item do not satisfy \BDN, and \BD \cite{rL12},
  \item separate \LPO, \WLPO, \MP and various variations of these \cite{mH15},
  \item separate various principles below \BDN \cite{bL11} (see Section \ref{Sec:belowBDN}).
\end{itemize}

Analysing the definition of a topological model one can see that the points of the underlying space are actually never mentioned, and that only the open sets are what matters. Therefore, we can just as well consider models over algebras that behaves like the open sets of a topological space do. Such algebras are known as \define{Heyting algebras} and have, indeed, been used for this purpose. Since there are Heyting algebras that cannot be defined in terms of the open of a topological space this is a richer model theory. These non-spatial Heyting valued models can be used to separate \FANf and \FANP \cite{bL11b}. 

\section{Realizability and Other Methods}

In a \define{realizability model} we extend intuitionistic logic by allowing witnesses of statements to be attached to  statements. This is, in particular, interesting in a constructive context, where we want to attach computable objects (``realizers'') that describe the computational content of a formula. For example, we would like to express the fact that a formula $\ex{m \in \NN}{P(\vec{x},m)}$ is realized (witnessed) by  a computable function $f$ when $P(\vec{x},f(\vec{x}))$ holds for all $\vec{x}$. There are many different ways to fix details of how the logical connectives and quantifiers are handled, leading to very different interpretations.\footnote{To mention some of the more famous realizability interpretations: there is Kleene's original version, modified realizability, Scott's graph model, models based on relativized computations, \dots} Most realizability models have a distinct recursive flavour and so we can say that, vaguely speaking, realizability models are to \RUSS what topological models are to \INT. 
We will not go into any details and only point the reader to \cite{Beeson1985}*{Chapter 7} and \cite{jvO08} for further details. 

Realizability models have been used to show that \BDN is not derivable in intuitionistic logic \cite{pL04}, and we would conjecture that they are the natural tool to separate \SinC, \KT, \SS.

As in the case of topological models, with realizability models one generally works in a classical meta-theory. This also means that there are strange realizability models, such as one based on infinite time turing machines \cite{aB15}, in which non-constructive principles like \LPO are validated. The model mentioned also has the curious property that $\BS$ is countable.

Finally, one can also use various proof interpretations \cite{uK08}, which can also be seen as types of realizabilities, to separate principles. This, for example and as was mentioned in the respective chapter, was used to show that \WMP is not derivable \cite{uK02}, which implies that \WLPO does not imply \MP.

\chapter{Bits'n'Pieces} \label{Ch:bitsnpieces}
\section{\texorpdfstring{\LLPOn{n}}{LLPOn}} \label{Sec:LLPOn}
In \cite{fR90} Richman introduced a natural weakening of \LLPO. 
\begin{principle}[LLPOn]{\LLPOn{n}} \label{PR:LLPOn} 
If $(P_{i})_{1 \leqslant i \leqslant n}$ is a partition of $\NN$ and $(a_n)_{n \geqslant 1}$ is a binary sequence with at most one $1$, then there exists $m$ such that  $a_{i} = 0$ for all $i \in P_{m}$.
\end{principle}
In \cite{fR02} we find the following version
\begin{principle}[LLPOnp]{\LLPOnp{n}} \label{PR:LLPOnp}
If  $(a_n)_{n \geqslant 1}$ is a binary sequence with at most one $1$, then there exists $0 \leqslant m < n$ such that  $a_{ni + m} = 0$ for all $i \in \NN$.
\end{principle}
It is clear that the first version implies the second if one considers the partition
\[ P_{m} = \set{mn+r}{ r \in \NN } \ .\]
A proof of the converse is not completely straightforward and is, to our knowledge, nowhere to be found in the literature, so we include one here.
\begin{Pro}
\LLPOnp{n} implies  \LLPOn{n} for every $n \in \NN$.
\end{Pro}
\begin{proof}
Let $(P_{i})_{1 \leqslant i \leqslant  n}$ be a partition of $\NN$ and $(a_{m})_{m \geqslant 1}$ a binary sequence with at most one $1$. Define a binary sequence $(b_{m})_{m \geqslant 1}$ by 
\[ 
b_{rn+i} = \begin{cases}
1 & \text{if }  r \in P_{i}  \text{ and } a_{r} = 1  \\
0 & \text{otherwise,}
\end{cases}
\]
for every $r \in \NN$ and $i \in \menge{0, \dots, n-1}$. Then $(b_{m})_{m \geqslant 1}$ has at most one $1$. Thus there exists $m$ such that $b_{ni +m} =0$ for all $i \in \NN$. Hence, by the definition of the sequence, there cannot be $i \in P_{m}$ with $a_{i} = 1$.
\end{proof}
The reader has probably already noticed that \LLPOn{2} is simply \LLPO.

\begin{Pro} \label{Pro:LLPOn-equiv-real}
\LLPOn{n} is equivalent to the statement that if $x_{1}, \dots, x_n$ are real numbers such that $x_{i}x_{j}=0$ for $i \neq j$, then there is $m$ such that $x_{m}=0$.
\end{Pro}
\begin{proof}
Assume \LLPOn{n} and let $x_{1}, \dots, x_n$ are real numbers such that $x_{i}x_{j}=0$ for $i \neq j$. Using countable choice we fix binary increasing sequences $(a^{(i)}_{m})_{m \geqslant 1}$ such that 
\begin{align*}
a^{(i)}_{m} = 0 & \implies x_{i} < \frac{1}{2^{m}} \ ,  \\
a^{(i)}_{m} = 1 & \implies x_{i} > \frac{1}{2^{m+1}} \ .
\end{align*}
Now let $b^{(i)}_{m} = a^{(i)}_{m+1} - a^{(i)}_{m}$, such that for every $i$ the sequence $b^{(i)}$ has at most one $1$ and that only if $a^{(i)}$ has one.  Finally combine all of these into one binary sequence $(c_{m})_{m \geqslant 1}$ by setting
\[ c_{in+ m} = b^{(i)}_{m} \] for every $m \in \NN$ and $1 \leqslant i \leqslant n$. Then $(c_{m})_{m \geqslant 1}$ has at most one $1$: For assume $c_{i} = c_{j} = 1$ with $i \neq j$. Furthermore choose $p,p^{\prime} \in \NN$ and $1 \leqslant r, r^{\prime} \leqslant n$ such that $i = p n+r$ and $j = p^{\prime} n + r^{\prime}$. If $p=p^{\prime}$, then $b^{(p)}_{r} = b^{(p)}_{r^{\prime}}$, which means that $r = r^{\prime}$ and therefore $i = j$; a contradiction. So $p \neq p^{\prime}$ and $b^{(p)}_{r} = b^{(p^{\prime})}_{r^{\prime}} = 1$. Hence $a^{(p)}_{r} = a^{(p^{\prime})}_{r^{\prime}} = 1$, which means that $x_{p} >0$ and $x_{p^{\prime}} >0$. So finally $x_{p} x_{p^{\prime}} > 0$, contradicting our assumptions. An application of \LLPOn{n} yields $1 \leqslant p \leqslant n$ such that $c_{pn + m} = 0$ for all $m \in \NN$. Hence $b^{(p)}_{m} = 0$ for all $m \in \NN$ and therefore $a^{(p)}_{m} = 0$ for all $m \in \NN$. 

Conversely let $(a_{m})_{m \geqslant 1}$ and for $1 \leqslant i \leqslant n$ define a real number 
\[ x_{i} = \sum_{m \geqslant 1} \frac{a_{in+m}}{2^{m}} \ .\]
It is easy to check that these numbers have the desired property. Now if there is $i$ such that $x_{i}=0$, then there cannot be $m \in \NN$ such that $a_{in+m}=1$, since in this case $x_{i}>0$.
\end{proof}

\begin{Rmk}
One might also consider the question whether the following is an interesting statement: 
\begin{quote}
If $x_1, \dots, x_n$ are real numbers such that $x_1 x_2 \dots x_n=0$, then there is a $i$ such that $x_i = 0$
\end{quote}
However this is trivially equivalent to \LLPO. If $x_1 x_2 \dots x_n = 0$, then by \LLPO either $x_1 x_2 \dots x_{n-1}$ or $x_n = 0$. Using \LLPO possibly up to $n-1$ times again we get that $x_1=0$, $x_2 = 0$, $\dots$, or $z=0$. Conversely, if $xy = 0$, then also $x\dots xy =0$, so either $x=0$ or $y = 0$.
\end{Rmk}

A more interesting equivalence is the following:
\begin{Pro}
\LLPOn{3} is equivalent to the statement that whenever $x,y$, and $z$ are real number such that $xyz=0$, 
\[ xy = 0 \lor yz=0 \lor xz = 0 \ .\]	
\end{Pro}
\begin{proof}
	The forward direction is easy: assume that $x,y,z$ are such that $xyz=0$ and  let $a=xy$, $b = yz$, and $c = xz$. Then $ab=bc=ac=xyz=0$. So we can apply \LLPOn{3} to decide whether $a=0$, $b=0$, or $c=0$; but that is exactly what we wanted to do. 
	
	For the converse let $x,y,z$ be such that 
	\begin{equation}\label{Eqn:YALPO}
	xy=0, yz=0, \text{ and } xz=0 \ .
	\end{equation}
	 Now let $a=\abs*{x-y}$, $b=\abs*{y-z}$, and $c=\abs*{x-z}$. Now we must have $abc=0$, since otherwise $x \neq y$, $y \neq z$, and $x \neq z$. And in that case at least two out of $x,y,z$ must be non-zero contradicting the assumption that their pair-wise product is zero. Thus we can decide whether $ab=0$, $bc=0$, or $ac =0$. Assume that w.l.o.g\ the first one holds. If $ab=0$, then $y =0$. Otherwise the assumption that $y \neq 0$ leads to a contradiction: if $y \neq 0$ we must have $x=z=0$ by Equation \eqref{Eqn:YALPO}, but then $ab=\abs*{y}^2 \neq 0$.
\end{proof}
\begin{Qu}
What about the general principle? Is \LLPOn{n}  equivalent to the statement that whenever $x_1, \dots, x_n$ are reals such that $\prod_{i=1}^n x_i = 0$, there exists $j$ such that 
\[\prod_{i=1, \ i \neq j}^n x_i = 0  \  ?\] 	
The forward direction is trivial. However the converse does not generalise straightforwardly
\end{Qu}

The reason that \LLPOn{n} does not feature more prominently in Constructive Reverse Mathematics is, that it has yet failed to have interesting equivalencies. So far it has actually only been used as a tool to show that certain results are non-construtive \cites{fR90,dB00b}; that is  it has been shown that some statements imply \LLPOn{n}, but not the converse.
The reason given for  \LLPOn{n} to be non-constructive is that it fails in \RUSS \cite{fR90}*{Theorem 5}. An alternative refutation relies on an argument about stable solutions by Beeson, which one could call ``a bit handwavey'' \cite{fR02}*{p. 112}.  Here we will show that it fails in the sheaf model of $\BS$, which also validates dependent choice.
\begin{Pro}
If a metric space $X$ has an accumulation point, then \LLPOn{n} fails in the topological model over $X$.
\end{Pro}
\begin{proof}
The proof builds on the one of Proposition \ref{Pro:topspace_omin_or_notMPv}. So let $x_{0},x_{1}, \dots$ and $r_{1}, r_{2}, \dots$ as there. Furthermore let $f^{(i)}: \RR^{+}_{0} \to \RR$ be defined to be piecewise linear through the points $(r_{3nj+3i+1}, \frac{1}{2^{j}})$ and $(r_{i^{\prime}},0)$ for $i^{\prime} \neq 3nj+3i+1$. Furthermore let $z_{i}$ be the reals in the sheaf model given by $\tilde{f}^{(i)} = f \circ d(x_{0}, \cdot)$. Then $z_{i} z_{j} = 0$ for $i \neq j$, since at every point at most one of the $f^{(i)}$ is non-zero. But also in every neighbourhood $U$ of $x_{0}$ and for every $i$ there exists $t$ such that $\tilde{f}(t) > 0$, so 
\[ U \nVdash \ex{i}{z_{i} =0} \ ,\]
and hence
\[ X \nVdash \ex{i}{z_{i} =0} \ .\]
Altogether $X \nVdash \LLPOn{n}$.
\end{proof}
\begin{Cor}
\LLPOn{n} fails in the sheaf model of $\RR$ on $\BS$. \end{Cor}
\begin{proof}
Dependent choice is validated by the sheaf model of $\RR$ on $\BS$. Also $\BS$ has an accumulation point.
\end{proof}

\section{Open Induction}
Coquand \cite{tC97}, U.~Berger \cite{uB04}, Schuster \cite{pS12}, and others have all investigated aspects of open induction. This principle is interesting, since, heuristically speaking, invocations of Zorn's Lemma in a classical proof can be replaced by using open induction. Now open induction is constructive at least for some sets, which also depends on the flavour of constructivism one practices. Thus, if open induction instance-wise implies a theorem, one can easily read off its constructive content up to exactly the level that one deems constructive.

We would also like to point out that Veldmann has done some work on open induction, however, the results are, at this stage, only partially published \cite{wV14}.

In classical mathematics, if $A \subset [0,1]$ such that $A$ is open, and \define{progressive} that is such that 
\begin{equation} \label{Eqn:progressive}
 \fa{x}{ \left( \left(\fa{y<x}{y \in A}\right) \implies x \in A \right)} 
 \end{equation}
then $A$ is actually the whole unit interval (notice that necessarily $0 \in A$).

There does not seem to be a good constructive theory of strict orders; or even any treatment of this theory. However, since in this section we will only deal with specific examples namely
\begin{itemize}
  \item the usual order $<$ on (subsets of) $\RR$ and
  \item the (strict) lexicographic order on (subsets of) $\CS$
\end{itemize}
 we actually do not need to fix a definition to state the principle of open induction for arbitrary strictly ordered sets.
\begin{principle}[OI]{\OI{X}} \label{PR:OI} Assume $X$ is a strictly ordered set equipped with the order topology and minimum element $\bot$. If $A \subset X$ is open and progressive, then $A = X$.
\end{principle}

Notice that such a set $X$ must necessarily contain all minimal elements $\bot$, since \ref{Eqn:progressive} is vacuously satisfied.

Veldmann has shown that \OI{[0,1]}  is implied by bar induction and in turn implies the fan theorem (that is the fan theorem for decidable bars.\footnote{We are unsure to what degree continuous choice is used in these, to our knowledge unpublished, proofs.})
\begin{Pro}
\OI{\CS} implies \FANf. 
\end{Pro}
\begin{proof}
Let $B$ be a bar. We will write $B_{u}$ for the set $\menge{u} \ast \CS \cap B$. Then 
\[ A = \set{\alpha \in \CS}{\ex{u_{1}, \dots, u_n \in \CS}{\fa{\beta< \alpha}{\beta \in \bigcup_{i =1}^n B_{u_{i}}}}} \ .\]
We want to show that $A$ is open and progressive. So let $\alpha \in A$. Now there exists $m$ such that $\gamma \in B$, as long as $\overline{\gamma}m = \overline{\alpha}m$. Furthermore there exists $u_{1}, \dots, u_n$ such that 
\[ \fa{\beta< \alpha}{\beta \in \bigcup_{i =1}^n B_{u_{i}}} \ , \] and therefore, with $u_{n+1} = \overline{\gamma}m $ also
\[ \fa{\beta< \gamma}{\beta \in \bigcup_{i =1}^{n+1} B_{u_{i}}} \ , \]
that is $\gamma \in A$. 

Along similar lines assume that $\alpha$ is such that $\beta \in A$ for all $\beta < \alpha$. Again choose $m$ such that $\overline{\alpha}m \in B$. Now either $\overline{\alpha}m = 0^{m}$ and we are done, or $\overline{\alpha}m = u 1 0 \dots 0$ for some $u \in \cS$. Then $\gamma < \alpha$ for $\gamma = u 0 1 1 \dots $. By our assumptions there exists $u_{1}, \dots , u_n$ such that 
\[ \fa{\beta< \gamma}{\beta \in \bigcup_{i =0}^n B_{u_{i}}} \ . \] Since, for any $ \beta <  \alpha$ we can decide whether $\overline{\beta} m = \overline{\alpha}m $ or $ \beta \leqslant \gamma$ we have that 
\[ \fa{\beta< \alpha}{\beta \in \bigcup_{i = 1}^{n+1} B_{u_{i}}} \ , \]
for $u_{n+1} = \overline{\alpha}m$. Hence $\alpha \in A$.

Since $B$ is a bar there exists $k$ such that $1^{k} = \overline{\one}k \in B$. 

By $\OI{\CS}$ we have $A=\CS$, and in particular $\one \in A$; therefore there exists $u_{1}, \dots,u_n$ such that $ \beta \in \bigcup_{i \leqslant n} B_{u_{i}}$ for any $\beta < \one$. Now for any $\gamma \in \CS$ we can decide whether $\overline{\gamma}k=1^{k}$ or $\gamma < \one$, and hence 
\[ \CS = \bigcup_{i =0}^n B_{u_{i}} \ ,  \]
where $u_{0}=1^{k}$. But this means that $B$ is a uniform bar.
\end{proof}

For some  sets $X$ we can prove $\OI{X}$ in pure \BISH.

\begin{Pro} $ $
\begin{enumerate}
\item If $X$ is a decidable subset of $\NN$, then $\OI{X}$ holds.
\item Let $\Ninf$ be the set of all increasing binary sequences (with the order and metric of Cantor space). Then $\OI{\Ninf}$ holds.
\end{enumerate}
\end{Pro}
\begin{proof} $ $
\begin{enumerate}
\item Let $A \subset X$ progressive (since $\NN$ has the discrete topology openness is irrelevant here). We need to show that $ X \subset A$. So let $x \in X$ be arbitrary. 
Since $X$ is decidable we can find natural numbers $a_{1} <  \dots < a_n$ such that $\set{y \in X}{y \leqslant x} = \menge{a_{1}, \dots, a_n}$, where $a_{1} = \bot_{X}$ and $a_n =x$. Now $a_{1} \in A$ by the observation above. Therefore, since $A$ is progressive also $a_{2} \in A$ and similarly after $n$ steps we have $a_n = x\in A$. 
\item As an abbreviation we write $\underline{n} = 0^n\one$ and $\omega=\zero$.
Now let $A \subset \Ninf$ be progressive and open. We can, as above, show  that $\underline{n} \in A$ for every $n \in \NN$. We can also show that $\omega \in A$, since for every $y < \omega$ we can find $n \in \NN$ such that $\underline{n}=y$.
It might be surprising that we are not done yet, and in fact we actually need to show that \emph{every} $x \in \Ninf$ is also in $A$, and not just the top element!\footnote{Of course classically we have $\Ninf = \set{\underline{n}}{n \in \NN} \cup \menge{\omega}$, but constructively we cannot prove this; to be precise this equality holds if and only if \LPO holds.}  However, since $A$ is open there exists $k \in \NN$ such that if $y > \underline{k}$ (that is if $y$ is in the neighbourhood of $\omega$ given by $k$), then $y \in A$. Furthermore for every $y \in \Ninf$ we can decide whether $y \leqslant \underline{k}$ and therefore $y = \underline{n}$ for some $n \in \NN$ or $y > \underline{k}$. In both cases we already know that $y \in A$. \qedhere
\end{enumerate}
\end{proof}

\begin{Pro}
The following are equivalent
\begin{enumerate}
\item \OI{\CS}
\item \OI{[0,1]}
\item \OI{\RR^{+}_{0}}
\end{enumerate}
\end{Pro}
\begin{proof}
Assume  \OI{\CS} holds and let $A \subset [0,1]$ be open, progressive and such that $0 \in A$. Furthermore let $F^{\nicefrac{2}{3}}:\CS \to [0,1]$ be the  function from Section \ref{Sec:linking-cantor-and-unitint}. Consider the set
\[ E = \set{\alpha \in \CS}{F^{\nicefrac{2}{3}}(\alpha) \in A} \ .\] Now $0 \in E$, since $F^{\nicefrac{2}{3}}(0)=0$. Since $E$ is the preimage of an open set under a continuous function it is also open. To see that it is progressive let $\alpha \in \CS$ be such that $\fa{\beta < \alpha}{\beta \in E}$. It suffices to show that for all $y< F^{\nicefrac{2}{3}}(\alpha)$ we have $y \in A$, since then $F^{\nicefrac{2}{3}}(\alpha) \in A$ by virtue of $A$ being progressive. By Lemma \ref{Lem:Fp} there  is $\gamma \in \CS$ with $F^{\nicefrac{2}{3}}(\gamma) =y$ and $\gamma < \alpha$. Hence $\gamma \in E$ and therefore $y \in A$. We are now in the position to apply \OI{\CS}  to show that $1 \in E$; and therefore $F(1)=1 \in A$.

Next, assume \OI{[0,1]} and let $E \subset \CS$ be open, progressive, and such that $0 \in E$. We will consider the set
\[A = \set{x \in [0,1]}{\fa{\alpha \in \CS}{ F^{\nicefrac{1}{3}}(\alpha) \leqslant x \implies \alpha \in E } }  \ , \]
where $F^{\nicefrac{1}{3}}$ is the canonical embedding of $\CS$ into the Cantor middle third set from Lemma \ref{Lem:Fp}. 
It is easy to see that $0 \in A$. We are going to show that if $y \in A $ for all $y < x$, then $A$ contains an open neighbourhood of $x$. Since $A$ is also downward closed this will show at the same time that $A$ is progressive as well as open. 

So let $x \in [0,1]$ such that $y \in A$ for all $y <x$. Let $\alpha$ be as in Lemma \ref{Lem:Fp}.\ref{Lem:Var-Bish}. We first want to show that $\beta \in E$ for all $\beta < \alpha$. It is easy to see that if $\beta < \alpha$, then $F^{\nicefrac{1}{3}}(\beta) < F^{\nicefrac{1}{3}}(\alpha)$. Now either $F^{\nicefrac{1}{3}}(\beta) < x$ and therefore $F^{\nicefrac{1}{3}}(\beta) \in A$ which in turn implies that $\beta \in E$, or $x < F^{\nicefrac{1}{3}}(\alpha)$. In the second case Lemma  \ref{Lem:Fp} yields $\fa{\beta < \alpha}{F^{\nicefrac{1}{3}}(\beta) \leqslant x}$, and thus, by the definition of $A$, $\beta \in E$.  So in all cases we have $\beta \in E$ for $\beta < \alpha$, and thus, since $E$ is progressive, $\alpha \in E $. Since $E$ is open there exists $n \in \NN$ such that $\gamma \in E$ whenever $\overline{\alpha}n = \overline{\gamma}n$. If $\overline{\alpha}n = 1^n$ we are done, since then $A=[0,1]$ and thus contains an open neighbourhood of $x$. In the case that $\overline{\alpha}n = u 0 1^{m} $ for some $u \in \cS $ and some $m \in \NN$ we have that for all $\beta < \gamma =u 1 1^{m} 0 \dots$ we  have that $\beta \in E$ and also $x < F^{\nicefrac{1}{3}}(\gamma)$. Hence also in this case $A$ contains an open neighbourhood of $x$. Thus \OI{[0,1]}  implies  \OI{\CS}.

To see that  \OI{[0,1]} implies  \OI{\RR^{+}_{0}}, we can use scaling: assume that $A \subset \RR^{+}_{0}$ is open and progressive, and let $z \in \RR^{+}_{0}$ be arbitrary. Then it is easy to see that $A \cap [0,z]$ is progressive and open (in $[0,z]$). Clearly \OI{[0,1]} implies \OI{[0,z]}, so $A \cap [0,z] = [0,z]$ which means that $[0,z] \subset A$. In particular $z \in A$.

Conversely Let $A \subset [0,1]$ be open, progressive, and such that $0 \in A$. Then 
\[ A^{\prime} = (A \cap (0,1)) \cup \set{x \in \RR^{+}_{0}}{A = [0,1]}  \ .\]
Then trivially $A^{\prime}$ is open. It is also progressive, since for $x \in A^\prime$ either $x \in A$, in which case we can use the progressiveness of $A$, or $x \in \set{x \in \RR^{+}_{0}}{A = [0,1]}$, in which case $A^\prime = \RR^{+}_{0}$, which is also clearly progressive. This means we can use \OI{\RR^{+}_{0}} to conclude that $A^\prime = \RR^{+}_{0}$. In particular $1 \in A^\prime$, which means that $A = [0,1]$.
\end{proof}

\section{The Limited Anti-Specker Property} 
Douglas Bridges, James Dent, and Maarten McKubre-Jordens \cites{dB14,jD13} have considered various weakenings of  the Anti-Specker property. One of them is the so-called \define{Limited Anti-Specker property}. 

\begin{principle}[ASLX]{\ASL{X}} \label{PR:ASL} If $X \cup \menge{\omega}$ is a one-point extension of a metric space $X$ and $(x_n)_{n \geqslant 1}$ is a sequence in $X \cup \menge{\omega}$ is eventually bounded away from each point of $X$, then there exists $n$ such that $x_n = \omega$.
\end{principle}

\begin{Pro}
$AS_{[0,1]}^{L} \iff AS_{\CS}^{L}$
\end{Pro}

\begin{proof} 
Let $F^{\nicefrac{1}{3}}:\CS \to [0,1]$ denote the canonical embedding of $\CS$ onto the Cantor middle third set as in Section \ref{Sec:linking-cantor-and-unitint}. We may assume that $d([0,1], \omega) = 2$. Let $(\alpha_n)_{n \geqslant 1}$ be a sequence in $\CS \cup \menge{\omega}$ that is bounded away from every point in $\CS$ and set 
\[ x_n = \begin{cases} \omega & \text{ if } \alpha_n = \omega  \\ F^{\nicefrac{1}{3}}(\alpha_n) &  \text{ if } \alpha_n \in \CS \ . \end{cases}\]
 We want to show that it is eventually bounded away from every point in $[0,1]$. To this end let $x \in [0,1]$ be arbitrary. By Proposition \ref{Lem:Fp}.\ref{Lem:Var-Bish} there exists $\alpha \in \CS$ such that 
\[ F^{\nicefrac{1}{3}}(\alpha) \neq x \implies \ex{ \varepsilon > 0}{ d(x,F^{\nicefrac{1}{3}}(\CS)) > \varepsilon} \ . \] Now find $N,M \in \NN$ such that for all $n \geqslant N$ we have $\overline{\alpha}M \neq \overline{\alpha_n}M$. It is an easy calculation\footnote{If $\overline{\beta}K \neq  \overline{\gamma}K$, then $\abs*{F^p(\beta) - F^p(\gamma)} \geqslant (1-2p)p^{M-1}$} to see that for all such $n$ we have 
\[ \abs*{F^{\nicefrac{1}{3}}(\alpha)-F^{\nicefrac{1}{3}}(\alpha_n)} \geqslant 3^{-M} = 2\delta . \]
  Now either $\abs*{x-F^{\nicefrac{1}{3}}(\alpha)}< \delta$ or $\abs*{x-F^{\nicefrac{1}{3}}(\alpha)} >0$. In the second case there exists $\varepsilon >0$ such that $d(x,F^{\nicefrac{1}{3}}(\CS)) > \varepsilon $, which implies that $\abs*{x - x_n} > \varepsilon$ for all $n \in \NN$ with $\alpha_n \in \CS$. In the first case for all $n \geqslant N$ for which $\alpha_n \in \CS$ we get 
\begin{align*}
\abs*{x-x_n}  = \abs*{x-F^{\nicefrac{1}{3}}(\alpha_n)} & \geqslant \abs*{F^{\nicefrac{1}{3}}(\alpha)  - F^{\nicefrac{1}{3}}(\alpha_n) } -  \abs*{x - F^{\nicefrac{1}{3}}(\alpha)}  \ , \\
&  \geqslant 2\delta -  \abs*{x - F^{\nicefrac{1}{3}}(\alpha)} > \delta \ .
\end{align*}
Notice that for $n$ with $\alpha_n=\omega$ we have $\abs*{x-x_n} \geqslant 1 > \frac{1}{2}$. So in all cases we have the desired $\abs*{x_n -x} > \min \menge{\varepsilon,\delta, \frac{1}{2}} > 0$ for all $n \geqslant N$. Therefore $x_n$ is bounded away from every point in $[0,1]$. Furthermore, if there exists $m$ such that $x_m = \omega$, then $\alpha_m = \omega$. Thus $AS_{[0,1]}^{L} \implies AS_{\CS}^{L}$. 

To see that the converse holds we can use \cite{hD09b}*{Lemma 0.3} and the fact that there exists a point-wise continuous surjection $\CS \to [0,1]$ (see Section \ref{Sec:linking-cantor-and-unitint}).
\end{proof}

As was shown in \cite{dB13} \ASL{X} is equivalent to the following version of \POS (remember that \POS is the same with the assumption that $f$ is uniformly continuous, which ensures that $\inf f$ exists).

\begin{principle}[POSpw]{$\POSp{X}$} \label{PR:POSp} If $f$ is a point-wise continuous, positive-valued mapping on a metric space, and if $\inf_{X} f$ exists, then $\inf_{X} f > 0$.
\end{principle}

\begin{Lem}
If $f$ is a point-wise continuous, positively valued mapping on a metric space, then $g =\frac{1}{f}$ is also a point-wise continuous, positively valued mapping; furthermore
\begin{itemize}
\item $\inf f$ exists if and only if $g(X)$ is upper order located, and
\item $\inf f > 0$ if and only if $g$ is bounded.
\end{itemize}
\end{Lem}
\begin{proof}
Straightforward.
\end{proof}

\begin{Cor}
The following are equivalent:
\begin{enumerate}
\item $\ASL{[0,1]}$
\item $\POSp{[0,1]}$
\item Every point-wise continuous function $f:[0,1] \to \RR$ such that $f([0,1])$ is located is bounded.
\end{enumerate}
\end{Cor}

Remember that in Corollary \ref{Cor:FAND_equiv_fullyloc} it was shown that \FAND is equivalent to every fully located, point-wise continuous function $[0,1] \to \RR$ being bounded. 

Therefore we get the following nice characterisations of POS, $POS_{p}^{X}$ and UCT being equivalent to all point-wise continuous functions on $[0,1]$ with an additional property being bounded.
\begin{center}
\rowcolors{1}{}{mylightgray}
\begin{tabular}{ll} \toprule
 & additional property X \\ \midrule
$\POS$, \FAND & $f([a,b])$ is located for all $a \leqslant b$ (that is $f$ is fully located) \\ 
 $\POSp{[0,1]}$, $\ASL{[0,1]}$ & $f([0,1])$ is located   \\ 
 \UCT & no assumption on $f$ \\ \bottomrule
\end{tabular}
\end{center}

\section{Increasing Specker Sequences}  \label{Sec:incss}

As mentioned above Specker's original sequence is increasing. So it is natural to follow Dent \cite{jD13} and consider the following principle.
 
\begin{principle}[iAS]{\iAS} \label{PR:iAS}
Every \emph{increasing} sequence $x_n$ in $\RR$ that is eventually bounded away from every point in $[0,1]$ is eventually bounded away from the entire interval.
\end{principle} 

If, for a sequence as in \iAS, we ever find $n \in \NN$ such that $x_n>1$, then from then on we have $x_m \geqslant x_n > 1$ for all $m \geqslant n$. Therefore \ASL{[0,1]} implies \iAS. In turn, we can show the following.

\begin{Pro}
$\iAS \implies \FAND$
\end{Pro}
\begin{proof}
Let $B$ be a decidable bar that is closed under extensions. Notice that we can decide for every $n\in \NN$, whether $n$ is a uniform bound of $B$. Now define a sequence $u_n \in \cS \cup \menge{\omega}$ by
\[ u_n = \begin{cases}
\omega & \text{if $n$ is a uniform bound} \\
u & \text{if $u \in 2^n$ is the (lexicographically) smallest element such that $u \notin B$}
\end{cases}
 \]
 It is easy to show that $x_n$ defined by 
 \[ x_n = \begin{cases}
2 & \text{if } u_n = \omega \\
F(u_n) & \text{if } u_n \in \cS
\end{cases}
 \]
 is an increasing sequence, that is eventually bounded away from every point in $[0,1]$. Thus there exists a $n$ such that $x_n > 1$, which translates back into $n$ being a uniform bound of $B$.
\end{proof}
Thus we have shown that 
\[ \FANc \implies \ASL{[0,1]} \implies \iAS \implies \FAND \ , \]
which of course raises the following question.
\begin{Qu}
Are all of the implications $\FANc \implies \ASL{[0,1]} \implies \iAS \implies \FAND$ strict?
\end{Qu}

Rather than focussing on the ``anti-Specker''-side of things we could, of course, ask  the same question about the recursive side and consider the following principle. 
\begin{principle}[iSS]{\iSS} \label{PR:iSS}
	There exists an increasing sequence
\end{principle}
Obviously we have that $\iSS \implies \SS$. So together with Proposition \ref{Pro:KT_impl_iSS} we have 
\[ \KT \implies \iSS \implies \SS \ . \]

\begin{Qu} \label{Qu:SS_iSS}
Does \SS imply \iSS? And does \iSS imply \KT?
\end{Qu}

Notice that even if we start with a Specker sequence there is no hope to construct a monotone sequence as a subsequence of the original sequence.
\begin{Pro}
If \SS holds and every sequence in $[0,1]$ that is bounded away from every point in $[0,1]$ contains a monotone subsequence, then \LPO holds.
\end{Pro}
\begin{proof}
Let $(a_n)_{n \geqslant 1}$ be a binary increasing sequence and $(x_n)_{n \geqslant 1}$ be a sequence in $[0,1]$ that is eventually bounded away from every point in $[0,1]$. By our assumption there exists a monotone subsequence $(x_{k_n})_{n \geqslant 1}$. Without loss of generality this sequence is increasing. Using dependent choice we may construct a subsequence $(x_{k_{\ell_n}})_{n \geqslant 1}$ of the subsequence such that \[ x_{k_{\ell_n}} < x_{k_{\ell_{n+1}}}  \  . \]
Now define a sequence $(y_n)_{n \geqslant 1}$ by
\[ 
y_n = 
\begin{cases}
x_{k_{\ell_n}} & \text{if } a_n=0 \\
1-x_{_{k_{\ell_n}}} & \text{if } a_n=1 \ .
\end{cases}
\]
Given $x \in [0,1]$ there exists $N$ and $\varepsilon >0$ such that for all $n \geqslant N$ 
\[ \abs*{x_{k_{\ell_n}} -x}, \abs*{x_{k_{\ell_n}} - (1-x)} > \varepsilon \ ,\] which means that also $\abs*{ (1-x_{k_{\ell_n}}) -x} > \varepsilon$. In both cases $\abs*{y_n - x} > \varepsilon$, and so $(y_n)_{n \geqslant 1}$ is eventually bounded away from every point in $[0,1]$. Now assume that there exists a monotone subsequence $y_{k_n}$. If it is decreasing, then $a_{k_{2}} =1$, since otherwise $y_{k_{1}} < x_{k_{2}}$. Similarly, if it is increasing, then $a_n=0 $ for all $n \in \NN$. In other words \LPO holds.
\end{proof}

\iSS also allows for the construction of a strong counterexample to a minimal intermediate value---more precisely, a function that around any root is constant $0$.
\begin{Pro}
\iSS implies the existence of an increasing, uniformly continuous function $f:[0,1] \to \RR$, such that $f(0) = -1$, $f(1) = 0$ and such that \[ f(x) = 0 \implies \ex{\delta>0}{f_{\restriction{B_{x}(\delta)}} = 0 }  \ . \]
\end{Pro}
\begin{proof}
First, assume that $(r_n)_{n \geqslant 1}$ is an increasing sequence in $[0,1]$ that is bounded away from every point in $[0,1]$. W.l.o.g.\ we may assume that $r_n < r_{n+1}$. Define a sequence of piecewise linear  functions $g_n: [0,1] \to \RR$ connecting the points 
\[ (0,-1),  \left(r_1, -\frac{1}{2}\right), \dots, \left(r_n, -\frac{1}{2^n}\right), (1,0) \ . \]
This increasing and uniformly continuous function converge uniformly to a limit $f:[0,1] \to \RR$. Therefore $f$ is also increasing,  uniformly continuous, and such that $f(0) = -1$, $f(1) = 0$. So let $x \in [0,1]$ such that $f(x) = 0$. Then we must have $x \geqslant r_n $ for all $n \in \NN$, since the assumption that $x < r_n$ for some $n \in \NN$ implies that $g_m(x) \leqslant \frac{1}{2^n}$ for all $m \geqslant n$ and therefore $f(x) \leqslant \frac{1}{2^n}$; a contradiction. Furthermore, since $(r_n)_{n \geqslant 1}$ is eventually bounded away from $x$, we must actually have that there exists $\delta > 0$ such that $y \geqslant r_n$ for all $y \in B_{x}(\delta)$ and $n \in \NN$. Thus $f(y) = 0$ for such $y$.
\end{proof}

Notice that it is not clear, whether the converse holds. At the same time the above construction can also not be done with a Specker sequence that is not necessarily increasing. This basically leads back to Question \ref{Qu:SS_iSS}. 

\section{Dirk Gently's Principle\texorpdfstring{: $(\varphi \implies \psi) \lor (\psi \implies \varphi)$}{}} \label{Sec:PimplQveeQimplP}

One statement that is notably missing from the list of the ``paradoxes of material implication'' equivalent to \LEM (see Proposition \ref{Pro:paradoxes_of_material_impl}) is \define{Dirk Gently's Principle}.\footnote{The name is based on the guiding principle of the protagonist of Douglas Adam's novel \emph{Dirk Gently's Holistic Detective Agency} who believes in the ``fundamental interconnectedness of all things.''} It is also known under the name \define{linearity}, and is the basis of G\"odel-Dummet logic \cite{jP03}.\footnote{Readers disliking our naming of this principle are encouraged to think of \DGP as standing for ``Dummet-Gödel Principle''.}
\begin{principle}[DGP]{\DGP} \label{Pr:DGP}
If $\varphi$ and $\psi$ are any syntactically correct statements, then
\[ (\varphi \implies \psi) \lor (\psi \implies \varphi) \ . \]
\end{principle}
\DGP is listed as Proposition 5.13 in Principia Mathematica \cite{PM}.
As a matter of fact the proof there is actually skipped and Russel and Whitehead ``merely indicate the proposition [...] used in the proofs.'' 
We can be a little bit more precise here.
\begin{Pro}
\DGP is implied by \LEM and implies \WLEM.	
\end{Pro}
\begin{proof}
The proof from \LEM is straightforward: either $\psi$ or $\lnot \psi$ holds. In the first case $\varphi \implies \psi$ (weakening) and in the second case $\psi \implies \varphi$ (ex falso quodlibet). To see that it implies \WLEM apply the principle to $\varphi$ and $\lnot \varphi$. Now either $\varphi \implies \lnot \varphi$ or $\lnot \varphi \implies \varphi$. In the first case the assumption that $\varphi$ holds leads to a contradiction, whence $\lnot \varphi$. Similarly, in the second case the assumption that $\lnot \varphi$ holds leads to a contradiction, and hence $\lnot \neg \varphi$.
\end{proof}
We can also show that it is equivalent to yet another ``paradox of material implication''.
\begin{Pro}
\DGP is equivalent to the statement that
\begin{equation}	
	((\varphi \land \psi) \implies \vartheta) \implies \left( (\varphi \implies \vartheta) \lor (\psi \implies \vartheta) \right)  \ ,
\end{equation} 
to the statement that
\begin{equation}	
	((\varphi \implies \vartheta) \land (\psi \implies \beta)) \implies \left( (\varphi \implies \beta) \lor ( \psi \implies \vartheta) \right)  \ ,
\end{equation} 
and to the statement that 
\begin{equation}	 \label{PMI19}
	(\varphi \implies \psi \lor \vartheta) \implies (\varphi \implies \psi) \lor (\varphi \implies \vartheta)  \ .
\end{equation}
\end{Pro}
\begin{proof}
To see the equivalence with the first statement assume that $(\varphi \land \psi) \implies \vartheta$. Now if \DGP holds, then either $\varphi \implies \psi$ or $\psi \implies \varphi$.
	In the first case, if $\varphi$ holds, then also $\varphi \land \psi$, and hence $\vartheta$ holds. Together that means that in the first case we have $\varphi \implies \vartheta$. In the second case, similarly, we can show that $\psi \implies \vartheta$.

Conversely, apply the statement to $\vartheta \equiv \varphi \land \psi$. Then the antecedent is always satisfied, which means that $ (\varphi \implies (\varphi \land \psi)) \lor (\psi \implies (\varphi \land \psi))$. Hence the desired $(\varphi \implies \psi) \lor (\psi \implies \varphi)$ holds.

For the second statement assume that $(\varphi \implies \vartheta) \land (\psi \implies \beta)$. By \DGP either $\varphi \implies \beta$ and we are done, or $\beta \implies \varphi$. But in that second case if we assume $\psi$ also $\beta$ holds, which in turn implies $\varphi$, which in turn implies $\vartheta$. Together, in the second case, $\psi \implies \vartheta$.

Conversely we can apply the statement to $\vartheta \equiv \varphi$ and $\beta \equiv \psi$, which yields 
\[((\varphi \implies \varphi) \land (\psi \implies \psi)) \implies \left( (\varphi \implies \psi) \lor ( \psi \implies \varphi) \right)  \ . \] 
But since the antecedent is always satisfied, we get the desired $(\varphi \implies \psi) \lor (\psi \implies \varphi)$.

The last statement is implied by \DGP: For either $q \implies r$ or $r \implies q$ \dots.
	Conversely we can apply \ref{PMI19} to $\varphi \land \psi$, $\varphi$, and $\psi$ to get:
	\[ (\varphi \land \psi \implies \varphi \lor \psi) \implies ((\varphi \land \psi \implies \varphi ) \lor (\varphi \land \psi \implies \psi)) \ .\] Now clearly the antecedent is always satisfied. So we have that 
	\[ (\varphi \land \psi \implies \varphi ) \lor (\varphi \land \psi \implies \psi) \ , \] which is equivalent to the desired 
\[ (\psi \implies \varphi ) \lor (\varphi \implies \psi) \ . \]
Hence \DGP holds.
\end{proof}

What makes the principle \DGP interesting for our point of view is that we can actually show that it lies strictly between \LEM and \WLEM.
\begin{Pro} \label{Pro:topmod_biimp}
There exist topological models $S$ and $T$ such that 
\[ S \nVdash \LEM  \qquad \text{ and } \qquad S \Vdash \DGP \ , \]	
and
\[ T \nVdash \DGP  \qquad \text{ and } \qquad T \Vdash \WLEM \ . \]	
\end{Pro}
\begin{proof}
Let $X = \menge{1,2,3,4}$ and let  the topological space $S = (X, \sigma)$ be given by the base \[ \menge{ \menge{1},\menge{1,2},\menge{3},\menge{3,4}} \ . \] 
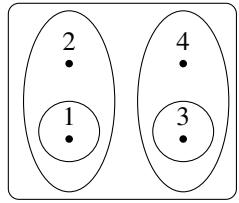
\begin{figure}[ht]
\centering
\begin{tikzpicture}
  \node [circle,fill=black,inner sep=1pt,label={above:$1$}] at (0,0) {};
  \node [circle,fill=black,inner sep=1pt,label={above:$2$}] at (0,1) {};
  \node [circle,fill=black,inner sep=1pt,label={above:$3$}] at (1.5,0) {};
  \node [circle,fill=black,inner sep=1pt,label={above:$4$}] at (1.5,1) {};
  \draw (0,0.1) circle (0.4cm);
  \draw (1.5,0.1) circle (0.4cm);
  \draw (0,0.5) ellipse (0.6cm and 1.2cm);
  \draw (1.5,0.5) ellipse (0.6cm and 1.2cm);
  \draw[rounded corners] (-0.8,-0.8) rectangle (2.2,1.8);
\end{tikzpicture}
\caption{A base of the space $S$}
\end{figure}
Then $S \nVdash \LEM$, since if $\varphi$ is such that  $\ext{\varphi} = \menge{1}$ we have 
\[ \ext{\lnot \varphi} = \Interior{\Complement{\menge{1}}} = \Interior{\menge{2,3,4}} = \menge{3,4}  \ .\]
Hence $\ext{\varphi \lor \lnot \varphi} = \menge{1,3,4} \neq X$. Conversely we can check by hand, or simply with the aid of a computer (see Appendix \ref{Appendix:sourcecode}) that $S \Vdash \DGP$.

Now consider  the topological space $T = (X, \tau)$ be given by the base $\menge{ \menge{1},\menge{1,2},\menge{1,3},\menge{1,2,3,4}}$.
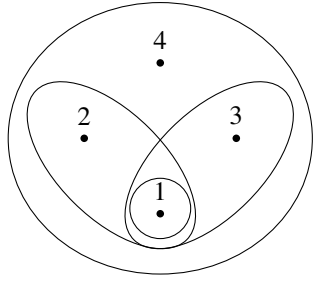
\begin{figure}[ht]
\centering
\begin{tikzpicture}
  \node [circle,fill=black,inner sep=1pt,label={above:$1$}] at (0,0) {};
  \node [circle,fill=black,inner sep=1pt,label={above:$2$}] at (-1,1) {};
  \node [circle,fill=black,inner sep=1pt,label={above:$3$}] at (1,1) {};
  \node [circle,fill=black,inner sep=1pt,label={above:$4$}] at (0,2) {};
  \draw (0,1) ellipse (2cm and 1.8cm);
  \draw (0,0.07) circle (0.4cm);
  \draw[rotate=45] (0,0.91) ellipse (0.7cm and 1.4cm);
  \draw[rotate=-45] (0,0.91) ellipse (0.7cm and 1.4cm);
\end{tikzpicture}
\caption{A base of the space $T$}
\end{figure}
 Again either through checking by hand, or simply with the aid of a computer (see Appendix \ref{Appendix:sourcecode}) that $T \Vdash \WLEM$. To see that \DGP fails consider $\varphi$ and $\psi$ such that $\ext{\varphi} = \menge{1,2}$ and $\ext{\psi} = \menge{1,3}$. Then 
\[ \ext{\varphi \implies \psi} = \Interior{(\Complement{\ext{\varphi}} \cup \ext{\psi} )} = \Interior{\menge{3,4} \cup \menge{1,3}} = \menge{1,3} \ , \]
and similarly $\ext{\psi \implies \varphi} = \menge{1,2}$. Hence 
\[ \ext{(\varphi \implies \psi) \lor (\psi \implies \varphi)} = \menge{1,2,3} \neq X \  .  \qedhere \]
\end{proof}

Thus we get the following simple hierarchy:

\begin{figure}[ht] \centering
\begin{tikzpicture}[node distance=2 cm, auto]
  \node (DNE2) {\LEM};
  \node (DGP2) [right of=DNE2] {\DGP};
  \node (WLEM2) [right of=DGP2] {\WLEM};
  \draw[double,->] (DNE2) to node {} (DGP2);
  \draw[double,->] (DGP2) to node {} (WLEM2);
\end{tikzpicture}
\end{figure}

As a final remark we would like to note that it seems that all simple propositional axiom schemata containing only few propositional symbols, and which are classically true (also see Propositions \ref{Pro:WLEM_equiv_DM1}, \ref{Pro:WLEM_equiv_DM12p}, and \ref{Pro:paradoxes_of_material_impl}) fall in one of these three categories (also see \cite{hD16cb}). 

\section{\texorpdfstring{$\Pi_1^0$ and $\Sigma_1^0$}{Pi01 and Sigma01}}
\label{Sec:Pi01andSigma01}
The purpose of this overview is to tie in some of our principles with other presentations such as \cites{Troelstra1988a,Troelstra1988,hI97}.

In the style of computability theory (of sets) we define a statement $\varphi$ to be a $\Sigma_1^0$ statement if there exists a binary sequence $(a_n)_{n \geqslant 1}$ such that 
\[ \varphi \iff \ex{n \in \NN}{a_n =1} \ , \]
and a 
a $\Pi_1^0$ statement if there exists a binary sequence $(b_n)_{n \geqslant 1}$ such that 
\[ \varphi \iff \fa{n \in \NN}{b_n =0} \ . \]
We can then identify fragments of logical principles by restricting them to $\Sigma_1^0$ statements or $\Pi_1^0$ statements respectively.
It easy, for example, to see that $\Sigma_1^0$--\LEM is just \LPO. Overall we get: 
\begin{Pro} The following equivalences hold, where $\vdash_i$ means equivalence to $\top$, i.e.\ provable in intuitionistic logic.
\newline
\begin{figure}[h] \centering
\begin{tabular}{lccc}
& $\Sigma_1^0$ & $\Pi_1^0$ & $\Sigma_1^0$--$\Pi_1^0$ \\ \hline
\LEM ($P \lor \lnot P$) & \LPO & \WLPO \\ 
\WLEM ($\lnot P \lor \lnot \neg P$)& \WLPO & \WLPO \\
Stability ($\lnot \neg P \implies P$) & \MP & $\vdash_i$ \\
\ref{DM1} ($\lnot (P \land Q) \implies \left( \lnot P \lor \lnot Q \right)$) & \LLPO & \MPv & \WLPO \\
\ref{DM2} ($\lnot (P \lor Q) \implies \left( \lnot P \land \lnot Q \right)$)  & $\vdash_i$ & $\vdash_i$ & $\vdash_i$  \\
\ref{DM1p} ($\lnot (\lnot P \land \lnot Q) \implies \left(  P \lor  Q \right)$) & \MP & \LLPO & \LPO \\
\ref{DM2p} ($\lnot (\lnot P \lor \lnot Q) \implies \left(  P \land  Q \right)$)  & \MP & $\vdash_i$ & \MP \\
\DGP $(P \implies Q) \lor (Q \implies P)$  & \LLPO & \LLPO & \WLPO
\end{tabular}
\end{figure}
\end{Pro}
\begin{proof}
Most of these equivalencies are trivial. Notice that $\Pi_{1}^0$ statements are stable, that is we can eliminate preceding double negations.

The equivalencies for the restricted versions of \DGP are non-trivial, and the proofs are entertaining: First assume that \LLPO holds. And consider two binary sequences $a_n$ and $b_n$. Now let $\alpha$ be the sequence \[ a_1b_1a_2b_2 \dots  \ ,\]
and let $\alpha^\prime$ be the sequence 
\[ \alpha^\prime_n = \begin{cases}
	1 & \text{if } \sum_{i=1}^n \alpha_n = 1 \\
	0 & \text{otherwise;}
\end{cases} \]
that is  $\alpha^\prime$ has at most one term equal to $1$. Since we assume \LLPO 
\begin{equation}
	\label{Eqn:SigmaDGP}
\fa{n \in \NN}{\alpha^\prime_{2n} = 0} \lor \fa{n \in \NN}{\alpha^\prime_{2n+1} = 0}  \ . 
\end{equation} 
In the first case assume that  $n \in \NN$ is such that $b_n =1$. Without loss of generality let $n$ be minimal. Then there must be $m \leqslant n$ such that $a_m = 1$, since if $a_m  = 0$ for all $m \leqslant n$, and $b_k =0$ for all $k < n$ we have $\alpha_{k} = 0$ for all $k < 2n$ and $\alpha_{2n} = b_n = 1$, which means that $\alpha^\prime_{2n} = 1$ contradicting our assumption. Hence in the first case we have $\ex{n \in \NN}{b_n = 1} \implies \ex{n \in \NN}{a_n = 1}$. Symmetrically in the second case we have that $\ex{n \in \NN}{a_n = 1} \implies \ex{n \in \NN}{b_n = 1}$. Thus we have shown that \LLPO implies \DGP restricted to $\Sigma_1^0$-formulas.

Conversely let $a_n$ be a binary sequence with at most one term equal to $1$. By $\Sigma_1^0$-\DGP we get that 
\[ \left( \ex{n \in \NN}{a_{2n}=1} \implies \ex{n \in \NN}{a_{2n+1}=1}\right)  \lor \left( \ex{n \in \NN}{a_{2n+1}=1} \implies \ex{n \in \NN}{a_{2n}=1}\right) \ . \]
In the first case we see that $\fa{n \in \NN}{a_{2n} =0}$, since the assumption that $\ex{n \in \NN}{a_{2n} = 1}$ leads to the contradiction that there is more than one term equalling $1$. Similarly in the second case we can show that $\fa{n \in \NN}{a_{2n+1} =0}$. Thus we have shown that $\Sigma_1^0$-\DGP implies \LLPO.

Next, let $a_n$ and $b_n$ be binary sequences and construct $\alpha$ and $\alpha^\prime$ as above. Using \LLPO we can, again, make the decision in Equation \eqref{Eqn:SigmaDGP}.
Furthermore, let us assume that the first alternative is the case, and assume that $\fa{n \in \NN}{a_n = 0}$. Also assume that there exists $n \in \NN$ such that $b_n = 1$; and without loss of generality $n$ is the minimal such number. Then $\alpha^\prime_{2n} = 1$, contradicting our assumption. Together we have that in case the first alternative holds we have \[ \fa{n \in \NN}{a_n = 0} \implies \fa{n \in \NN}{b_n = 0} \ .\] Similarly, in case the second alternative holds we can show that \[ \fa{n \in \NN}{b_n = 0} \implies \fa{n \in \NN}{a_n = 0} \ .\]
That means that .

Conversely let $a_n$ be a binary sequence with at most one term equal to $1$. By $\Pi_1^0$-\DGP we get that 
\[ \left( \fa{n \in \NN}{a_{2n}=0} \implies \fa{n \in \NN}{a_{2n+1}=0}\right)  \lor \left( \fa{n \in \NN}{a_{2n+1}=0} \implies \fa{n \in \NN}{a_{2n}=0}\right) \ . \]
In the first case assume that there is $n$ such that $a_{2n+1} = 1$. Then, since there is at most one term equal to $1$ we have that $\fa{n \in \NN}{a_{2n} = 0}$, which together with the case we are in means that $\fa{n \in \NN}{a_{2n+1}=0}$; a contradiction. Hence in that case we must actually have $\fa{n \in \NN}{a_{2n+1}=0}$. Similarly in case the second alternative holds we can show that $\fa{n \in \NN}{a_{2n}=0}$. Thus  $\Pi_1^0$-\DGP implies \LLPO.

Interestingly, the mixed case is equivalent to \WLPO: let $a_n$ and $b_n$ be binary sequences (the sequence $b_n$ actually plays no role in what follows). By \WLPO either $\fa{n \in \NN}{a_n =0}$ or $\lnot \fa{n \in \NN}{a_n =0}$. In the first case, also $\ex{n \in \NN}{b_n =1} \implies \fa{n \in \NN}{a_n =0}$. In the second case $\fa{n \in \NN}{a_n =0} \implies \ex{n \in \NN}{b_n =1}$, since the antecedent contradicts our assumption. 

Conversely, let $a_n$ be an arbitrary binary sequence, and apply $\Sigma_1^0-\Pi_1^0$-\DGP to this sequence (and itself). This yields
\[  \ex{n \in \NN}{a_n =1} \implies \fa{n \in \NN}{a_n =0} \lor  \fa{n \in \NN}{a_n =0} \implies \ex{n \in \NN}{a_n =1}  \ . \] 
In case the first alternative holds we must have $a_n =0$ for all $n \in \NN$, since the assumption that there is $n$ with $a_n=1$ leads to the contradiction that $a_n =0$ for that same $n$. In case the second alternative holds the assumption that $\fa{n \in \NN}{a_n =0}$ leads to a contradiction, hence in that case $\lnot \fa{n \in \NN}{a_n =0}$.
\end{proof}

\begin{Qu}
Which principles are $\Sigma_1^0-\Sigma_1^0$--$\textrm{PP}$, $\Pi_1^0-\Pi_1^0$--$\textrm{PP}$, $\Pi_1^0-\Sigma_1^0$--$\textrm{PP}$, and $\Sigma_1^0-\Pi_1^0$--$\textrm{PP}$ equivalent to, where $\textrm{PP}$ is Peirce's law, as in Proposition \ref{Pro:paradoxes_of_material_impl}.\ref{LEMequiv1}? (It is easy to see that the $\Sigma_1^0$ and $\Pi_1^0$ versions of the special case $(\lnot \varphi \implies \varphi) \implies \varphi$ of Peirce's law is equivalent to \MP and provable in intuitionistic logic respectively.)
\end{Qu}

\section{\texorpdfstring{\nWLPO and \nLPO}{}} \label{Sec:nLPO}

Interestingly enough, not just $\WLPO$ has notable equivalences, also its negation does. The following is, just as its positive version Proposition \ref{Pro:WLPO-Equiv-disc}, very much folklore. It is worth pointing out that as a principle ``all functions are non-discontinuous'' is a very stable property under changes in the underlying spaces.
\begin{Pro} \label{Pro:nWLPO_equiv_non-disc}The following are equivalent to \nWLPO.
\begin{enumerate}
\item \label{non-dis-3} Every mapping $f:\CS \to \menge{0,1}$ is non-discontinuous.
\item \label{non-dis-2} Every mapping $f:\BS \to \NN$ is non-discontinuous. 
\item \label{non-dis-4} Every mapping $f:[0,1] \to \RR$ is non-discontinuous.
\item \label{non-dis-1} Every mapping of a complete metric space into a metric space is non-discontinuous.
\end{enumerate}
\end{Pro}
\begin{proof}
This is simply by taking the contrapositive of Proposition \ref{Pro:WLPO-Equiv-disc}, and by using the lemma following that proposition.
\end{proof}

The next proposition relies on a property proposed by M.\ Escard\'{o}. A set $X$ is called \define{searchable},\footnote{Notice that Escard\'{o}'s definition is slightly different, but is equivalent to this one in the presence of unique choice.} if we can decide any predicate on it. That is if for any $p:X \to 2$ 
\[ \fa{x\in X}{p(x)=0} \lor  \ex{x\in X}{p(x)=1} \ . \] 
\begin{Pro}
The following are equivalent:
\begin{enumerate}
\item $\nLPO$
\item Every searchable subset of $\NN$ is not unbounded.
\item Every searchable subset of $\NN$ is bounded.
\end{enumerate}
\end{Pro}
\begin{proof}
Assume $\nLPO$ and let $S \subset \NN$ be searchable. Assume furthermore that $S$ is unbounded, that is, with a bit of work, there exists an bijection $s: \NN \to S$. Then, if $a_n$ is an arbitrary binary sequence consider $p: S \to 2$ defined by
\[ p(n) = a_{s^{-1}(n)}  \ . \]
Since $S$ is searchable we can therefore decide, whether $\fa{n \in \NN}{a_n =0}$ or whether $\ex{n \in \NN}{a_n = 1}$. In other words \LPO holds; a contradiction and hence $S$ is not unbounded.
\end{proof}

Even though $\nLPO$ and $\nWLPO$ do not seem to have many natural equivalences they extend the recursive hierarchy nicely.

\[ \SinC \implies \KT \implies \iSS \implies \SS \implies \lnot \WLPO \implies \lnot \LPO \]

\appendix
\chapter{List of Open Questions}
\setcounter{Que}{0}
\includecollection{qu}

\chapter{Source Code}
\section{Topological Models} \label{Appendix:sourcecode}
The Python\footnote{Version 2.7 or 3.x} program used to check all possibilities in the proof of Proposition \ref{Pro:topmod_biimp} is

\inputminted[linenos,
               frame=lines,
               framesep=2mm, breaklines]{python}{topology.py}

\bibliographystyle{abbrv}
\bibliography{All}

\printindex

\end{document}